\documentclass[12pt]{article}
\usepackage{mathrsfs}
\usepackage{amsmath,amssymb, amsthm, amsopn,
  amsfonts,epsfig,graphics,relsize,exscale}{}
\usepackage{hyperref}
\hypersetup{
    colorlinks=true,
    linkcolor=red,
    filecolor=magenta,      
    urlcolor=cyan,
}
\usepackage[all]{xy}
\usepackage{graphicx}
\usepackage[T1]{fontenc}
\usepackage[latin1]{inputenc}
\usepackage{color,tocvsec2}
\usepackage{typearea}
\usepackage{fouriernc}
\usepackage{enumerate}
\usepackage{lscape}
\usepackage[normalem]{ulem}
\usepackage{multicol}        

\usepackage{array}
\definecolor{rd}{rgb}{1,0.3,0.35}
\newcommand{\red}{}

\usepackage[a4paper,width=14.66cm,top=2.52cm,bottom=2.52cm]{geometry}
\usepackage{makeidx}
\makeindex

\newcommand{\field}[1]{\mathbb{#1}} \newcommand{\rz}{\field{R}}
\newcommand{\cz}{\field{C}} \newcommand{\nz}{\field{N}}
\newcommand{\zz}{\field{Z}}

\newcommand{\Id}{{\textrm{Id}}}
\newcommand{\Imag}{{\textrm{Im\,}}}

\DeclareMathOperator\Real{\rz e\,}
\DeclareMathOperator\supp{\textrm{supp}}

\newcommand{\ccap}{\mathop{\cap}}
\newcommand{\ccup}{\mathop{\cup}}
\newcommand{\fkh}{\mathfrak{h}}
\newcommand{\fkf}{\mathfrak{f}}
\newcommand{\fkF}{\mathfrak{F}}
\newtheorem{theorem}{Theorem}[section]
\newtheorem{lemma}[theorem]{Lemma}

\newtheorem{proposition}[theorem]{Proposition}
\newtheorem{definition}[theorem]{Definition}
\newtheorem{remark}[theorem]{Remark}
\newtheorem{corollary}[theorem]{Corollary}

\def\dim{\textrm{dim\;}}

\title{Bismut  hypoelliptic Laplacians for  manifolds with boundaries.
}
\author{
  F.~Nier
\thanks{LAGA, Universit{\'e} de Sorbonne-Paris Nord (Paris XIII), 99 avenue
J.B.~Cl{\'e}ment, F-93430~Villetaneuse. nier@math.univ-paris13.fr}\\
S.~Shen \thanks{Institut de Math{\'e}matiques de Jussieu-Paris Rive
  Gauche, Sorbonne Universit{\'e}, Case Courrier 247, 4 place Jussieu,
  F-75252 Paris~Cedex~05, shu.shen@imj-prg.fr }
}
\begin{document}
\bibliographystyle{plain}
\maketitle{}
\begin{abstract}
Boundary conditions for Bismut's hypoelliptic Laplacian which
naturally correspond to Dirichlet and Neumann boundary conditions for
Hodge Laplacians are considered. Those are related with specific
boundary conditions for the differential and its various
adjoints. Once the closed realizations of those operators are well
understood, the commutation of the differential with the resolvent of
the hypoelliptic Laplacian is checked with other properties like the
PT-symmetry, which are important for the spectral analysis.
\end{abstract}
\noindent\textbf{MSC2020: 35G15, 35H10, 35H20, 35R01, 47D06, 58J32, 58J50, 58J65,
 60J65}\\
\noindent\textbf{Keywords:} 
Bismut's hypoelliptic Laplacian, boundary value problem, subelliptic
estimates, commutation of unbounded operators.
\\
\tableofcontents{}
\section{Introduction}
\label{sec:intro}
This text is devoted to boundary conditions which extend naturally to
Bismut's hypoelliptic Laplacians, Dirichlet and Neumann's boundary
conditions for Hodge and Witten Laplacians. Actually such boundary
conditions were proposed in \cite{Nie} and the functional analysis was
carried out, essentially relying on the scalar principal parts while
neglecting the complicated lower order terms related with curvature
tensors. While considering the commutation of the resolvent of this operator
with the differential, with the suitable boundary conditions,  a good understanding of the geometrical content of the whole
operators cannot be
skipped. This article answers one of the questions asked at the end
of \cite{Nie} about the proper boundary conditions for the
differential and Bismut's codifferential, which ensure the commutation
with the resolvent of Bismut's hypoelliptic Laplacian. A motivation
for this comes from the accurate spectral analysis of the low-lying
spectrum which is now well understood for Witten Laplacian in the low
temperature limits, for possibly non Morse potential functions. An
instrumental tool in \cite{LNV1}\cite{LNV2}, was the introduction of artificial Dirichlet and Neumann
realizations of Witten Laplacians, as a localization technique. 
Although it is a first step, ignoring for the moment the asymptotic
analysis issues, the results of this article pave
the way to such an analysis for Bismut's hypoelliptic Laplacian.\\
This introduction rapidly presents the objects, the main results, some
notations and conventions. It ends by pointing out the problems to be
solved and the followed strategy. The outline of the article is given
in this last paragraph.

\subsection{Bismut hypoelliptic Laplacian}
\label{sec:bismhyp}
When $Q$ is a closed (compact) riemannian manifold endowed with the metric
$g=g^{TQ}=g_{ij}(\underline{q})d\underline{q}^{i}d\underline{q}^{j}$\,,
the hypoelliptic Laplacian introduced by J-M.~Bismut in
\cite{Bis04}\cite{Bis05} is a Hodge type Laplacian defined on the
cotangent space $X=T^{*}Q$\,, $\pi_{X}:X=T^{*}Q\to Q$\,, of which we briefly
recall the construction here.
Details will be given in the text and
may be found in \cite{Bis04}\cite{Bis04-2}\cite{Bis05}\cite{BiLe}\cite{Leb1}\cite{Leb2}.\\
The Levi-Civita connection associated with $g^{TQ}$ gives rise to the
decomposition into horizontal and vertical space
$$
TX=\underbrace{TX^{H}}_{\simeq \pi_{X}^{*}(TQ)}\oplus
\underbrace{TX^{V}}_{\simeq \pi_{X}^{*}(T^{*}Q)}\,,
$$
and by duality to 
$$
T^{*}X=\underbrace{T^{*}X^{H}}_{\simeq \pi_{X}^{*}( 
T^{*}Q)}\oplus \underbrace{T^{*}X^{V}}_{\pi_{X}^{*}(TQ)}\,.
$$
By using the dual metric
$g^{T^{*}Q}=g^{ij}(\underline{q})\frac{\partial}{\partial
  \underline{q}^{i}}\frac{\partial}{\partial\underline{q}^{j}}$\,,
one may define metrics on $TX$ and $T^{*}X$ which make the
horizontal vertical decomposition orthogonal. 
By tensorization this
provides metrics on the vector bundles $E'=\Lambda TX$ and $E=\Lambda
T^{*}X$\,.
 It is convenient to introduce from the begining a
$p$-dependent weight where an element of
$X=T^{*}Q$ is locally written $x=(q,p)$\,, by setting
$$
|p|_{q}=\sqrt{g_{q}^{T^{*}Q}(p,p)}=\sqrt{g^{ij}(q)p_{i}p_{j}}\quad,\quad
\langle p\rangle_{q}=\sqrt{1+g_{q}^{T^{*}Q}(p,p)}=\sqrt{1+g^{ij}(q)p_{i}p_{j}}\,.
$$
The metric $g^{E}$ and $g^{E'}$\,, dual to each other, are given by
$$
g^{E}=\langle p\rangle_{q}^{N_{V}-N_{H}}\pi_{X}^{*}(g^{\Lambda
T^{*}Q}\otimes g^{\Lambda TQ})\quad,\quad g^{E'}=\langle p\rangle_{q}^{N_{H}-N_{V}}\pi_{X}^{*}(g^{\Lambda
TQ}\otimes g^{\Lambda T^{*}Q})\,,
$$
where $N_{V}$ and $N_{H}$ are the vertical and horizontal number operators.
The volume associated with $g^{E}$ coincides with the symplectic volume on
$X$ and it is simply written $dv_{X}$\,.\\
Additionally we may add a hermitian vector bundle structure by
starting from  $(\fkf, \nabla^{\frak{f}}\,,
g^{\fkf})$ where $\pi_{\fkf}:\fkf\to Q$ is  a
complex vector bundle, endowed with a flat connection
$\nabla^{\fkf}$\,, and $g^{\fkf}$
is a hermitian metric. While identifying the anti-dual vector bundle
with $\fkf$ via the metric $g^{\fkf}$\,, the anti-dual metric remains
$g^{\fkf}$ but the anti-dual flat connection differ and will be
denoted by $\nabla^{\fkf'}$\,.
By setting
$$
\omega(\nabla^{\fkf},g^{\fkf})=(g^{\fkf})^{-1}\nabla^{\fkf}g^{\fkf}\in
\mathcal{C}^{\infty}(Q;T^{*}Q\otimes L(\fkf))
$$
we get
$$
\nabla^{\fkf'}=\nabla^{\fkf}+\omega(\nabla^{\fkf},g^{\fkf})
$$
and we can introduce the unitary connection on $(\fkf,g^{\fkf})$
$$
\nabla^{\fkf,u}=\nabla^{\fkf}+\frac{1}{2}\omega(\nabla^{\fkf},g^{\fkf})\,.
$$
Simple examples of such vector bundles on $Q$ are given by
\begin{itemize}
\item $\fkf=Q\times \cz$\,, $\nabla^{\fkf}$ is the trivial  connection
  $\nabla$\,,
$g^{\fkf}(z)=e^{-2V(q)}|z|^{2}$\,,
$\omega(\nabla^{\fkf},g^{\fkf})=-2dV(q)$\,, where $V\in
\mathcal{C}^{\infty}(Q;\rz)$ is a potential function,
\item the
orientation bundle $\fkf=(\mathrm{or}_{Q}\times\cz)/\zz_{2}$ where
$\mathrm{or}_{Q}\to Q$ is the orientation double cover of $Q$\,. The
trivial connection and trivial metric on $\mathrm{or}_{Q}\times \cz$
induce the connection $\nabla^{\fkf}$ and $g^{\fkf}$\,.
\end{itemize}
The total vector   bundles $F=E\otimes \pi_{X}^{*}(\fkf)
=\pi_{X}^{*}(\Lambda T^{*}Q\otimes \Lambda
TQ\otimes \fkf)$
(resp. $F'=E'\otimes\pi^{*}_{X}(\fkf)$) endowed with the metrics
\begin{eqnarray*}
  && g^{F}=g^{E}\otimes \pi_{X}^{*}g^{\fkf}
=\langle p\rangle_{q}^{N_{V}-N_{H}}\pi^{*}_{X}(g^{\Lambda
     T^{*}Q}\otimes g^{\Lambda TQ}\otimes g^{\fkf})\\
\text{resp.}&&
g^{F'}=g^{E'}\otimes \pi_{X}^{*}g^{\fkf}=
\langle p\rangle_{q}^{N_{H}-N_{V}}\pi^{*}_{X}(
g^{\Lambda  TQ}
\otimes g^{\Lambda T^{*}Q}\otimes g^{\fkf})\,.
\end{eqnarray*}
When $\nabla^{Q,g}$ is connection on $\Lambda
T^{*}Q\otimes\Lambda TQ$ induced by the Levi-Civita connection on $TQ$\,, and after the identification of 
$TX=TX^{H}\oplus TX^{V}$ with $\pi^{*}_{X}(TQ\oplus T^{*}Q)$\,, the
connection $\nabla^{E,g}$ on $E$ is nothing but
$\pi^{*}_{X}(\nabla^{Q,g})$\,. When $\nabla^{\fkf}$\,,
$\nabla^{\fkf'}=\nabla^{\fkf}+\omega(\nabla^{\fkf},g^{\fkf})$
and
$\nabla^{\fkf,u}=\nabla^{f}+\frac{1}{2}\omega(\nabla^{\fkf},g^{\fkf})$ 
are the three above connections on
$\fkf$\,, the connections on $F$ and $F'$ are given by
\begin{eqnarray*}
  && \nabla^{F,g}=\pi^{*}_{X}(\nabla^{Q,g}+\nabla^{\fkf})\\
&&\nabla^{F',g}=\pi^{*}_{X}(\nabla^{Q,g}+\nabla^{\fkf'})\\
&&\nabla^{F,g,u}=\pi^{*}_{X}(\nabla^{Q,g}+\nabla^{\fkf,u})\,.
\end{eqnarray*}
When there is no ambiguity with a fixed metric $g^{TQ}$\,, the exponent
${}^{g}$ will be dropped and the above connections will be simply
written $\nabla^{E}$\,, $\nabla^{F}$\,, $\nabla^{F'}$\,,
$\nabla^{F,u}$\,.\\
The $L^{2}(X;F)$  (resp. $L^{2}(X;F')$) space is the space of
$L^{2}$-sections for the Hilbert scalar product
\begin{eqnarray*}
  &&\langle s\,,\,s'\rangle_{L^{2}}=\int_{X}\langle s(x)\,,\,
     s'(x)\rangle_{g^{F}}~dv_{X}(x)\\
\text{resp.}&& \langle t\,,\,t'\rangle_{L^{2}}=\int_{X}\langle t(x)\,,\,t'(x)\rangle_{g^{F'}}~dv_{X}(x)\,.
\end{eqnarray*}
From those structures one could define a Hodge Laplacian for sections
of $F$\,. However this is not the way to introduce
Bismut's hypoelliptic Laplacian. Instead one works with the non
degenerate bilinear form $\phi_{b}$ on $TX$ combining the metric and
the symplectic form $\sigma=d\theta$\,, where $\theta=p_{i}dq^{i}$ is the
tautological one-form on $X$\,, with $X=T^{*}Q$\,. For $b\neq 0$\,, the isomorphism
$\phi_{b}:TX\to T^{*}X$ extended to $\phi_{b}:E'=\Lambda TX\to
E=\Lambda T^{*}X$ is the one given by the non degenerate form
$$
\eta_{\phi_{b}}(U,V)=g^{TQ}(\pi_{X,*}(U),\pi_{X,*}(V))+b\sigma(U,V)=U.\phi_{b}V\,,\quad
U,V\in TX\,.
$$
When $(\underline{e}_{1},\ldots,\underline{e}_{d})$ is a local  frame
of $TQ$\,, with the dual frame
$(\underline{e}^{1},\ldots,\underline{e}^{d})$ in $T^{*}Q$\,, we
denote by
$e_{i}\in TX^{H}$ and 
$\hat{e}^{j}\in TX^{V}$\,, the corresponding vectors obtained via
$TX^{H}\simeq \pi_{X}^{*}(TQ)$ and $TX^{V}\simeq \pi_{X}^{*}(X^{V})$
and we notice that
$(e_{i},\hat{e}^{j})$ is a symplectic frame of $TX$\,. The dual basis
is  $e^{i}\in T^{*}X^{H}$ and
$\hat{e}_{j}\in T^{*}X^{V}$\,. Then the
matrix of $\eta_{\phi_{b}}$ or $\phi_{b}:TX\to T^{*}X$ in those bases is
given by 
$$
\phi_{b}=
\begin{pmatrix}
 g^{TQ}&-b\Id\\
b\Id&0
\end{pmatrix}\,.
$$
The dual bilinear form on $\Lambda T^{*}X$ is then given by
$$
\eta_{\phi_{b}}^{*}(\omega,\theta)=(\phi^{-1}_{b}\omega).\theta\quad,\quad\omega,\theta\in
\Lambda T^{*}X\,.
$$
Where the matrix of $\eta_{\phi_{b}}^{*}\big|_{T^{*}X\times T^{*}X}$ or of
${}^{t}\phi_{b}^{-1}:T^{*}X\to TX$ equals
$$
{}^{t}\phi_{b}^{-1}=
\begin{pmatrix}
  0&-b^{-1}\Id\\
b^{-1}\Id&b^{-2}g^{TQ}
\end{pmatrix}
\,.
$$
The tensorization with $\fkf$ is done by writing
$\eta^{*}_{\phi_{b},\fkf}=\eta^{*}_{\phi_{b}}\otimes
\pi^{*}_{X}(g^{\fkf})$ on
$F=E\otimes\pi^{*}_{X}(\fkf)$ so that
the non-degenerate sesquilinear form on sections of
$F=E\otimes\pi^{*}_{X}(\fkf)$ is given by
$$
\langle s\,,\,s'\rangle_{\phi_{b}}=\int_{X}\eta^{*}_{\phi_{b},\fkf}(s(x),s'(x))~dv_{X}(x)\,.
$$
By
introducing the kinetic energy
$$
\fkh(q,p)=\frac{|p|_{q}^{2}}{2}=\frac{g^{ij}(q)p_{i}p_{j}}{2}
$$
and the deformed differential
$$
d_{\fkh}=e^{-\fkh}de^{\fkh}=d+d\fkh\wedge
$$
Bismut's codifferential $d^{\phi_{b}}_{\fkh}$ is the formal
adjoint of the differential $d_{\fkh}$ for the above  duality product
$\langle~,~\rangle_{\phi_{b}}$\,.
Bismut's hypoelliptic Laplacian is nothing but
$$
B^{\phi_{b}}_{\fkh}=\frac{1}{4}(d^{\phi_{b}}_{\fkh}+d_{\fkh})^{2}=\frac{1}{4}(d^{\phi_{b}}_{\fkh}d_{\fkh}+d_{\fkh}d^{\phi_{b}}_{\fkh})\,.
$$
In \cite{Bis05} Bismut proved the following Weitzenbock formula
\begin{eqnarray*}
B^{\phi_{b}}_{\fkh}&=&\frac{1}{4b^{2}}
   \left[-\Delta^{V}+|p|_{q}^{2}-\frac{1}{2}\langle
   R^{TQ}(e_{i},e_{j})e_{k}\,,\,e_{\ell}\rangle
                       e^{i}e^{j}\mathbf{i}_{\hat{e}^{k}\hat{e}^{\ell}}+2N_{V}-\mathrm{dim}~Q\right]\\
&&\hspace{2cm}
-\frac{1}{2b}\Big[
\mathcal{L}_{Y_{\fkh}}+\frac{1}{2}\omega(\nabla^{\fkf},g^{\fkf})(Y_{\fkh})+\frac{1}{2}e^{i}\mathbf{i}_{\hat{e}^{j}}\nabla^{F}_{e_{i}}\omega(\nabla^{\fkf},g^{\fkf})(e_{j})
\\
&&
\hspace{5cm}
+\frac{1}{2}\omega(\nabla^{\fkf},g^{\fkf})(e_{i})\nabla^{F}_{\hat{e}^{i}}
\Big]
\end{eqnarray*}
where $\Delta^{V}$ is the vertical Laplacian and
$\mathcal{O}=\frac{-\Delta^{V}+|p|_{q}^{2}}{2}$ the scalar harmonic
oscillator in the vertical direction while
$Y_{\fkh}=g^{ij}(q)p_{i}e_{i}$ is the Hamilton vector field for the
kinetic energy $\fkh=\frac{|p|_{q}^{2}}{2}$\,.
The other terms, which involve various curvatures, like 
$R^{TQ}$ the riemann curvature tensor of
$g^{TQ}$ pulled-back by $\pi^{*}_{X}$\,, and
$\omega(\nabla^{\fkf},g^{\fkf})$\,, are actually lower order terms
controlled in the analysis by the scalar principal part.\\
For the analysis of those operators G.~Lebeau introduced in
\cite{Leb1} the scale of Sobolev spaces
$\mathcal{W}^{\mu}(X;E)$\,, $\mu\in\rz$\,, 
modelled on the fact that
horizontal derivations $\nabla^{F,u}_{e_{i}}$
{\red\,,
$\frac{\partial}{\partial q^{i}}$\,, and weighted vertical derivations
$\langle p\rangle_{q} \nabla^{F,u}_{\hat{e}^{j}}$\,,
$\langle p\rangle_{q}\frac{\partial}{\partial p_{j}}$ are of order $1$\,,
while the  multilplications  $p_{j}\times$ and $\langle
p\rangle_{q}\times$ are of order
$\frac{1}{2}$\,, with $\mathcal{W}^{0}(X;F)=L^{2}(X;F)$\,.}
 They satisfy $\ccap_{\mu\in\rz}\mathcal{W}^{\mu}(X;F)=S(X;F)$ the space
 of $\mathcal{C}^{\infty}$ vertically rapidly decaying sections while
 $\ccup_{\mu\in\rz}\mathcal{W}^{\mu}(X;F)=S'(X;F)$ is the space of
 tempered distributional sections.
The maximal subelliptic estimates were proved in \cite{Leb2}. They
say the following things:
\begin{itemize}
\item Given $b\neq 0$\,, there exist a constant $C_{b}>0$ and for any $\mu\in \rz$ a
  constant $C_{b,\mu}>0$ such that
$$
\hspace{-1cm}\| \mathcal{O} s\|_{\mathcal{W}^{\mu}}+\|
\nabla^{F,u}_{Y_{\fkh}}s\|_{W^{\mu}}+\delta_{0,\mu}\langle
\lambda\rangle^{1/2}\|s\|_{\mathcal{W}^{\mu}}+\|s\|_{\mathcal{W}^{\mu+2/3}}\leq C_{b,\mu}\|(C_{b}+B^{\phi_{b}}_{\fkh}-i\delta_{0,\mu}\lambda)s\|_{\mathcal{W}^{\mu}}
$$
for all $s\in S'(X;F)$ and all $\lambda\in\rz$ such that
$(C_{b}+B^{\phi_{b}}_{\fkh}-i\delta_{0,\mu}\lambda)s\in \mathcal{W}^{\mu}(X;F)$\,.
\item The operator $C_{b}+B^{\phi_{b}}_{\fkh}$ with domain
  $D(B^{\phi_{b}}_{\fkh})=\left\{s\in L^{2}(X;F)\,,
    B^{\phi_{b}}_{\fkh}s\in L^{2}(X;F)\right\}$ is maximal accretive
  in $L^{2}(X;F)$\,.
\end{itemize}
The reason for the factor $\delta_{0,\mu}$ which says that the
$\lambda$-dependent case holds here only for $\mu=0$ is due to the
fact that the more general $\lambda$-dependent estimates of \cite{Leb2} require
$\lambda$-dependent $\mathcal{W}^{\mu}$-norms for $\mu\neq 0$ which
won't be  considered
in this text.
\subsection{Boundary conditions and results}
\label{sec:bdyres}
Let $\overline{Q}_{-}=Q_{-}\sqcup Q'$\,, $Q'=\partial Q_{-}$\,, be
a compact riemannian manifold with boundary.  The cotangent bundle is
a manifold with boundary:
$\overline{X}_{-}=X_{-}\sqcup X'$\,, $X_{-}=T^{*}Q_{-}$\,,
$X'=\partial X_{-}=T^{*}Q\big|_{\partial Q_{-}}$\,.
A collar neighborhood $Q_{(-\varepsilon,0]}$ of $Q'$ can be chosen such
that $Q_{(-\varepsilon,0]}\simeq (-\varepsilon,0]\times Q'$ and the
metric $g^{TQ}$ equals
$(d\underline{q}^{1})^{2}+m^{TQ'}(\underline{q}^{1},\underline{q'})$ with
$m^{TQ'}(\underline{q}^{1},\underline{q}')$ is a
$\underline{q}^{1}$-dependent metric on $Q'$\,.
Although the following constructions will be checked to make sense
geometrically, we follow here the shortest presentation in terms of
coordinates.
The decomposition
$$
TQ_{(-\varepsilon,0]}=\rz\frac{\partial}{\partial
  \underline{q}^{1}}\mathop{\oplus}^{\perp}TQ'
\quad 
T^{*}Q_{(-\varepsilon,0]}=\rz d\underline{q}_{1}\mathop{\oplus}^{\perp}T^{*}Q'\,,
$$
provides  coordinates $x=(q^{1},q',p_{1},p')$ where
$(q',p')=(q^{i'},p_{j'})_{2\leq i',j'\leq d}$ are local
coordinates on $T^{*}Q'$\,, $q^{1}=\underline{q}^{1}(\pi_{X}(x))$ and
$p_{1}=\frac{\partial}{\partial \underline{q}^{1}}.p$\,.
Take $\underline{e}_{i}=\frac{\partial}{\partial\underline{q}^{i}}$
and $\underline{e}^{i}=d\underline{q}^{i}$\,. The construction given
in Subsection~\ref{sec:bismhyp}, relying on $TX=TX^{H}\oplus TX^{V}\simeq
\pi_{X}^{*}(TQ\oplus T^{*}Q)$\,, provides the basis
$(e_{i},\hat{e}^{j})_{1\leq i,j\leq \dim Q}$ of $TX$ and the dual
basis 
$(e^{i},\hat{e}_{j})_{1\leq i,j\leq\dim Q}$ of $T^{*}X$\,.\\
A section $s$ of $E=\Lambda T^{*}X$ (or
$F=E\otimes\pi^{*}_{X}(\fkf)$) reads in those coordinates
$$
s=s_{I}^{J}(q^{1},q',p_{1},p')e^{I}\hat{e}_{J}\,.
$$
When $s$ is a regular enough section to admit the following traces we 
consider
\begin{eqnarray*}
  && s\big|_{X'}=s_{I}^{J}(0,q',p_{1},p')e^{I}\hat{e}_{J}\,,\\
&& \mathbf{i}_{e_{1}}e^{1}\wedge s\big|_{X'}=s_{I'}^{J}(0,q',p_{1},p')e^{I'}\hat{e}_{J}\,,\\
&& \hat{e}_{1}\wedge \mathbf{i}_{\hat{e}^{1}}s\big|_{X'}=
s_{I}^{\left\{1\right\}\cup
   J'}(0,q',p_{1},p')e^{I}\hat{e}_{\left\{1\right\}\cup J'}\,.
\end{eqnarray*}
Note that $s\big|_{X'}$ is not the classical pull back to $X'$\,.\\
When $\nu$ is a flat unitary involution of $\fkf\big|_{Q'}$\,,  $\nu\in \mathcal{C}^{\infty}(Q';L(\fkf\big|_{Q'}))$ 
such that
  its covariant derivative along $Q'$ vanishes
$\nabla^{L(\fkf|_{Q'})}\nu=0$ 
(the simplest and essential  example being $\nu=\pm
\Id_{\fkf}$)\,, the transformation $\hat{S}_{\nu}$ acting on sections of
$F\big|_{X'}$ is given by
$$
\hat{S}_{\nu}(s_{I}^{J}(0,q',p_{1},p')e^{I}\hat{e}_{J})=(-1)^{\left|\left\{1\right\}\cap
  I\right|+\left|\left\{1\right\}\cap J\right|}\nu
s_{I}^{J}(0,q^{1},-p_{1},p')e^{I}\hat{e}_{J}\,,
$$
where we use the same notation $\nu$ for the pulled-backed unitary map
$\pi^{*}_{X'}(\nu)$\,. Note that  $(\hat{S}_{\nu})^{2}=\mathrm{Id}$
and $\frac{1-\hat{S}_{\nu}}{2}$ is a projection.\\
We shall prove the following results for closed realizations of the
differential operators $P=d_{\fkh}$\,, $P=d^{\phi_{b}}_{\fkh}$ and $P=B^{\phi_{b}}_{\fkh}$\,, where an important step consists
in proving trace theorems for sections $s\in L^{2}(X_{-};F)$ such that
$Ps\in L^{2}(X_{-};F)$\,, so  that the definition of the domains makes
sense.
\begin{theorem}
\label{th:dg} The operator
$(\overline{d}_{g,\fkh},D(\overline{d}_{g,\fkh}))$ in $L^{2}(X_{-};F)$
defined by
\begin{eqnarray*}
  &&D(\overline{d}_{g,\fkh})=\left\{s\in L^{2}(X_{-};F)\,,\quad
     d_{\fkh}s\in L^{2}(X_{-};F)\,,\quad
     \frac{1-\hat{S}_{\nu}}{2}\mathbf{i}_{e_{1}}e^{1}\wedge
     s\big|_{X'}=0 \right\}\\
&& \forall s\in D(\overline{d}_{g,\fkh})\,,\quad \overline{d}_{g,\fkh}s=d_{\fkh}s\,,
\end{eqnarray*}
is closed and satisfies $\overline{d}_{g,\fkh}\circ
\overline{d}_{g,\fkh}=0$\,.\\
The set $\mathcal{C}^{\infty}_{0}(\overline{X}_{-};F)\cap
D(\overline{d}_{g,\fkh})$ is a core for $\overline{d}_{g,\fkh}$\,.
\end{theorem}
\begin{theorem}
\label{th:dphibg} The operator
$(\overline{d}^{\phi_{b}}_{g,\fkh},D(\overline{d}^{\phi_{b}}_{g,\fkh}))$ in $L^{2}(X_{-};F)$
defined by
\begin{eqnarray*}
  &&D(\overline{d}^{\phi_{b}}_{g,\fkh})=\left\{s\in L^{2}(X_{-};F)\,,\quad
     d^{\phi_{b}}_{\fkh}s\in L^{2}(X_{-};F)\,,\quad
     \frac{1-\hat{S}_{\nu}}{2}\hat{e}_{1}\wedge\mathbf{i}_{\hat{e}^{1}}
     s\big|_{X'}=0 \right\}\\
&& \forall s\in D(\overline{d}^{\phi_{b}}_{g,\fkh})\,,\quad \overline{d}^{\phi_{b}}_{g,\fkh}s=d^{\phi_{b}}_{\fkh}s\,,
\end{eqnarray*}
is closed and satisfies $\overline{d}^{\phi_{b}}_{g,\fkh}\circ
\overline{d}^{\phi_{b}}_{g,\fkh}=0$\,.\\
The set $\mathcal{C}^{\infty}_{0}(\overline{X}_{-};F)\cap
D(\overline{d}^{\phi_{b}}_{g,\fkh})$ is a core for $\overline{d}^{\phi_{b}}_{g,\fkh}$\,.
\end{theorem}
\begin{theorem}
\label{th:Bphibg} The operator $(\overline{B}^{\phi_{b}}_{g,\fkh},
D(\overline{B}^{\phi_{b}}_{g,\fkh}))$ defined in $L^{2}(X_{-};F)$ by
\begin{eqnarray*}
  && D(\overline{B}^{\phi_{b}}_{g,\fkh})=\left\{s\in
     L^{2}(X_{-};F)\,,\quad \nabla^{F}_{\frac{\partial}{\partial p}}s~\text{and}~B^{\phi_{b}}_{\fkh}s\in L^{2}(X_{-};F)\,,\quad \frac{1-\hat{S}_{\nu}}{2}s\big|_{X'}=0
\right\}\\
&&\forall s\in D(\overline{B}^{\phi_{b}}_{g,\fkh})\,,\quad \overline{B}^{\phi_{b}}_{g,\fkh}s=B^{\phi_{b}}_{\fkh}s\,,
\end{eqnarray*}
is closed and there exists a constant $C_{b}\in\rz$ such that
$C_{b}+\overline{B}^{\phi_{b}}_{g,\fkh}$ is maximal accretive.\\
The set $\mathcal{C}^{\infty}_{0}(\overline{X}_{-};F)\cap
D(\overline{B}^{\phi_{b}}_{g,\fkh})$ is a core for
$\overline{B}^{\phi_{b}}_{g,\fkf}$\,.\\
There exists a constant $C'_{b}>0$ such that the estimates
\begin{eqnarray*}
  &&
\left.
     \begin{array}[c]{l}
\|(1+\mathcal{O})^{1/2}s\|_{L^{2}}+\|s\|_{\mathcal{W}^{1/3}}\\
+\langle
     \lambda\rangle^{1/4}\|s\|_{L^{2}}+\|\langle
     p\rangle_{q}^{-1}s\big|_{X'}\|_{L^{2}(X',|p_{1}|dv_{X'})}
     \end{array}
\right\}
     \leq
     C'_{b}\|(1+C_{b}+\overline{B}^{\phi_{b}}_{g,\fkf}-i\lambda)s\|_{L^{2}}\,,\\
&& \|(1+\mathcal{O})^{1/2}s\|_{L^{2}}\leq C_{b}'\Real \langle s\,,\, (1+C_{b}+\overline{B}^{\phi_{b}}_{g,\fkf})s\rangle
\end{eqnarray*}
hold for all $s\in D(\overline{B}^{\phi_{b}}_{g,\fkf})$ and all
$\lambda\in\rz$\,.\\
The semigroup $(e^{-t\overline{B}^{\phi_{b}}_{g,\fkh}})_{t\geq 0}$
preserves  $D(\overline{d}_{g,\fkh})$ and
$D(\overline{d}^{\phi_{b}}_{g,\fkh})$ with
\begin{eqnarray*}
  &&\forall s\in D(\overline{d}_{g,\fkh})\,, \forall t\geq 0\,,\quad
\overline{d}_{g,\fkh}e^{-t\overline{B}^{\phi_{b}}_{g,\fkh}}s=e^{-t\overline{B}^{\phi_{b}}_{g,\fkh}}\overline{d}_{g,\fkh}s\,,\\
&&
\forall s\in D(\overline{d}^{\phi_{b}}_{g,\fkh})\,, \forall t\geq 0\,,\quad
\overline{d}^{\phi_{b}}_{g,\fkh}e^{-t\overline{B}^{\phi_{b}}_{g,\fkh}}s=e^{-t\overline{B}^{\phi_{b}}_{g,\fkh}}\overline{d}^{\phi_{b}}_{g,\fkh}s\,.
\end{eqnarray*}
For all $z\in \cz\setminus
\textrm{Spec}(\overline{B}^{\phi_{b}}_{g,\fkh})$ the resolvent 
$(z-\overline{B}^{\phi_{b}}_{g,\fkh})^{-1}$ preserves 
$D(\overline{d}_{g,\fkh})$ and $D(\overline{d}^{\phi_{b}}_{g,\fkh})$ with
\begin{eqnarray*}
  &&\forall s\in D(\overline{d}_{g,\fkh})\,, \quad
\overline{d}_{g,\fkh}(z-\overline{B}^{\phi_{b}}_{g,\fkh})^{-1}s=(z-\overline{B}^{\phi_{b}}_{g,\fkh})^{-1}\overline{d}_{g,\fkh}s\,,\\
&&
\forall s\in D(\overline{d}^{\phi_{b}}_{g,\fkh})\,, \quad
\overline{d}^{\phi_{b}}_{g,\fkh}(z-\overline{B}^{\phi_{b}}_{g,\fkh})^{-1}s=(z-\overline{B}^{\phi_{b}}_{g,\fkh})^{-1}\overline{d}^{\phi_{b}}_{g,\fkh}s\,.
\end{eqnarray*}
\end{theorem}
\subsection{Some notations and conventions}
\label{sec:notconv}
Although we already introduced some notations, let us fix some
conventions and notations used throughout the article.\\

\noindent\textbf{Coordinates:}
Local coordinates systems on $Q$ or $\overline{Q}_{-}$ will be
underlined and written
$(\underline{q}^{1},\ldots,\underline{q}^{d})$\,, $d=\dim Q$\,. While
working in a neighborhood of $Q'=\partial Q_{-}$\,, they will be
chosen such that $g^{TQ}=(d\underline{q}^{1})^{2}+m^{TQ'}(\underline{q}^{1})$\,.\\
Primed exponents (or indices) like in  $\underline{q}^{i'}$ or
$d\underline{q}^{I'}=d\underline{q}^{i_{1}'}\wedge\ldots\wedge
d\underline{q}^{i'_{p}}$\,, mean that the value $1$ is excluded,
$i'\neq 1$ or $1\not\in I'$\,.\\
On $X=T^{*}Q$ local coordinates will be denoted
$(q^{1},\ldots,q^{d},p_{1},\ldots, p_{d})$ with
$q^{i}=\underline{q}^{i}(\pi_{X}(x))$ and
$p_{i}=\frac{\partial}{\partial \underline{q}^{i}}.p$\,, with the same convention
for primed exponents and indices. Different local coordinate systems
on $X$,
which will be specified later, will be used and then they will be
written
$(\tilde{q}^{1},\ldots,\tilde{q}^{d},\tilde{p}_{1},\ldots,\tilde{p}_{d})$\,.\\

\noindent\textbf{Local frames:} We use
$\underline{e}_{i}=\frac{\partial}{\partial \underline{q}^{i}}$ and
$\underline{e}^{i}=d\underline{q}^{i}$
and  the notations
$(e_{i},\hat{e}^{j})$ and $(e^{i},\hat{e}_{j})$ refer to the
associated frames of
$TX=TX^{H}\oplus TX^{V}\simeq \pi^{*}_{X}(TQ\oplus T^{*}Q)$
 and $T^{*}X=T^{*}X^{H}\oplus TX^{V}\simeq \pi^{*}_{X}(T^{*}Q\oplus TQ)$\,.
Similar frames constructed near $X'$ for the metric
$g_{0}^{TQ}=(d\underline{q}^{1})^{2}+m^{TQ'}(0)$ will  be denoted by
$(f_{i},\hat{f}^{j})$ and $(f^{i},\hat{f}_{j})$\,.  Those frames will
be abreviated by $(e,\hat{e})$ or $(f,\hat{f})$\,. Finally while using
a symmetry argument on the double copy $Q=Q_{-}\sqcup Q'\sqcup Q_{+}$
we will use the frames $(e_{\mp},\hat{e}_{\mp})$ on
$X_{\mp}=T^{*}Q_{\mp}$\,. They will be glued in a suitable way along
$X'$ and we
will use the writing
$(e,\hat{e})=1_{X_{\mp}}(x)(e_{\mp},\hat{e}_{\mp})$ for their global
definition on $X=X_{+}\sqcup X'\sqcup X_{+}$\,.\\

\noindent\textbf{Fiber bundles, metrics and connections:} A general vector
 bundle being written
$\pi_{\fkF}:\fkF\to M$ with its natural projection
$\pi_{\fkF}$\,.
 The restricted fiber bundle to $M'\subset M$
will be written $\fkF\big|_{M'}$\,.\\
A metric on $\fkF$ will be written
$g^{\fkF}$\,. A connection will be written
$\nabla^{\fkF}$ and $\nabla^{\fkF}_{U}$ for $U\in TM$ will denote
the covariant derivative w.r.t $U$\,. One exception is the Levi-Civita connection
acting on tensors above the riemannian manifold $(M,g^{TM})$ which will
be denoted by $\nabla^{M}$ or $\nabla^{M,g^{TM}}$\,, the latter being
used when it is necessary to specify the metric dependence. Other
exponents or indices may be used for specifying the connection and we
already made the difference between the flat connection
$\nabla^{\fkf}$ on $\pi_{\fkf}:\fkf\to Q$ and the unitary connection
for the metric $g^{\fkf}$\,, $\nabla^{\fkf,u}$\,.\\
When the vector bundle $\pi_{\fkF}:\fkF\to M$ is endowed with the connection
$\nabla^{\fkF}$\,, the exterior
covariant derivative  acting on  $\mathcal{C}^{\infty}(M;\Lambda T^{*}M\otimes
\fkF)$\,, as an exterior derivative, will be denoted $d^{\nabla^{\fkF}}$\,,
instead of the sometimes used notation $\nabla^{\fkF}$\,.
We recall that $\nabla^{\fkF}$ is a flat connection when
$d^{\nabla^{\fkF}}\circ d^{\nabla^{\fkF}}=0$\,.\\
\noindent\textbf{Functional spaces:} We shall use the notation
$\mathcal{F}(M;\fkF)$ for sections of a vector bundle
$\pi_{\fkF}:\fkF\to M$ with the regularity specified
by $\mathcal{F}$\,. Example given: We may take
$\mathcal{F}=\mathcal{C}^{\infty}$\,, $\mathcal{F}=L^{2}_{loc}$\,,
$\mathcal{F}=L^{2}$ once a metric $g^{\fkF}$ is fixed,
$\mathcal{F}=\mathcal{C}^{\infty}_{0}$\,,
$\mathcal{F}=\mathcal{D}'$\,. In particular
$\mathcal{C}^{\infty}(M;\Lambda T^{*}M)$ stands for
$\Omega(M)$ usually denoting the set of smooth differential forms on
$M$\,.\\
The local  spaces $\mathcal{C}^{\infty}(M;\fkF)$\,, $\mathcal{C}^{\infty}_{0}(M;\fkF)$\,, $L^{2}_{loc}(M;\fkF)$ and
$L^{2}_{comp}(M,\fkF)$ do not depend on the chosen metric on
$\fkF$\,. When $\fkF'$ denotes the dual vector bundle
with a duality product denoted by $u.v$ (possibly right-antilinear and
left-linear for complex vector bundles),
remember the duality between $L^{2}_{loc~comp}(M;\fkF)$ and
$L^{2}_{comp~loc}(M;\fkF')$ given by
$$
\langle s\,,\,s'\rangle=\int_{M}s(x).s'(x)~dv_{M}(x)
$$
where $dv_{M}$ is any given smooth volume measure on $M$\,, which can be
specified in local charts.\\
Accordingly the local spaces
$W^{\mu,2}_{loc~comp}(M;\fkF)$\,, $\mu\in\rz$\,, saying that there
are $\mu$ derivatives in $L^{2}_{loc~comp}$ when $\mu\in \nz$\,, do
not depend on the chosen metric and the dual of
$W^{\mu,2}_{comp}(M;\fkF)$ is
$W^{-\mu,2}_{loc}(M;\fkF)$\,.\\
Once metrics are fixed on $TM$ and $\fkF$ the global Sobolev space is denoted
$W^{\mu,2}(M;\fkF)$ or $W^{\mu,2}(M;\fkF,g^{TM},g^{\fkF})$\,.\\
When $M=X$ and $\fkF=E$ or $\fkF=F$\,, the Sobolev scale
introduced by G.~Lebeau in \cite{Leb1} will be denoted by
$\mathcal{W}^{\mu}(X;E)$ or $\mathcal{W}^{\mu}(X;F)$\,.\\

\noindent\textbf{Manifolds with boundaries:} All the manifolds with
boundaries $\overline{M}=M\sqcup \partial M$\,, namely
$\overline{Q}_{\mp}$ or $\overline{X}_{\mp}$\,, will have a
$\mathcal{C}^{\infty}$ boundary. By following the general $\mathcal{C}^{\infty}$-reflection
principle (see \cite{ChPi}) modeled on half-space problems,
 $\mathcal{C}^{\infty}$ functions, vector
bundle structures, and sections of vector bundles are well defined on
$\overline{M}$ as restriction of $\mathcal{C}^{\infty}$ objects on an
extended neighborhood $\tilde{M}$ of $\overline{M}$\,. Accordingly
$\mathcal{C}^{\infty}_{0}(\overline{M};\fkF)$ will denote
the space of $\mathcal{C}^{\infty}$ sections of the vector bundle $\fkF$\,,
which have a compact support in $\overline{M}$\,. The same thing
applies to the local spaces
$L^{2}_{loc}(\overline{M};\fkF)$\,,
$L^{2}_{comp}(\overline{M};\fkF)$ which must not be confused
with $L^{2}_{loc}(M;\fkF)$ and
$L^{2}_{comp}(M;\fkF)$\,.\\
The definition of local Sobolev spaces
$W_{loc~comp}^{\mu,2}(\overline{M};\fkF)$\,, $\mu\in\rz$\,,
follows the presentation of \cite{ChPi}\,, as the set of restrictions
to $M$ of elements of
$W^{\mu,2}_{loc~comp}(\tilde{M};\fkF)$\,. When $\mu>1/2$\,, any element of
$W^{\mu,2}_{loc}(\overline{M};\fkF)$ admits a trace
in $W^{\mu-1/2,2}_{loc}(\partial M;\fkF\big|_{\partial M})$\,, while  $\mathcal{C}^{\infty}_{0}(M;\fkF)$ is dense
in $W^{\mu,2}_{loc~comp}(\overline{M};\fkF)$ iff
$\mu\leq 1/2$\,. Remember also that for $\mu\geq 0$\,,
 $W_{loc}^{-\mu,2}(\overline{M};\fkF)$
is the dual of
$W^{\mu,2}_{0,comp}(\overline{M};\fkF)=\overline{\mathcal{C}^{\infty}_{0}(M;\fkF)}^{W_{comp}^{\mu,2}}$\,.\\
The definition of the global Sobolev scale
introduced by Lebeau on the manifold with boundary 
$\overline{X}_{-}$\,,
$\mathcal{W}^{\mu}(\overline{X}_{\mp};F)$\,,   follows the same scheme
and we refer to Subsection~\ref{sec:globSob} for details.

\noindent\textbf{Operators:} On a $\mathcal{C}^{\infty}$ vector
 bundle
$\pi_{\fkF}:\fkF\to M$\,, on a closed manifold $M$\,,
and when $g^{\fkF}$ is a metric on $\fkF$\,,
a differential operator $P$ with
$\mathcal{C}^{\infty}(M;L(\fkF))$ will not be distinguished by
notations from its maximal closed realization with domain 
$$
D(P)=\left\{s\in
  L^{2}(M;\fkF)\,,\quad Ps\in
  L^{2}(M;\fkF)\right\}\,.
$$
This will be the case for the
the differentials $d$\,, $d_{\fkh}$\,, $d^{\phi_{b}}_{\fkh}$ and
Bismut's hypoelliptic Laplacian $B^{\phi_{b}}_{\fkh}$ acting on
sections of $F=E\otimes \pi^{*}_{X}(\fkf)$\,.\\
The situation is different on a manifold with boundary where closed
realizations are related with a choice of boundary conditions. Then we
will use the notation $\overline{P}_{\alpha}$ for the closed
realization where $\alpha$ is a parameter which specifies the boundary
conditions among an admissible family. In our case the parameter $\alpha$
will essentially be $g=g^{TQ}$\,.\\

\noindent\textbf{Matched piecewise $\mathcal{C}^{\infty}$ structures:}
While using symmetry arguments on the glued double copies $Q_{-}\sqcup
Q'\sqcup Q_{+}$ or $X_{-}\sqcup X'\sqcup X_{+}$\,,
$X_{\mp}=T^{*}Q_{\mp}$\,, we are led to use piecewise
$\mathcal{C}^{\infty}$ objects, continuous or not. In order to
remember the possible discontinuities we shall use the notation
$\widehat{\fkF}_{g}$ for matched fiber bundles, the index $_{g}$
recalling when it is the case,  that the matching along $X'$ depends
on the chosen metric $g=g^{TQ}$\,. Accordingly closed realizations of a
differential operators $P$ with piecewise $\mathcal{C}^{\infty}$
coefficients and interface conditions along $X'$ which may depend on $g^{TQ}$
will be denoted by $\widehat{P}_{g}$\,. Redundant $\widehat{~}$
notations will be avoided. Example given in
$\widehat{B}^{\phi_{b}}_{g,\fkh}$ will be used instead of
$\widehat{B}^{\hat{\phi}_{b}}_{\hat{g},\hat{\fkh}}$ despite
$\widehat{B}^{\phi_{b}}_{g,\fkh}$ is actually associated with the
piecewise $\mathcal{C}^{\infty}$-versions of
$\phi_{b}$\,, $g^{TQ}$ and $\fkh$\,.
\subsection{Issues and strategy}
\label{sec:issStrat}
As a first remark, Bismut's hypoelliptic Laplacian, the differential
$d_{\fkh}$ and Bismut's codifferential $d^{\phi_{b}}_{\fkh}$ are all
first order differential operators in the position variable
$q$\,. Boundary conditions must only involve first, and possibly
partial first, traces along the boundary $X'=\partial
X_{-}$\,. Although the general geometry of $Q'=\partial Q_{-}$ and
$X'=\partial X_{-}$ depends on curvature terms and in particular the
second fundamental form of $Q'\subset (Q,g^{TQ })$\,, those curvatures
should have a limited effect on the analysis of those operators. The analysis carried out in \cite{Nie} worked
directly on Geometric Kramers-Fokker-Planck operators as defined by
G.~Lebeau in \cite{Leb1}, which is a larger class of operators
including Bismut's hypoelliptic Laplacian and where lower order
curvature dependent terms can be neglected. Following the dyadic
partition unity in the vertical variable already used by G.~Lebeau in
\cite{Leb1}\cite{Leb2}\,, it was possible to consider terms like
$A^{ik}_{j}(q)p_{k}p_{i}\frac{\partial}{\partial p_{j}}$ as parameter
dependent perturbations of $\frac{-\Delta_{p}+|p|^{2}}{2}$ and to
absorb the large $p$ contribution of the second fundamental form of
$Q'$ in $(\overline{Q}_{-},g^{TQ})$\,. This led to subelliptic
estimates where the curvature of the boundary nevertheless
deteriorates the exponents (compare the maximal hypoellipticity result of
G.~Lebeau recalled at the end of Subsection~\ref{sec:bismhyp} with
Theorem~\ref{th:Bphibg}). It is not known for the
moment whether the subelliptic estimates of Theorem~\ref{th:Bphibg} which
are the same as the ones of \cite{Nie} are
optimal.\\
However while considering the exact commutation of
$\overline{d}_{g,\fkh}$\,, $\overline{d}^{\phi_{b}}_{g,\fkh}$ with
$(z-\overline{B}^{\phi_{b}}_{g,\fkh})^{-1}$ stated in
Theorem~\ref{th:Bphibg}, a careful treatment the geometry involved by all the terms of
$B^{\phi_{b}}_{\fkh}$\,, $d_{\fkh}$ and $d^{\phi_{g}}_{\fkh}$ cannot
be skipped.\\
The heuristic leading to the boundary conditions of
$\overline{d}_{g,\fkh}$\,, $\overline{d}^{\phi_{b}}_{g,\fkh}$ and
$\overline{B}^{\phi_{b}}_{g,\fkh}$ given in
Theorems~\ref{th:dg}-\ref{th:dphibg}-\ref{th:Bphibg}, relies on the 
doubling of the manifold $\overline{Q}_{-}$ into $Q_{-}\sqcup Q'\sqcup Q_{+}$
and then to associate Dirichlet (resp. Neumann) boundary condition to odd
(resp. even) sections in the
symmetry $(\underline{q}^{1},\underline{q}')\mapsto
(-\underline{q}^{1},\underline{q}')$ between $Q_{-}$ and $Q_{+}$\,. By working in $X=X_{-}\sqcup
X'\sqcup X_{+}$\,, with a totally geodesic boundary $Q'$\,, namely
when $g^{TQ}=(d\underline{q}^{1})^{2}+m^{TQ'}(0)$\,, this leads
naturally to the boundary conditions given in
Theorems~\ref{th:dg}-\ref{th:dphibg}-\ref{th:Bphibg}. The analysis
comes from a straightforward translation of G.~Lebeau's maximal
hypoelliptic results on a closed manifold, because the symmetrization
preserves in this case all the $\mathcal{C}^{\infty}$ structures. When
$Q'$ has a non trivial second fundamental form, this is no more
possible, e.g. the symmetrized metric
$(d\underline{q}^{1})^{2}+m^{TQ'}(|\underline{q}^{1}|)$ is  no more
$\mathcal{C}^{1}$ and only piecewise $\mathcal{C}^{\infty}$-structures
are preserved on $X$\,.\\
Another unusual thing comes from the fact that the boundary conditions
for $\overline{d}_{g,\fkh}$ in Theorem~\ref{th:dg} actually depend on
the chosen metric $g^{TQ}$ on $\overline{Q}_{-}$\,. This is not the
case in the elliptic framework of Hodge or Witten Laplacian and this
is again a side effect on the cotangent space $X=T^{*}Q$ of the non
trivial second fundamental form of $Q'\subset (Q,g^{TQ})$ which requires a $g^{TQ}$-dependent
matching along $X'$ in order to speak of continuity and traces along
$X'$\,.\\
However the followed strategy is reminiscent of what we learned from
the carefull analysis of Witten and Hodge~Laplacians: Avoid as  long
as possible the  complicated curvature terms, while focussing firstly
on the differential $d_{\fkh}$ and secondly translate the result on
codifferentials $d^{\phi_{b}}_{\fkh}$ by duality. This does not ends
the game because it is not possible  to write
$\overline{B}^{\phi_{b}}_{g,\fkh}$ as a square of
$\frac{1}{2}(\overline{d}^{\phi_{b}}_{g,\fkh}+\overline{d}_{g,\fkh})$\,. This
has to be combined with the results of \cite{Nie}\,, with specific trace
theorems for $B^{\phi_{b}}_{\fkh}$\,, and with an explicit commutation
result for a dense set of smooth sections, where the latter cannot be
the same for the commutations with  $\overline{d}_{g,\fkh}$ or with
$\overline{d}^{\phi_{b}}_{g,\fkh}$\,. A consequence of the 
pseudospectral subelliptic estimates (with
respect to the imaginary spectral parameter $i\lambda$)
 ensures
that
$(1+C_{b}+\overline{B}^{\phi_{b}}_{g,\fkh})^{n}e^{-t\overline{B}^{\phi_{b}}_{g,\fkh}}$
is bounded as soon as $t>0$ for any $n\in\nz$\,. A bootstrap
regularity argument where Lebeau's maximal subelliptic estimates play
again a crucial role, shows that taking $n\in\nz$ large enough implies
$(1+C_{b}+\overline{B}^{\phi_{b}}_{g,\fkh})^{-n}:L^{2}(X_{-};F)\to
\mathcal{W}^{1}(\overline{X}_{-};F)\cap
D(\overline{B}^{\phi_{b}}_{g,\fkh})\subset
D(\overline{d}_{g,\fkh})\cap D(\overline{d}^{\phi_{b}}_{g,\fkh})$\,.
Actually all this analysis, and especially the use of Lebeau's maximal
subelliptic estimates for closed manifold, 
is carried out on the symmetrized phase
space
$X=X_{-}\sqcup X'\sqcup
X_{+}$ but for  the piecewise $\mathcal{C}^{\infty}$ and continuous  vector bundles
$\hat{E}_{g}$ or $\hat{F}_{g}$\,.\\

Below is the outline of the article:
\begin{itemize}
\item Section~\ref{sec:geomcot} specifies the geometry of the
  cotangent bundle $X=T^{*}Q$ when $(Q,g^{TQ})$ is a riemannian
  manifold. Several aspects of the parallel transport for the
  Levi-Civita connections $\nabla^{Q,g}$ and the pulled-back
  connection $\pi_{X} ^{*}(\nabla^{Q,g})$ will be specified. This
  leads to a natural definition of the piecewise
  $\mathcal{C}^{\infty}$ and continuous vector bundles $\hat{E}_{g},\hat{E'}_{g},\hat{F}_{g},\hat{F'}_{g}$\,.
\item In Section~\ref{sec:funcsp}, details are given for various
  functional spaces. In particular the independence of Lebeau's spaces
  with respect to the chosen metric $g^{TQ}$ is recalled. Functional
  spaces on the piecewise $\mathcal{C}^{\infty}$
  vector bundles $\hat{E}_{g}, \hat{E}'_{g}, \hat{F}_{g}, \hat{F}'_{g}$ are
 specified with the
  help of parallel transport introduced in
  Section~\ref{sec:geomcot}. Isomorphisms  and invariances  of those functional spaces induced by vector bundle
  isomorphisms  are reviewed.
\item Section~\ref{sec:closeddiff} is devoted to the definition of
  $\overline{d}_{g,\fkh}$ and its symmetrized version
  $\hat{d}_{g,\fkh}$ after proving the suitable trace
  theorems. A specific paragraph is  devoted to checking
  $\hat{d}_{g,\fkh}\circ \hat{d}_{g,\fkh}=0$ coming from a
  $\mathcal{C}^{\infty}$-interpretation of $\hat{E}_{g},\hat{F}_{g}$\,.
\item After defining the $F'$ adjoint of
  $\overline{d}_{g,\fkh}$\,, and the $\widehat{F}'_{g}$ adjoint of
  $\hat{d}_{g,\fkh}$\,, Section~\ref{sec:adjdiff} 
specifies the symplectic
  codifferential $\overline{d}^{\sigma}_{g,\fkh}$ for $\phi=\sigma$
  and finally Bismut's codifferential
  $\overline{d}^{\phi_{b}}_{g,\fkh}$ for $\phi=\phi_{b}$\,. This follows the scheme of
  J.M.~Bismut in \cite{Bis05}. However for the boundary or interface
  value problem, the choice of coordinates or
  $\mathcal{C}^{\infty}$-structures differ for those three steps and
  can be put together only at the level of piecewise
  $\mathcal{C}^{\infty}$ and continuous vector bundles.
\item Section~\ref{sec:closedhypo} after recalling details about
  Bismut's hypoelliptic Laplacians and general Geometric
  Kramers-Fokker-Planck operators, provides a  trace theorem for
  local versions of
  $B^{\phi_{b}}_{g,\fkh}$\,. After the definition of
  $\overline{B}^{\phi_{b}}_{g,\fkh}$\,, Theorem~\ref{th:Bphibg} is
  proved with additional properties concerning bootraped regularity,
  for powers of the resolvent and the semigroup, and the (formal)
  PT-symmetry. Note that the commutation
  $(z-\overline{B}^{\phi_{b}}_{\fkh})^{-1}\overline{d}_{g,\fkh}=\overline{d}_{g,\fkh}(z-\overline{B}^{\phi_{b}}_{\fkh})^{-1}$
  is rather proved in the spirit of \cite{ABG} by making use of the
  semigroup with
  $e^{-t\overline{B}^{\phi_{b}}_{\fkh}}\overline{d}_{g,\fkh}=\overline{d}_{g,\fkh}e^{-t\overline{B}^{\phi_{b}}_{\fkh}}$
  for $t\geq 0$\,.
\end{itemize}

\section{Geometry of the cotangent bundle}
\label{sec:geomcot}
This section gathers all the geometric information concerned with : a)
the
decomposition~$TX=TX^{H}\oplus TX^{V}$ associated with $g=g^{TQ}$\,; b)
more generally parallel transport for $\nabla^{Q,g}$ and
$\pi_{X}^{*}(\nabla^{Q,g})$\,, $\nabla^{E}$\,, $\nabla^{F}$\,,
$\nabla^{E'}$\,, $\nabla^{F'}$\,; c) the doubled
manifolds $Q=Q_{-}\sqcup Q'\sqcup Q_{+}$\,, and  the doubled cotangent
$X=X_{-}\sqcup X'\sqcup X_{+}$\,. The piecewise $\mathcal{C}^{\infty}$
vector bundles $\hat{E}_{g}, \hat{E}'_{g}, \hat{F}_{g}, \hat{F'}_{g}$
are introduced and some specific vector bundle isomorphisms are
studied. All those presentations are done in a coordinate free way
and they ensure the independence w.r.t a choice of coordinates. The
reader willing to grasp a concrete realization, can first look at the
final paragraph where those constructions are expressed in terms of  local coordinates. 

\subsection{The cotangent bundle of a manifold without boundary} 
\label{sec:cot}
Let $Q$ be a smooth manifold (without boundary at the moment). Denote
by $X$  the total space of the cotangent bundle $T^{*}Q$ endowed with
the natural projection $\pi_{X}:X=T^{*}Q\to Q$\,.\\
The vertical subbundle of the tangent vector bundle on $X$\,, $\pi_{TX}:TX\to X$\,,
is nothing but
\begin{align}\label{eqTVX}
	TX^{V}=\pi_{X}^{*}(T^{*}Q)\,. 
\end{align} 
It is a subbundle of $TX$ with the exact
sequence of smooth vector bundles on $X$
\begin{align}\label{eqTVH}
	0\to TX^{V}\to TX\to \pi_{X}^{*}(TQ)\to 0\,. 
\end{align} 
By duality $T_{x}^{*}X^{H}=\left\{\alpha\in T_{x}^{*}X\,, \forall t\in
  T_{x}X^{V}\,, \quad \alpha.t=0\right\}$ identifies $T^{*}X^{H}$ as the
subbundle
\begin{align}
  \label{eq:TsHX}
T^{*}X^{H}=\pi_{X}^{*}(T^{*}Q)\,,
\end{align}
with the exact sequence of smooth vector bundles on $X$
\begin{align}
  \label{eq:TsHV}
  0\to T^{*}X^{H}\to T^{*}X\to \pi_{X}^{*}(TQ)\to 0\,.
\end{align}
Those constructions do not involve any metric.\\
Now when $g=g^{TQ}$ is a riemannian metric on $Q$\,, the Levi-Civita
connection  $\nabla^{TQ,g}$ induces a connection on  tensor algebras
written simply $\nabla^{Q,g}$\,, in particular on $X=T^{*}Q$\,.  This
defines a horizontal vector subbundle of $TX$
\begin{align}\label{eqTHX}
	TX^{H}\simeq \pi_{X}^{*}(TQ)\,, 
\end{align} 
with the $g$-dependent direct sum decomposition
\begin{align}\label{eqTX1}
	TX\stackrel{g}{=} TX^{H}\oplus TX^{V}\stackrel{g}{=} \pi_{X}^{*}(TQ\oplus T^{*}Q). 
\end{align} 
The duality defines
\begin{eqnarray}
 \label{eqTsXV} &&T^{*}X^{V}\simeq \pi_{X}^{*}(TQ)\\
\label{eqTsX1}&&
T^{*}X\stackrel{g}{=}T^{*}X^{H}\oplus T^{*}X^{V}\stackrel{g}{=}\pi_{X}^{*}(T^{*}Q\oplus TQ)\,.
\end{eqnarray}
Let $\pi_{\fkf}:\fkf \to Q$ be a smooth vector bundle on $Q$ endowed
with a flat connection $\nabla^{\fkf}$ and a smooth  hermitian metric
$g^{\fkf}$\,. It is identified via the metric  with the antidual flat connection
$\nabla^{\fkf'}$\,. If
$\omega(\nabla^{\fkf},g^{\fkf})=(g^{\fkf})^{-1}\nabla^{\fkf}g^{\fkf}$
then $\nabla^{\fkf'}=\nabla^{\fkf}+\omega(\fkf,\nabla^{\fkf})$\,.
The vector bundles $\fkF=E,F,E',F'$\,, $\pi_{\fkF}:\fkF\to Q$\,,
 are defined by
\begin{eqnarray*}
  &&E=\Lambda T^{*}X\quad,\quad E'=\Lambda
     TX\\
\text{and}&&
F=E\otimes \pi^{*}_{X}(\fkf)\quad,\quad F'=E'\otimes\pi_{X}^{*}(\fkf)\,,
\end{eqnarray*}
with the $g$-dependent identifications
\begin{eqnarray*}
  && E\stackrel{g}{=}\pi_{X}^{*}(\Lambda T^{*}Q\otimes
     \Lambda TQ)\quad,\quad  
E'\stackrel{g}{=}\pi_{X}^{*}(\Lambda TQ\otimes \Lambda T^{*}Q)\\
&& F\stackrel{g}{=}\pi_{X}^{*}(\Lambda T^{*}Q\otimes
     \Lambda TQ\otimes \fkf)\quad,\quad 
F'\stackrel{g}{=}\pi_{X}^{*}(\Lambda TQ\otimes \Lambda
   T^{*}Q\otimes \fkf)
\end{eqnarray*}
The metrics on those vector bundles involve the weight
\begin{equation}
  \label{eq:lapq}
\langle p\rangle_{q}=\sqrt{1+g^{T^{*}Q}_{q}(p,p)}
\end{equation}
and are defined by
\begin{eqnarray}
\label{eq:gE}
&& g^{E}=\langle p\rangle_{q}^{-N_{H}+N_{V}}\pi_{X}^{*}(g^{\Lambda
     T^{*}Q}\otimes g^{\Lambda TQ})\,,\\
\label{eq:gEpr}
&&
g^{E'}=\langle p\rangle_{q}^{N_{H}-N_{V}}
\pi_{X}^{*}(g^{\Lambda TQ}\otimes g^{\Lambda T^{*}Q})\,,\\
\label{eq:gF}
&&g^{F}=\langle p\rangle_{q}^{-N_{H}+N_{V}}
\pi_{X}^{*}(g^{\Lambda T^{*}Q}\otimes g^{\Lambda TQ}\otimes g^{\fkf})\,,
\\
\label{eq:gFpr}
&&
g^{F'}=\langle p\rangle_{q}^{N_{H}-N_{V}}
\pi_{X}^{*}(
g^{\Lambda TQ}\otimes g^{\Lambda T^{*}Q}\otimes g^{\fkf})\,.
\end{eqnarray}
The Levi-Civita connection on $\Lambda TQ\otimes \Lambda T^{*}Q$
associated with $g=g^{TQ}$ being denoted by $\nabla^{Q,g}$ there is a
natural connection on $\fkF=E,F,E',F'$ simply given by 
\begin{eqnarray}
\label{eq:nabEg}
&&
\nabla^{E,g}=\pi_{X}^{*}(\nabla^{Q,g})\quad,\quad
\nabla^{E',g}=\pi_{X}^{*}(\nabla^{Q,g})\,,
\\
\label{eq:nabFg}
&&\nabla^{F,g}=\pi_{X}^{*}(\nabla^{Q,g}+\nabla^{\fkf})\quad,\quad
\nabla^{F',g}=\pi_{X}^{*}(\nabla^{Q,g}+\nabla^{\fkf'})\,.
\end{eqnarray}
\begin{remark}
Discerning what depends on the metric $g=g^{TQ}$ is of outmost
importance when boundary value problems are considered in particular
because $g^{TQ}$ does not have a product structure near the
boundary, in particular when the second fundamental form of the
boundary does not vanish.
\end{remark}

\subsection{Manifold with boundary} 
\label{sec:mfldbdy}
From now on, we will assume that $\overline{Q}_{-}=Q_{-}\sqcup Q'$\,,  is a compact manifold with 
boundary $Q'=\partial Q_{-}$\,. Before considering the metric aspects, 
$\overline{Q}_{-}$ can be considered as a domain of the doubled
manifold $Q=Q_{-}\sqcup Q'\sqcup Q_{+}$ where $Q_{+}$
(resp. $\overline{Q}_{+}=Q'\sqcup Q_{+}$) is a copy of $Q_{-}$
(resp. $\overline{Q}_{-}$) and the $\mathcal{C}^{\infty}$-structures
are matched along $Q'$\,. By following the
$\mathcal{C}^{\infty}$-reflection principle (see \cite{ChPi}-I-7), there is a
canonical $\mathcal{C}^{\infty}$ structure on $Q$ which is unique modulo
diffeomorphims preserving $Q'$\,. However its concrete realization may depend on the
choice of a normal bundle with its differential structure, which is equivalent to the choice of a
tubular neighborhood of $Q'=\partial Q_{-}$ in $\overline{Q}_{-}$ according to  \cite{Lan}-Chap~IV-6. This may lead
to different  realizations of the doubled manifold
$Q$ which are all diffeomorphic. Actually our analysis is done with a family of metrics for which the
normal bundle $N_{Q'}\overline{Q}_{-}$ is not changed. So both
approaches, starting from an abstract definition of $Q$ or from its
construction after fixing the normal bundle, are equivalent.
The metric $g_{-}=g^{TQ_{-}}\in \mathcal{C}^{\infty}(\overline{Q}_{-};
T^{*}Q_{-}\odot T^{*}Q_{-})$ can thus be thought as the restriction of a
$\mathcal{C}^{\infty}$ metric  $g=g^{TQ}$ on $Q$ (another
metric will be put on $Q$ in the next paragraph). All the objects,
smooth vector bundles and functional spaces (see the $\mathcal{C}^{\infty}$-reflection
principle in \cite{ChPi}-I-7)
which are related to the $\mathcal{C}^{\infty}$ structure of
$\overline{Q}_{-}$ can be thought as restrictions to
$\overline{Q}_{-}$ (or to $Q_{-}$) 
of objects on $Q$\,. Those  objects
 will be specified later when necessary.\\
Hence we can consider the case of a closed hypersurface 
$Q'\subset Q$\,, of the compact riemannian manifold $(Q,g^{TQ})$ \,, which admits a global unit normal 
vector $\underline{e}_{1}$:
$$
TQ\big|_{Q'}=TQ'\mathop{\oplus}^{\perp} \rz \underline{e}_{1}\,,\quad \underline{e}_{1}\in
\mathcal{C}^{\infty}(Q';N_{Q'}Q)\,,\quad g^{TQ}(\underline{e}_{1},\underline{e}_{1})=1\,,
$$
where $N_{Q'}Q$ is the normal vector bundle of $Q'\subset
(Q,g^{TQ})$\,. For the manifold $\overline{Q}_{-}$ with boundary
$Q'$\,, $\underline{e}_{1}$ is the outward unit normal vector.\\
For $\underline{q}'\in Q'$\,
, let $(\exp^{Q,g}_{\underline{q}'}(t\underline{e}_{1}))_{t\in (-\varepsilon,\varepsilon)}$ 
be the geodesic curve on $Q$ starting from $\underline{q}'$ in the direction
$\underline{e}_{1}$ which is  well defined for $t\in (-\varepsilon,\varepsilon)$\,, where
$\varepsilon>0$ can be chosen uniform w.r.t $\underline{q}'\in Q'$ by
compactness. This provides diffeomorphisms
\begin{eqnarray}
\label{eqQfQe1}
  &&\begin{array}[c]{rcl}
(-\varepsilon,\varepsilon)\times Q' &\to& \left\{q\in Q\,,
                                          d_{g}(q,Q')<\varepsilon\right\}=Q_{(-\varepsilon,\varepsilon)}\,,
\\
(\underline{q}^{1},\underline{q}')&\mapsto& \exp^{Q,g}_{\underline{q}'}(\underline{q}^{1}\underline{e}_{1})
\end{array}\\
\label{eqQfQe2}
&&
\begin{array}[c]{rcl}
(-\varepsilon,0]\times Q' &\to& \left\{q\in \overline{Q}_{-}\,,
                                          d_{g}(q,Q')<\varepsilon\right\}=Q_{(-\varepsilon,0]}\,,
\\
(\underline{q}^{1},\underline{q}')&\mapsto& \exp^{Q,g}_{\underline{q}'}(\underline{q}^{1}\underline{e}_{1})
\end{array}
\end{eqnarray}
In the sequel, we will not distinguish the global coordinate 
$(\underline{q}^{1},\underline{q}')$  with the natural projections 
\begin{align}
&	\underline{q}^{1}:Q_{I}\to I,&\underline{q}': 
Q_{I}\to Q'\quad\text{for}~I=(-\varepsilon,\varepsilon)~\text{or}~I=(-\varepsilon,0]\,. 
\end{align} 
By \eqref{eqQfQe1}\eqref{eqQfQe2}, we have 
\begin{align}\label{eqGaTQ}
	TQ_{I}\simeq \underline{q}^{1*}TI\oplus \underline{q}^{\prime *}TQ'. 
\end{align} 
Gauss Lemma, over 
$Q_{I}$\,, says
\begin{align}\label{eqgTQq}
	 	g^{TQ}=(d\underline{q}^{1})^{2}+m^{TQ'}(\underline{q}^{1}),
\end{align} 
where $m=m^{TQ'}$ is a $\underline{q}^{1}$-dependent metric on $Q'$\,.
\\
By following Bismut-Lebeau in \cite{BiLe91}-VIII, for 
$(\underline{q}^{1},\underline{q}')\in Q_{(-\epsilon,0]}$, we can identify $T_{\underline{q}'}Q'$ with 
$T_{(\underline{q}^{1},\underline{q}')}Q_{I}$ by the parallel transport with 
respect to the Levi-Civita connection $\nabla^{Q,g}$\,,
  along the geodesic 
$\exp^{Q,g}_{\underline{q}'}(t\underline{e}_{1})$ from  $t=0$ to $t=\underline{q}^{1}$\,. 
This gives  over 
$Q_{I}$, a smooth identification of vector bundles,  
\begin{align}\label{eqTQp}
TQ_{I}= \underline{q}^{\prime*}(\mathbb R \underline{e}_{1}\oplus  TQ').
\end{align}
Contrarily to \eqref{eqGaTQ} the latter decomposition defined as the
pull-back of an abstract vector bundle with
$(\underline{q}^{1},\underline{q}')\mapsto \underline{q}'$ does not
give rise in general to an integrable decomposition of
$TQ_{I}$\,. The extrinsic curvature of $Q'\subset (Q_{I},g)$ when
$\partial_{\underline{q}^{1}}m(0)\neq 0$ prevents from integrability. \\
Since the parallel transport is an isometry, via \eqref{eqTQp}, the
metric $g=g^{TQ}$ becomes
\begin{align}\label{eqgTQ0}
	g^{TQ}=(d\underline{q}^{1})^{2}+ \underline{q}^{\prime*} m(0). 
\end{align} 
Note that the identification \eqref{eqTQp} depends on $\underline{q}^{1}$, while 
the right hand side of \eqref{eqgTQ0} is independent of
$\underline{q}^{1}$\,. \\
More generally, if  $\pi_{\fkF}:\fkF\to Q$ is a vector bundle on $Q$ endowed with a 
 connection $\nabla^{\fkF}$\,, the fiber $\fkF_{\underline{q}'}$ above
 $\underline{q}'\in Q'$ can be identified  with 
 $\fkF_{(\underline{q}^{1},\underline{q}')}$\,, $\underline{q}^{1}\in (-\varepsilon,\varepsilon)$\,,  by using the parallel
 transport along
 $(\exp^{Q,g}_{\underline{q'}}(t\underline{e}_{1}))_{t\in
   (-\varepsilon,\varepsilon)}$ associated with  $\nabla^{\fkF}$\,.\\
 Hence over $Q_{(-\varepsilon,\varepsilon)}$\,, 
 \begin{align}\label{eqFqF}
	\fkF= \underline{q}^{\prime *}\fkF|_{Q'}. 
\end{align} 
 Under this identification, the covariant derivatives equals
 \begin{align}\label{idF}
 	\nabla^{\fkF}_{(U^{1}\underline{e}_{1}+U')}=U^{1}\frac{\partial}{\partial 
 	\underline{q}^{1}}+\nabla_{U'}^{\fkF\big|_{Q'},\underline{q}^{1}},
 \end{align} 
 where $\nabla^{\fkF_{|Q'},\underline{q}^{1}}$ is a $\underline{q}^{1}$-dependent connection on  
 $\fkF_{|_{Q'}}$\,. 
The exterior covariant derivative $d^{\nabla^{\fkF}}$ is then
\begin{equation}
  \label{eq:dFnorm}
  d^{\nabla^{\fkF}}=d\underline{q}^{1}\wedge \frac{\partial}{\partial \underline{q}^{1}}+d^{\nabla^{\fkF|_{Q'},\underline{q}^{1}}}\,.
\end{equation}
More precisely, for $\underline{q}^{1}\in (-\varepsilon,\varepsilon)$, there is a section  
 $A_{\underline{q}^{1}}\in C^{\infty}(Q',T^{*}Q'\otimes {\rm End}(\fkF\big|_{Q'}))$ which depends 
 smoothly  on $\underline{q}^{1}$\,, such that 
 \begin{align}
	\nabla^{\fkF\big|_{Q'},\underline{q}^{1}}=\nabla^{\fkF\big|_{Q'},0}+A_{\underline{q}^{1}}\,.
\end{align} 
In $A_{\underline{q}^{1}}$ there is no component of 
	$d\underline{q}^{1}$\,, because the identification 
	\eqref{idF} is obtained by parallel transport.\\
This general construction will be applied with the following vector
bundles:
\begin{itemize}
\item $\fkF=X=T^{*}Q$\,, $\nabla^{\fkF}=\nabla^{Q,g}$\, which is
  actually obtained by duality from the case
  $\fkF=TQ_{(-\varepsilon,\varepsilon)}$ treated above\,;
\item $\fkF=\fkf$\,, where $\fkf$ is
  endowed with  the hermitian metric
  $g^{\fkf}$\,,  the flat connection $\nabla^{\fkf}$\,, or its
  antidual flat connection $\nabla^{\fkf'}$\,, after identifying
  $\fkf$ with its antidual via the metric;
\item $\fkF=\Lambda TQ\otimes \Lambda T^{*}Q\otimes \fkf$\,,
  $\nabla^{\fkF}=\nabla^{Q,g}+\nabla^{\fkf}$ and $\fkF=\Lambda
  T^{*}Q\otimes \Lambda TQ\otimes \fkf$\,, $\nabla^{\fkF}=\nabla^{Q,g}+\nabla^{\fkf'}$\,.
\end{itemize}

\subsection{The doubled riemannian  manifold $(Q,\hat{g})$}
\label{sec:dbriem}
Like in the previous paragraph consider $\overline{Q}_{-}\subset
Q_{-}\sqcup Q'\sqcup Q_{+}$ and use the identification via the
exponential map for a smooth metric $g=g^{TQ}$\,, $g^{TQ}\big|_{\overline{Q}_{-}}=g_{-}$\,,  $Q_{I}\simeq I\times Q'$ for an
interval $I\subset (-\varepsilon,\varepsilon)$\,. \\
With this isomorphism the map $S_{Q}$
\begin{equation}
  \label{eq:RQ}
\begin{array}[c]{rrcl}
 S_{Q}:& Q_{(-\varepsilon,\varepsilon)}&\to &
                                         Q_{(-\varepsilon,\varepsilon)}\\
&(\underline{q}^{1},\underline{q}')&\mapsto& (-\underline{q}^{1},\underline{q}')
\end{array}
\end{equation}
is an involutive diffeomorphism. The push-forward and pull-back maps
coincide. They are
given by,
\begin{eqnarray}
\nonumber  &&
S_{Q,*}=S^{*}_{Q}:TQ_{(-\varepsilon,\varepsilon)}\to
TQ_{(-\varepsilon,\varepsilon)}\quad,\\
\label{eq:SgQ*1}&& S_{Q,*}(\alpha \underline{e}_{1}, t')=(-\alpha
   \underline{e}_{1},t')\in
   T_{(-\underline{q}^{1},\underline{q}')}Q\quad\text{for}~(\alpha\underline{e}_{1},
   t')\in T_{(\underline{q}^{1},\underline{q}')}Q\,,\\
\nonumber
&&S_{Q,*}=S^{*}_{Q}:T^{*}Q_{(-\varepsilon,\varepsilon)}\to
   T^{*}Q_{(-\varepsilon,\varepsilon)}\quad,\\
\label{eq:SgQ*2}
&&  S_{Q,*}(\alpha \underline{e}^{1},\theta')=(-\alpha
   \underline{e}^{1},\theta')\in
T^{*}_{(-\underline{q}^{1},\underline{q}')}Q\quad\text{for}~(\alpha\underline{e}^{1},
   \theta')\in T^{*}_{(\underline{q}^{1},\underline{q}')}Q\,,
\end{eqnarray}
and they
have a natural action on tensors.\\
In particular we can define the metric $g_{+}=S_{Q,*}g_{-}$ on
$TQ_{(-\varepsilon,\varepsilon)}$ with
$$
g_{+}(\underline{q}^{1},\underline{q}')=(d\underline{q}^{1})^{2}+m(-\underline{q}^{1},\underline{q}')\,.
$$
Because $g_{+}=g_{-}$ on $Q'$\,, we can define the continuous metric
$\hat{g}=1_{Q_{(-\varepsilon,0)}}g_{-}+1_{Q_{[0,\varepsilon)}}g_{+}$
 on $TQ_{(-\varepsilon,\varepsilon)}$ by
$$
\hat{g}(\underline{q}^{1},\underline{q'})=g_{-}(-|\underline{q}^{1}|,\underline{q}')=(d\underline{q}^{1})^{2}+m(-|\underline{q}^{1}|,
\underline{q}')\,.
$$
In general when $\partial_{\underline{q}^{1}}m(0,\underline{q}')\neq
0$\,, which corresponds to a non vanishing second fundamental form of
$Q'\subset (Q,g_{-})$\,, the metric $\hat{g}$ is only piecewise
$\mathcal{C}^{\infty}$ on $\overline{Q}_{-}$ and $\overline{Q}_{+}$
and continuous (with a discontinuous $\partial_{\underline{q}^{1}}$
derivative along $Q'$)\,.
As noticed before, $g=g_{-}$ and $g_{+}$ induce the same
identification \eqref{eqQfQe1}, the same involutions
\eqref{eq:RQ}\eqref{eq:SgQ*1}\eqref{eq:SgQ*2} and the same vector field
$\underline{e}_{1}$ obtained via \eqref{eqTQp}. Nevertheless   the idenfications of $\underline{q}^{\prime
  *}(TQ')$
in  \eqref{eqTQp}  depends  on the chosen metric $g=g^{TQ}$\,.

When $g=g_{-}=(d\underline{q}^{1})^{2}+m(0,\underline{q}')=g_{0}$ the metric
$\hat{g}_{0}=g_{0}$ is the initial smooth metric on
$Q_{(-\varepsilon,\varepsilon)}$\,.  
The identifications
$Q_{(-\varepsilon,\varepsilon)}=(-\varepsilon,\varepsilon)\times Q'$
and $TQ_{(-\varepsilon,\varepsilon)}=\underline{q}^{\prime *}(\rz
\underline{e}_{1}\oplus TQ')$\,,
made for the metric
$\hat{g}=(d\underline{q}^{1})^{2}+m(-|\underline{q}^{1}|;\underline{q}')$ and for the
metric $g_{0}=(d\underline{q}^{1})^{2}+m(0,\underline{q}')$ 
provide a piecewise
$\mathcal{C}^{\infty}$ diffeomorphims continously coinciding with
$\mathrm{Id}_{Q_{(-\varepsilon,\varepsilon)}}$ and a piecewise
$\mathcal{C}^{\infty}$  and continuous isometry $\widehat{\Psi}_{Q}^{g,g_{0}}$from
$(TQ_{(-\varepsilon,\varepsilon)},g_{0}^{TQ})$ to
$(TQ_{(-\varepsilon,\varepsilon)}, \hat{g}^{TQ})$ such that the
following diagram commutes
\begin{equation}
  \label{eq:psiQgg}
\xymatrix{
(TQ_{(-\varepsilon,\varepsilon)},
g_{0}^{TQ}) 
\ar[d]_{S_{Q,*}}
\ar[r]^{\widehat{\Psi}_{Q}^{g,g_{0}}} 
&
(TQ_{(-\varepsilon,\varepsilon)},\hat{g}^{TQ})
\ar[d]^{S_{Q,*}}\quad\\ 
(TQ_{(-\varepsilon,\varepsilon)},
g_{0}^{TQ}) 
\ar[u]\ar[d]_{\pi_{TQ}}
\ar[r]^{\widehat{\Psi}_{Q}^{g,g_{0}}} 
&
(TQ_{(-\varepsilon,\varepsilon)},\hat{g}^{TQ})
\ar[u]\ar[d]^{\pi_{TQ}}\quad\\ 
Q_{(-\varepsilon,\varepsilon)}
\ar[r]_{\mathrm{Id}_{Q_{(-\varepsilon,\varepsilon)}}}& 
Q_{(-\varepsilon,\varepsilon)}\quad\,
}
\end{equation}
with $\widehat{\Psi}_{Q}^{g,g_{0}}\big|_{Q'}=\Id_{TQ\big|_{Q'}}$\,. 
A similar result holds for $T^{*}Q_{(-\varepsilon,\varepsilon)}$
endowed with the dual metrics $g_{0}^{T^{*}Q}$ and $\hat{g}^{T^{*}Q}$\,, with natural tensorial extensions.
\\
The situation already encountered with a piecewise
$\mathcal{C}^{\infty}$-identification, above
$\overline{Q}_{(-\varepsilon,0]}$ and
$\overline{Q}_{[0,\varepsilon)}$\,,  
 of
$\pi_{TQ}:TQ_{(-\varepsilon,\varepsilon)}\to
Q_{(-\varepsilon,\varepsilon)}$ can be generalized for a general
vector bundle $\pi_{\fkF}:\fkF\to Q_{(-\varepsilon,\varepsilon)}$
endowed with a connection $\nabla^{\fkF}$\,.\\
When
$\pi_{\fkF}:\fkF\to Q$ is a vector bundle on $Q$ endowed with a 
 connection $\nabla^{\fkF}$\,, formula \eqref{eqFqF} remains valid
 \begin{equation}
\label{eq:FqFhg}
	\fkF= \underline{q}^{\prime *}\fkF|_{Q'}\,.
 \end{equation}
In order to complete the picture we spectify the double of 
the flat vector bundle  $\pi_{\fkf}:\fkf\to
\overline{Q}_{-}$ endowed with the flat connection
$\nabla^{\fkf}$ and the hermitian metric $g^{\fkf}$\,. 
Because it is
flat the parallel transport along $\exp_{\underline{q}'}^{Q,\hat{g}}$
is trivial but applications, in particular the treatment of Dirichlet
and Neumann boundary conditions, requires an additional modification
along $Q'$\,. As mentionned in the introduction $\nu\in \mathcal{C}^{\infty}(Q';L(\fkf\big|_{Q'}))$ is an involutive isometry of
$(\fkf\big|_{Q'}, g^{\fkf})$  such that the covariant derivative vanishes
$\nabla^{L(\fkf|_{Q'})}\nu=0$\,.\\
Actually this is equivalent to
$\fkf\big|_{Q'}=\fkf_{+}\big|_{Q'}\oplus^{\perp} \fkf_{-}$ with
$\nu\big|_{\fkf_{\pm}}=\pm \Id_{\fkf_{\pm}}$ and all the theory can be done by
assuming $\nu=\pm \Id_{\fkf\big|_{Q'}}$\,, which is also our main
concern.\\
\begin{definition}
  \label{de:doublef} The double, with respect to $\nu$\,, of  $\pi_{\fkf}:\fkf\to
  \overline{Q}_{-}$ endowed with the smooth flat connection
 $\nabla^{\fkf}\in
  \mathcal{C}^{\infty}(Q;T^{*}Q\otimes L(\fkf))$ and the
  metric $g^{\fkf}$\,, is the double copy, still denoted by,
  $\pi_{\fkf}:\fkf\to Q$ using $Q=Q_{-}\sqcup Q'\sqcup Q_{+}$\,,
  $\overline{Q}_{+}\simeq \overline{Q}_{-}$ endowed with the flat
  connection
  $\nabla^{\fkf}\big|_{\overline{Q}_{\mp}}=\nabla^{\fkf}_{{\overline{Q}}_{-}}$\,,
  with 
  the metric
  $\hat{g}^{\fkf}(\underline{q}^{1},\underline{q'})=g^{\fkf}(-|\underline{q}^{1}|,\underline{q}')$
  and the continuity condition
$$
\fkf_{(0^{+},\underline{q}')}\ni
(0^{+},\underline{q}',\nu v)=(0^{-},\underline{q}',v)\in \fkf_{(0^{-},\underline{q}')}\,.
$$
\end{definition}
Since $\nu$ is flat, 
when $(v_{1},\ldots,v_{d_{f}})$ is a local flat frame of
$\pi_{\fkf}:\fkf\to \overline{Q}_{-}$ around $q_{0}\in Q'$\,,
$[1_{(-\varepsilon,0]}(\underline{q}^{1})+1_{(0,\varepsilon)}(\underline{q}^{1})\nu]
v_{i}$\,, $i=1,\ldots,d_{f}$ is a local flat frame of
$\pi_{\fkf}:\fkf\to Q$\,. Thus $\pi_{\fkf}:\fkf \to Q$ has a natural
$\mathcal{C}^{\infty}$ structure associated with $\nabla^{\fkf}$\,.
The involution $S_{Q}$ lifts to $\fkf$ and therefore to $\Lambda
TQ\otimes \Lambda T^{*}Q\otimes \fkf$\,. 
This lifting on  $\fkf$ or $\Lambda
TQ\otimes \Lambda T^{*}Q\otimes \fkf$ will be denoted by
$S_{Q,\nu}$ in order to recall that the symetric extension of $\fkf$
to $Q$ depends on $\nu$\,. With the symmetric metric $\hat{g}^{\fkf}$\,,
$S_{Q,\nu}$ is an isometry of
$(\fkf,\hat{g}^{\fkf})$ and 
the diagram \eqref{eq:psiQgg} can be completed by replacing
$TQ_{(-\varepsilon,\varepsilon)}$ with 
$$
\fkF=\Lambda TQ_{(-\varepsilon,\varepsilon)}\otimes 
 \Lambda T^{*}Q_{(-\varepsilon,\varepsilon)}\otimes
 \fkf\,,
$$
\begin{equation}
 \label{eq:psiQgg2}
\xymatrix{
(\fkF,
 g_{0}^{\Lambda TQ}\otimes g_{0}^{\Lambda T^{*}Q}\otimes \hat{g}^{\fkf}) 
 \ar[d]_{S_{Q,\nu}}
 \ar[r]^{\widehat{\Psi}_{Q}^{g,g_{0}}} 
 &
 (\fkF,\hat{g}^{\Lambda TQ}\otimes \hat{g}^{\Lambda
   T^{*}Q}\otimes \hat{g}^{\fkf})
 \ar[d]^{S_{Q,\nu}}\quad\\ 
 (\fkF,
 g_{0}^{\Lambda TQ}\otimes g_{0}^{\Lambda T^{*}Q}\otimes \hat{g}^{\fkf}) 
 \ar[u]\ar[d]_{\pi_{\fkF}}
 \ar[r]^{\widehat{\Psi}_{Q}^{g,g_{0}}} 
 &
 (\fkF,\hat{g}^{\Lambda TQ}\otimes \hat{g}^{\Lambda
   T^{*}Q}\otimes \hat{g}^{\fkf})
 \ar[u]\ar[d]^{\pi_{\fkF}}\quad\\ 
 Q_{(-\varepsilon,\varepsilon)}
 \ar[r]_{\mathrm{Id}_{Q_{(-\varepsilon,\varepsilon)}}}& 
 Q_{(-\varepsilon,\varepsilon)}
 }
\end{equation}
Because the flat connection $\nabla^{\fkf}$ differs from the unitary
connection $\nabla^{\fkf,u}$ it is important to keep the
$\underline{q}^{1}$-dependent metric $g^{\fkf}$ in $g_{0}^{\Lambda
  TQ}\otimes g_{0}^{\Lambda T^{*}Q}\otimes g^{\fkf}$ while
$g_{0}^{TQ}=(d\underline{q}^{1})^{2}\oplus m(0,\underline{q'})$\,.\\
Since the metric $\hat{g}^{\fkf}$ is only piecewise
$\mathcal{C}^{\infty}$ and continuous, attention must be paid to  the
identification with $\fkf$ of the antidual $\fkf'$  via the metric. The double
$\pi_{\fkf'}:\fkf'\to Q$ can be thought as
$\mathcal{C}^{\infty}$-vector bundle on $Q$ with the flat connection
$\nabla^{\fkf'}$ antidual to $\nabla^{\fkf}$ but the identification with
$\fkf$ gives a new $\mathcal{C}^{\infty}$-structure on $\fkf$\,.
This construction yields the following statement.
\begin{proposition}
\label{pr:dualf} The identification of $\pi_{\fkf'}:\fkf'\to Q$ the
antidual to the $\mathcal{C}^{\infty}$ flat hermitian
 vector bundle $(\fkf,\nabla^{\fkf},\hat{g}^{\fkf})$ can be identified
 via the metric $\hat{g}^{\fkf}$ with $\pi_{\fkf}:\fkf\to Q$ in the
 class of piecewise $\mathcal{C}^{\infty}$ and continous vector
 bundle. The antidual flat connection $\nabla^{\fkf'}$ on $\pi_{\fkf}:\fkf\to Q$
 differs from $\nabla^{\fkf}$ in general and gives rise to a different
 $\mathcal{C}^{\infty}$ structure on $\fkf$ (remember
 $\nabla^{\fkf'}=\nabla^{\fkf}+\omega(\fkf,\hat{g}^{\fkf})=\nabla^{\fkf}+(\hat{g}^{\fkf})^{-1}\nabla^{\fkf}\hat{g}^{\fkf}$). 
\end{proposition}
\begin{remark}
Since $Q$ is smooth or when $Q$ is endowed with the smooth metric $g$
such that $g\big|_{Q_{-}}=g_{-}$\,, all the differential geometric 
constructions make sense on $Q$\,. However, all the 
constructions which  involve the Riemannian structure with the
symmetric metric $\hat{g}$\,, will be only piecewise 
smooth, and even sometimes not continuous, extending the subtlety
already appearing in Proposition~\ref{pr:dualf} with the metric $\hat{g}^{\fkf}$\,. In particular 
 on  the total space $X=T^{*}Q$\,: the tangent and
cotangent space $TX$ and $T^{*}X$ are smooth vector bundles but the 
horizontal subbundle  $TX^{H}$  and vertical subbundle $TX^{V}$\,,
$TX^{V}$\,, which rely on the chosen metric $g^{TQ}$ will lead to piecewise
$\mathcal{C}^{\infty}$ and a priori discontinuous
 structures for the non smooth metric  $\hat{g}^{TQ}$\,. The
 continuity issue is discussed in the next Subsection.
\end{remark}

\subsection{The doubled cotangent and its vector bundles}
\label{sec:dbcot}
The manifold $X$ is the total space of the cotangent $T^{*}Q$\,,
$Q=Q_{-}\sqcup Q'\sqcup Q_{+}$ and $\overline{X}_{-}=X_{-}\sqcup X'$
is the boundary manifold
$\overline{X}_{-}=T^{*}Q\big|_{\overline{Q}_{-}}$ with boundary
$X'=T^{*}Q\big|_{Q'}$\,. So $X$ (resp. $\overline{X}_{-}$) can be considered as
$\mathcal{C}^{\infty}$-vector bundles on  $Q$
(resp. $\overline{Q}_{-}$) with projections $\pi_{X}:X\to Q$
(resp. $\pi_{\overline{X}_{-}}:\overline{X}_{-}\to
\overline{Q}_{-}$) and as a symplectic manifold (resp. the domain of
a symplectic manifold). We follow the two steps approach of
Subsections~\ref{sec:mfldbdy} and \ref{sec:dbriem} by first considering
the smooth case with a smooth metric $g_{-}=g=g^{TQ}$ and then the symmetric
non smooth metric $\hat{g}^{TQ}$on $Q_{(-\varepsilon,\varepsilon)}$\,,
with
$\hat{g}\big|_{Q_{(-\varepsilon,0]}}=g\big|_{Q_{(-\varepsilon,0]}}=g_{-}\big|_{Q_{(-\varepsilon,0]}}$
and $\hat{g}\big|_{Q_{[0,\varepsilon)}}=
g_{+}\big|_{Q_{[0,\varepsilon]}}$\,.
\begin{definition}
\label{de:SigmaXI} For an interval $I\subset
(-\varepsilon,\varepsilon)$\,, $X_{I}$ will denote
$X\big|_{Q_{I}}=T^{*}Q\big|_{Q_{I}}$\,.\\
  The map $S_{Q,*}:T^{*}Q_{(-\varepsilon,\varepsilon)}\to
  T^{*}Q_{(-\varepsilon,\varepsilon)}$ will be denoted
  $\Sigma:X_{(-\varepsilon,\varepsilon)}\to
  X_{(-\varepsilon,\varepsilon)}$ as a symplectic smooth involution of
  $X_{(-\varepsilon,\varepsilon)}$ with push-forward and pull-back
  $\Sigma_{*}=\Sigma^{*}$ acting on $\Lambda
  TX_{(-\varepsilon,\varepsilon)}$ and $\Lambda
  T^{*}X_{(-\varepsilon,\varepsilon)}$\,.\\
On $F=\Lambda T^{*}X\otimes \pi_{X}^{*}(\fkf)$ or
$F'=\Lambda TX\otimes \pi_{X}^{*}(\fkf)$\,, where $\pi_{\fkf}:\fkf\to
Q$ given of Definition~\ref{de:doublef} depends on the isometric
smooth involution $\nu$\,, the involution $\Sigma_{*}\otimes
\pi_{X}^{*}(S_{Q,\nu})$ will be denoted by $\Sigma_{\nu}$\,.
\end{definition}
By using the decomposition \eqref{eqGaTQ} we can write
\begin{eqnarray*}
  &&
X_{(-\varepsilon,\varepsilon)}=\underline{q}^{1,*}(\rz \underline{e}^{1})+\underline{q}^{*}(T^{*}Q')
\\
&& g^{T^{*}Q}=\underline{e}_{1}\otimes \underline{e}_{1}\oplus^{\perp}m^{T^{*}Q}(\underline{q}^{1})
\end{eqnarray*}
so that $X_{I}=T^{*}I\oplus T^{*}Q'$\,. 
Hence we can write $x\in
X_{(-\varepsilon,\varepsilon)}$ as $x=(q^{1},q',p_{1},p')$ with
$(q^{1},p_{1})\in I\times \rz=T^{*}I$ and $(q',p')\in T^{*}Q'$\,,
$(q^{1},q')=\pi_{X}(x)$ and
\begin{eqnarray*}
  && \Sigma(q^{1},q',p_{1},p')=(-q^{1},q',-p_{1},p')\\
\text{and}&&
2\fkh(x)=g^{T^{*}Q}_{q}(p,p)=p_{1}^{2}+m^{T^{*}Q'}(q^{1},q')(p',p')\,. 
\end{eqnarray*}
The domain
$X_{(-\varepsilon,0]}$ is a natural collar neighborhood of
$X'$ in $\overline{X}_{-}$\,.\\
Additionally, this shows that  the kinetic energy $\fkh$ is not
 invariant by $\Sigma$\,, in general.\\
The latter point 
 is solved by introducing the metric $\hat{g}^{TQ}$ of
Subsection~\ref{sec:dbriem} and the kinetic energy
$$
2\hat{\fkh}(x)=p_{1}^{2}+m^{T^{*}Q}(-|q^{1}|,q')(p',p')\,.
$$
But this leads to a discontinuous  Levi-Civita connection $\nabla^{Q,\hat{g}}$
and therefore to  discontinuous horizontal-vertical decomposition.
This discontinuity must be handled in the vector bundles $E=\Lambda
T^{*}X$\,, $E'=\Lambda TX$\,, $F=E\otimes \pi_{X}^{*}(\fkf)$ and
$F'=E'\otimes \pi_{X}^{*}(\fkf)$\,, where we recall that
$(\fkf,\nabla^{\fkf},\hat{g}^{\fkf})$ used for $F$ and
$(\fkf,\nabla^{\fkf'},\hat{g}^{\fkf'})$ used for $F'$ are two
$\mathcal{C}^{\infty}$-flat vector bundles on $Q$\,, with antidual flat
connections identified via $\hat{g}^{\fkf}$ and possibly different 
$\mathcal{C}^{\infty}$-structures (see
Definition~\ref{de:doublef} and Proposition~\ref{pr:dualf}).
\\
However such a discontinuity  as well as the isometry with the case
when $g_{-}=g_{0}$\,, $\hat{g}_{0}=g_{0}$\,, and all the constructions are smooth\,, can be solved by a  repeated application of
$\fkF=\underline{q}^{\prime *}\fkF\big|_{Q'}$ written in \eqref{eq:FqFhg}.\\
\textbf{A)}$ \boxed{\mathbf\fkF=X_{(-\varepsilon,\varepsilon)}=T^{*}Q_{(-\varepsilon,\varepsilon)}~:}$
With $\fkF=X_{(-\varepsilon,\varepsilon)}=T^{*}Q_{(-\varepsilon,\varepsilon)}$
\eqref{eq:FqFhg} provides the vector bundle isomorphism
$\widehat{\Psi}_{Q}^{g,g_{0}}$ (see diagrams \eqref{eq:psiQgg} and
\eqref{eq:psiQgg2}) which is piecewise $\mathcal{C}^{\infty}$ and
continuous. 
\begin{definition}
\label{de:tildeqp} 
On $X_{(-\varepsilon,\varepsilon)}=T^{*}Q_{(-\varepsilon,\varepsilon)}$ the piecewise
$\mathcal{C}^{\infty}$ and continuous vector bundle isomorphism
$(\widehat{\Psi}^{g,g_{0}}_{Q})$ will be denoted $\hat\varphi_{X}^{g,g_{0}}$ and
the coordinates
$(\tilde{q},\tilde{p})=(\tilde{q}^{1},\tilde{q}',\tilde{p}_{1},\tilde{p}')$
of $x\in TX_{(-\varepsilon,\varepsilon)}$
will be given by
$$
\tilde{q}=q\quad,\quad
\tilde{p}_{1}=p_{1}[(\hat\varphi_{X}^{g,g_{0}})^{-1}(x)]\quad,\quad \tilde{p}'=p'[(\varphi_{X}^{g,g_{0}})^{-1}(x)]\,.
$$
\end{definition}
With those new coordinates
$$
2\fkh(x)=\tilde{p}_{1}^{2}+m^{T^{*}Q'}(0,\tilde{q}')(\tilde{p}',\tilde{p}')\,,
$$
the parallel transport in $X_{(-\varepsilon,\varepsilon)}=T^{*}Q_{(-\varepsilon,\varepsilon)}$
along the geodesic $(\exp_{\underline{q}'}^{Q,\hat{g}}(te_{1}))_{t\in
  (-\varepsilon,\varepsilon))}$ in $Q_{(-\varepsilon,\varepsilon)}$ is
nothing but
$(t,\tilde{q}',\tilde{p}_{1},\tilde{p}')_{t\in
  (-\varepsilon,\varepsilon)}$ and
$e_{1}=\frac{\partial}{\partial
  \tilde{q}^{1}}\in TX^{H}\big|_{X_{-}\cup X_{+}}$\,. Finally the
diagrams \eqref{eq:psiQgg} and \eqref{eq:psiQgg2} and
Definition~\ref{de:SigmaXI} ensure 
\begin{equation}
  \label{eq:Sigmati}
  \Sigma(\tilde{q}^{1},\tilde{q}',\tilde{p}_{1},\tilde{p}')=(-\tilde{q}^{1},\tilde{q}',-\tilde{p}_{1},\tilde{p}')\,.
\end{equation}
However this change of variables does not preserve the symplectic form
$\sigma$ on $X_{(-\varepsilon,\varepsilon)}$\,. We won't use the coordinates $(\tilde{q},\tilde{p})$ when the
symplectic structure of $X$ is required.\\
\noindent\textbf{B)}$\boxed{\fkF=\Lambda
  T^{*}Q_{(-\varepsilon,\varepsilon)}\otimes \Lambda
  TQ_{(-\varepsilon,\varepsilon)}\otimes \fkf~:}$ We focus on
$F=\Lambda T^{*}X\otimes \pi_{X}^{*}(\fkf)$ but the
constructions have obvious translations in
 $F'=\Lambda TX\otimes \pi_{X}^{*}(\fkf)$\,.\\
Although the horizontal-vertical decomposition $T^{*}X=T^{*}X^{H}\oplus
T^{*}X^{V}$ made on $X_{-}$ with $g_{-}=\hat{g}\big|_{X_{-}}$ and on
$X_{+}$ with $g_{+}=\hat{g}\big|_{X_{+}}$ first appears discontinous,
with 
$$
F\big|_{\overline{X}_{-}}\stackrel{g_{-}}{\simeq}\pi_{X}^{*}(\fkF\big|_{\overline{Q}_{-}})\quad,\quad F\big|_{\overline{X}_{+}}\stackrel{g_{+}}{\simeq}\pi_{X}^{*}(\fkF\big|_{\overline{Q}_{+}})\,,
$$
we may define a continuous vector bundle after taking a quotient via
the map $\pi_{X,*}$\,.
\begin{definition}
\label{de:hatEF} The vector bundle $\hat{F}_{g}$ is defined 
is defined as the quotient vector bundle 
\begin{eqnarray*}
  &&
\hat{F}_{g}=\left(F\big|_{\overline{X}_{-}}\sqcup
     F\big|_{\overline{X}_{+}}\right)/\sim
\\
&& 
F\big|_{\overline{X}_{-}}\stackrel{g_{-}}{\simeq}\pi_{X}^{*}(\fkF\big|_{\overline{Q}_{-}})\quad,\quad
   F\big|_{\overline{X}_{+}}\stackrel{g_{+}}{\simeq}\pi_{X}^{*}(\fkF\big|_{\overline{Q}_{+}})\,,\quad
   \mathcal{F}=\Lambda T^{*}Q\otimes \Lambda TQ\otimes \fkf\,,\\
&&
\left(
   \begin{array}[c]{c}
     (x_{-},v_{-})\sim (x_{+},v_{+})\\
   (x_{\mp},v_{\mp})\in F\big|_{\overline{X}_{\mp}}
   \end{array}
\right)
\Leftrightarrow\left(
  \begin{array}[c]{c}
    x_{-}=x_{+}\in X'=\partial X_{\mp}\\
   \pi_{X,*}(v_{-})=\pi_{X,*}(v_{+}) \in \mathcal{F}_{\pi_{X}(x)}\,.
  \end{array}
\right)
\end{eqnarray*}
The same definition is used for the continuous vector bundles 
$\hat{E}_{g},\hat{E}'_{g},\hat{F}'_{g}$ by using respectively
$\fkF=\Lambda T^{*}Q\otimes \Lambda TQ$\,, $\Lambda TQ\otimes \Lambda
T^{*}Q$ and $\Lambda TQ\otimes \Lambda T^{*}Q\otimes\fkf$\,.
\end{definition}
By construction $\hat{E}_{g}$\,, $\hat{E}'_{g}$\,, $\hat{F}_{g}$ and
$\hat{F}'_{g}$ are piecewise $\mathcal{C}^{\infty}$ and continuous
vector bundles. Additionally because $g_{-}\big|_{Q'}=g_{+}\big|_{Q'}$
the horizontal vertical decomposition coincide in
$\hat{F}_{g}\big|_{X'}$ (resp. $\hat{E}_{g}\big|_{X'}$\,,
$\hat{E}'_{g}\big|_{X'}$\,, $\hat{F}'_{g}\big|_{X'}$) and the metric 
$g_{-}^{F}\big|_{X'}$ (resp. $g_{-}^{E}\big|_{X'}$\,,
$g_{-}^{E'}\big|_{X'}$\,, $g_{-}^{F'}\big|_{X'}$) and
$g_{+}^{F}\big|_{X'}$  (resp. $g_{+}^{E}\big|_{X'}$\,,
$g_{+}^{E'}\big|_{X'}$\,, $g_{+}^{F'}\big|_{X'}$) coincide. 
Therefore we can write
\begin{eqnarray}
\label{eq:hatehatg}
  &&\hat{E}_{g}\stackrel{\hat{g}}{\simeq}\pi_{X}^{*}(\Lambda T^{*}Q\otimes
     \Lambda TQ)\,,\\
\label{eq:hatfhatg}
&& \hat{F}_{g}\stackrel{\hat{g}}{\simeq}\pi_{X}^{*}(\Lambda T^{*}Q\otimes
     \Lambda TQ\otimes \fkf)\,,\\
 \label{eq:hatephatg}
 &&\hat{E}'_{g}\stackrel{\hat{g}}{\simeq}\pi_{X}^{*}(\Lambda TQ\otimes
     \Lambda T^{*}Q)\,,\\
\label{eq:hatfphatg}&& \hat{F}'_{g}\stackrel{\hat{g}}{\simeq}\pi_{X}^{*}(\Lambda TQ\otimes
     \Lambda T^{*}Q\otimes \fkf)\,,
\end{eqnarray}
where the identification holds in the class of piecewise
$\mathcal{C}^{\infty}$ and continuous vector bundles. Remember that the doubled fiber bundle
$\pi_{\fkf}:\fkf\to Q$ is the one of Definition~\ref{de:doublef} with
the flat connection $\nabla^{\fkf}$ for $\hat{F}_{g}$ while its antidual
version of Proposition~\ref{pr:dualf} is used for $\hat{F}'_{g}$\,.\\
 The complexification of  $\fkF=E,E'$ is of course a particular
case of $\fkF=F,F'$ with $\fkf=Q\times \cz$ endowed with the trivial
metric and flat connection and $\nu=1$\,. But it is convenient to have
a specific notation.\\
The metric
$\hat{g}^{F}=1_{\underline{X}_{-}}(x)g_{-}^{F}+1_{X_{+}}(x)g_{+}^{F}$ (resp.
$\hat{g}^{E,E',F'}=1_{\underline{X}_{-}}(x)g_{-}^{E,E',F'}+1_{X_{+}}(x)g_{+}^{E,E',F'}$)
is a piecewise $\mathcal{C}^{\infty}$ and continuous metric on
$\hat{F}_{g}$ (resp. $\hat{E}_{g}$\,, $\hat{E}'_{g}$\,,
$\hat{F}'_{g}$)\,.\\
When $g_{-}^{TQ}=g_{0}^{TQ}$\,, the quotient vector bundle
$\hat{F}_{g_{0}}$ is nothing but $F$\,.\\
The
diagram \eqref{eq:psiQgg2} is associated with
$\fkF=\underline{q}^{\prime *}(\fkF\big|_{Q'})$\,.  Actually this can be
lifted to $X$ as follows: 
\begin{itemize}
\item 
By A), the exponential map
$\exp_{\underline{q}'}^{Q,\hat{g}}(t\underline{e}_{1})$ is lifted to
$X_{(-\varepsilon,\varepsilon)}$ as
$(t,\tilde{q}',\tilde{p}_{1},\tilde{p}')$ and this defines the map
$\tilde{x}': X_{(-\varepsilon,\varepsilon)}\to X'$ with
$\tilde{x}'(\tilde{q}^{1},\tilde{q}',\tilde{p}_{1},\tilde{p})=(0,\tilde{q}',\tilde{p}_{1},\tilde{p}')$\,. 
Hence we get
$
\hat{F}_{g}=\tilde{x}^{\prime *}(\hat{F}_{g}\big|_{X'})$\,, 
where the parallel transport for the connection 
$\nabla^{Q,\hat{g}}+\nabla^{\fkf}$ along
$(\exp_{\underline{q}'}^{Q,\hat{g}}(t\underline{e}_{1}))_{t\in
  (-\varepsilon,\varepsilon)}$ is lifted to the parallel transport
along $t\mapsto (t,\tilde{x}')$ for $\nabla^{F,\hat{g}}$\,. 
When $\hat{F}_{g}\big|_{X_{(-\varepsilon,\varepsilon)}}$ is endowed
with the metric $\pi_{X}^{*}(g^{\Lambda T^{*}Q}\otimes g^{\Lambda
  T^{*}Q}\otimes g^{\fkf})$\,, pulling back \eqref{eq:psiQgg2} to
$X_{(-\varepsilon,\varepsilon)}$ says $\tilde{x}^{\prime
  *}(\hat{F}_{g}\big|_{X'})$ is isometric to
$F\big|_{X_{(-\varepsilon,\varepsilon)}}$ endowed with the metric
$\pi_{X}^{*}(g_{0}^{\Lambda T^{*}Q}\otimes g_{0}^{\Lambda TQ}\otimes g^{\fkf})$\,.
\item  The weight $\langle p\rangle_{q,\hat{g}}$ involved in the metrics
$\hat{g}^{E}$ and $g^{F}$ satisfies
$$
\langle
p\rangle_{\hat{g},q}=\sqrt{1+\tilde{p}_{1}^{2}+m^{i'j'}(0,\tilde{q}')\tilde{p}_{i'}\tilde{p}_{j'}}\,,
$$
and it is constant along the curve $t\mapsto (t,\tilde{x}')$\,.
With the identification
 $\hat{F}_{g}=\tilde{x}^{\prime *}(\hat{F}_{g}\big|_{X'})$
the weight $\langle p\rangle_{\hat{g},q}^{N_{H}-N_{V}}$ is thus sent to
$\langle p\rangle_{g_{0},q}^{N_{H}-N_{V}}$\,.
\item It will appear with the explicit coordinate writing
  \eqref{eq:Sigmae1}\eqref{eq:Sigmae2} or with the
  $\mathcal{C}^{\infty}$-structure associated with $\hat{F}_{g}$ in
  Subsection~\ref{sec:diffhatE} that the isometric involution on
  $\hat{F}_{g}\big|_{X\setminus X'}=F\big|_{X\setminus X'}$ is well
  defined on the quotient vector bundle $\hat{F}_{g}$\,.
\end{itemize}
\begin{definition}
\label{de:hPsigg0}
  The identification $\hat{F}_{g}=\tilde{x}^{\prime *}(F\big|_{X'})$ provides a
  piecewise $\mathcal{C}^{\infty}$ and continuous vector bundle
  isometry from
  $(F\big|_{X_{(-\varepsilon,\varepsilon)}},\hat{g}_{0}^{F}=g_{0}^{E}\otimes \hat{g}^{\fkf})$ to
  $(\hat{F}_{g}\big|_{X_{(-\varepsilon,\varepsilon)}},\hat{g}^{F})$ which is denoted by
  $\widehat{\Psi}_{X}^{g,g_{0}}$\,. The same notation is used for
  $\widehat{\Psi}_{X}^{g,g_{0}}:
  (F'\big|_{X_{(-\varepsilon,\varepsilon)}},\hat{g}_{0}^{F'}=g_{0}^{E'}\otimes
\hat{g}^{\fkf})\to
  (\hat{F}'_{g}\big|_{X_{(-\varepsilon,\varepsilon)}},\hat{g}^{F'})$ and when $F,F'$ are
  replaced by $E,E'$\,.
\end{definition}
The diagram \eqref{eq:psiQgg2} is now lifted to 
\begin{equation}
 \label{eq:psiXgg}
\xymatrix{
(F\big|_{X_{(-\varepsilon,\varepsilon)}},
 g_{0}^{E}\otimes \hat{g}^{\fkf}) 
 \ar[d]_{\Sigma_{0,\nu}}
 \ar[r]^{\widehat{\Psi}_{X}^{g,g_{0}}} 
 &
 (\hat{F}_{g}\big|_{X_{(-\varepsilon,\varepsilon)}},\hat{g}^{F})
 \ar[d]^{\Sigma_{\nu}}\quad\\ 
 (F\big|_{X_{(-\varepsilon,\varepsilon)}},
 g_{0}^{E}\otimes\hat{g}^{\fkf}) 
 \ar[u]\ar[d]_{\pi_{F}}
 \ar[r]^{\widehat{\Psi}_{X}^{g,g_{0}}} 
 &
 (\hat{F}_{g}\big|_{X_{(-\varepsilon,\varepsilon)}},
\hat{g}^{F})
 \ar[u]\ar[d]^{\pi_{F}}\quad\\ 
 X_{(-\varepsilon,\varepsilon)}
 \ar[r]_{\hat\varphi_{X}^{g,g_{0}}}& 
 X_{(-\varepsilon,\varepsilon)}
 \quad\\ 
 X'
 \ar[u]_{i}
 \ar[r]_{\mathrm{Id}_{X'}}& 
 X'\ar[u]_{i}
}
\end{equation}
with similar diagrams when $F$ is replaced by $E,E',F'$\,.
\begin{remark}
The $\mathcal{C}^{\infty}$-structure of $\hat{F}_{g}$ is based on the
non symplectic coordinates $(\tilde{q},\tilde{p})$ with the collar
neighborhood $\tilde{q}^{1}\in (-\varepsilon,0]$ and an additional
twist presented in Subsection~\ref{sec:diffhatE}. However many
different structures have to be considered in this analysis: the
differential structure, the symplectic structure and the riemannian
structure. The presentation of $\hat{F}_{g}$ as a piecewise
$\mathcal{C}^{\infty}$ and continuous vector bundle, defined as a
quotient, is actually the one where all those different aspects are
simply formulated.
\end{remark}

\subsection{The geometric constructions in local coordinates}
\label{sec:loccoor}
The previous sections ensure that the various changes of variables or
isomorphisms of vector bundles have a natural geometric meaning
independent of a coordinate system. However it is instructive for the
analysis to explain them in terms of some specific
 local coodinates systems. Let us first recall a few facts about the
 smooth case and then we will shift to the description of piecewise
 $\mathcal{C}^{\infty}$ isometries associated with the possibly non
 smooth metric $\hat{g}$\,.

\subsubsection{The smooth case}
\label{sec:coorsmooth}
Let us start with a smooth riemannian manifold $(Q,g)$ and a local
coordinate system $(\underline{q}^{1},\ldots,\underline{q}^{d})$ in an
open neighborhood $U$ of $q_{0}\in Q$\,. The Levi-Civita connection
$\nabla^{Q,g}$ associated with
$g=g^{TQ}=g_{ij}(\underline{q})d\underline{q}^{i}d\underline{q}^{j}$
and the dual metric
$g^{T^{*}Q}(p,p)=g^{ij}(\underline{q})p_{i}p_{j}$ can be specified
with the Christoffel symbols
\begin{equation}
  \label{eq:christof}
\Gamma_{ij}^{k}(\underline{q})=\frac{1}{2}g^{k\ell}\left[\frac{\partial
    g_{j\ell}}{\partial \underline{q}^{i}}+\frac{\partial
    g_{i\ell}}{\partial \underline{q}^{j}}-\frac{\partial
    g_{ij}}{\partial \underline{q}^{\ell}} \right]\,.
\end{equation}
It is given by
$$
\nabla^{Q,g}_{\frac{\partial}{\partial \underline{q}^{i}}}\frac{\partial}{\partial
  \underline{q}^{j}}=\Gamma_{ij}^{k}(\underline{q})\frac{\partial}{\partial \underline{q}^{k}}\quad,\quad
\nabla^{Q,g}_{\frac{\partial}{\partial \underline{q}^{i}}}d\underline{q}^{j}=-\Gamma_{ik}^{j}(\underline{q})d\underline{q}^{k}\,.
$$
Because it is torsion free we have the symmetry
$\Gamma_{ij}^{k}=\Gamma_{ji}^{k}$\,.\\
An horizontal curve $t\mapsto (q(t),p(t))$ on $X=T^{*}Q$ being
characterized by
$\nabla_{\dot{q}(t)}^{Q,g}p(t)=0$\,, a basis
of $TX^{H}$\,, $e_{i}\in TX^{H}$ such that
$\pi_{X,*}(e_{i})=\frac{\partial}{\partial q^{i}}$\,,  and
$\hat{e}^{i}\in TX^{V}$\,,
$\pi_{X,*}(\hat{e}^{i})=d\underline{q}^{i}$\,, is thus given by
\begin{equation}
  \label{eq:defeiej1}
e_{i}=\frac{\partial}{\partial
  q^{i}}+\Gamma_{ij}^{k}(q)p_{k}\frac{\partial}{\partial
  p_{j}}\in T_{(q,p)}X^{H}\quad,\quad
\hat{e}^{j}=\frac{\partial}{\partial p_{j}}\in T_{(q,p)}X^{V}\,.
\end{equation}
Its dual basis on $T^{*}X$ is
\begin{equation}
  \label{eq:defeiej2}
e^{i}=dq^{i}\in T^{*}_{(q,p)}X^{H}\quad,\quad
\hat{e}_{j}=dp_{j}-\Gamma_{ji}^{k}(q)p_{k}dq^{i}\in T^{*}_{(q,p)}X^{V}\,. 
\end{equation}
Due to the possible curvature of $(Q,g^{TQ})$ 
\begin{equation}
  \label{eq:curveiej}
[e_{i},e_{j}]=R_{ijk}^{TQ;\ell}(q)p_{\ell}\frac{\partial}{\partial
  p_{\ell}}\in TX^{V}
\end{equation}
where the Riemann curvature tensor $R^{TQ}=R^{TQ}_{ij}d\underline{q}^{i}\wedge d\underline{q}^{j}$ is the
$\mathrm{End}(TQ)$ valued two-form given
$$
R^{TQ}(S,T)=\nabla_{S}^{TQ,g}\nabla_{T}^{TQ,g}-\nabla_{T}^{TQ,g}\nabla_{S}^{TQ,g}-\nabla^{TQ,g}_{[S,T]}
$$
when $S=S^{i}(\underline{q})\frac{\partial}{\partial \underline{q}^{i}}$ and
$T(\underline{q})=T^{j}(\underline{q})\frac{\partial}{\partial
  \underline{q}^{j}}$\,.\\
However $(e_{i},\hat{e}^{j})$ is a symplectic basis of $TX=T(T^{*}Q)$
endowed with its canonical symplectic form
$\sigma\stackrel{loc}{=}dp_{i}\wedge dq^{i}$ and
$$
\sigma=dp_{i}\wedge dq^{i}=\hat{e}_{i}\wedge e^{i}\,.
$$
General elements of $E=\Lambda T^{*}X$  and $E'=\Lambda TX$  will be
written  locally
$$
\omega_{I}^{J}e^{I}\hat{e}_{J}\quad,\quad u^{I}_{J}e_{I}\hat{e}^{J}\,,
$$
with repeated summation convention w.r.t $I,J\subset
\left\{1,\ldots,d\right\}$\,.\\
The vertical, total and
horizontal number operators, $N_{V}, N_{H}$ and $N$\,, are given on 
$\Lambda T^{*}X$ and $\Lambda TX$ by
\begin{eqnarray*}
  &&
     N_{V}(\omega_{I}^{J}e^{I}\hat{e}_{J})=|J|\omega_{I}^{J}e^{I}\hat{e}_{J}\quad,\quad
     N(\omega_{I}^{J}e^{I}\hat{e}_{J})=(|I|+|J|)\omega_{I}^{J}e^{I}\hat{e}_{J}\quad,\quad
     N_{H}=N-N_{V}\\
&&N_{H}(u^{I}_{J}e_{I}\hat{e}^{J})=|I|u^{I}_{J}e_{I}\hat{e}^{J}\quad,\quad
   N(u^{I}_{J}e_{I}\hat{e}^{J})=(|I|+|J|)u^{I}_{J}e_{I}\hat{e}^{J}\quad,\quad N_{V}=N-N_{H}\,.
\end{eqnarray*}
Because the Levi-Civita connection preserves the metric on $X=T^{*}Q$
the horizontal vector fields are tangent to $|p|_{q}^{2}=Cte$~:
$$
e_{i}f(|p|_{q}^{2})=0\quad,\quad
|p|_{q}^{2}=g^{ij}(q)p_{i}p_{j}=2\fkh(q,p)\quad,\quad f\in
\mathcal{C}^{1}(\rz;\rz)\,. 
$$
When necessary we will specify the metric $g$ in the notation
with an additional index by writing $|p|_{q}^{2}=|p|^{2}_{g,q}$ and
$\langle p\rangle_{g,q}=\langle p\rangle_{q}$\,.\\
The metric $g^{E}=\langle p\rangle_{q}^{-N_{H}+N_{V}}\pi_{X}^{*}(g^{\Lambda T^{*}Q}\otimes g^{\Lambda
  TQ})$ and $g^{E'}=\langle
p\rangle_{q}^{N_{H}-N_{V}}\pi_{X}(g^{\Lambda TQ}\otimes g^{\Lambda
  T^{*}Q})$ on $E=\Lambda T^{*}X$ and $E'=\Lambda TX$ 
already
introduced in \eqref{eq:gE} and \eqref{eq:gEpr} are such that
\begin{eqnarray}
\label{eq:normei}
&& |e^{i}|_{(q,p)}=|dq^{i}|_{(q,p)}=\langle
   p\rangle_{q}^{-1/2}\sqrt{g^{ii}(q)}\,,\\
\label{eq:normehej}
&&|\hat{e}_{j}|_{(q,p)}=\left|dp_{j}-\Gamma_{ji}^{k}p_{k}dq^{i}\right|_{(q,p)}=\langle
   p\rangle_{q}^{1/2}\sqrt{g_{jj}(q)}=\mathcal{O}(\langle
   p\rangle_{q})|e^{i}|_{(q,p)}\,,\\
&& \langle e^{i},\hat{e}_{j}\rangle_{g^{E}}=0\,,\\
\label{eq:normefei}
&&    |\hat{e}^{j}|_{(q,p)}=\left|\frac{\partial}{\partial
     p_{j}}\right|_{(q,p)}=\langle
   p\rangle_{q}^{-1/2}\sqrt{g^{jj}(q)}\,,\\
\label{eq:normefhej}
  && |e_{i}|_{(q,p)}=\left|\frac{\partial}{\partial
     q^{i}}+\Gamma_{ij}^{k}p_{k}\frac{\partial}{\partial
     p_{j}}\right|_{(q,p)}=\langle
     p\rangle_{q}^{1/2}\sqrt{g_{ii}(q)}=\mathcal{O}(\langle
     p\rangle_{q})|\hat{e}^{j}|_{(q,p)}\,\\
&&\langle e_{i}\,,\, \hat{e}_{j}\rangle_{g^{E'}}=0\,.
\end{eqnarray}
With this choice the riemannian volume on $X$\,, $\mathrm{vol}_{g^{E'}}$\,, is
nothing but the symplectic volume
$$
d\mathrm{vol}_{g^{E'}}=|dqdp|=\frac{1}{d!}|\sigma^{d}|\quad,\quad d=\dim Q\,,
$$
and coincides with the standard Lebesgue measure in any symplectic
coordinates system. Note that $X$ is orientable with the non vanishing
volume form $\frac{1}{d!}\sigma^{d}$\,.\\
The connections $\nabla^{E}$ and $\nabla^{E'}$ introduced in
\eqref{eq:nabEg} satisfy
\begin{eqnarray*}
&&
\nabla_{e_{i}}^{E}e^{\ell}=-\Gamma_{ik}^{\ell}(q)e^{k}\quad,\quad
   \nabla_{e_{i}}^{E}\hat{e}_{j}=\Gamma_{ij}^{k}(q)\hat{e}_{k}\quad,\quad
   \nabla_{\hat{e}^{j}}^{E}e^{\ell}=\nabla^{E}_{\hat{e}_{j}}\hat{e}_{k}=0\,,\\
  &&
\nabla_{e_{i}}^{E'}e_{\ell}=\Gamma_{i\ell}^{k}(q)e_{k}\quad,\quad
\nabla^{E'}_{e_{i}}\hat{e}^{j}=-\Gamma_{ik}^{j}(q)\hat{e}^{k}\quad,\quad
\nabla_{\hat{e}^{j}}^{E'}e_{i}=\nabla_{\hat{e}_{j}}^{E'}\hat{e}^{\ell}=0\,.
\end{eqnarray*}
Remember that they are defined as pull-backed connections and do not
coincide exactly with the Levi-Civita connection associated with
$g^{E'}$ due to the weight $\langle p\rangle_{q}^{N_{H}-N_{V}}$\,. \\
Let us finish with the flat vector bundle $\fkf$ endowed with the flat
connection $\nabla^{\fkf}$ and the hermitian metric
$g^{\fkf}$\,. Locally above $U\ni q_{0}$\,, there is a frame
$(v^{1},\ldots,v^{d_{f}})$ of $\fkf$ such that
$\nabla_{\frac{\partial}{\partial \underline{q}^{i}}}^{\fkf}v^{k}=0$
and $(\fkf,\nabla^{\fkf})\simeq (U\times \cz^{d_{f}},\nabla)$ with the
trivial connection $\nabla$ and the covariant derivative
$\nabla_{\frac{\partial}{\partial
    \underline{q}_{i}}}=\frac{\partial}{\partial
  \underline{q}_{i}}$\,. The metric
is given by the matrix
$A^{ij}(\underline{q})=g^{\fkf}(v^{i},v^{j})$\,,
$A(\underline{q})=A(\underline{q})^{*}$ and
$(A(\underline{q})^{-1})_{ij}=A_{ij}(\underline{q})$\,. 
The antidual $(\fkf',\nabla^{\fkf'})$ is also isomorphic
to $ (U\times \cz^{d}, \nabla)$
but identifying $\fkf'$ with $\fkf$  via
$g^{\fkf}(v_{\ell}^{*},v^{k})=\delta_{\ell}^{k}$ says
$v_{\ell}^{*}=A^{-1}(\underline{q})v^{\ell}=A_{\ell,k}(\underline{q})v^{k}$\,. With
$$
0=\nabla^{\fkf'}_{\frac{\partial}{\partial
    \underline{q}^{i}}}v_{\ell}^{*}=
\nabla^{\fkf'}_{\frac{\partial}{\partial \underline{q}^{i}}}
[A^{-1}(\underline{q})v^{\ell}]
=(\partial_{\underline{q}^{i}}A^{-1}) v^{\ell}+A^{-1}\nabla^{\fkf'}_{\frac{\partial}{\partial
    \underline{q}^{i}}}v^{\ell}
$$
we deduce
$$
\nabla^{\fkf'}v^{\ell}-\underbrace{
\nabla^{\fkf}v^{\ell}}_{=0}=-A(dA^{-1})(\underline{q})v^{\ell}=(dA)A^{-1}(\underline{q})v^{\ell}\,.
$$
Hence we obtain
\begin{eqnarray*}
  &&
     \omega(g^{\fkf},\nabla^{\fkf})=\nabla^{\fkf'}-\nabla^{\fkf}=
     (dA)A^{-1}(\underline{q})\,.
\\
&&\nabla^{\fkf,u}=\nabla^{\fkf}+\frac{1}{2}(dA)A^{-1}(\underline{q})
\stackrel{\fkf\simeq
   I\times \cz^{d_{f}}}{=}
\nabla +\frac{1}{2}(dA)A^{-1}(\underline{q})\,.
\end{eqnarray*}
We are especially interested in the line bundle $\fkf=Q\times \cz$
with $\nabla^{\fkf}=\nabla$ and $g^{\fkf}(z)=e^{-2V(q)}|z|^{2}$\,. 
Then $\nabla^{\fkf'}=\nabla -2(dV(\underline{q}))$ and
$\nabla^{\fkf,u}=\nabla-dV(\underline{q})$\,.\\
By conjugating with $e^{-V(q)}$\,, we can actually consider
$\fkf=Q\times\cz$ with $g^{\fkf}(z')=|z'|^{2}$ with the flat
connection $\nabla^{\fkf}=\nabla+dV(q)$ and its dual flat connection $\nabla^{\fkf'}=\nabla-dV(q)$\,.
\subsubsection{The non smooth doubles with the metric $\hat{g}^{TQ}$}
\label{sec:coornonsmooth}
With $Q=Q_{-}\sqcup Q'\sqcup Q_{+}$ local coordinates in a
neighborhood $U$ of $q_{0}\in Q'$\,, ($U\subset Q_{(-\varepsilon,\varepsilon)}$)
are chosen such  that:
\begin{itemize}
\item
  $g_{-}^{TQ}=(d\underline{q}^{1})^{2}\oplus^{\perp}m^{i'j'}(\underline{q}^{1},\underline{q}')d\underline{q}^{i'}d\underline{q}^{j'}$\,.
\item 
$\hat{g}^{TQ}=(d\underline{q}^{1})^{2}\oplus^{\perp}m^{i'j'}(-|\underline{q}^{1}|,\underline{q}')d\underline{q}^{i'}d\underline{q}^{j'}$\,,
$\hat{m}=m^{TQ'}(-\underline{q}^{1},\underline{q}')$ and the
corresponding Christoffel symbols are denoted by
$\hat{\Gamma}_{ij}^{k}$\,. We will keep the notation $\Gamma_{ij}^{k}$
for $g^{TQ}=g^{TQ}_{-}$\,.
\item The associated symplectic coordinates on $\pi_{X}^{-1}(U)\simeq
  U\times\rz^{d}$ are written $(q,p)=(q^{i},p_{j})_{1\leq i,j\leq d}$ and
 $X'\cap \pi_{X}^{-1}(U)=\left\{(q^{1},q',p_{1},p')\in U\times \rz^{d}\,, q^{1}=0\right\}$\,.
\end{itemize}
Remember the convention that $i'$ (resp. $I'\subset \left\{1,\ldots
  d\right\}$) denotes an index $i'\neq 1$ (resp. $1\not\in I'$)\,.\\
Three things must be noticed with those coordinates
\begin{itemize}
\item From \eqref{eq:christof} a Christoffel symbol
  $\Gamma^{k}_{ij}(\underline{q})$ vanishes if $1$ appears more than
  once in $i,j,k$\,.
\item For the metric $\hat{g}$\,, the Christoffel symbols
  $\hat{\Gamma}_{i'j'}^{k'}(\underline{q})$ are continuous but not
  $\mathcal{C}^{1}$ on
  $Q_{(-\varepsilon,\varepsilon)}$ while the possibly non continuous Christoffel symbols
  $\hat{\Gamma}_{1i'}^{k'}=\hat{\Gamma}_{i'1}^{k'}=\frac{\partial
    \hat{m}_{i'k'}}{2\partial \underline{q}^{1}}$ and
  $\hat{\Gamma}^{1}_{i'j'}=-\frac{\partial\hat{m}_{i'j'}}{2\partial\underline{q}^{1}}$
  satisfy
  $\hat{\Gamma}_{1i'}^{k'}(0^{+},\underline{q}')=-\hat{\Gamma}_{1i'}^{k'}(0^{-},\underline{q}')$
  and
  $\hat{\Gamma}_{i'j'}^{1}(0^{+},\underline{q}')=-\hat{\Gamma}(0^{-},\underline{q}')$\,.
\item When
  $g^{TQ}=g_{0}^{TQ}=\hat{g}^{TQ}=(d\underline{q}^{1})^{2}+m^{TQ'}(0,\underline{q}')$
  everything is smooth and $\Gamma_{ij}^{k}=0$ when $1$ appears in $i,j,k$\,.
\end{itemize}
\begin{definition}
\label{de:frameef}
Work with the local coordinates $(q,p)=(q^{1},\ldots,q^{d},p_{1},\ldots, p_{d})$
in $\pi_{X}^{-1}(U)$\,, $q_{0}\in Q'\cap U\subset Q_{(-\varepsilon,\varepsilon)}$\,.\\
 The frame \eqref{eq:defeiej1} and \eqref{eq:defeiej2} associated with
 $g_{-}^{TQ}$ (resp. $g_{+}^{TQ}$) are  denoted
by $(e_{-,i},\hat{e}_{-}^{j})$ and
 $(e^{i}_{-},\hat{e}_{-,j})$
  (resp.   $(e_{-,i},\hat{e}_{-}^{j})$ and
  $(e^{i}_{-},\hat{e}_{-,j})$).\\
The abbreviated version is simply
 $(e_{\mp},\hat{e}_{\mp})$\,.\\
The notations $(e_{i},\hat{e}^{j})$ and $(e^{i},\hat{e}_{j})$ now
refer to the metric $\hat{g}$ with
$(e,\hat{e})=1_{Q_{-}}(q)(e_{-},\hat{e}_{-})+1_{Q_{+}}(q)(e_{+},\hat{e}_{+})$
on $\pi_{X}^{-1}(U)\setminus X'$\,, while working in $\hat{E}_{g}$\,,
$\hat{E}'_{g}$\,, $\hat{F}_{g}$\,, $\hat{F}'_{g}$  means that $(e_{-},\hat{e}_{-})$ and $(e_{+},\hat{e}_{+})$
are identified along $X'$ and $(e,\hat{e})$ makes sense on $\pi_{X}^{-1}(U)$\,.\\
When $g^{TQ}=g_{0}^{TQ}$ those frames are simply denoted
$(f_{i},\hat{f}^{j})$ and $(f^{i},\hat{f}_{j})$\,. 
\end{definition}
Below are the detailed expressions of those frames in the coordinates $(q,p)$:
\begin{eqnarray}
\label{eq:fv}
f_{1}&=&\frac{\partial}{\partial q^{1}}\quad,\quad
  f_{i'}=\frac{\partial}{\partial
  q^{i'}}+\Gamma_{i',j'}^{k'}(0,q')p_{k'}\frac{\partial}{\partial
  p_{j'}}\quad,\quad
\hat{f}^{j}=\frac{\partial}{\partial p_{j}}\,,\\
\label{eq:ff}
f^{i}&=&dq^{i}\quad,\quad \hat{f}_{1}=dp_{1}\quad,\quad
         \hat{f}_{j'}=dp_{j'}-\Gamma_{j'i'}^{k'}(0,q')p_{k'}dq^{i'}\,,\\
\label{eq:ev1}
e_{\mp, 1}&=&\frac{\partial}{\partial
   q^{1}}+\hat{\Gamma}_{1j'}^{k'}(q)p_{k'}\frac{\partial}{\partial
   p_{j'}}\stackrel{\text{on}~X'}{=}f_{1}\pm\Gamma_{1j'}^{k'}(0,q')p_{k'}\frac{\partial}{\partial
   p_{j'}}\quad,\\
\label{eq:ev}
e_{\mp, i'}&=&\frac{\partial}{\partial
   q^{i'}}+\hat{\Gamma}_{i'j'}^{k'}(q)p_{k'}\frac{\partial}{\partial
   p_{j'}}+\hat{\Gamma}_{i'1}^{k'}(q)p_{k'}\frac{\partial}{\partial
   p_{1}}+\hat{\Gamma}_{i'j'}^{1}(q)p_{1}\frac{\partial}{\partial
   p_{j'}}\\
\label{eq:ev0}
&\stackrel{\text{on}~X'}{=}&f_{i'}\pm \Gamma_{i'1}^{k'}(0,q')p_{k'}\frac{\partial}{\partial
   p_{1}}\pm \Gamma_{i'j'}^{1}(0,q')p_{1}\frac{\partial}{\partial
   p_{j'}}\quad,\\
\label{eq:hev}
\hat{e}_{\mp}^{j}&=&\frac{\partial}{\partial
                     p_{j}}=\hat{f}^{j}
\\
\label{eq:ef}
e_{\mp}^{i}&=&dq^{i}=f^{i}
\\
\label{eq:hef1}
\hat{e}_{\mp,1}&=&
   dp_{1}-\hat{\Gamma}_{1i'}^{k'}(q)p_{k'}dq^{i'}
\stackrel{\text{on}~X'}{=}\hat{f}_{1}\mp \Gamma_{1i'}^{k'}(0,q')p_{k'}dq^{i'}\quad,
\\
\label{eq:hef}
\hat{e}_{\mp, j'}&=&
   dp_{j'}-\hat{\Gamma}_{j'i'}^{k'}(q)p_{k'}dq^{i'}
-\hat{\Gamma}_{j'1}^{k'}(q)p_{k'}dq^{1}
-\hat{\Gamma}_{j'i'}^{1}(q)p_{1}dq^{i'}\\
\label{eq:hef0}
&\stackrel{\text{on}~X'}{=}&
f_{j'}\mp \Gamma_{j'1}^{k'}(0,q')p_{k'}dq^{1}
\mp \Gamma_{j'i'}^{1}(0,q')p_{1}dq^{i'}
\,.
\end{eqnarray}
We see in particular
on \eqref{eq:ev1}\eqref{eq:ev0}\eqref{eq:hev} for $E'=\Lambda TX$ and
on \eqref{eq:ef}\eqref{eq:hef1}\eqref{eq:hef0} for $E=\Lambda
T^{*}X$ that
\begin{eqnarray}
\label{eq:Sigmae1}
  &&
     \Sigma_{*}(e_{-,i})\big|_{X'}=(-1)^{\delta_{i1}}e_{+,i}\big|_{X'}\quad,\quad
     \Sigma_{*}(\hat{e}_{-}^{j})\big|_{X'}=(-1)^{\delta_{j1}}\hat{e}_{+}^{j}\big|_{X'}\\
\label{eq:Sigmae2}
&&
     \Sigma_{*}(e_{-}^{i})\big|_{X'}=(-1)^{\delta_{i1}}e_{+}^{i}\big|_{X'}\quad,\quad
     \Sigma_{*}(\hat{e}_{-,j})\big|_{X'}=(-1)^{\delta_{j1}}\hat{e}_{+,j}\big|_{X'}\,.
\end{eqnarray}
The coordinates $(\tilde{q},\tilde{p})$ introduced in
Definition~\ref{de:tildeqp}, and the expression of the frame
$(e,\hat{e})$ in those new coordinates can be specified locally. They are characterized by
$$
\tilde{q}=q\quad,\quad e_{1}\tilde{p}=0\quad,\quad \tilde{p}\big|_{X'}=p\,, 
$$ 
or
\begin{eqnarray*}
   &&
      \begin{pmatrix}
        \tilde{q}\\
 \tilde{p}_{1}\\
 \tilde{p}'
      \end{pmatrix}
=
\begin{pmatrix}
   \textrm{Id}&0&0\\
 0&1&0\\
 0&0&\psi(q^{1},q')  
\end{pmatrix}
 \begin{pmatrix}
        q\\p_{1}\\p'
\end{pmatrix}\\
\\
 \text{with}&&
 \left\{\begin{array}[c]{l}
 \frac{\partial}{\partial q^{1}}
\psi_{j'}^{k'}=-\hat{\Gamma}_{1,j'}^{\ell'}(q^{1},q')\psi_{\ell'}^{k'}
\\
 \psi_{j'}^{k'}(0,q')=\delta_{j'}^{k'}\,.                
 \end{array}
 \right.
\end{eqnarray*}
We deduce
\begin{eqnarray*}
  &&
\frac{\partial}{\partial \tilde{q}^{\ell}}=\frac{\partial}{\partial
q^{\ell}}-[\frac{\partial \psi}{\partial
q^{\ell}}\psi^{-1}]_{\ell'}^{k'}p_{k'}\frac{\partial}{\partial 
p_{\ell'}}
\quad,\quad
\frac{\partial}{\partial\tilde{p}_{1}}=\frac{\partial}{\partial
  p_{1}}\quad,\quad \frac{\partial}{\partial
  \tilde{p}_{j'}}=[\psi^{-1}]_{\ell'}^{j'}\frac{\partial}{\partial p_{\ell'}}
\\
&& d\tilde{q}^{\ell}=dq^{\ell}\quad,\quad
   d\tilde{p}_{1}=dp_{1}\quad,\quad
   d\tilde{p}_{j'}=[\psi(q)]_{j'}^{k'}dp_{k'}+[\frac{\partial \psi}{\partial
   q^{\ell}}]_{j'}^{k'}p_{k'}dq^{\ell}
\end{eqnarray*}
This leads to 
\begin{eqnarray}
\label{eq:tv1}
  &&
\frac{\partial}{\partial
  \tilde{q}^{1}}=\frac{\partial}{\partial q^{1}}+\hat{\Gamma}_{1j'}^{k'}p_{k'}\frac{\partial}{\partial
p_{j'}}=e_{1}\quad,\quad
     \hat{e}^{1}=\frac{\partial}{\partial\tilde{p}_{1}}\\
\label{eq:tvv}
&& \hat{e}^{j'}=\frac{\partial}{\partial
   p_{j'}}=\psi_{k'}^{j'}(q)\frac{\partial}{\partial \tilde{p}_{k'}}
\stackrel{on~X'}{=}\frac{\partial}{\partial \tilde{p}_{j'}}\,,\\
\nonumber
&&
e_{i'}=\frac{\partial}{\partial \tilde{q}^{i'}}+[\psi
   \frac{\partial \psi}{\partial
   q^{i}}]_{\ell'}^{k'}\tilde{p}_{k'}\frac{\partial}{\partial
   \tilde{p}_{\ell'}}
+[\psi\hat{\Gamma}_{i',.}^{.}\psi^{-1}]_{j'}^{k'}\tilde{p}_{k'}\frac{\partial}{\partial
   \tilde{p}_{j'}}\\
\label{eq:tvh}
&&\hspace{4cm}+[\hat{\Gamma}_{i',1}^{.}\psi^{-1}]^{\ell'}\tilde{p}_{\ell'}\frac{\partial}{\partial\tilde{p}_{1}}
+[\psi\hat{\Gamma}_{i',.}^{1}]_{j'}\tilde{p}_{1}\frac{\partial}{\partial
   \tilde{p}_{j'}}
\\
\label{eq:tvh0}
&&e_{\mp,
   i'}\stackrel{on~X'}{=}\frac{\partial}{\partial\tilde{q}^{i'}}+
   \Gamma_{i'j'}^{k'}(0,q')\tilde{p}_{k'}\frac{\partial}{\partial
   \tilde{p}_{j'}}\pm
   \Gamma_{i',1}^{k'}(0,q')\tilde{p}_{k}\frac{\partial}{\partial\tilde{p}_{1}}\pm
   \Gamma_{i'j'}^{1}(0,q')\tilde{p}_{1}\frac{\partial}{\partial \tilde{p}_{j'}}\,,
\end{eqnarray}
and to
\begin{eqnarray}
 \label{eq:tfh}
&&
 e^{i}=dq^{i}=d\tilde{q}^{i}
\\
\label{eq:tfv1}
&&
   \hat{e}_{1}=d\tilde{p}_{1}-[\hat{\Gamma}_{1j'}^{.}\psi^{-1}]^{k'}\tilde{p}_{k'}d\tilde{q}^{j'}\\
\label{eq:tfv10}
&&\hat{e}_{1}\stackrel{on~X'}{=}d\tilde{p}_{1}\mp
   \Gamma_{1j'}^{k'}(0,q')\tilde{p}_{k'}d\tilde{q}^{i'}\\
\nonumber
&&
   \hat{e}_{j'}=[\psi^{-1}]_{j'}^{k'}d\tilde{p}_{k'}-[\psi^{-1}\frac{\partial
   \psi}{\partial
   q^{\ell}}\psi^{-1}]_{j'}^{k'}\tilde{p}_{k'}d\tilde{q}^{\ell}
-[\hat{\Gamma}_{j'.}^{.}\psi^{-1}]_{i'}^{k'}\tilde{p}_{k'}d\tilde{q}^{i'}
\\
\label{eq:tfv}
&&\hspace{4cm}
-[\hat{\Gamma}_{j'1}^{.}\psi^{-1}]^{k'}\tilde{p}_{k'}d\tilde{q}^{1}-\hat{\Gamma}_{j'i'}^{1}\tilde{p}_{1}d\tilde{q}^{i'}
\\
\label{eq:tfv0}
&&
\hat{e}_{\mp,j'}\stackrel{on~X'}{=}d\tilde{p}_{j'}
  - \Gamma_{j'i'}^{k}(0,q')\tilde{p}_{k'}d\tilde{q}^{i'}\mp \Gamma^{1}_{j'i'}(0,q')\tilde{p}_{1}d\tilde{q}^{i'}\,.
\end{eqnarray}
Remember that the map $\Sigma:(q^{1},q',p_{1},p')\to
(-q^{1},q',-p_{1},p')$ of Definition~\ref{de:SigmaXI} keeps
 the same
form $\Sigma(\tilde{q}^{1},\tilde{q}',\tilde{p}_{1},\tilde{p'})=(-\tilde{q}^{1},\tilde{q}',-\tilde{p}_{1},\tilde{p}')$
in the new coordinates $(\tilde{q},\tilde{p})$\,. In particular
\eqref{eq:Sigmae1}\eqref{eq:Sigmae2} can be deduced from
\eqref{eq:tv1}\eqref{eq:tvv}\eqref{eq:tvh0} and \eqref{eq:tfh}\eqref{eq:tfv10}\eqref{eq:tfv0}.\\
The continuous matching along $X'$ vector bundles $\hat{E}_{g}$ and
$\hat{E}'_{g}$  simply says that
$(e_{-},\hat{e}_{-})\big|_{\partial X_{-}}$ and
$(e_{+},\hat{e}_{+})\big|_{\partial X_{+}}$ are identified.\\
The lifting to $X\cap U$ of the  geodesic curve
$\exp_{\underline{q}'}^{Q,\hat{g}}(t\underline{e}_{i})$ is nothing but
the curve $t\mapsto (t,\tilde{q}',\tilde{p})$ and the map
$\tilde{x}':X\times U\to X'$ is nothing but
$$
\tilde{x}'(\tilde{q}^{1},\tilde{q}^{i'},\tilde{p}_{j})=(0,\tilde{q}^{i'},\tilde{p}_{j})\,.
$$
Using the lifted connection $\nabla^{E,\hat{g}}$ and
$\nabla^{E',\hat{g}}$\,, the frames
$(e,\hat{e})\big|_{X'}=(e_{-},\hat{e}_{-})\big|_{\partial
  X_{-}}=(e_{+},\hat{e}_{+})\big|_{\partial X_{+}}$ are lifted to the
new frames $\tilde{x}^{\prime *}\left[(e,\hat{e})\big|_{X'}\right]$\,. 
The  piecewise $\mathcal{C}^{\infty}$ and continuous vector
bundle isometry  $\widehat{\Psi}_{X}^{g,g_{0}}$ of diagram
\eqref{eq:psiXgg} is nothing but
$$
\Psi_{X}^{g,g_{0}}:(q,p,f,\hat{f})\to (\tilde{q},\tilde{p}, \tilde{x}^{*}[(e,\hat{e})\big|_{X'}])\,.
$$ 
For the continuity properties in $\hat{E}_{g}$ and $\hat{E}'_{g}$ we
can work more simply with the frame
$(e,\hat{e})=1_{Q_{\mp}}(q)(e_{\mp},\hat{e}_{\mp})$
than with the frame $\tilde{x}^{\prime *}[(e,\hat{e})\big|_{X'}]$\,.\\
Let us conclude with the vector bundle $\pi_{\fkf}:\fkf\to
Q$ of Definition~\ref{de:doublef} endowed with the metric
$\hat{g}^{\fkf}$ and the two connections $\nabla^{\fkf}$ and
$\nabla^{\fkf'}$ (see also Proposition~\ref{pr:dualf}) with now:
$$
\nabla^{\fkf'}-\nabla^{\fkf}=\omega(\nabla^{\fkf},\hat{g}^{\fkf})\,.
$$
 With
the example of $\fkf=\overline{Q}_{-}\times\cz$ and with
$\nabla^{\fkf}=\nabla$ the
trivial connection, and $g^{\fkf}(z)=g_{-}^{\fkf}(z)=e^{-2V(\underline{q})}|z|^{2}$
we get
$\hat{g}^{\fkf}(z)=e^{-2\hat{V}(\underline{q})}|z|^{2}$ with 
$\hat{V}(\underline{q})=V(-|\underline{q}^{1}|,\underline{q}')$
and
$$
\omega(\nabla^{\fkf},\hat{g}^{\fkf})=-2d\hat{V}=-2\left[\textrm{sign}(-\underline{q}^{1})
    d_{\underline{q}^{1}}\wedge \frac{\partial V}{\partial \underline{q}^{1}}+d_{\underline{q}'}V\right](-|\underline{q}^{1}|,\underline{q}')
$$
with a discontinuity along $\underline{q}^{1}=0$ when
$\partial_{\underline{q}^{1}}V(0,\underline{q}')\neq 0$\,. Then $\nabla^{\fkf'}$ and
$\nabla^{\fkf ,u}$ become piecewise $\mathcal{C}^{\infty}$ and
not continuous for the initial $\mathcal{C}^{\infty}$-structure of
$Q\times \cz$\,.\\ 
Note that if we take $\fkf=\overline{Q}_{-}\times \cz$ with the metric
$g^{\fkf}(z)=|z|^{2}$ but with the connection 
$\nabla^{\fkf}=\nabla+dV(q)$\,, the connection on the doubled vector
bundle $\pi_{\fkf}:\fkf\to Q$ of Definition~\ref{de:doublef} is 
now $\nabla^{\fkf}=\nabla +d\hat{V}(q)$ and
while $\nabla^{\fkf'}=\nabla-d\hat{V}(q)$ and
$\nabla^{\fkf,u}=\nabla$\,. Remember also that the continuity in
$\pi_{\fkf}:\fkf\to Q$ means  a change of sign across $Q'$ when $\nu=-1$\,.
\begin{remark}
  The relations \eqref{eq:Sigmae1}\eqref{eq:Sigmae2} suggest another
  interpretation of $\hat{E}_{g}$ and $\hat{E}'_{g}$ as the exterior
  algebras of the cotangent and tangent bundle of
 a smooth manifold, $\overline{X}_{-}$ and
  $\overline{X}_{+}$ being glued by identifying
  $(0^{-},\tilde{q}',\tilde{p}_{1},\tilde{p})$ and
  $(0^{+},\tilde{q}',-\tilde{p}_{1},\tilde{p}')$\,. This will be used
  in Subsection~\ref{sec:diffhatE}. However for most of the analysis
  the above presentation of $\hat{E}_{g}$ and $\hat{E}'_{g}$ as
  piecewise $\mathcal{C}^{\infty}$ and continuous vector bundles on
  $X$ is safer and  more convenient.
\end{remark}

\section{Functional spaces and invariances}
\label{sec:funcsp}
We review the functional spaces that we will use. First we start with
local spaces in the smooth case, which do not depend on any chosen
riemannian metric, then we discuss the case of sections in
$\widehat{\fkF}_{g}$ for $\fkF=E,E',F,F'$\,, where the metric enters
in the game only in the continuity or jump condition along $X'$\,. 
Finally we study how global spaces depend on the chosen metric
$g^{TQ}$\,. In particular, global spaces of sections of
$\widehat{\fkF}_{g}$ are characterized after considering the
restrictions $s_{\mp}=s\big|_{\overline{X}_{\mp}}$ like in the smooth case with
a boundary and then possibly adding  the local continuity condition
$s_{-}\big|_{X'}=s_{+}\big|_{X'}$ in $\widehat{\fkF}_{g}\big|_{X'}$\,.
Invariances and isomorphisms of those functional spaces via the change
of variables or vector bundle isomorphisms introduced in
Section~\ref{sec:geomcot} are discussed.

\subsection{Local spaces for smooth vector bundles}
\label{sec:locspaces}

Let $M$ be a $\mathcal{C}^{\infty}$ manifold and $\pi_{\fkF}:\fkF\to
M$ a $\mathcal{C}^{\infty}$-vector bundle.  A smooth manifold with boundary
is denoted $\overline{M}=M\sqcup \partial M$ and accordingly
$\pi_{\fkF}:\fkF\to \overline{M}$ is a
$\mathcal{C}^{\infty}(\overline{M})$-vector bundle endowed with any
smooth connection. The cases that we
have in mind are $M=Q,X,Q_{-},Q_{+},X_{-},X_{+}$ and
$\overline{M}=\overline{Q}_{-},\overline{Q}_{+},\overline{X}_{-},\overline{X}_{+}$\,. 
By following \cite{ChPi} all the spaces $\mathcal{F}(M;\fkF)$ or
$\mathcal{F}(\overline{M};\fkF)$ are defined independently of any
riemannian structures or metric on $\fkF$~: $\mathcal{F}=\mathcal{C}^{\infty}$\,,
$\mathcal{F}=\mathcal{C}^{\infty}_{0}$\,, $\mathcal{F}=L^{2}_{loc}$\,,
$\mathcal{F}=L^{2}_{comp}$\,, $\mathcal{F}=W^{\mu,2}_{loc}$ and
$\mathcal{F}=W^{\mu,2}_{comp}$\,, $\mu\in \rz$\,, where ``$W^{\mu,2}$
counts $\mu$-derivatives in $L^{2}$'' when $\mu\in\nz$\,.\\

When $M$ is a smooth manifold (no boundary), taking $\fkF'$ the dual
bundle of $\fkF$ and
fixing any smooth volume element $dv_{M}$ on $M$ provide a duality product
$$
\langle s\,,\,s'\rangle=\int_{M}\langle s\,,\,s'\rangle_{\fkF',\fkF}(x)~dv_{M}(x)\,,
$$
and we will use without distinction real or sesquilinear (left-anti
linear and right $\cz$-linear) duality products.\\
The set of distributional sections $\mathcal{D}'(M;\fkF')$ is defined
as the dual of $\mathcal{C}^{\infty}_{0}(M;\fkF)$ and this duality
holds between $W^{\mu,2}_{loc~comp}(M;\fkF)$ and
$W_{comp~loc}^{-\mu,2}(M;\fkF')$ for $\mu\in\rz$\,.\\

A smooth manifold with boundary $\overline{M}$\,, can be considered as a domain of a smooth closed manifold $\tilde{M}$ and
$\fkF=\tilde{\fkF}\big|_{\overline{M}}$ where
$\pi_{\tilde{\fkF}}:\tilde{\fkF}\to \tilde{M}$ is a
$\mathcal{C}^{\infty}$-vector bundle.  
According to \cite{ChPi}, the above functional space
 $\mathcal{F}(\overline{M};\fkF)$ is defined as the set of
 restrictions to $M$ of elements of
 $\mathcal{F}(\tilde{M};\tilde{\fkF})$:
$$
\mathcal{F}(\overline{M};\fkF)=\left\{u\in
  \mathcal{D}'(M;\fkF)\,,\exists \tilde{u}\in
  \mathcal{F}(\tilde{M};\tilde{\fkF})\,, \quad u=\tilde{u}\big|_{M}\right\}
$$
endowed with the quotient topology.
On a manifold with boundary $\overline{M}=M\sqcup \partial M$\,,
compact sets of $M$ and $\overline{M}$ differ and the spaces 
$\mathcal{F}(M,\fkF)$ and $\mathcal{F}(\overline{M};\fkF)$ are
distinguished, the later specifying the information up to the boundary
$\partial M$\,. Finally when 
$dv_{M}$ a $\mathcal{C}^{\infty}(\overline{M})$ volume element, the
duality holds between $L^{2}_{loc~comp}(\overline{M};\fkF)$ and
$L^{2}_{comp~loc}(\overline{M};\fkF')$\,.\\
A section $s\in \mathcal{F}(M;\fkF)$ (resp. $s\in
\mathcal{F}(\overline{M};\fkF)$) if for any locally finite partition
of unity $\sum_{j\in J}\chi_{j}\equiv 1$\,, $\chi_{j}\in
\mathcal{C}^{\infty}_{0}(M;\rz)$ (resp. $\chi_{j}\in
\mathcal{C}^{\infty}_{0}(\overline{M};\rz)$) one has $\chi_{j}s\in
\mathcal{F}(M;\fkF)$ for all $j\in J$\,, the latter being checked in
any local coordinate system. Those spaces are invariant by
$\mathcal{C}^{\infty}$ diffeomorphisms on $M$ (resp $\overline{M}$)
and $\mathcal{C}^{\infty}$ vector bundle isomorphisms of $\fkF$\,. \\

When $M=X=T^{*}Q$ or $\overline{M}=\overline{X}_{\mp}$\,, any given
riemannian metric $g^{TQ}$ provides the function
$|p|_{q}^{2}=2\fkh(q,p)=g^{ij}(q)p_{i}p_{j}$ and $s\in
W^{\mu,2}_{loc}(X;\fkF)$ (resp. $s\in
W^{\mu,2}_{loc}(\overline{X}_{\mp};\fkF)$)  if and only if for any
$\chi\in \mathcal{C}^{\infty}_{0}(\rz;\rz)$\,, $\chi(|p|^{2}_{q})s\in
W^{\mu,2}_{comp}(M;\fkF)$ (resp. $\chi(|p|_{q}^{2})s\in
W^{\mu,2}_{comp}(\overline{M};\fkF)$)\,. Note also that on $X$ and
$\overline{X}_{\mp}$ the symplectic volume $dv_{X}=|dqdp|$ is fixed
independently of any chosen metric.\\
{\red $W^{(\mu_{1},\mu_{2}),2}$ removed}\\
We now define spaces associated with a continuous operator $P:\mathcal{D}'(M)\to
\mathcal{D}'(M)$\,, which will be in practice  a differential operator
with smooth coefficients.
\begin{definition}
  \label{de:calE} Let $\pi_{\fkF}:\fkF\to M$ be a smooth vector bundle on
  $M$  (resp $\overline{M}=M\sqcup \partial M$) 
and let $P$ be a differential
  operator with smooth coefficients $P:\mathcal{D}'(M;\fkF)\to
  \mathcal{D}'(M;\fkF)$\,. The spaces $\mathcal{E}_{loc}(P,\fkF)$ and
  $\mathcal{E}_{comp}(P,\fkF)$ are defined as
  \begin{eqnarray*}
    &&
\mathcal{E}_{\bullet}(P,\fkF)=\left\{ \omega\in
  L^{2}_{\bullet}(M;\fkF)\,,\quad P\omega\in
  L^{2}_{\bullet}(M;\fkF)\right\}\,,\quad \bullet={loc}~\text{or}~comp\,.\\
\text{resp.}
&&
\mathcal{E}_{\bullet}(P,\fkF)=\left\{ \omega\in
  L^{2}_{\bullet}(\overline{M};\fkF)\,,\quad P\omega\in
  L^{2}_{\bullet}(\overline{M};\fkF)\right\}\,,\quad \bullet={loc}~\text{or}~comp\,.
  \end{eqnarray*}
\end{definition}
\subsection{Local spaces for
  $\hat{E}_{g},\hat{E}'_{g},\hat{F}_{g},\hat{F}'_{g}$}
The vector bundles $\hat{E}_{g}$\,, $\hat{E}_{g}'$\,,
$\hat{F}_{g}=\hat{E}_{g}\otimes \pi_{X}^{-1}(\fkf)$ and
$\hat{F}'_{g}=\hat{E}_{g}\otimes \pi_{X}^{-1}(\fkf)$  of
Definition~\ref{de:hatEF} are defined as piecewise 
$\mathcal{C}^{\infty}$  vector bundles\,, with some 
matching conditions along $X'=\partial
X_{-}=\partial X_{+}$\,. Local functional spaces
$\mathcal{\mathcal{F}}(X;\widehat{\fkF}_{g})$\,, with
$\widehat{\fkF}_{g}=\hat{E}_{g},\hat{E}'_{g},\hat{F}_{g},\hat{F}'_{g}$
will be specified accordingly by using
$\mathcal{F}(\overline{X}_{\mp};\fkF)$ with the corresponding
continuity condition along $X'$ when it makes sense.  
Remember that $\pi_{\fkf}:\fkf\to Q$ may have different 
$\mathcal{C}^{\infty}$ structures according to
Proposition~\ref{pr:dualf} in $\fkF=E\otimes \pi_{X}^{*}(\fkf)$ and
$\fkF=E'\otimes \pi_{X}^{*}(\fkf)$\,.\\
The space $L^{2}_{loc~comp}$ is defined
piecewise and the distinction between $\widehat{\fkF}_{g}$ and $\fkF$ can
be forgotten according to 
$$
L^{2}_{loc}(X;\widehat{\fkF}_{g})
=L^{2}_{loc}(\overline{X}_{-};\fkF)\oplus L^{2}_{loc}(\overline{X}_{+};\fkF)=L^{2}_{loc}(X;\fkF)\,.
$$
\begin{definition}
\label{de:evodloc} 
When $\Sigma:X_{(-\varepsilon,\varepsilon)}\to
X_{(-\varepsilon,\varepsilon)}$  and $\Sigma_{\nu}$ are the maps of
Definition~\ref{de:SigmaXI} and $\fkF=E,E',F,F'$ with $\nu=\pm 1$ when
$\fkF=E,E'$,
the set of even and odd sections of
$L^{2}_{loc}(X_{(-\varepsilon,\varepsilon)};\fkF)$\,,
is defined by
\begin{equation}
  \label{eq:deL2evoddloc}
L_{ev~odd,loc}^{2}(X_{(-\varepsilon,\varepsilon)};\fkF)
=\left\{s\in
  L^{2}_{loc}(X_{(-\varepsilon,\varepsilon)};\fkF)\,,\quad
  \Sigma_{\nu}s=\pm s\right\}\,.
\end{equation}
For $s\in L^{2}_{loc}(X_{(-\varepsilon,0]};\fkF)$ we define
\begin{equation}
  \label{eq:desev}
s_{ev}=1_{X_{-}}(x)s+1_{X_{+}}(x) \Sigma_{\nu}s)\,.
\end{equation}
\end{definition}
Note that the spaces of even and odd sections are interchanged by a
simple change of the unitary flat involution $\nu:\fkf\big|_{Q'}\to \fkf\big|_{Q'}$ into
$-\nu$\,.\\
We will make an
extensive use of the set of  smooth compactly supported sections,
$\mathcal{C}_{0,g}(\widehat{\fkF}_{g})$ defined below.
\begin{definition}
\label{de:Cg} For $\fkF=E,E',F,F'$  possibly restricted to
$X_{(-\varepsilon,\varepsilon)}$\,, the space
$\mathcal{C}_{0,g}(\widehat{\fkF}_{g})$ is defined as
\begin{equation}
  \label{eq:deCghF}
\mathcal{C}_{0,g}(\widehat{\fkF}_{g})
=\left\{s\in \mathcal{C}^{0}(X;\widehat{\fkF}_{g})\,,\quad
  s\big|_{\overline{X}_{\mp}}\in \mathcal{C}^{\infty}_{0}(\overline{X}_{\mp};\fkF)\right\}\,.
\end{equation}
The set of even elements  is defined by
\begin{equation}
  \label{eq:deCgev}
\mathcal{C}_{0,g,ev}(\widehat{\fkF}_{g})=\mathcal{C}_{0,g}(\widehat{\fkF}_{g})\cap
L^{2}_{ev, comp}(X_{(-\varepsilon,\varepsilon)};\fkF)\,.
\end{equation}
Finally the space $\mathcal{C}_{0,g}(L(\widehat{\fkF}_{g}))$ is defined
like \eqref{eq:deCghF} by
$$
\mathcal{C}_{0,g}(L(\widehat{\fkF}_{g}))=\left\{s\in
  \mathcal{C}^{0}(X;L(\widehat{\fkF}_{g}))\,,\quad
  s\big|_{\overline{X}_{\mp}}\in \mathcal{C}^{\infty}(\overline{X}_{\mp};L(\fkF))\right\}\,.
$$ 
\end{definition}
Definition~\ref{de:hatEF} actually provides a $\mathcal{C}^{\infty}$
structure for  $\widehat{\fkF}_{g}$ by taking the one of the
right-hand side in the equalities
\eqref{eq:hatehatg}\eqref{eq:hatfhatg}\eqref{eq:hatephatg}\eqref{eq:hatfphatg},
depending on the case.  The local Sobolev spaces
$W^{\mu,2}_{loc}(X;\widehat{\fkF}_{g})$\,, $\mu\in \rz$\\
{\red $W^{(\mu_{1},\mu_{2}),2}$ removed}
 can be defined
for this $\mathcal{C}^{\infty}$-structure\,. We recall the standard result concerning the existence of
traces (see e.g. \cite{ChPi}):
 \begin{itemize}
 \item The trace map 
$\gamma:W^{\mu,2}_{loc}(\overline{X}_{\mp};\widehat{\fkF}_{g})=W^{\mu,2}_{loc}(\overline{X}_{\mp};\fkF)\to
W^{\mu-1/2,2}_{loc}(\partial X_{\mp};\fkF)$\,, $\gamma
s=s\big|_{\partial X_{\mp}}$\,,  is well defined for
$\mu>\frac{1}{2}$\,.
\item For $\mu\in [0,1/2[$\,,  $s\in
  W^{\mu,2}_{loc}(X;\widehat{\fkF}_{g})$ 
if and only if $s_{\mp}=s\big|_{X_{\mp}}\in
  W^{\mu,2}(\overline{X}_{\mp};\widehat{\fkF}_{g})=W^{\mu,2}(\overline{X}_{\mp};\fkF)$\,.
\item For $\mu\in ]1/2,3/2[$\,,  $s\in W^{\mu,2}(X;\widehat{\fkF}_{g})$
  if and only if  $s_{\mp}=\big|_{X_{\mp}}\in
  W^{\mu,2}_{loc}(\overline{X}_{\mp};\widehat{\fkF}_{g})=W^{\mu ,2}_{loc}(\overline{X}_{\mp};\fkF)$ and the traces
  along $X'=\partial X_{-}=\partial X_{+}$ coincide
  $s_{-}\big|_{\partial X_{-}}=s_{+}\big|_{\partial X_{+}}$ in $\widehat{\fkF}_{g}\big|_{X'}$\,.
\end{itemize}
In all the analysis we will avoid trace issues for half-integer 
exponents
$\mu=\frac{1}{2}+n$\,, $n\in \nz$\,, which as it is well known
(see \cite{LiMa}-Chap~11) is a subtle critical case.\\ 
The equality of traces in $\widehat{\fkF}_{g}\big|_{X'}$\,, for
$\mu_{1}\in ]1/2,3/2[$ means that
the frames $(e_{-},\hat{e}_{-})$ and $(e_{+},\hat{e}_{+})$ are
identified along $X'=\partial X_{-}=\partial X_{+}$ and this is
actually a jump condition with the usual
$\mathcal{C}^{\infty}$-structure of $E=\Lambda T^{*}X$ and $E'=\Lambda
TX$\,, which also corresponds to the case when
$g^{TQ}=g_{0}^{TQ}=(d\underline{q}^{1})^{2}+m^{TQ}(0,\underline{q}')$\,. For
such a metric $g_{0}^{TQ}$\,, we write simply
$W^{\mu,2}_{loc}(X;\widehat{\fkF}_{g_{0}})=W^{\mu,2}_{loc}(X;\fkF)$\,.
For a general metric $g^{TQ}$\,, the spaces
$W^{\mu,2}_{loc}(X;\widehat{\fkF}_{g})$ will be
considered with
$\mu=[0,1]\setminus\left\{1/2\right\}$\,. Since differential operators with 
possibly discontinuous coefficients along $X'$ and non obvious
effects on the jump condition will be studied, it is better to split the analysis
on $\overline{X}_{-}$ and $\overline{X}_{+}$ and check  separately the
matching condition along $X'$\,.\\
We keep  of course the notation $W^{\mu,2}_{loc}(X;\fkF)$
when $g^{TQ}=g_{0}^{TQ}$ for any $\mu \in \rz$ such that
but  for a
general $g^{TQ}$\,, we take the following definition equivalent to the
previous construction.
\begin{definition}
\red
 For $\fkF=E,E',F,F'$\,, $\mu\in [0,1]\setminus
   \left\{1/2\right\}$\,,  the space
 $W^{\mu,2}_{loc}(X;\widehat{\fkF}_{g})$ is defined
as the space of sections $s\in L^{2}_{loc}(X;\fkF)$ such that:
\begin{itemize}
\item $s_{\mp}=s\big|_{X_{\mp}}\in
  W^{\mu,2}_{loc}(\overline{X}_{\mp};\fkF)$\,;
\item if $\mu\in ]1/2,1]$\,, $s_{-}\big|_{\partial
    X_{-}}=s_{+}\big|_{\partial X_{+}}$ in $\widehat{\fkF}_{g}\big|_{X'}$\,.
\end{itemize}
The set of even sections of
$W^{\mu,2}_{loc}(X;\hat{E}_{g})$ is defined as
$$
W^{\mu,2}_{ev,
  loc}(X;\widehat{\fkF}_{g})=W^{\mu,2}_{loc}(X;\widehat{\fkF}_{g})\cap
L^{2}_{ev, loc}(X;\fkF)\,.
$$
\end{definition}
\begin{definition}
\label{de:S1Snu} In $X'=\partial X_{-}$ the map $S_{1}=\Sigma\big|_{X'}$
is given by 
\begin{equation}
  \label{eq:deS1}
S_{1}(0,q',p_{1},p')=(0,q',-p_{1},p')\,.
\end{equation}
For $\fkF=F$ (resp. $\fkF=F'$)
 the map $\hat{S}_{\nu}:\mathcal{D}'(X';\fkF\big|_{X'})\to
\mathcal{D}'(\fkF\big|_{X'})$ is given by
\begin{eqnarray}
\label{eq:dehSnu1}
&&
\hat{S}_{\nu}(\omega_{I}^{J}(x')e_{-}^{I}\hat{e}_{-,J})=\nu
(-1)^{|\left\{1\right\}\cap I|+|\left\{1\right\}\cap J|}\omega_{I}^{J}(S_{1}(x'))e_{-}^{I}\hat{e}_{-,J}
\\
\label{eq:dehSnu2}
\text{resp.}
&&
\hat{S}_{\nu}(u_{J}^{I}(x')e_{-,I}\hat{e}_{-}^{J})=\nu
(-1)^{|\left\{1\right\}\cap I|+|\left\{1\right\}\cap
   J|}u_{J}^{I}(S_{1}(x'))e_{-,I}\hat{e}_{-}^{J}\,.
\end{eqnarray}
\end{definition}
\begin{proposition}
\label{pr:partrace}
\red
For a section $s\in L^{2}_{loc }(\overline{X}_{-};\fkF)$ and
$\mu\in [0,1]\setminus\left\{1/2\right\}$\,,
and  $\fkF=F,F'$\,, there is an
equivalence between:
\begin{description}
\item[a)] $s_{ev}\in \mathcal{C}_{0,g}(\widehat{\fkF}_{g})$ 
(resp. $s_{ev}\in W^{\mu,2}_{loc}(X;\widehat{\fkF}_{g})$)\,;
\item[b)]  
  $s\in \mathcal{C}^{\infty}_{0}(\overline{X}_{-};\fkF)$ (resp. $s\in
  W^{\mu,2}_{loc}(\overline{X}_{-};\fkF)$) 
and  (resp. when $\mu>1/2$)
$$
\hat{S}_{\nu}s\big|_{\partial X_{-}}=s\big|_{\partial X_{-}}\,.
$$
\end{description}
\end{proposition}
\begin{proof}
  Simply write that $s_{ev}=1_{X_{-}}(x)s+1_{X_{+}}(\Sigma_{\nu}s)$  satisfies
  $s_{ev}\big|_{X_{-}}=s\big|_{X_{-}}$ and admits a trace  in $\mathcal{D}'(X';\widehat{\fkF}_{g}\big|_{X'})$ when
  $s_{ev}\in \mathcal{C}_{0,g}(X;\widehat{\fkF}_{g})$ or {\red$s_{ev}\in
  W^{\mu,2}_{loc}(X;\widehat{\fkF}_{g})$\,,
$\mu>1/2$\,, with:}
  \begin{itemize}
  \item 
 $\Sigma\big|_{X'}=S_{1}$\,,
  \item 
    $\Sigma_{*}(e_{-}^{i},\hat{e}_{-,j})=((-1)^{\delta_{i1}}e_{+}^{i},(-1)^{\delta_{j1}}\hat{e}_{+,j})$ tensorized
  with $\nu$ in $\pi_{\overline{X}_{-}}^{*}(\fkf)$\,,
\item and
  $(e_{-},\hat{e}_{-})\big|_{X'}=(e_{+},\hat{e}_{+})\big|_{X'}$
  in $\widehat{\fkF}_{g}\big|_{X'}$\,.
\end{itemize}
\end{proof}
\subsection{Global functional spaces}
\label{sec:globfunc}

\subsubsection{Global $L^{2}$-spaces, duality and adjoints}
\label{sec:globL2}

Like for 
$L^{2}_{loc}(X;\widehat{\fkF}_{g})=L^{2}_{loc}(X;\fkF)=L^{2}_{loc}(\overline{X}_{-};\fkF)\oplus
L^{2}_{loc}(\overline{X}_{+};\fkF)$ for
$\fkF=E,E',F,F'$ we can simply work with the vector bundles $E=\Lambda
T^{*}X$\,, $E'=\Lambda TX$\,, $F=E\otimes \pi_{X}^{*}(\fkf)$ and
$F'=E'\otimes \pi_{X}^{*}(\fkf)$ and forget the distinction between
$\hat{\fkF}_{g}$ and $\fkF$\,.\\
The cotangent space $X$ is endowed with the symplectic volume
$$
dv_{X}=|\frac{1}{d!}\sigma^{d}|=|\frac{1}{d!}(e^{i}\wedge\hat{e}_{j})^{d}|=|dqdp|
$$
where $|dqdp|$ stands for the Lebesgue
measure 
in the local coordinates
$(q,p)=(q^{1},\ldots,q^{d},p_{1},\ldots,p_{d})$\,.\\
The local coordinates
$(\tilde{q},\tilde{p})=(\tilde{q}^{1},\ldots,\tilde{q}^{d},\tilde{p}_{1},\ldots,\tilde{p}_{d})$
of Definition~\ref{de:tildeqp} are not symplectic coordinates and
\eqref{eq:tfh}\eqref{eq:tfv1}\eqref{eq:tfv}  lead to 
\begin{equation}
  \label{eq:dvxtilde}
dv_{X}=|\frac{1}{d!}\sigma^{d}|=|\frac{1}{d!}(e^{i}\wedge\hat{e}_{i})^{d}|=|\det
\psi^{-1}(\tilde{q})||d\tilde{q}d\tilde{p}|\,.
\end{equation}
Remember that the metrics $g^{E},g^{E'},g^{F},g^{F'}$ on $E,E',F,F'$ given by
\eqref{eq:gE}\eqref{eq:gEpr}\eqref{eq:gF}\eqref{eq:gFpr} include the
weight $\langle p\rangle_{q}^{\pm N_{H}\mp N_{V}}$ and the same is
done for $\hat{g}^{E},\hat{g}^{E'},\hat{g}^{F},\hat{g}^{F'}$with
$g^{TQ}$ replaced by $\hat{g}^{TQ}$\,.
\begin{definition}
\label{de:L2glob} Let $\fkF=E,E',F,F'$ be endowed with the metric
$\tilde{g}^{\fkF}=g^{\fkF}$ or $\tilde{g}^{\fkF}=\hat{g}^{\fkF}$\,. The space
$L^{2}(X;\fkF)$\,, or $L^{2}(X;\fkF,\tilde{g}^{\fkF})$ when we want to
specify the metric,  is
$$
L^{2}(X;\fkF)=\left\{s\in L^{2}_{loc}(X;\fkF)\,,\quad
  \int_{X}|s(q,p)|^{2}_{\tilde{g}^{\fkF}}~|dqdp|<+\infty\right\}\,.
$$
The scalar product and the duality product between $L^{2}(X;\fkF)$ and
$L^{2}(X;\fkF')$ are given by
\begin{eqnarray}
\label{eq:scalL2}
  &&\langle s\,,\,s'\rangle_{L^{2}}=\langle
     s\,,\,s'\rangle_{L^{2}(\tilde{g}^{\fkF})}=\int_{X}\langle
     s\,,\,s'\rangle_{\tilde{g}^{\fkF}_{(q,p)}}~|dqdp|=\overline{\langle
     s',s\rangle_{L^{2}}}\\
\label{eq:dualL2}
&&\langle t\,,\,s\rangle=\int_{X}\langle t\,,\,
   s\rangle_{\fkF'_{(q,p)},\fkF_{(q,p)}}~|dqdp|=\overline{\langle s\,,\,t\rangle}
\end{eqnarray}
for any $s,s'\in L^{2}(X;\fkF)$\,, $t\in L^{2}(X;\fkF')$\,.\\
Finally the set $L^{2}_{ev~odd}(X;\fkF,\hat{g}^{\fkF})$ equals
$$
L^{2}_{ev~odd}(X;\fkF,\hat{g}^{\fkF})=L^{2}(X;\fkF,\hat{g}^{\fkF})\cap L^{2}_{ev~odd,loc}(X;\fkF)\,.
$$
\end{definition}
\begin{proposition}
  \label{pr:L2glob} Let $\fkF=E,E',F,F'$ be endowed with the metric
  $\tilde{g}^{\fkF}=g^{\fkF}$ or $\hat{g}^{\fkF}$ and let
  $(\fkF',\tilde{g}^{\fkF'})$ its antidual, $\fkF'=E',E,F',F$ respectively.
  \begin{enumerate}
  \item[a)] With the duality product
\eqref{eq:dualL2}, the dual of  $L^{2}(X;\fkF,\tilde{g}^{\fkF})$
is nothing but
    $L^{2}(X;\fkF',\tilde{g}^{\fkF'})$\,.
\item[b)] When $\varphi:Q\to Q$ is a $\mathcal{C}^{\infty}$ (resp. 
  piecewise $\mathcal{C}^{\infty}$ on $\overline{Q}_{\mp}$)
  diffeomorphism, the 
  the push-forward
 $\psi=\varphi_{*}:X=T^{*}Q\to X$ viewed as (resp. piecewise) 
diffeomorphism in $X$ defines a continuous isomorphism
  $\psi_{*}:L^{2}(X;\fkF,\tilde{g}^{\fkF})\to
  L^{2}(X;\fkF,\tilde{g}^{\fkF})$\,.
\item[c)]
 If $\Psi:\Lambda TQ\otimes \Lambda T^{*}Q\to \Lambda TQ\otimes
 \Lambda T^{*}Q$ be a $\mathcal{C}^{\infty}$ (or piecewise
 $\mathcal{C}^{\infty}$ on $\overline{Q}_{\mp}$)  vector bundle
isomorphism\,, then
  $[\pi_{X}^{*}(\Psi)]_{*}:L^{2}(X;\fkF,\tilde{g}^{\fkF})\to
  L^{2}(X,\fkF,\tilde{g}^{\fkF})$  is a continuous isomorphism.
\item[d)] The vector bundle isomorphism $\Sigma_{*}:\fkF\to \fkF$ defines
  a unitary involution of $L^{2}(X;\fkF,\hat{g}^{\fkF})$ when $\fkF=E,E'$\,. The same
  holds for $\Sigma_{\nu}$ when $\fkF=F,F'$ and 
  $L^{2}(X;\fkF,\hat{g}^{\fkF})=L ^{2}_{ev}(X,\fkF;\hat{g}^{\fkF})\oplus^{\perp}L ^{2}_{odd}(X,\fkF;\hat{g}^{\fkF})$\,.
\item[e)] When $g^{TQ}_{0}=(d\underline{q}^{1})^{2}+m^{TQ'}(0,\underline{q}')$
  and
  $g^{TQ}=(d\underline{q}^{1})^{2}+m^{TQ'}(\underline{q}^{1},\underline{q}')$\,,
  the vector bundle isomorphim $\widehat{\Psi}_{X}^{g,g_{0}}$ of
  Definition~\ref{de:hPsigg0} defines a continuous isomorphism
  $(\widehat{\Psi}^{g,g_{0}}_{X})_{*}$  from $L^{2}(X_{(-\varepsilon,\varepsilon)};\fkF,\hat{g}_{0}^{\fkF})$ to
  $L^{2}(X_{(-\varepsilon,\varepsilon)};\fkF,\hat{g}^{\fkF})$\,, with 
  $\hat{g}_{0}^{\fkF}=g_{0}^{E}\otimes \hat{g}^{\fkf}$ when $\fkF=F$ and
  $\hat{g}_{0}^{\fkF}=g_{0}^{E'}\otimes \hat{g}^{\fkf}$ when $\fkF=F'$\,.
  \end{enumerate}
\end{proposition}
\begin{proof}
 \noindent\textbf{a)}It is simply the pointwise duality.\\
\noindent\textbf{b)} When $\varphi:Q\to Q$ is a (piecewise) $\mathcal{C}^{\infty}$-diffeomorphism
$$
(Q,P)=\psi(q,p)=\varphi_{*}(q,p)=(\varphi(q),
{}^{t}D\varphi_{q}^{-1}(q)p)=(\varphi(q), A(q)p)
$$
the frames $(e,\hat{e})$ are transformed according to 
$$
dQ^{i}=[D\varphi_{q}]^{i}_{j}dq^{j}+[DA_{q})^{k}_{ij}]p_{k}dq^{j}\quad,\quad dP_{j}=[A(q)]_{j}^{k}dp_{k}\,,
$$
where all the $q$-dependent coefficients are uniformly bounded.\\
The weight $\langle p\rangle_{q}^{\pm N_{H}\mp N_{V}}$ depending on
the case for $\fkF$\,, ensures 
$$
\|\psi^{*}s\|_{L^{2}(\tilde{g}^{\fkF})}\leq C\|s\|_{L^{2}(\tilde{g}^{\fkF})}
$$
and the same can be done for $\psi_{*}=(\psi^{-1})^{*}=(\varphi^{*})^{*}$\,.\\
\noindent\textbf{c)} When $\fkF=E,E'$ this comes from the identification of
$TX=TX^{H}\oplus^{\perp}TX^{V}=\pi_{X}^{*}(TQ\oplus^{\perp} T^{*}Q)$
and
$T^{*}X=TX^{V}\oplus^{\perp}T^{*}X^{V}=\pi_{X}^{*}(T^{*}Q\oplus^{\perp}TQ)$\,. When
$\fkF=F,F'$ we set $\varphi=\pi_{\Lambda TQ\otimes \Lambda T^{*}Q}$
and we first extend $\Psi$ to $(\Lambda TQ\otimes \Lambda
T^{*}Q)\otimes \fkf$ and to $(\Lambda T^{*}Q\otimes\Lambda TQ)\otimes
\fkf$ as $\Psi\otimes \varphi_{*}$ and then pull it back via $\pi^{*}_{X}$\,.\\
\noindent\textbf{d)}  $\Sigma_{\nu}:(\fkF,\hat{g}^{\fkF})\to
(\fkF,\hat{g}^{\fkF})$ is an isometry and $\Sigma:X\to X$ is a
symplectic map.\\
\noindent\textbf{e)} We already know that
$\widehat{\Psi}_{X}^{g,g_{0}}:(\fkF,\hat{g}_{0}^{\fkF})\to
(\widehat{\fkF}_{g},\hat{g}^{\fkF})$ is an isometry projected to
$\pi_{X}(\widehat{\Psi}_{X}^{g,g_{0}})=\hat{\varphi}^{g,g_{0}}_{X}$\,,
while the map $\hat{\varphi}^{g,g_{0}}_{X}$ is given locally by
$\hat{\varphi}_{X}^{g,g_{0}}(q,p)=x$ with $\tilde{q}(x)=q$ and
$\tilde{p}(x)=p$ according to the Definition~\ref{de:tildeqp} of the
coordinates $(\tilde{q},\tilde{p})$\,. With \eqref{eq:dvxtilde} the
map $|\det \psi(\tilde{q})|^{1/2}(\hat{\Psi}_{X}^{g,g_{0}})_{*}$ is unitary from
$L^{2}(X_{(-\varepsilon,\varepsilon)};\fkF,\hat{g}_{0}^{\fkF})$ to
$L^{2}(X_{(-\varepsilon,\varepsilon)};\fkF,\hat{g}^{\fkF})$ while the
multiplication by $|\det \psi(\tilde{q})|^{\pm 1/2}$ is an isomorphism.
\end{proof}
For traces along $X'=\partial X_{-}=\partial X_{+}$ we also need
global $L^{2}$-spaces. It actually suffices to specify the volume
element along $X'$
\begin{definition}
\label{de:volX'} On $X'$ the volume element denoted $dv_{X'}=|dq'dp|$ equals 
$$
dv_{X'}=\left|\mathbf{i}_{e_{1}}\frac{1}{d!}^{d}\sigma\right|
$$
The volume element $|p_{1}|dv_{X'}=|p_{1}||dp_{1}dq'dp'|$ equals
$|(\mathbf{i}_{Y_{\fkf}}\frac{\sigma^{d}}{d!})\big|_{TX'}|$ where $Y_{\fkh}$ is the
Hamiltonian vector field associated with $\fkh$\,.
\end{definition}
The above definition does not rely on a coordinate
system. However with the local coordinates $(q,p)$ and
$(\tilde{q},\tilde{p})$ with
$(\tilde{q},\tilde{p})\big|_{X'}=(q,p)\big|_{X'}$\,, we get
$$
dv_{X'}=|dq'dp|=|d\tilde{q}'d\tilde{p}|\quad\text{and}\quad |p_{1}|dv_{X'}=|p_{1}||dq'dp_{1}dp'|=|\tilde{p}_{1}||d\tilde{q}'d\tilde{p}_{1}d\tilde{p}'|\,.
$$
Because $L^{2}(X;\fkF')$ is the dual of $L^{2}(X;\fkF)$
via the duality product  \eqref{eq:dualL2},
the adjoints of operators acting in $L^{2}(X;\fkF)$ can be defined in $L^{2}(X;\fkF')$\,.
\begin{definition}
\label{de:formaladjoint}
The vector bundle $\fkF=E,E',F,F'$ with dual $\fkF'=E',E,F,F'$ is endowed with the metric
$\tilde{g}^{\fkF}=g^{\fkF}$ or $\tilde{g}^{\fkF}=\hat{g}^{\fkF}$\,.\\
Let $\Omega$ be an open set in $X$ and let
$(P,D(P))$ be a densely defined operator in
$L^{2}(\Omega;\fkF,\tilde{g}^{\fkF})$\,.
The adjoint
denoted $(\tilde{P},D(\tilde{P}))$ in $L^{2}(\Omega;\fkF',\tilde{g}^{\fkF'})$
is defined by
\begin{eqnarray*}
  && \left(t\in D(\tilde{P})\right)\Leftrightarrow
\left( \exists C_{t}\geq 0\,,\forall s'\in D(P)\,, |\langle
     t,Ps'\rangle|\leq C_{t}\|s'\|_{L^{2}}\right)\\
&& \forall t\in D(\tilde{P})\,, \forall s'\in D(P)\,, \quad 
\langle \tilde{P}t\,,\, s'\rangle=\langle t\,,\, Ps\rangle\,.
\end{eqnarray*}
\end{definition}
Since  \eqref{eq:scalL2} and \eqref{eq:dualL2} give a unitary mapping
$U:L^{2}(\Omega;\fkF',\tilde{g}^{\fkF'})\to L^{2}(\Omega;\fkF,\tilde{g}^{\fkF})$\,, there is a simple relation
with the usual adjoint $(P^{*},D(P^{*}))$ for the
$L^{2}(X;\fkF,\tilde{g}^{\fkF})$ scalar product,
$(\tilde{P},D(\tilde{P}))=(U^{-1}P^{*}U, U^{-1}D(P^{*}))$\,.\\
However Bismut constructions of adjoint in
$L^{2}(\Omega;\fkF,\tilde{g}^{\fkF})$ involves
 a non symmetric (resp.
non hermitian) non degenerate bilinear (resp. sesquilinear) on $E$
(resp. $F$)\,.
For an isomorphism  $\phi:E'\to E$ with adjoint ${}^{t}\phi:E'\to E$
we keep the same notation for $\phi=\phi\otimes
\mathrm{Id}_{\pi_{X}^{*}(\fkf)}:F'\to F$\,.
Let the vector bundle isomorphism  $\phi:E'\to E$  be such that
\begin{equation}
  \label{eq:estimunivphi}
\exists C>0\,, \forall x\in X\,, \forall t\in \fkF_{x}\,,\quad C^{-1}|\phi t|_{\tilde{g}^{\fkF}_{x}}\leq
|t|_{\tilde{g}^{\fkF'}_{x}}\leq C|\phi t|_{\tilde{g}^{\fkF}_{x}}\,,
\end{equation}
This gives  two sesquilinear forms on $\fkF'$ and $\fkF$ dual to each other
$$
\eta_{\phi}(U,V)=\langle U\,,\,\phi V\rangle_{\fkF',\fkF}\quad,\quad
\eta^{*}_{\phi}(\omega,\theta)=\langle (\phi^{-1}\omega)\,,\,\theta\rangle_{\fkF',\fkF}\,.
$$
For sections $s,s'\in L^{2}(X,\fkF)$ 
we  set
\begin{equation}
  \label{eq:bilphi}
\langle s\,,\,s'\rangle_{\phi}=\int_{X}\eta^{*}_{\phi}(s,s')~dv_{X}\,.
\end{equation}
Because $\phi$\,, $\eta_{\phi}$ and $\eta^{*}_{\phi}$ are
 not assumed to be  neither symmetric (or hermitian) nor
 anti-symmetric (or anti-hermitian),  left and
right adjoints must be distinguished.
\begin{definition}
 \label{de:phiformadj}
The vector bundle $\fkF=E,E',F,F'$ with dual $\fkF'=E',E,F,F'$ is endowed with the metric
$\tilde{g}^{\fkF}=g^{\fkF}$ or $\tilde{g}^{\fkF}=\hat{g}^{\fkF}$\,.\\
We assume that $\phi:F'\to F$ satisfies \eqref{eq:estimunivphi} and
$\Omega$ is an open subset of $X$\,.\\
The left $\phi$-adjoint of a densely defined operator $(P,D(P))$ in
$L^{2}(\Omega;\fkF)$\,, denoted by
$(P^{\phi},D(P^{\phi}))$
is defined by
\begin{eqnarray*}
  && \left(s\in D(P^{\phi})\right)\Leftrightarrow
\left( \exists C_{s}\geq 0\,,\forall s'\in D(P)\,, |\langle
     s\,,\,Ps'\rangle_{\phi}|\leq C_{s}\|s'\|_{L^{2}}\right)\\
&& \forall s\in D(P^{\phi})\,, \forall s'\in D(P)\,, \quad 
\langle P^{\phi}s\,,\, s'\rangle_{\phi}=\langle s\,,\, Ps\rangle_{\phi}\,.
\end{eqnarray*}
The right $\phi$-adjoint is defined similarly by considering the
continuity of 
$s'$-dependent anti-linear form $D(P)\ni s\mapsto
\langle
P s\,,\, s'\rangle_{\phi}$ \,.
\end{definition}
\begin{proposition}
\label{pr:adjlr}
The left $\phi$-adjoint $(P^{\phi},D(P^{\phi}))$ equals
$(\phi\tilde{P}\phi^{-1}, \phi D(\tilde{P}))$ while the right $\phi$-adjoint
equals
$(P^{{}^{t}\phi}, D(P^{{}^{t}\phi}))=
({}^{t}\phi \tilde{P}{}^{t}\phi^{-1}; {}^{t}\phi D(\tilde{P}))$\,.\\
As adjoints of densely defined operators, the operators $(P^{\phi},D(P^{\phi}))$ and
$(P^{{}^{t}\phi},D(P^{{}^{t}\phi}))$ are closed in
$L^{2}(\Omega,\fkF)$\,.
When $(P,D(P))$ is closed and densely defined then the same holds for
$(P^{\phi}, D(P^{\phi}))$ and $(P^{\phi})^{{}^{t}\phi}=P$\,.\\
When $(P,D(P))$ is densely defined and closable,
$(P^{\phi})^{{}^{t}\phi}=\overline{P}$ the closure of $P$\,.
\end{proposition}
\begin{proof}
  Just write
$$
\langle \phi^{-1}\tilde{P}^{\phi}s\,,\,s'\rangle=
\langle P^{\phi}s\,,\,
s'\rangle_{\phi}=\langle s\,,\, Ps'\rangle_{\phi}=
\langle \phi^{-1}s\,,\, Ps\rangle=\langle \tilde{P}\phi^{-1}s\,,\, s'\rangle
$$
and use the definitions to get $P^{\phi}=\phi \tilde{P}\phi^{-1}$\,.\\
With $\overline{\langle s\,,\, s'\rangle_{\phi}}
=\langle s'\,,\,
s\rangle_{{}^{t}\phi}$
write for the right $\phi$-adjoint $P^{r,\phi}$
$$
\langle Ps\,,\, s'\rangle_{\phi}=\langle s\,,\, P^{r,\phi}s'\rangle_{\phi}
$$
in the form
$$
\langle s'\,,\, Ps\rangle_{{}^{t}\phi}=\langle P^{r,\phi}s'\,, s'\rangle_{{}^{t}\phi}\,,
$$
so that  $P^{r,\phi}=P^{{}^{t}\phi}={}^{t}\phi
\tilde{P}{}^{t}\phi^{-1}$\,.\\
The relation $(P^{\phi})^{{}^{t}\phi}=\overline{P}$ comes from the
standard theory of adjoints in Hilbert (therefore reflexive Banach)
spaces, according to
$$
(P^{\phi})^{{}^{t}\phi}={}^{t}\phi
\widetilde{P^{\phi}}~{}^{t}\phi^{-1}\quad,\quad P^{\phi}=\phi
\tilde{P}~\phi^{-1}\quad,\quad  \widetilde{P^{\phi}}={}^{t}\phi^{-1}\tilde{\tilde{P}}~{}^{t}\phi={}^{t}\phi^{-1}P~{}^{t}\phi\,.
$$
\end{proof}
\subsubsection{Global Sobolev scale}
\label{sec:globSob}
In \cite{Leb1}, G.~Lebeau introduced the Sobolev scale
$\mathcal{W}^{\mu}(X;\Lambda T^{*}X\otimes \fkf)$\,, $\mu\in \rz$\,, when $(Q,g^{TQ})$ is a smooth closed compact
riemannian manifold, which is adapted to the geometry of $X=T^{*}Q$
and to the analysis of Bismut's hypoelliptic Laplacian.
We adapt those definitions to our case with our notations.
\begin{definition}
\label{de:Ws} Let $\fkF=E,E',F,F'$ be endowed with the smooth metric
$g^{\fkF}$\,.\\
 For $n\in \nz$\,, the space $\mathcal{W}^{n}(X;\fkF)$
is the set of sections $s\in L^{2}(X;\fkF,g^{\fkF})$ for which there
exists $C_{s}>0$ such that
$$
\red
\Big\|\langle
p\rangle_{q}^{2n_{3}}\big(\prod_{k=1}^{n_{1}}\nabla_{\tilde{U}_{k}}^{\fkF}\big)\big(\prod_{\ell=1}^{n_{2}}\langle
     p\rangle_{q}\nabla_{\tilde{V}^{\ell}}^{\fkF}\big)s\Big\|_{L^{2}(g^{\fkF})}\leq
     C_{s}\big(\prod_{k=1}^{n_{1}}\|U_{k}\|_{W^{n,\infty}(Q;TQ)}\big)
\big(\prod_{\ell=1}^{n_{2}}\|V^{\ell}\|_{W^{n,\infty}(Q;TQ)}\big)
$$
where
$$
\red
n_{1}+n_{2}+n_{3}\leq n\,,\quad U_{k}\in
              \mathcal{C}^{\infty}(Q;TQ)\,, V^{\ell}\in
              \mathcal{C}^{\infty}(Q;T^{*}Q)\,,
$$
and where $\tilde{U}_{k}$ (resp. $\tilde{V}^{\ell}$) is the horizontal
(resp. vertical) lift of $U_{k}$ (resp. $V^{\ell}$)\,.\\
The space $\mathcal{W}^{n}(X;\fkF)$ can be given a Hilbert space
structure (see \eqref{eq:normWn} below) and $\mathcal{W}^{\mu}(X;\fkF)$ for $\mu\in\rz$\,, is then defined by duality
and interpolation.\\
For $\Omega=X_{\mp}$\,, $\mathcal{W}^{\mu}(\overline{\Omega};\fkF)$\,,
$\mu\in\rz$\,, is defined as 
$$
\mathcal{W}^{\mu}(\overline{\Omega};\fkF)=\left\{u\in
  \mathcal{D}'(\Omega;\fkF)\,, \exists \tilde{u}\in
  \mathcal{W}^{\mu}(X;\fkF)\,, u=\tilde{u}\big|_{\Omega}\right\}\,.
$$
\end{definition}
Here are some explanations and we refer the reader to
\cite{Leb1}\cite{Leb2} for details.
For $\mu=n\in\nz$\,, $s\in \mathcal{W}^{n}(X;\fkF)$ can be checked by
introducing a partition of unity on $Q$\,,
$\sum_{m}\chi_{m}^{2}(\underline{q})=1$ subordinate to an atlas
and by taking the $U_{k}$\,, $V^{\ell}$ in the local frame
$(\frac{\partial}{\partial\underline{q}^{i}},d\underline{q}^{j})$
with 
\begin{eqnarray}
\label{eq:orderei}
  &&\tilde{U}_{k}=e_{i}=\overbrace{\frac{\partial}{\partial
     q^{i}}}^{order~1}+\Gamma_{ij}^{k}(q){\red\overbrace{p_{k}\frac{\partial}{\partial
     p_{j}}}^{order~1}}\\
\label{eq:orderej}&&
   \langle p\rangle_{q}\tilde{V}^{\ell}= {\red \langle p\rangle_{q}}\hat{e}^{j}=\red\overbrace{\langle p\rangle_{q} 
\frac{\partial}{\partial p_{j}}}^{order~1}\,.
\end{eqnarray}
By setting for $\alpha\in \nz^{d}$\,,
$\nabla^{\fkF, \alpha}_{e}=\nabla^{\fkF, \alpha_{1}}_{e_{1}}\ldots
\nabla_{e_{d}}^{\fkF,\alpha_{d}}$ and $\nabla^{\fkF, \alpha}_{\hat{e}}=
\nabla_{\hat{e}^{1}}^{\fkF,\alpha_{1}}\ldots
\nabla_{\hat{e}^{d}}^{\fkF,\alpha_{d}}$ a Hilbert norm on
$\mathcal{W}^{n}(X,\fkF)$ is given by
\begin{equation}
  \label{eq:normWn}
\|s\|_{\mathcal{W}^{n}}^{2}=
\sum_{m}\sum_{|\alpha|+|\beta|+n_{3}\leq n}
\|\langle
p\rangle^{2n_{3}+|\beta|}_{q}\nabla_{e}^{\fkF,\alpha}\nabla_{\hat{e}}^{\fkF,\beta}[\chi_{m}(q)s]\|_{L^{2}(g^{\fkF})}^{2}\,.
\end{equation}
Although the multiplication by $\langle p\rangle^{n_{3}}$ and the covariant
derivatives themselves do not commute, changing the order in the above expression
gives an equivalent norm because:
\begin{itemize}
\item $\nabla^{\fkF}$ is the pull-back of a connection of the
 fiber
  bundle $\Lambda TQ\otimes \Lambda T^{*}Q\otimes \fkf$ (resp. 
$\Lambda TQ\otimes \Lambda T^{*}Q\otimes \fkf$) on $Q$ with
  $\nabla^{\fkF}_{\tilde{U}_{1}}\nabla^{\fkF}_{\tilde{U}_{2}}-\nabla^{\fkF}_{\tilde{U}_{2}}\nabla^{\fkF}_{\tilde{U}_{1}}=\nabla^{\fkF}_{\widetilde{[U_{1},U_{2}]}}+\pi^{*}_{X}R(U_{1},U_{2})$
  where $R(U_{1},U_{2})$ is a smooth endomorphim valued section on $Q$\,,
therefore  independent of $p$\,.
\item The above additional term $\nabla^{\fkF}_{\widetilde{[U_{1},U_{2}]}}$ is
  estimated with \eqref{eq:curveiej}\,.
\item The covariant vertical derivatives are the trivial ones.
\item Changing the position of the weight multiplication brings lower order
  corrections owing to
$$
e_{i}(f(\fkh))=0\quad,\quad \langle
p\rangle_{q}^{t}\frac{\partial}{\partial p_{j}}\langle
p\rangle_{q}^{-t}=\frac{\partial}{\partial p_{j}}+\mathcal{O}(\langle p\rangle_{q}^{-1})
$$
for any $f\in \mathcal{C}^{1}(\rz;\rz)$\,, $\langle p\rangle_{q}=(1+2\fkh)^{1/2}$\,, and any $t\in \rz$\,.
\end{itemize}
The abstract definition of $\mathcal{W}^{\mu}(X;\fkF)$ by duality and
interpolation can be specified as follows (see \cite{Leb1}): Once the
localisation in $q$ is made, assume $s=\chi_{m}(q)s$\,, take a dyadic
partition of unity
$\theta^{2}_{0}(|p|_{q}^{2})+\sum_{m'=1}^{\infty}\theta_{1}^{2}(2^{m'}|p|_{q})=\sum_{m'=0}^{\infty}\tilde{\chi}_{m'}^{2}(|p|_{q})\equiv
1$
then $s\in \mathcal{W}^{\mu}(X;\fkF)$ can be replaced by
$2^{m'}$-dependent estimates of $2^{m'd/2}(\tilde{\chi}_{m'}s)(q,2^{m'}p)$ in
$\red W^{\mu,2}_{comp}(X,\fkF)$\,, with a fixed compact support in
$(q,p)\in \rz^{2d}$\,. And this can be characterized by standard
pseudodifferential calculus.

 Bismut and Lebeau in
\cite{Leb1}\cite{Leb2}\cite{BiLe} work actually 
with  the metrics
\begin{eqnarray*}
  && \tilde{g}^{E'}=\langle p\rangle_{q}^{-2N_{V}}\pi_{X}^{*}(g^{\Lambda TQ}\otimes
     g^{\Lambda T^{*}Q})=\langle
     p\rangle_{q}^{-N}g^{E'}\quad\text{on}~E'=\Lambda TX\,,\\
\text{and}&&
\tilde{g}^{E}=
\langle p\rangle_{q}^{2N_{V}}\pi_{X}^{*}(g^{\Lambda T^{*}Q}\otimes
     g^{\Lambda TQ})=\langle
     p\rangle_{q}^{N}g^{E}\quad\text{on}~E=\Lambda T^{*}X\,,
\end{eqnarray*}
with the corresponding metric $\tilde{g}^{F'}=\tilde{g}^{E'}\otimes
g^{\fkf}$ and $\tilde{g}^{F}=\tilde{g}^{E}\otimes g^{\fkf}$\,.  But
this is a particular case of the weighted $\mathcal{W}^{\mu}$-spaces
which is discussed below.
\begin{proposition}
\label{pr:indepWmu}
 Let $\fkF=E,E',F,F'$ be endowed with the smooth metric
$g^{\fkF}$\,.\\
The spaces $\mathcal{W}^{\mu}(X;\fkF)$\,, $\mu\in\rz$\,, have the
following properties:
\begin{description}
\item[a)] They do not depend on the chosen metric $g^{TQ}$\,.
\item[b)] When $\varphi:Q\to Q$ is a diffeomorphism and the vector bundle
  isomorphisms $\psi=\varphi_{*}:X=T^{*}Q\to X$ is viewed as as
  diffeomorphism of $X$\,, defines an isomorphism
  $\psi_{*}:\mathcal{W}^{\mu}(X;\fkF)\to \mathcal{W}^{\mu}(X;\fkF)$\,.
\item[c)] If $\Psi:\Lambda TQ\otimes \Lambda T^{*}Q\to \Lambda TQ\otimes
 \Lambda T^{*}Q$ be a $\mathcal{C}^{\infty}$ vector bundle
 isomorphism, then  $[\pi_{X}^{*}(\Psi)]_{*}:\mathcal{W}^{\mu}(X;\fkF)\to
 \mathcal{W}^{\mu}(X,\fkF)$ is a continuous isomorphism.
\item[d)] The space $\mathcal{W}^{\mu}_{loc}(X;\fkF)$ is nothing but
  $\red W^{\mu,2}_{loc}(X;\fkF)$\,. In particular the trace
  $s\big|_{X'}$ is well defined as soon as $s\in
  \mathcal{W}^{\mu}(X;\fkF)$ with $\mu>1/2$\,. Additionally for
  $\mu\in ]1/2,1]$\,, $s\in \mathcal{W}^{\mu}(X;\fkF)$ is equivalent
  to $s_{\mp}=s\big|_{X_{\mp}}\in
  \mathcal{W}^{\mu}(\overline{X}_{\mp};\fkF)$ and
  $s_{-}\big|_{\partial X_{-}}=s_{+}\big|_{\partial X_{+}}$\,, while
  the trace condition is dropped when $\mu\in[0,1/2[$\,.
\item[e)] Let $G$ be vector bundle isomorphism, $G\in \mathcal{C}^{\infty}(X;L(\fkF))$\,, such
  that 
  \begin{eqnarray*}
    &&
(\nabla_{e_{i}}^{L(\fkF)}G)G^{-1}\,,\,
G^{-1}(\nabla_{e_{i}}^{L(\fkF)}G)\in \mathcal{L}(\mathcal{W}^{\mu}(X;\fkF);\mathcal{W}^{\mu-1}(X;\fkF))
\\
&&
(\nabla_{\hat{e}^{j}}^{L(\fkF)}G)G^{-1}\,,\,
   G^{-1}(\nabla^{L(\fkF)}_{\hat{e}^{j}}G)\in \mathcal{L}(\mathcal{W}^{\mu}(X;\fkF);\mathcal{W}^{\mu-1/2}(X;\fkF))\,,
  \end{eqnarray*}
  then the norm of $G\mathcal{W}^{n}(X;\fkF)=\left\{s\in
    {\red W^{n,2}_{loc}(X;\fkF)}\,, G^{-1}s\in
    \mathcal{W}^{n}(X;\fkF) \right\}$ can be given by the same
  expression as \eqref{eq:normWn} where only the metric, $g^{\fkF}$\,, and
  $L^{2}$-norm,
  $\|~\|_{L^{2}(g^{\fkF})}$\,, are replaced 
  respectively by
  $\tilde{g}^{\fkF}(v,v)=g^{\fkF}(G^{-1}v,G^{-1}v)$ and $\|~\|_{L^{2}(\tilde{g}^{\fkF})}$\,. The weighted
  spaces $G\mathcal{W}^{\mu}(X;\fkF)$\,, $\mu\in \rz$\,, can thus be characterized
  without changing the connection $\nabla^{\fkF}$\,.
\end{description}
\end{proposition}
\begin{proof}
\noindent\textbf{a)}  It suffices to consider the case
$\mathcal{W}^{n}(X;\fkF)$ for $n\in\nz$\,, where the result is already
known for $n=0$\,, and to work locally with
the coordinates $(q,p)\in U\times \rz^{d}$\,, $U$ open set of $\rz^{d}$\,. We take the euclidean
metric $g_{e}=g^{TU}_{e}=\sum_{i=1}^{d}(d\underline{q}^{i})^{2}$ as a
reference metric for which the local frame in $TX$ and $T^{*}X$ are
simply $\frac{\partial}{\partial q^{i}}$\,, $\frac{\partial}{\partial
  p_{j}}$ and $dq^{i},dp_{j}$\,, while $\nabla^{\fkF,g_{e}}$ can be
chosen as the trivial connection. The weights which are powers of  $\langle p\rangle^{2}_{g,q}=(1+g^{ij}(q)p_{i}p_{j})$
and $\langle p\rangle^{2}_{g_{e},q}=1+\sum_{j}^{d}p_{j}^{2}$
are uniformly equivalent with all the derivatives
$$
\partial_{q}^{\alpha}\partial_{p}^{\beta}\left(\frac{\langle
    p\rangle_{g,q}}{\langle p\rangle_{g_{e},q}}\right)^{\pm 1}
$$
uniformly bounded. It thus suffices to compare the covariant
derivatives:
\begin{eqnarray*}
  && \nabla^{\fkF,g}_{e_{i}}-\nabla^{\fkF,g_{e}}_{\frac{\partial}{\partial
    q^{i}}}=
     (\nabla^{\fkF,g}-\nabla^{\fkF,g_{e}})_{e_{i}}+e_{i}-\frac{\partial}{\partial q^{i}}
=\pi_{X,*}(\nabla^{Q,g}-\nabla^{Q,g_{e}})_{e_{i}}+\Gamma_{ij}^{k}(q)\frac{\partial}{\partial
     p_{j}}\,,\\
&&\nabla_{\frac{\partial}{\partial
   p_{j}}}^{\fkF,g}=\nabla^{\fkF,g_{e}}_{\frac{\partial}{\partial
   p_{j}}}=\frac{\partial}{\partial p_{j}}\,.
\end{eqnarray*}
Since $\nabla^{Q,g}-\nabla^{Q,g_{e}}\in
\mathcal{C}^{\infty}(Q;L(\Lambda TQ\otimes \Lambda T^{*}Q))$\,, we
deduce that the local expression of the  norm of $\mathcal{W}^{n}(X;\fkF)$\,,
for $\chi_{m}s$ with the metric $g$ and $g_{e}$ a neighborhood of the support of
$\chi_{m}$ are uniformly equivalent. This provides the local result
for two different metric $g_{1}$\,, $g_{2}$ and taking the full finite
sum in \eqref{eq:normWn} ends the proof.\\

\noindent\textbf{b)} Again we can work locally and by \textbf{a)} we can take
the euclidean metric $g^{TQ}=g^{TQ}_{e}$
 on $U$ and $\varphi(U)$\,. Write $(Q,P)=\psi(q,p)=(\varphi(q),
 {}^{t}[D\varphi_{q}]^{-1}p)=(\varphi(q), A(q)p)$ and
 \begin{eqnarray*}
   && \overbrace{\frac{\partial}{\partial
      q^{i}}}^{order~1}=[D\varphi_{q}]_{i}^{j}\overbrace{\frac{\partial}{\partial
      Q^{j}}}^{order~1}+[DA_{q}]_{ij}^{k}
{\red\overbrace{P_{k}\frac{\partial}{\partial P_{j}}}^{order~1}}\\
&&\red \overbrace{\langle p\rangle\frac{\partial}{\partial
   p_{j}}}^{order~1}=A_{k}^{j}(q)\frac{\langle p\rangle}{\langle A(q)p\rangle}\overbrace{\langle P\rangle_{q}\frac{\partial}{\partial P_{k}}}^{order~1}\,. 
 \end{eqnarray*}
Since locally with $g=g_{e}$\,, the connection $\nabla^{\fkf,g_{e}}$
becomes a trivial one, the equivalence of the norm \eqref{eq:normWn}
of $\chi_{m}s$ and $\psi_{*}[\chi_{m}s]$ follows. By \textbf{a)} this
equivalence holds for the metric $g^{\fkF}$ put on $U\supset \supset
\mathrm{supp}\chi_{m}s$ and $\varphi(U)$\,, and we conclude by summing
with respect to $m$ in \eqref{eq:normWn}\,.\\
\noindent\textbf{c)} By \textbf{b)} the problem is reduced to the case when
$\varphi=\pi_{\Lambda TQ\otimes \Lambda T^{*}Q}(\Psi)=\Id_{Q}$  and
$\Psi\in \mathcal{C}^{\infty}(Q;L(\Lambda TQ\otimes \Lambda T^{*}Q))$
and the result comes from
$\nabla^{\fkF}=\pi_{X}^{*}(\nabla^{Q,g}\otimes
\nabla^{\fkf~or~\fkf'})$ while we already know the result for $n=0$ by
Proposition~\ref{pr:L2glob}-c).\\
\noindent\textbf{d)} Locally\,, that is while considering
$\chi s$ with $\chi\in \mathcal{C}^{\infty}_{0}(X;\rz)$\,,
 the weight $\red\langle p\rangle_{q}^{2n_{3}+|\beta|}$
can be forgotten and $\mathcal{W}^{n}_{loc}(X;\fkF)$ is nothing but
$\red W^{n,2}_{loc}(X;\fkF)$\,. Choosing $\chi$ with a small enough
support we can even consider the map $s\mapsto \chi s$ as a continuous
map from $\mathcal{W}^{n}(X;\fkF)$ to $\red W^{n,2}(\rz^{2d}; 
\cz^{N_{d,\fkf}})$ and the continuity from $\mathcal{W}^{\mu}(X;\fkF)$
to $\red W^{\mu,2}(\rz^{2d};\cz^{N_{d,\fkf}})$ for any $\mu\in
\rz$ holds true by duality and interpolation. This proves
$\mathcal{W}^{\mu}_{loc}(X;\fkF)=\red W^{\mu,2}_{loc}(X;\fkF)$ for
all $\mu\in \rz$\,. \\
\noindent\textbf{e)}The norm of $s\in G\mathcal{W}^{n}(X,\fkF)$ equals
$\|s\|_{G\mathcal{W}^{n}}=\|G^{-1}s\|_{\mathcal{W}^{n}}$ while $G^{-1}\langle
p\rangle_{q}^{n_{3}}=\langle p\rangle_{q}^{n_{3}}G^{-1}$\,. 
The expression \eqref{eq:normWn} gives
$$
\red
\|s\|_{G\mathcal{W}^{n}}^{2}=
\sum_{m}\sum_{|\alpha|+|\beta|+n_{3}\leq n}
\|\langle
p\rangle^{2n_{3}+|\beta|}_{q}\tilde{\nabla}_{e}^{\fkF,\alpha}\tilde{\nabla}_{\hat{e}}^{\fkF,\beta}[\chi_{m}(q)s]\|_{L^{2}(\tilde{g}^{\fkF})}^{2}\,,
$$
with 
\begin{eqnarray*}
  && \tilde{\nabla}^{\fkF}=G\nabla^{\fkF} G^{-1}=
\nabla^{\fkF} +G\nabla^{L(\fkF)}G^{-1}=\nabla^{\fkF}-(\nabla^{L(\fkF)} G)G^{-1}\\
&& \nabla^{\fkF}=G^{-1}\tilde{\nabla}^{\fkF}G=\tilde{\nabla}^{\fkF}+G^{-1}(\nabla^{L(\fkF)} G)\,,\\
\text{and}&& \tilde{g}^{\fkF}(v,v)=g^{\fkF}(G^{-1}v,G^{-1}v)\,.
\end{eqnarray*}
The assumptions are exactly the ones which ensure the equivalence with
the squared norm 
$$
\red
\sum_{m}\sum_{|\alpha|+|\beta|+n_{3}\leq n}
\|\langle
p\rangle^{n_{3}+|\beta|}_{q}
\nabla_{e}^{\fkF,\alpha}
\nabla_{\hat{e}}^{\fkF,\beta}
[\chi_{m}(q)s]
\|_{L^{2}(\tilde{g}^{\fkF})}^{2}\,,
$$
where the initial connection $\nabla^{\fkF}$ is used.
\end{proof}
The result \textbf{e)} will be used with two types of weights.\\
\noindent
$\boxed{G=\langle p\rangle_{q}^{\pm\frac{N_{V}+N_{H}}{2}}}$~:
The sign depends on the case $\fkF=E'$ or
$\fkF=E$:
 \begin{itemize}
 \item It changes the
  metric $g^{E'}=\langle
p\rangle_{q}^{N_{H}-N_{V}}\pi_{X}^{*}(g^{\Lambda TQ}\otimes g^{\Lambda T^{*}Q})$ into
$\tilde{g}^{E'}=\langle
p\rangle_{q}^{-2N_{V}}\pi_{X}^{*}(g^{\Lambda TQ}\otimes g^{\Lambda
  T^{*}Q})$
\item It changes the metric $g^{E}=
\langle
p\rangle_{q}^{-N_{H}+N_{V}}\pi^{*}_{X}(g^{\Lambda T^{*}Q}\otimes
g^{\Lambda TQ})
$ into $\tilde{g}^{E}=\langle
p\rangle_{q}^{2N_{V}}\pi^{*}_{X}(g^{\Lambda T^{*}Q}\otimes g^{\Lambda TQ})$\,.
 \end{itemize}
Proposition~\ref{pr:indepWmu}-\textbf{e)} applies because
$$
(\partial_{T}G) G^{-1}\,,\, G^{-1}(\partial_{T}G)\in
\mathcal{L}(\mathcal{W}^{\mu}(X;E))
$$
for $T=e_{i}$ and $T=\hat{e}^{j}$ and for any $\mu\in\nz$ (and
therefore for any $\mu\in \rz$) because $\partial_{q^{i}}(\langle
p\rangle_{q}^{t})=\mathcal{O}(\langle p\rangle_{q}^{t})$ and
$\partial_{p_{j}}(\langle p\rangle_{q}^{t})=\mathcal{O}(\langle
p\rangle_{q}^{t-1})$ for $t\in \rz$\,.\\
This choice  allows to transfer at once the estimates of
\cite{Leb1}\cite{Leb2}\cite{BiLe} where the metrics $\tilde{g}^{E}$
and $\tilde{g}^{E'}$ were chosen. The advantage
of our choice is that the tensorized
 map $\phi=\sigma:TX\to T^{*}X$  sends isometrically $(E',g^{E'})$ to
 $(E,g^{E})$ while it sends isometrically $(E',\tilde{g}^{E'})$ to
 $(E, \langle p\rangle_{q}^{-2N}\tilde{g}^{E})$\,.\\
\noindent
$\boxed{G=e^{\pm (\fkh(q,p)+V(q)}}$~:
 In
\cite{BiLe}\cite{Bis05} the $L^{2}$-norm on $\tilde{F}=E=\Lambda T^{*}X\otimes
\pi_{X}^{*}(\tilde{\fkf})$\,, is given by
$$
\int_{X}|s(q,p)|_{g^{\tilde{F}}}^{2}e^{-2\fkh(q,p)}~|dqdp|=
\|s\|^{2}_{e^{\fkh}L^{2}(g^{\tilde{F}})}\,.
$$
Additionally the metric on $\tilde{\fkf}=Q\times \cz$ is given by
$g^{\tilde{\fkf}}(z)=e^{-2V(q)}|z|^{2}$ while the flat connection is the
trivial one $\nabla^{\fkf}=\nabla$\,. Taking $z'=e^{-V(q)}z$ gives
$z=e^{V(q)}z'$\,. It is thus the same as choosing
 $\fkf=Q\times \cz$\,,
$\nabla^{\fkf}=\nabla+dV(q)$ and $g^{\fkf}(z')=|z'|^{2}$ and the above
squared norm equals
$$
\|s\|^{2}_{e^{\fkh}L^{2}(g^{\tilde{F}})}=\|s\|^{2}_{e^{\fkh +V}L^{2}(g^{F})}\,.
$$
while its dual norm satisfies
$$
\|t\|^{2}_{e^{-\fkh}L^{2}(g^{\tilde{F}'})}=\|s\|^{2}_{e^{-(\fkh+V)}L^{2}(g^{F'})}\,,
$$
with $\fkf'=Q\times \cz=\fkf$\,, $g^{\fkf'}=g^{\fkf}$\,, but
$\nabla^{\fkf'}=\nabla-dV(q)$\,.\\
With
  \begin{eqnarray*}
    &&
(\partial_{q^{i}}G) G^{-1}=G^{-1}(\partial_{q^{i}}G)=\pm \frac{\partial
 (\fkh+V(q))}{\partial q^{i}}(q,p)\in \mathcal{L}(\mathcal{W}^{\mu}(X;E);
\mathcal{W}^{\mu-1}(X;E))
\\
\text{and}&&
(\partial_{p_{j}}G) G^{-1}=G^{-1}(\partial_{p_{j}}G)
=
\pm 
\frac{\partial (\fkh +V(q))}{\partial p_{j}}(q,p)\in
             \mathcal{L}(\mathcal{W}^{\mu}(X;E);\mathcal{W}^{\mu-1/2}(X;E))\,.
\end{eqnarray*}
the result of  Proposition~\ref{pr:indepWmu}-\textbf{e)} ensures that the
regularity estimates are equivalent after simply applying the weight
to the $L^{2}$-space.
\begin{definition}
  When $\mu\in [0,1]\setminus\left\{\frac{1}{2}\right\}$ we define
$\mathcal{W}^{\mu}(X;\widehat{\fkF}_{g})$ as the set of sections $s\in
L^{2}(X;\fkF)$ such that
\begin{itemize}
\item $s_{\mp}=s\big|_{X_{\mp}}$ belongs to
  $\mathcal{W}^{\mu}(\overline{X}_{\mp};\fkF)$\,,\\
\item if $\mu\in ]1/2,1]$\,, $s_{-}\big|_{\partial
    X'}=s_{+}\big|_{\partial X'}$ in $\widehat{\fkF}_{g}\big|_{X'}$\,.
\end{itemize}
Finally for $\mu\in[0,1]$\,,
$\mathcal{W}^{\mu}_{ev}(X;\widehat{\fkF}_{g})=\mathcal{W}^{\mu}(X;\widehat{\fkF}_{g})\cap
L^{2}_{ev}(X;\fkF)$\,.
\end{definition}
  We already noticed that
  $\red W^{\mu,2}_{loc}(X;\widehat{\fkF}_{g_{0}})=W^{\mu,2}_{loc}(X;\fkF)$\,. Since
  the definition of $\mathcal{W}^{\mu}(X;\widehat{\fkF}_{g_{0}})$
  simply adds the global estimates which can be checked separated on
  both sides we get $\mathcal{W}^{\mu}(X;\widehat{\fkF}_{g_{0}})=\mathcal{W}^{\mu}(X;\fkF)$\,.
\begin{proposition}
\label{pr:hatPsiW} 
The isomorphism
$(\widehat{\Psi}^{g,g_{0}}_{X})_{*}:
L^{2}(X_{(-\varepsilon,\varepsilon)};\fkF, \hat{g}_{0}^{\fkF})\to 
L^{2}(X_{(-\varepsilon,\varepsilon)};\fkF,\hat{g}^{\fkF})$\,, 
of Definition~\ref{de:hPsigg0} and
Proposition~\ref{pr:L2glob}-\textbf{e)}, is actually an continuous isomorphism from
$\mathcal{W}^{\mu}(X_{(-\varepsilon,\varepsilon)};\fkF)$ to
$\mathcal{W}^{\mu}(X_{(-\varepsilon,\varepsilon)};\widehat{\fkF}_{g})$ for $\mu\in [0,1]$\,.
\end{proposition}
\begin{proof}
For $\mu=0$\,, $(\Psi^{g,g_{0}}_{X})_{*}:L^{2}(X;\fkF,g_{0}^{\fkF})\to
L^{2}(X;\fkF,g^{\fkF})$ is an isomorphism.\\
Consider now $\mu=1$\,.
The vector bundle morphism $(\Psi^{g,g_{0}}_{X})_{*}$ is a piecewise
$\mathcal{C}^{\infty}$ vector bundle isomorphism from $\fkF$ to
$\hat{\fkF}_{g}$ which transforms the continuity condition
$s_{-}\big|_{X'}=s_{+}\big|_{X'}$ in $\fkF\big|_{X'}$ into the same
continuity condition in $\widehat{\fkF}_{g}\big|_{X'}$\,. Additionally
by Proposition~\ref{pr:indepWmu}-\textbf{b)}\textbf{c)} applied on
both sides (or more exactly for the $\mathcal{C}^{\infty}$-metrics
$g_{-}^{TQ}$ and $g_{+}^{TQ}$ and then restricted to $X_{\mp}$)\,, we
obtain 
$$
C^{-1}\|s_{\pm}\|_{\mathcal{W}^{1}(\overline{X}_{\mp})}\leq \|
(\Psi^{g_{\mp},g_{0}}_{X})_{*}s\|_{\mathcal{W}^{1}(\overline{X}_{\mp})}
\leq C\|s_{\pm}\|_{\mathcal{W}^{1}(\overline{X}_{\mp})}\,.
$$
This proves that
$(\widehat\Psi^{g,g_{0}}_{X})_{*}:\mathcal{W}^{1}(X;\fkF)\to
\mathcal{W}^{1}(X;\widehat{\fkF}_{g})$ is an
isomorphism. Interpolation yields the result 
for $\mu\in [0,1]$\,.
\end{proof}
Finally, Proposition~\ref{pr:indepWmu}-\textbf{e)} works for
$\hat{G}\mathcal{W}^{\mu}(X;\widehat{\fkF}_{g})$\,, $\mu\in [0,1]$\,, with the weights $\hat{G}$
described in the three examples, after replacing the smooth metric
$g^{TQ}$ by $\hat{g}^{TQ}$\,, namely
$$
\hat{G}=\langle p\rangle_{\hat{g},q}^{\pm\frac{N_{H}+N_{V}}{2}}
\quad
\text{and}\quad
\hat{G}=e^{\pm (\hat{\fkh}(q,p)+\hat{V}(q))}\,,
$$
because those weights are continuous w.r.t $q^{1}$\,.

\section{Closed realizations of the differential}
\label{sec:closeddiff}

In this section trace theorems and boundary conditions for the
differential, more generally the exterior covariant derivative, are
considered. No riemannian metric is really required here and all the
analysis is made by using the proper $\mathcal{C}^{\infty}$ structure
of manifolds made of the two pieces $\overline{X}_{-}$ and
$\overline{X}_{+}$ glued in the proper way. In particular, the new
manifold $M_{g}$ is introduced in
Subsection~\ref{sec:diffhatE}. Its construction relies on the coordinates
$(\tilde{q},\tilde{p})$ related with the parallel transport in
$X=T^{*}Q$ for the Levi-Civita connection associated with
$\hat{g}^{TQ}$\,. It depends on the metric $g_{-}^{TQ}$ initially chosen
on $\overline{Q}_{-}$ and accordingly the boundary conditions for the
differential finally depend on $g_{-}^{TQ}$\,.

\subsection{General partial trace results}
\label{sec:gentrace}
Let $M$ (resp. $\overline{M}=M\sqcup M'$) be a smooth oriented manifold
(resp. with boundary $\partial M=M'$) and let $\pi_{\fkF}:\fkF\to M$
(resp. $\pi_{\fkF}:\fkF\to \overline{M}$)
be a $\mathcal{C}^{\infty}$ vector bundle on $M$
(resp. on $\overline{M}$) endowed with the non necessarily smooth
connection
\begin{eqnarray}
  \label{eq:nonsmoothconn}
  &&\nabla^{\fkF}: \mathcal{C}^{\infty}(M;\fkF)\to L^{\infty}_{loc}(M;
  T^{*}M\otimes \fkF)\\
  \label{eq:nonsmoothconn2}
\text{resp.} &&\nabla^{\fkF}:
  \mathcal{C}^{\infty}(\overline{M};\fkF)\to
                L^{\infty}_{loc}(\overline{M}; T^{*}M\otimes \fkF)\,.
\end{eqnarray}
 The exterior covariant derivatives   $d^{\nabla^{\fkF}}$ acting on
 sections of
 $\Lambda T^{*}M\otimes \fkF$  is written in
 local coordinates
$$
d^{\nabla^{\fkF}}=(dx^{i}\wedge) \frac{\partial}{\partial
  x^{i}}\otimes \mathrm{Id_{\fkF}}+
(dx^{i}\wedge)\otimes
 \nabla^{\fkF}_{\frac{\partial}{\partial x^{i}}}\,.
$$
 Remember the Definition~\ref{de:calE} of
 $\mathcal{E}_{loc~comp}(d^{\nabla^{\fkF}},\Lambda T^{*}M\otimes \fkF)$
 in the two cases $M'=\emptyset$ and
 $M'\neq \emptyset$\,.
\begin{proposition}
\label{pr:partialtrace}
\noindent\textbf{a)} If $\fkF$ is a smooth vector bundle on $M$  (resp. the manifold
with boundary $\overline{M}=M\sqcup M'$) and $\nabla^{\fkF}_{1}$\,,
$\nabla^{\fkF}_{2}$ are two connections on $\fkF$ which fulfill
\eqref{eq:nonsmoothconn} (resp. \eqref{eq:nonsmoothconn})\,, then 
$$
\mathcal{E}_{\bullet}(d^{\nabla_{1}^{\fkF}}, \Lambda T^{*}M\otimes
     \fkF)=\mathcal{E}_{\bullet}(d^{\nabla_{2}^{\fkF}}, \Lambda
     T^{*}M\otimes \fkF)\quad \bullet =loc~\text{or}~comp.
$$
\noindent\textbf{b)}When $M'\subset M$ a smooth hypersurface of $M$ (resp. a manifold
with boundary $\overline{M}=M\sqcup M'$) with the natural embedding
$j_{M'}:M'\to M$\,, the tangential trace map $s\mapsto j_{M'}^{*}s$ is well
defined an continuous from $\mathcal{E}_{loc}(d^{\nabla^{\fkF}},\Lambda
T^{*}M\otimes \fkF)$ to
$\mathcal{D}'(M';(\Lambda
T^{*}M'\otimes \fkF\big|_{M'})$\,.\\
\noindent\textbf{c)} The space $\mathcal{C}^{\infty}_{0}(M;\Lambda T^{*}M\otimes \fkF)$
(resp. $\mathcal{C}^{\infty}_{0}(\overline{M};\Lambda T^{*}M\otimes
\fkF)$) is dense in the two spaces $\mathcal{E}_{comp}(d^{\nabla^{\fkF}},\Lambda T^{*}M\otimes \fkF)$
and $\mathcal{E}_{loc}(d^{\nabla^{\fkF}},\Lambda T^{*}M\otimes \fkF)$\,.\\
\noindent\textbf{d)} In the case $\overline{M}=M\sqcup M'$ and
$\fkF=\overline{M}\times\cz$ with the trivial connection\,,
 Stokes formula 
$$
\int_{M}d\overline{s}\wedge s'
+(-1)^{\mathrm{deg}s}\overline{s}\wedge
(ds')=\int_{M'}\overline{j_{M'}^{*}s}\wedge j_{M'}^{*}s'\,,
$$
holds
$s\in \mathcal{E}_{loc}(d,\Lambda T^{*}M\otimes \cz)$ and all 
$s'\in
\mathcal{E}_{comp}(d,\Lambda T^{*}M\otimes \cz)$\,, where the
right-hand side is the unique  sesquilinear continuous extension from
$\mathcal{C}^{\infty}_{0}(\overline{M};\Lambda T^{*}M\otimes \cz)$\,. \\
\noindent\textbf{e)} When $M=M_{-}\sqcup M'\sqcup M_{+}$ and $\overline{M}_{\pm}$
are smooth domains of $M$\,, $s\in \mathcal{E}_{loc}(d^{\nabla^{\fkF}},\Lambda
T^{*}M\otimes\fkF)$ iff $s_{\mp}=s\big|_{M_{\mp}}$ belongs to
$\mathcal{E}_{loc}(d^{\nabla^{\fkF}}, (\Lambda T^{*}M\otimes \fkF)\big|_{\overline{M}_{\mp}})$ and
$$
j^{*}_{M'}s_{-}=j^{*}_{M'}s_{+}\quad \text{in}~\mathcal{D}'(M';\Lambda
T^{*}M'\otimes \fkF\big|_{M'})
$$
\end{proposition}
\begin{proof}
We work with a complex vector bundle $\fkF$\,. It does not change anything
here.\\
\noindent\textbf{a)}  The equality is due to
$\nabla_{1}^{\fkF}-\nabla_{2}^{\fkF}\in
L^{\infty}_{loc}(M;T^{*}M\otimes L(\fkF))$ and 
$$
d^{\nabla_{1}^{\fkF}}-d^{\nabla_{2}^{\fkF}}=dx^{i}\wedge
[\nabla^{\fkF}_{1,\frac{\partial}{\partial
    x^{i}}}-\nabla^{\fkF}_{2,\frac{\partial}{\partial x^{i}}}]\,,
$$
in local coordinates.
The result is then a consequence of the equivalences
 for  $s\in L^{2}_{loc}(M;\Lambda T^{*}M\otimes \fkF)$ (resp. $s\in
L^{2}_{loc}(\overline{M};\Lambda T^{*}M\otimes\fkF)$):
\begin{eqnarray*}
  && \left(d^{\nabla^{\fkF}_{1}}s\in L^{2}_{loc}(M;\Lambda
     T^{*}M\otimes \fkF)\right)
\Leftrightarrow
\left(
d^{\nabla^{\fkF}_{2}}s\in L^{2}_{loc}(M;\Lambda T^{*}M\otimes \fkF)
\right)\,,
\\
\text{resp.}&&
\left(d^{\nabla^{\fkF}_{1}}s\in L^{2}_{loc}(\overline{M};\Lambda
     T^{*}M\otimes \fkF)\right)
\Leftrightarrow
\left(
d^{\nabla^{\fkF}_{2}}s\in L^{2}_{loc}(\overline{M};\Lambda T^{*}M\otimes \fkF)
\right)\,.
\end{eqnarray*}

\noindent\textbf{b)} For the existence of a trace, the case with a boundary
$\pi_{\fkF}:\fkF\to \overline{M}$ is contained in the case without
boundary with $M'\subset M$\,, by writing $\overline{M}$ as a smooth
domain of the smooth manifold 
$\tilde{M}$ and $\fkF=\tilde{\fkF}\big|_{\overline{M}}$\,.\\
Because $d^{\nabla^{\fkF}}\chi=\chi d^{\nabla^{\fkF}}+d\chi\wedge$ for $\chi\in
\mathcal{C}^{\infty}_{0}(M;\rz)$\,, $s\in
\mathcal{E}_{loc}(d^{\nabla^{\fkF}},\Lambda T^{*}M\otimes \fkF)$
is equivalent to $\chi_{j}s\in \mathcal{E}_{comp}(d^{\nabla^{\fkF}},\Lambda
T^{*}M\otimes \fkF\big|_{U_{j}})$ for all $j$\,, when
$\sum_{j}\chi_{j}\equiv 1$ is a locally finite partition of unity
subordinate to a trivializing  atlas $M=\cup_{j}U_{j}$  for $\fkF$\,,
$\fkF\big|_{U_{j}}\simeq U_{j}\times \cz^{d_{f}}$\,. With \textbf{a)},
the connection $\nabla^{\fkF}$ can be replaced by the trivial
connection on $U_{j}\times\cz^{d_{f}}$\,. By possibly refining the
atlas we can assume $U_{j}=(-\varepsilon,\varepsilon)^{m}$ in a local
coordinate system $(x^{1},\ldots,x^{m})$ such that $U_{j}\cap
M'=\left\{0\right\}\times (-\varepsilon,\varepsilon)^{m-1}$ for
$U_{j}\cap M'\neq \emptyset$\,.
We now have to verify that
$$
\chi_{j}s\in \mathcal{E}_{comp}(d;\Lambda T^{*}(-\varepsilon,\varepsilon)^{m}\otimes
\cz^{d_{f}})
$$
has a trace along $\left\{x^{1}=0\right\}$\,.
Finally the local $L^{2}$-estimates of 
 $\chi_{j}s\in \mathcal{E}_{comp}(d;\Lambda
T^{*}U_{j}\otimes\cz^{d_{f}})$ can be expressed with the
euclidean metric on $(-\varepsilon,\varepsilon)^{m}$\,.
Working separtely on components in $\cz^{d_{f}}$ reduces the problem
to the scalar case.\\
In $U_{j}=(-\varepsilon,\varepsilon)^{m}$ with the euclidean metric,
$\chi_{j}s=s_{I}dx^{I}\in \mathcal{E}_{comp}(d; \Lambda T^{*}(-\varepsilon,\varepsilon)^{m}\otimes \cz)$ means
$$
\chi_{j}s=s_{I}(x)dx^{I}\in L^{2}_{comp}((-\varepsilon,\varepsilon)^{m};
\Lambda \cz^{m}) \quad\text{and}\quad
d\chi_{j}s\in L^{2}_{comp}(d; \Lambda \cz^{m} )
$$
gives
$$
\frac{\partial s_{I'}}{\partial x^{1}}dx^{1}\wedge dx^{I'}
=ds-\frac{\partial s_{I}}{\partial x^{i'}}dx^{i'}\wedge dx^{I} \in
L^{2}((-\varepsilon,\varepsilon);
W^{-1,2}((-\varepsilon,\varepsilon)^{m-1};\Lambda \cz^{m}))\,.
$$
Therefore every $s_{I'}$\,, $1\not\in I'$\,,
belongs to
$W^{1,2}((-\varepsilon,\varepsilon);W^{-1,2}((-\varepsilon,\varepsilon)^{m-1};\Lambda
\cz^{m-1})$ and admits a trace in
$$
W^{-1,2}((-\varepsilon,\varepsilon)^{m-1};\Lambda \cz^{m-1}
)\subset 
\mathcal{D}'((-\varepsilon,\varepsilon)^{m-1}; \Lambda \cz^{m-1}         )\,.
$$
Hence, $j^{*}_{M'}(s_{I}dx^{I})=s_{I'}dx^{I'}$ is well defined in
$\mathcal{D}'((-\varepsilon,\varepsilon)^{m-1}; \Lambda \rz^{m-1}\otimes \cz)$\,.\\
By summing the locally finite different pieces of
$s=\sum_{j}\chi_{j}s$ where all the
$\chi_{j}s$ belong to 
$\mathcal{E}_{comp}(d^{\nabla^{\fkF}}, (\Lambda
T^{*}M\otimes\fkF)\big|_{U_{j}})$\,, we
conclude  that 
$$
j_{M'}^{*}:\mathcal{E}_{loc}(d^{\nabla^{\fkF}}, \Lambda T^{*}M\otimes \fkF))\to
\mathcal{D}'(M';\Lambda T^{*}M'otimes\fkF\big|_{M'})
$$
is well defined and continuous.\\
\noindent\textbf{c)} For the density and with the local reduction to  $U_{j}=(-\varepsilon,\varepsilon)^{m}$ used
in \textbf{b)}, it
suffices to approximate $\chi_{j}s\in
\mathcal{E}_{comp}(d;\Lambda
T^{*}(-\varepsilon,\varepsilon)^{m}\otimes \cz^{d_{f}})$ by 
$\varphi_{\eta}*(\chi_{j}s)$ as $\eta\to 0^{+}$\,, with
$\varphi_{\eta}(x)=\eta^{-m}\varphi_{1}(\eta^{-1}x)$\,,
$\varphi_{1}\in \mathcal{C}^{\infty}_{0}(\rz^{m})$\,,
$\int_{\rz^{m}}\varphi_{1}=1$\,. From
$d[\varphi_{\eta}*(\chi_{j}s)]=\varphi_{\eta}*d(\chi_{j}s)$
we deduce 
$$
\lim_{\eta\to 0^{+}}\|\chi_{j}s-\varphi_{\eta}*(\chi_{j}s)\|_{L^{2}}+\|d[\chi_{j}s-\varphi_{\eta}*(\chi_{j}s)]\|_{L^{2}}=0\,,
$$
while $\varphi_{\eta}*(\chi_{j}s)\in
\mathcal{C}^{\infty}_{0}((-\varepsilon,\varepsilon)^{m};\Lambda
\rz^{m}\otimes \cz^{d_{f}})$\,. On a fixed compact set $K\subset M$\,, only a
finite number  of $j$'s in $\sum_{j}\chi_{j}s$ have to be considered
and this proves the density of $\mathcal{C}^{\infty}_{0}(M;\Lambda
T^{*}M\otimes \fkF)$ in $\mathcal{E}_{loc~comp}(d^{\nabla^{\fkF}}, \Lambda
T^{*}M\otimes \fkF)$\,.\\
\noindent\textbf{d)} Consider the case $\overline{M}=M\sqcup M'$ and 
$\fkF=\overline{M}\times\cz$\,.
When $s\in \mathcal{E}_{comp}(d,\Lambda T^{*}M\otimes \fkF)$
and $s'\in \mathcal{E}_{loc}(d,\Lambda T^{*}M\otimes \fkF)$
there exists $\chi\in \mathcal{C}^{\infty}_{0}(\overline{M};\rz)$ such
that  $s\wedge s'=\chi s\wedge \chi s'$\,. We thus assume $s, s'\in
\mathcal{E}_{comp}(d;\Lambda T^{*}M\otimes \fkF)$\,.\\
But the sesquilinear map
$$
(s,s')\in \mathcal{E}_{comp}(d,\Lambda T^{*}M\otimes \fkF) \to
\underbrace{d\overline{s}\wedge s'+(-1)^{\mathrm{deg}~s}\overline{s}\wedge
ds'}_{=d(\overline{s}\wedge s')} \in
L^{1}_{comp}(M;\Lambda T^{*}M\otimes \fkF)
$$
is continuous. 
By \noindent\textbf{c)}, for any $s,s'\in
\mathcal{E}_{comp}(d,\Lambda T^{*}M\otimes \fkF)$ there exists
two sequences $(\omega_{n})_{n\in\nz}$ and $(\theta_{n})_{n\in \nz}$
in $\mathcal{C}^{\infty}_{0}(\overline{M};\Lambda T^{*}M\otimes \fkF)$
which converge respectively to $s$ and $s'$ in
$\mathcal{E}_{comp}(d,\Lambda T^{*}M\otimes \fkF)$\,. 
For any such sequence $\overline{\omega}_{n}\wedge \theta_{n}\in
\mathcal{C}^{\infty}_{0}(\overline{M};\Lambda T^{*}M\otimes \fkF)$ and
Stokes formula says
$$
\int_{M}d[\overline{\omega}_{n}\wedge
\theta_{n}]=\int_{M'}j^{*}_{M'}(\overline{\omega}_{n}\wedge
\theta_{n})=\int_{M'}(j_{M'}^{*}\overline{\omega}_{n})\wedge (j_{M'}^{*}\theta_{n})\,.
$$
The left-hand  side converges to 
$\int_{M}d\overline{s}\wedge s'+(-1)^{\mathrm{deg}~s}\overline{s}\wedge
ds'$ which is a continuous sesquilinear form on
$\mathcal{E}_{comp}(d,\Lambda T^{*}M\otimes \fkF)$ and this ends the
proof of the extended Stokes formula.\\
\noindent\textbf{e)} One implication is trivial by restriction to
$\overline{M}_{\mp}$\,.\\
So  assume $s_{\mp}\in
\mathcal{E}_{loc}(d^{\nabla^{\fkF}},(\Lambda T^{*}M\otimes \fkF)\big|_{\overline{M}_{\mp}})$ and
$j^{*}_{M'}s_{-}=
j^{*}_{M'}s_{+}$ in $\mathcal{D}'(M';(\Lambda
T^{*}M\otimes\fkF)\big|_{M'})$\,. With a locally finite partition of
unity $\sum_{j}\chi_{j}\equiv 1$ in $M$\,,  we want to prove $\chi_{j}s\in
\mathcal{E}_{comp}(d^{\nabla^{\fkF}}, \Lambda T^{*}U_{j}\otimes \fkF\big|_{U_{j}})$ for
all $j$\,. By following the scheme of \textbf{b)} it suffices to
consider $U_{j}=(-\varepsilon,\varepsilon)^{m}$ and
$\fkF\big|_{U_{j}}=U_{j}\times \cz$ endowed with the trivial
connection $\nabla$ and $d^{\nabla}=d$\,.
The Stokes formula of \textbf{d)} is applied with
$\chi_{j}s\big|_{\overline{M}_{\mp}}$ and $s'\in
\mathcal{C}^{\infty}_{0}((-\varepsilon,\varepsilon)^{m};\Lambda \cz^{m})$:
\begin{equation*}
\left.  \begin{array}[c]{l}
\int_{(-\varepsilon,0]\times (-\varepsilon,\varepsilon)^{m-1}} d\overline{\chi_{j}s}\wedge s'
+(-1)^{\mathrm{deg}s}\overline{\chi_{j}s}\wedge
(ds')
+\\
\int_{[0,\varepsilon)\times (-\varepsilon,\varepsilon)^{m-1}} d\overline{\chi_{j}s}\wedge s'
+(-1)^{\mathrm{deg}s}\overline{\chi_{j}s}\wedge
(ds')
\end{array}
\right\}
        =\int_{(-\varepsilon,\varepsilon)^{m-1}}j_{M'}^{*}\overline{[\chi_{j}s_{+}-\chi_{j}s_{-}]}\wedge
j_{M'}^{*}s'
\end{equation*}
where the right-hand is $0$ for all $s'\in
\mathcal{C}^{\infty}_{0}((-\varepsilon,\varepsilon)^{m};\Lambda
\cz^{m})$\,. This implies the
existence of a constant $C_{s}$ such that
$$
\forall s'\in \mathcal{C}^{\infty}_{0}((-\varepsilon,\varepsilon)^{m};
\Lambda \cz^{m})\,,
\left|\int_{(-\varepsilon,\varepsilon)^{m}}\overline{\chi_{j}s}\wedge
  ds'\right|
\leq C_{s}\|s'\|_{L^{2}}\,.
$$
But the linear form
$$
s'\in
\mathcal{C}^{\infty}_{0}((-\varepsilon,\varepsilon)^{m};\Lambda \cz^{m})
\mapsto -(-1)^{\deg~{s}}\int_{(-\varepsilon,\varepsilon)^{m}}\overline{\chi_{j}s}\wedge
ds'
$$
is the definition of $d(\chi_{j}s)$ as a current, i.e. an element of
$\mathcal{D}'((-\varepsilon,\varepsilon)^{m}, \Lambda  \cz^{m})$\,, which therefore
belongs to $L^{2}_{comp}((-\varepsilon,\varepsilon)^{m};\Lambda T^{*}(-\varepsilon,\varepsilon)^{m}\otimes\cz)$\,.\\ 
Doing this for all components in $\cz^{d_{f}}$\,, 
 proves $\chi_{j}s \in \mathcal{E}_{comp}(d^{\nabla^{\fkF}}, (\Lambda
 T^{*}M\otimes \fkF)\big|_{U_{j}})$ and therefore $s\in
 \mathcal{E}_{loc}(d^{\nabla^{\fkF}},\Lambda T^{*}M\otimes \fkF)$\,.
\end{proof}
\begin{remark}
\label{re:j*}
  Locally with coordinates such that $M'=\left\{(x^{1},x')\in M\,,
    x^{1}=0\right\}$ the partial trace $j_{M'}^{*}s$ can be replaced
  by $\mathbf{i}_{\frac{\partial}{\partial x^{1}}}dx^{1}\wedge
  s\big|_{M'}=s_{I'}dx^{I'}\big|_{M'}$ when $s=s_{I}dx^{I}$\,.
\end{remark}
\begin{remark} Although the differential $d$ defines an elliptic
  complex (see \cite{ChPi}), the operator $d$ is not elliptic. In
  particular the partial trace defined in $\mathcal{D}'(M';\Lambda T^{*}M)$
  does not have neither the $\mathcal{W}^{1/2,2}_{loc}$ nor the $L^{2}_{loc}$ regularity associated
  with order 1 elliptic differential operators, as shows the example
  $r^{-\alpha}dr=d\frac{r^{1-\alpha}}{1-\alpha}$\,, $\alpha\in ]1/2,1[$\,, $M'=\rz\times\left\{0\right\}$\,, in $\rz^{2}$ with polar
  coordinates $(r,\theta)$\,.
\end{remark}
For Proposition~\ref{pr:partialtrace} we used the  (local) duality
between $\mathcal{C}^{\infty}_{0}(M;\Lambda^{p} T^{*}M)$ and
$\mathcal{D}'(M;\Lambda^{\dim M-p} T^{*}M)$ and made integration by
parts via Stokes theorem. We may instead use the duality between
$\mathcal{C}^{\infty}_{0}(M;\Lambda T^{*}M)$ and
$\mathcal{D}'(M;\Lambda TM)$ given by the natural duality between $\Lambda TM$
and $\Lambda T^{*}M$\,. It is not necessary to assume $M$ oriented
here but let us keep this assumption which is fulfilled in our applications. We put a volume element $dv_{M}$  and by
assuming that the hypersurface $M'$ admits a global defining function
$x^{1}\in \mathcal{C}^{\infty}(M;\rz)$\,,
$M'=(x^{1})^{-1}(\left\{0\right\})$\,, $dx^{1}\big|_{M'}\neq 0$\,,
this defines a volume element $dv_{M'}$ on $M'$ by writing
$dv_{M}(x)=|dx^{1}|dv_{M'}(x')$ with local coordinates
$x=(x^{1},x')$\,. 
Let $\fkF'$ be the anti-dual $\mathcal{C}^{\infty}$ vector bundle and
let $\nabla^{\fkF'}$ be the anti-dual connection of $\nabla^{\fkF}$ 
characterized by
\begin{equation*}
\frac{\partial}{\partial
  x^{i}}(t.s)=(\nabla^{\fkF'}_{\frac{\partial}{\partial
    x^{i}}}t).s+t.(\nabla^{\fkF}_{\frac{\partial}{\partial x^{i}}}s)
\quad,\quad t\in \mathcal{C}^{\infty}(M;\fkF')\,,\, s\in
     \mathcal{C}^{\infty}(M,\fkF)\,,
\end{equation*}
where $t.s(x)$ stands for the natural $\fkF'_{x}-\fkF_{x}$ duality\,.
It satisfies
\begin{eqnarray}
  \label{eq:nonsmoothconnd}
  &&\nabla^{\fkF'}: \mathcal{C}^{\infty}(M;\fkF')\to L^{\infty}_{loc}(M;
  T^{*}M\otimes \fkF')\\
  \label{eq:nonsmoothconnd2}
\text{resp.} &&\nabla^{\fkF'}:
  \mathcal{C}^{\infty}(\overline{M};\fkF')\to
                L^{\infty}_{loc}(\overline{M}; T^{*}M\otimes \fkF')\,.
\end{eqnarray}
The  interior covariant derivative
 $\tilde{d}^{\nabla^{\fkF'}}$ acting on sections of
$\Lambda TM\otimes \fkF'$ is written in
local coordinates
$$
\tilde{d}^{\nabla^{\fkF'}}=-\mathbf{i}_{dx^{i}}\frac{\partial}{\partial
              x^{i}}\otimes \mathrm{Id}_{\fkF'}
 -\mathbf{i}_{dx^{i}}\otimes \nabla^{\fkF'}_{\frac{\partial}{\partial x^{i}}}\,.
$$
and when $dv_{M}(x)=\lambda(x)|dx|$ 
$$
\tilde{d}^{\nabla^{\fkF'},v_{M}}=-\mathbf{i}_{dx^{i}}\frac{\partial}{\partial
              x^{i}}\otimes\mathrm{Id}_{\fkF'}
 -\mathbf{i}_{dx^{i}}\otimes \nabla^{\fkF'}_{\frac{\partial}{\partial
     x^{i}}}-\mathbf{i}_{dx^{i}}\frac{\partial \lambda}{\partial x^{i}}\lambda^{-1}\otimes\mathrm{Id}_{\fkF'}
$$
The operator $\tilde{d}^{\nabla^{\fkF'},v_{M}}:\mathcal{D}'(M;\Lambda
TM\otimes \fkF')\to \mathcal{D}'(M;\Lambda
TM\otimes \fkF')$ is characterized by
$$
\forall s\in \mathcal{C}^{\infty}_{0}(M;\Lambda T^{*}M\otimes
\fkF)\,,\quad \int_{M}(\tilde{d}^{\nabla^{\fkF'},v_{M}}t).s~dv_{M}=
\int_{M}t.(d^{\nabla^{\fkF}}s)~dv_{M}\,.
$$
The following result will be used in Section~\ref{sec:adjdiff}.
\begin{proposition}
\label{pr:partialtrace2}
\noindent\textbf{a)} If $\fkF$ is a smooth vector bundle on $M$  (resp. the manifold
with boundary $\overline{M}=M\sqcup M'$) and $\nabla^{\fkF}_{1}$\,,
$\nabla^{\fkF}_{2}$ are two connections on $\fkF$ which fulfill
\eqref{eq:nonsmoothconn} (resp. \eqref{eq:nonsmoothconn} with antidual versions
$\fkF'$\,, $\nabla^{\fkF'_{1}}$ and $\nabla^{\fkF'_{2}}$ and if
$dv_{M,1}$ and $dv_{M,2}$ are two Lipschitz continuous volume elements\,, then 
$$
\mathcal{E}_{\bullet}(d^{\nabla_{1}^{\fkF'}, v_{M,1}}, \Lambda TM\otimes
     \fkF')=\mathcal{E}_{\bullet}(d^{\nabla_{2}^{\fkF'}, v_{M,2}}, \Lambda
     TM\otimes \fkF')\quad \bullet =loc~\text{or}~comp.
$$
\noindent\textbf{b)}When $M'\subset M$ a smooth hypersurface of $M$ (resp. a manifold
with boundary $\overline{M}=M\sqcup M'$) with a global defining
function $x^{1}$\,, 
 the partial trace map $t\mapsto \mathbf{i}_{dx^{1}}t\big|_{M'}$ is well
defined an continuous from $\mathcal{E}_{loc}(d^{\nabla^{\fkF'},v_{M}},\Lambda
TM\otimes \fkF')$ to
$\mathcal{D}'(M';\Lambda
TM'\otimes \fkF'\big|_{M'})$\,.\\
\noindent\textbf{c)} The space $\mathcal{C}^{\infty}_{0}(M;\Lambda TM\otimes \fkF')$
(resp. $\mathcal{C}^{\infty}_{0}(\overline{M};\Lambda TM\otimes
\fkF')$) is dense in the two spaces $\mathcal{E}_{comp}(d^{\nabla^{\fkF',v_{M}}},\Lambda TM\otimes \fkF')$
and $\mathcal{E}_{loc}(d^{\nabla^{\fkF'},v_{M}},\Lambda TM\otimes \fkF')$\,.\\
\noindent\textbf{d)} In the case $\overline{M}=M\sqcup M'$ and
$x^{1}<0$ in $M$\,, the
integration by parts
$$
\int_{M}t.(d^{\nabla^{\fkF}}s)~dv_{M}-
\int_{M}(\tilde{d}^{\nabla^{\fkF'},v_{M}}t).s~dv_{M}
=\int_{M'}(\mathbf{i}_{dx^{1}}t).s~dv_{M'}
$$
holds for all
$t\in \mathcal{E}_{loc}(d,\Lambda TM\otimes \fkF')$ and all 
$\in
\mathcal{E}_{comp}(d,\Lambda T^{*}M\otimes \fkF)$\,, where the
right-hand side is the unique  sesquilinear continuous extension from
$\mathcal{C}^{\infty}_{0}(\overline{M};\Lambda TM\otimes \fkF')\times
\mathcal{C}^{\infty}_{0}(\overline{M};\Lambda T^{*}M\otimes \fkF)$\,. \\
\noindent\textbf{e)} When $M=M_{-}\sqcup M'\sqcup M_{+}$ and $\overline{M}_{\pm}$
are smooth domains of $M$\,, $t\in \mathcal{E}_{loc}(d^{\nabla^{\fkF',v_{M}}},\Lambda
TM\otimes\fkF')$ iff $t_{\mp}=t\big|_{M_{\mp}}$ belongs to
$\mathcal{E}_{loc}(d^{\nabla^{\fkF'},v_{M}}, (\Lambda TM\otimes \fkF')\big|_{\overline{M}_{\mp}})$ and
$$
\mathbf{i}_{dx^{1}}t_{-}\big|_{M'}=\mathbf{i}_{dx^{1}}t_{+}\quad \text{in}~\mathcal{D}'(M';\Lambda
TM'\otimes \fkF'\big|_{M'})\,.
$$
\end{proposition}
\begin{proof}
  The proofs of a)b)c) are essentially the same as for a)b)c) in
  Proposition~\ref{pr:partialtrace}. The statement e) is a consequence
  of d). The proof of d) simply relies on the following computation
  for $t\in \mathcal{C}^{\infty}_{0}(\overline{M}; \Lambda TM\otimes
  \fkF')$ and $s\in \mathcal{C}^{\infty}_{0}(\overline{M};\Lambda
  T^{*}M\otimes \fkF)$ supported in a chart open domain, $x^{i'}\in
  (-\varepsilon,\varepsilon)$\,, $x^{1}\in (-\varepsilon,0]$~:
  \begin{eqnarray*}
    \int_{M}(t.(d^{\nabla^{\fkF}}s)-(\tilde{d}^{\nabla^{\fkF'}}t).s)~\lambda(x)|dx|&=&\int_{(-\varepsilon,0]\times
    (-\varepsilon,\varepsilon)^{d-1}} \frac{\partial}{\partial
    x^{i}}[((\mathbf{i}_{dx^{i}}t).s)\times \lambda]~|dx|
\\
&=&
\int_{(-\varepsilon,\varepsilon)^{d-1}}(\mathbf{i}_{dx^{1}}t).s~
\lambda(0,x')|dx'|=\int_{M'}(\mathbf{i}_{dx^{1}}t).s~dv_{M'}\,.
  \end{eqnarray*}
The sign of the final right-hand side is changed if we assume $x^{1}>0$ in $M$\,.
\end{proof}
\subsection{The differential structure of $\hat{E}_{g}$ and $\hat{F}_{g}$} 
\label{sec:diffhatE}
The vector bundle $\hat{E}_{g}$ was introduced in
Definition~\ref{de:hatEF} as a piecewise $\mathcal{C}^{\infty}$
and continuous vector bundle above $X=T^{*}Q$ where
$\hat{E}_{g}\big|_{\overline{X}_{\mp}}=\Lambda
T^{*}X\big|_{\overline{X}_{\mp}}\stackrel{g_{\mp}}{=}\pi_{X}^{*}(\Lambda T^{*}Q\otimes
\Lambda TQ\big|_{\overline{Q}_{\mp}})$ with the matching condition
$$
e^{i}_{-}\big|_{\partial X_{-}}=e^{i}_{+}\big|_{\partial
  X_{+}}\quad,\quad
\hat{e}_{-,j}\big|_{\partial X_{-}}=\hat{e}_{+,j}\big|_{\partial
  X_{+}}\quad \partial X_{-}=\partial X_{+}=X'\,.
$$
We used the frame $(e_{\mp},\hat{e}_{\mp})$ of Definition~\ref{de:frameef}.\\
It can be given another interpretation. In the manifold
$\overline{X}_{-}=X_{-}\cup X'$\,, the coordinates
$(\tilde{q},\tilde{p})=(\tilde{q}^{1},\tilde{x}')$ of Definition~\ref{de:tildeqp} identify
$X_{(-\varepsilon,0]}$ as the tubular neighborhood
$(-\varepsilon,0]\times X'$\,. Meanwhile $S_{1}:X'\to X'$ of Definition~\ref{de:S1Snu}, 
$$
S_{1}(0,q',p_{1},p')=(0,q',-p_{1},p')\quad (0,q',p_{1},p')=(0,\tilde{q}',\tilde{p}_{1},\tilde{p}')\,,
$$
is a diffeomorphism of $X'$\,. By following Milnor in
\cite{Mil}-Theorem~1.4 there is a  $\mathcal{C}^{\infty}$-manifold,
unique modulo diffeomorphism,
\begin{eqnarray}
  \label{eq:Mg1}
  &&M_{g}=X_{-}\cup X'\cup X_{+}\quad\text{such~that}\\
\label{eq:Mg2}
&& \hspace{-1cm}\left((0^{-},\tilde{q}',\tilde{p}_{1},\tilde{p}')=(0^{+},\tilde{Q}',\tilde{P}_{1},\tilde{P}')\right)\Leftrightarrow ((\tilde{Q}',\tilde{P}_{1},\tilde{P}')=S_{1}(\tilde{q}',\tilde{p}_{1},\tilde{p}')=(\tilde{q}',-\tilde{p}_{1},\tilde{p}'))\,.
\end{eqnarray}
The subscripted notation $M_{g}$ keeps track of the fact that the
construction of the coordinates $(\tilde{q},\tilde{p})$ actually
depend on the chosen metric $g_{-}=g=g^{TQ}$ on $\overline{Q}_{-}$\,,
symmetrized as $\hat{g}^{TQ}$\,.\\
We recall the Definition~\ref{de:doublef} of the double copy  $\pi_{\fkf}:\fkf\to Q$\,, when
$Q=Q_{-}\sqcup Q'\sqcup Q_{+}$\,, given with
$(0^{+},\underline{q}',\nu v)=(0^{-},\underline{q}',v)$ and
$\pi_{M_{g}}^{*}(\fkf)$ is a flat
$\mathcal{C}^{\infty}$ vector bundle on the $\mathcal{C}^{\infty}$
manifold $M_{g}$\,. The exterior covariant derivative denoted by
$d^{\nabla^{\fkf}}_{M_{g}}$ satisfies:
\begin{itemize}
\item $d^{\nabla^{\fkf}}_{M_{g}}:\mathcal{F}(M_{g};\Lambda
  T^{*}M_{g}\otimes \pi_{M_{g}}^{*}(\fkf))\to \mathcal{F}(M_{g};\Lambda
  T^{*}M_{g}\otimes \pi_{M_{g}}^{*}(\fkf))$
for $\mathcal{F}=\mathcal{C}^{\infty}_{0}$ and for $\mathcal{F}=\mathcal{D}'$\,;
\item $d_{M_{g}}^{\nabla^{\fkf}}\circ d_{M_{g}}^{\nabla^{\fkf}}=0$ in
  $\mathcal{C}^{\infty}_{0}(M_{g};\Lambda T^{*}M_{g}\otimes
  \pi_{X}^{*}(\fkf))$ and in $\mathcal{D}'(M_{g};\Lambda
  T^{*}M_{g}\otimes \pi_{X}^{*}(\fkf))$\,;
\item  in particular
  $d_{M_{g}}^{\nabla^{\fkf}}:\mathcal{E}_{\bullet}(d_{M_{g}}^{\nabla^{\fkf}},
  \Lambda T^{*}M_{g}\otimes \pi_{M_{g}}^{*}(\fkf))\to \mathcal{E}_{\bullet}(d_{M_{g}}^{\nabla^{\fkf}},
  \Lambda T^{*}M_{g}\otimes \pi_{M_{g}}^{*}(\fkf))$ respectively for $\bullet=loc$ or
  $comp$\,;
\item the trace map $s\mapsto j_{X'}s$ or $s\mapsto \mathbf{i}_{e_{1}} e^{1}\wedge
  s\big|_{X'}$ after Remark~\ref{re:j*}, where we recall
  $e_{1}=\frac{\partial}{\partial\tilde{q}^{1}}$ and $e^{1}=d\tilde{q}^{1}$\,, is well defined and continuous from 
 $\mathcal{E}_{loc}(d_{M_{g}}^{\nabla^{\fkf}},
  \Lambda T^{*}M_{g}\otimes \pi_{M_{g}}^{*}(\fkf))$ to
  $\mathcal{D}'(X';\Lambda T^{*}M_{g}\otimes
  \pi_{M_{g}}^{*}(\fkf)\big|_{X'})$\,;
\item  
$\mathcal{C}^{\infty}_{0}(M_{g};\Lambda T^{*}M_{g}\otimes
  \pi_{M_{g}}^{*}(\fkf))$ is dense in
  $\mathcal{E}_{loc~comp}(d^{\nabla^{\fkf}}_{M_{g}},\Lambda
  T^{*}M_{g}\otimes \pi_{M_{g}}^{*}(\fkf))$\,;
\item if $\Sigma_{M_{g}}:M_{g}\to M_{g}$ is the natural symmetry
  specified locally by
  $\Sigma_{M_{g}}(\tilde{q}^{1},\tilde{q}',\tilde{p})=(-\tilde{q}_{1},\tilde{q}',\tilde{p})$
 with its push-forward $\Sigma_{M_{g},*}$ acting on
 $\mathcal{D}'(M_{g};\Lambda T^{*}M_{g}\otimes \pi_{M_{g}}^{*}(\fkf))$\,,
 then
 $d^{\nabla^{\fkf}}\Sigma_{M_{g},*}=\Sigma_{M_{g},*}d^{\nabla^{\fkf}}$
 and $d^{\nabla^{\fkf}}$ preserves the parity with respect to $\Sigma_{M_{g},*}$\,.
\end{itemize}
For the smooth manifold $M_{g}$ it is convenient to introduce the
following notations which have already introduced 
counterparts on $\hat{E}_{g}$ and
$\hat{F}_{g}$\,.
\begin{definition}
\label{de:symMg}
Let $M_{g}$ be the manifold defined by \eqref{eq:Mg1}\eqref{eq:Mg2}\,.
The manifold $M_{g,(-\varepsilon,\varepsilon)}$ is the open domain
characterized by $|\tilde{q}^{1}|<\varepsilon$\,.\\
On $M_{g,(-\varepsilon,\varepsilon)}$ the symmetry $\Sigma_{M_{g}}$ is
simply given by
$\Sigma_{M_{g}}(\tilde{q}^{1},\tilde{x}')=(-\tilde{q}^{1},\tilde{x}')$\,.\\
The set $L^{2}_{loc,ev}(M_{g,(-\varepsilon,\varepsilon)}; \Lambda
T^{*}M\otimes \pi_{M_{g}}^{*}(\fkf))$ is defined by
$$
L^{2}_{loc,ev}(M_{g,(-\varepsilon,\varepsilon)}; \Lambda
T^{*}M\otimes \pi_{M_{g}}^{*}(\fkf))
=\left\{s\in L^{2}_{loc}(M_{g,(-\varepsilon,\varepsilon)}; \Lambda
T^{*}M\otimes \pi_{M_{g}}^{*}(\fkf))\,, \quad \Sigma_{M_{g},*}s=\nu s\right\}\,,
$$
and $L^{2}_{loc,odd}$ has the same definition with the condition
$\Sigma_{M_{g},*}s=-\nu s$\,.\\
For $\fkF= \Lambda
T^{*}M_{g}\otimes
\pi_{M_{g}}^{*}(\fkf)\big|_{M_{g,(-\varepsilon,\varepsilon)}}$, the set $\mathcal{C}_{0}(\fkF)$ equals 
$$
\mathcal{C}^{\infty}_{0}(M_{g,(-\varepsilon,0]};\fkF)\cap \mathcal{C}^{\infty}_{0}(M_{g,[0,\varepsilon)};\fkF)\cap
\mathcal{C}^{0}(M_{g,(-\varepsilon,\varepsilon)};\fkF)\quad
\text{with}\quad\fkF= \Lambda
T^{*}M_{g}\otimes \pi_{M_{g}}^{*}(\fkf)\,,
$$
and
$\mathcal{C}_{0,ev~odd}(\fkF)=\mathcal{C}_{0}(\fkF)\cap L^{2}_{loc,
  ev~odd}(M_{g,(-\varepsilon,\varepsilon)};\fkF)$\,.
\end{definition}
\begin{lemma}
\label{le:Mg} Let $M_{g,(-\varepsilon,\varepsilon)}$  be the
neighboorhood of $X'$ given by $|\tilde{q}^{1}|<\varepsilon$ and let
the map $\tilde{S}_{1}:M_{g,(-\varepsilon,\varepsilon)}\to
X_{(-\varepsilon,\varepsilon)}$ given by
$$
\tilde{S}_{1}(\tilde{q}^{1},\tilde{q}',\tilde{p}_{1},\tilde{p}')=
\left\{
  \begin{array}[c]{ll}
    (\tilde{q}^{1},\tilde{q}',\tilde{p}_{1},\tilde{p}')&\text{if}~\tilde{q}^{1}\leq
                                                         0\,,\\
(\tilde{q}^{1},\tilde{q}',-\tilde{p}_{1},\tilde{p}')&\text{if}~\tilde{q}^{1}>0\,.
  \end{array}
\right.
$$
When $\Sigma_{M_{g}}(\tilde{q}^{1},\tilde{x}')=(-\tilde{q}^{1},x')$ on
$M_{g,(-\varepsilon,\varepsilon)}$ we get $\tilde{S}_{1}\circ
\Sigma_{M_{g}}=\Sigma
\circ\tilde{S}_{1}:M_{g,(-\varepsilon,\varepsilon)}\to
X_{(-\varepsilon,\varepsilon)}$\,.\\
Moreover $\tilde{S}_{1,*}:  \Lambda
T^{*}M_{g}\otimes
\pi_{M_{g}}^{*}(\fkf)\big|_{M_{g,(-\varepsilon,\varepsilon)}}\to
\hat{F}_{g}\big|_{X_{(-\varepsilon,\varepsilon)}}$ is a piecewise
$\mathcal{C}^{\infty}$ and continuous vector bundle isomorphism  such
that $\tilde{S}_{1,*}$ sends
$L^{2}_{loc}(M_{g,(-\varepsilon,\varepsilon)};\Lambda T^{*}M\otimes
\pi_{M_{g}}^{*}(\fkf))$ into
$L^{2}_{loc}(X_{(-\varepsilon,\varepsilon)};\hat{F}_{g})$\,,
$\mathcal{C}_{0}(\Lambda T^{*}M_{g}\otimes \pi_{M_{g}}^{*}(\fkf)\big|_{M_{g,(-\varepsilon,\varepsilon)}})$ into
$\mathcal{C}_{0}(\hat{F}_{g}\big|_{X_{(-\varepsilon,\varepsilon)}})$
of Definition~\ref{de:Cg}
  and transforms the
parity with respect to $\Sigma_{M_{g},*}\otimes \nu$ into the parity with
respect to $\Sigma_{\nu}$\,.
\end{lemma}
\begin{proof}
  We can forget the vector bundle $\pi_{M_{g}}^{*}(\fkf)$\,. We focus on
  $\Lambda T^{*}M_{g}\big|_{M_{g,(-\varepsilon,\varepsilon)}}$ and
  $\hat{E}_{g}\big|_{X_{(-\varepsilon,\varepsilon)}}$\,.
The equality $\tilde{S}_{1}\circ \Sigma_{M_{g}}=\Sigma\circ
\tilde{S}_{1}$ is obvious. Meanwhile the push-forward
$\tilde{S}_{1,*}$\,, $\Sigma_{M_{g},*}$ and $\Sigma_{*}$ are
$\mathcal{C}^{\infty}$ vector bundle isomorphisms when $\tilde{q}^{1}$
is restricted to $(-\varepsilon,0]$ or $[0,+\infty)$\,. It thus
suffices to check that
$\tilde{S}_{1,*}:T^{*}M_{g}\big|_{M_{g,(-\varepsilon,\varepsilon)}}\to
T^{*}X_{(-\varepsilon,\varepsilon)}$ is continuous along
$X'$\,. Restricted to $M_{g,(-\varepsilon,0]}=X_{(-\varepsilon,0]}$\,, $\tilde{S}_{1,*}$
is the identity and a smooth local frame of
$T^{*}M_{g,(-\varepsilon,0]}$ is given by
$(e^{i}_{-},\hat{e}_{j,-})$\,. On the $\mathcal{C}^{\infty}$ manifold
$M_{g,(-\varepsilon,\varepsilon)}$\,, it is the restriction of a
smooth frame $(\tilde{e}^{i},\tilde{e}_{j})$ such that
$\tilde{e}^{i},\tilde{e}_{j})\big|_{\tilde{q}^{1}=0^{+}}=(e^{i}_{-},\hat{e}_{j,-})\big|_{\tilde{q}^{1}=0^{-}}$\,.
By using 
$$
\tilde{S}_{1,*}\big|_{\tilde{q}^{1}\leq 0}=\mathrm{Id}\quad,\quad
\Sigma_{*}\tilde{S}_{1,*}\Sigma_{M_{g},*}=\tilde{S}_{1,*}
$$
we deduce from \eqref{eq:tfh}\eqref{eq:tfv10}\eqref{eq:tfv0}
\begin{eqnarray*}
  &&
     \tilde{S}_{1,*}(\tilde{e}^{i},\tilde{e}_{j})\big|_{\tilde{q}^{1}=0^{-}}=(e^{i}_{-},\hat{e}_{j,-})\big|_{\tilde{q}^{1}=0^{-}}\\
&& \tilde{S}_{1,*}(\tilde{e}^{i},\tilde{e}_{j})\big|_{\tilde{q}^{1}=0^{+}}=(e^{i}_{+},\hat{e}_{j,+})\big|_{\tilde{q}^{1}=0^{+}}\,,
\end{eqnarray*}
and the continuous local frame of $(\tilde{e}^{i},\tilde{e}_{j})$ of
$T^{*}M_{g}\big|_{M_{g,(-\varepsilon,\varepsilon)}}$ is sent to the
continuous local frame
$(e^{i},\hat{e}_{j})=1_{\rz_{\mp}}(\tilde{q}^{1})(e^{i}_{\mp},\hat{e}_{j,\mp})$
of $\hat{E}_{g}\big|_{X_{(-\varepsilon,\varepsilon)}}$\,.\\
The parity properties then come from
$\Sigma_{*}\circ\tilde{S}_{1,*}=\tilde{S}_{1,*}\circ\Sigma_{M_{g},*}$\,.
\end{proof}
The trace properties of Proposition~\ref{pr:partialtrace} have been used
with the closed manifold $M_{g}\supset X'$ to review the properties of
the exterior covariant derivative $d^{\nabla^{\fkf}}_{M_{g}}$\,. 
A separate use in the
manifolds with boundaries $\overline{X}_{\mp}=X_{\mp}\sqcup X'$
provides a definition of the differential acting on sections of
$\hat{E}_{g}$ and $\hat{F}_{g}$\,.
\begin{definition}
\label{de:dhatE}
Let $E=\Lambda T^{*}X$ and $F=E\otimes \pi_{X}^{*}(\fkf)$ and let $d^{\nabla^{\fkf}}$
be the covariant exterior derivative for the flat vector bundle
$(\pi_{X}^{*}(\fkf), \pi_{X}^{*}(\nabla^{\fkf}))$\,.
With the notation
$(e,\hat{e})=1_{\rz_{\mp}}(q^{1})(e_{\mp},\hat{e}_{\mp})$ of
Definition~\ref{de:frameef}\,,
 a section $\omega=\omega_{I}^{J}e^{I}\hat{e}_{J}\in
  L^{2}_{loc}(X;\hat{F}_{g})$ belongs to
  $\mathcal{E}_{loc}(\hat{d}_{g},\hat{F}_{g})$ (resp. $\mathcal{E}_{comp}(\hat{d}_{g},\hat{F}_{g})$) if its restrictions
  $\omega_{\mp}=\omega\big|_{X_{\mp}}=\omega_{I}^{J}e_{\mp}^{I}\hat{e}_{\mp,J}$
  belong to
    $\mathcal{E}_{loc}(d^{\nabla^{\fkf}}, F\big|_{\overline{X}_{\mp}})$
    (resp. $\mathcal{E}_{comp}(d^{\nabla^{\fkf}}, F\big|_{\overline{X}_{\mp}})$)
    with
    \begin{equation}
      \label{eq:bcdhat}
\mathbf{i}_{e_{+,1}}e_{+}^{1}\wedge\omega_{+}\big|_{\partial X_{+}}=
\mathbf{i}_{e_{-,1}}e_{-,1}\wedge \omega_{-}\big|_{\partial
  X_{-}}\quad\text{in}~\mathcal{D}'(X';\hat{F }_{g}\big|_{X'})\,,
\end{equation}
or
$$
\omega_{I'}^{J}(0^{+},.)=\nu \omega_{I'}^{J}(0^{-},.)\quad\text{in}~~
\mathcal{D}'(X',\pi_{X'}^{*}(\fkf\big|_{Q'}))
$$
for all $I',J\subset\left\{1,\ldots,d\right\}$\,, $1\not\in I'$\,. The
differential $\hat{d}_{g}$ with domain $\mathcal{E}_{loc}(\hat{d}_{g};\hat{F}_{g})$ is
then defined by $\hat{d}_{}\omega\big|_{X_{\mp}}=d^{\nabla^{\fkf}}\omega_{\mp}$\,.
\end{definition}
The properties of $\hat{d}_{g}$ and $\mathcal{E}_{loc}(\hat{d}_{g},\hat{F}_{g})$ are
deduced from the one of $d^{\nabla^{\fkf}}_{M_{g}}$ and
$\mathcal{E}_{loc}(d^{\nabla^{\fkf}}_{M_{g}},\Lambda T^{*}M_{g}\otimes
\pi_{M_{g}}^{*}(\fkf))$\,.
\begin{proposition}
\label{pr:dod}
The differential $\hat{d}_{g}$ defined on $\mathcal{E}_{loc}(\hat{d}_{g},\hat{F}_{g})$ satisfies
$\hat{d}_{g}\mathcal{E}_{loc}(\hat{d}_{g},\hat{F}_{g})\subset
\mathcal{E}_{loc}(\hat{d}_{g},\hat{F}_{g})$ and
$\hat{d}_{g}\circ \hat{d}_{g}=0$\,.\\
The map $\Sigma_{\nu}:L^{2}_{loc}(X;\hat{F}_{g})\to
L^{2}_{loc}(X;\hat{F}_{g})$ preserves 
$\mathcal{E}_{loc}(\hat{d}_{g},\hat{F}_{g})$ and $\hat{d}_{g}\Sigma_{\nu}=\Sigma_{\nu}\hat{d}_{g}$ so
that $\hat{d}_{g}$ preserves the parity with respect to $\Sigma_{\nu}$\,.\\
 The space $\mathcal{C}_{0,g}(\hat{F}_{g})$  of Definition~\ref{de:Cg} is densely and
continuously embedded  in
$\mathcal{E}_{loc}(\hat{d}_{g},\hat{F}_{g})$ and
$\mathcal{E}_{comp}(\hat{d}_{g},\hat{F}_{g})$\,.\\
Moreover the exists a dense set
$\hat{\mathcal{D}}_{g,\nabla^{\fkf}}$ of $\mathcal{C}_{0,g}(\hat{F}_{g})$\,,
such that $\hat{d}_{g}\mathcal{D}_{g,\nabla^{\fkf}}\subset \mathcal{C}_{0,g}(\hat{F}_{g})$\,.
\end{proposition}
\begin{proof}
  With Lemma~\ref{le:Mg} and with
  $j_{X'}^{*}s$ written according to Remark~\ref{re:j*} as $\mathbf{i}_{e_{1}}e^{1}\wedge
  s\big|_{X'}=\mathbf{i}_{\frac{\partial}{\partial \tilde
      q^{1}}}d\tilde{q}^{1}\wedge s\big|_{X'}$\,,   the map
  $\tilde{S}_{1,*}$ sends $L^{2}_{loc}(M_{g}; \Lambda T^{*}M\otimes
  \pi_{M_{g}}^{*}(\fkf))$ to $L^{2}_{loc}(X;\hat{F}_{g})$ with
  $d^{\nabla^{\fkf}}\tilde{S}_{1,*}\big|_{M_{g}\setminus
    X'}=\tilde{S}_{1,*}d^{\nabla^{\fkf}}_{M_{g}}\big|_{M_{g}\setminus X'}$ while the trace condition
  $j_{X'}s\big|_{\partial X_{-}}=j_{X'}^{*}s\big|_{\partial X_{+}}$ in
  $\mathcal{D}'(X';\Lambda T^{*}X'\otimes
  \pi_{X}^{*}(\fkf)\big|_{X'})$  is transformed into
  \eqref{eq:bcdhat}.\\
This proves
  $\tilde{S}_{1,*}\mathcal{E}_{loc}(d_{M_{g}}^{\nabla^{\fkf}},\Lambda
  T^{*}M_{g}\otimes
  \pi_{M_{g}}^{*}(\fkf))=\mathcal{E}_{loc}(d,\hat{F}_{g})$ and
  $\hat{d}_{g}=\tilde{S}_{1,*}d_{M_{g}}^{\nabla^{\fkf}}\tilde{S}_{1}^{*}$\,.\\
The property $\hat{d}_{g}\circ \hat{d}_{g}=0$ is thus the consequence of
$d^{\nabla^{\fkf}}_{M_{g}}\circ d^{\nabla^{\fkf}}_{M_{g}}$ on the
$\mathcal{C}^{\infty}$ manifold $M_{g}$ endowed with the flat exterior
covariant derivative $d_{M_{g}}^{\nabla^{\fkf}}$\,.\\
The dense an continuous embeddings
$$ \mathcal{C}^{\infty}_{0}(M_{g};\Lambda T^{*}M_{g}\otimes
\pi_{M_{g}}^{*}(\fkf))\subset \mathcal{C}_{0}(\Lambda T^{*}M_{g}\otimes
\pi_{M_{g}}^{*}(\fkf))\subset \mathcal{E}_{loc~comp}(d^{\nabla^{\fkf}}_{M_{g}},\Lambda
T^{*}M_{g}\otimes \pi_{M_{g}}^{*}(\fkf))$$
with
$$ d_{M_{g}}^{\nabla^{\fkf}}\mathcal{C}^{\infty}_{0}(M_{g};\Lambda T^{*}M_{g}\otimes
\pi_{M_{g}}^{*}(\fkf))\subset \mathcal{C}^{\infty}_{0}(M_{g};\Lambda T^{*}M_{g}\otimes
\pi_{M_{g}}^{*}(\fkf))\subset  \mathcal{C}_{0}(\Lambda T^{*}M_{g}\otimes
\pi_{M_{g}}^{*}(\fkf))\,,
$$
combined with
$$
\tilde{S}_{1,*} \mathcal{C}_{0}(\Lambda T^{*}M_{g}\otimes
\pi_{M_{g}}^{*}(\fkf))=\mathcal{C}_{0,g}(\hat{F}_{g})\,,
$$
provides the density results by taking
$\hat{\mathcal{D}}_{g,\nabla^{\fkf}}=\tilde{S}_{1,*}\mathcal{C}^{\infty}_{0}(M_{g};\Lambda
T^{*}M_{g}\otimes \pi_{M_{g}}^{*}(\fkf))$\,.
\end{proof}

\subsection{Boundary conditions for $d_{b'\fkh}$ and properties of the
  associated closed operator}
\label{sec:BCd}
Remember 
$$
s_{ev}(x)=1_{X_{-}}(x)s(x)+1_{X_{+}}(x)\Sigma_{\nu}s(x)\quad\text{when}~s\in L^{2}_{loc}(X_{-};F)\,,
$$
and
$$
\hat{\fkh}=\frac{|p|_{q}^{2}}{2}=\frac{p_{1}^{2}+m^{i'j'}(-|q^{1}|,q')p_{i'}p_{j'}}{2}=\frac{\tilde{p}_{1}^{2}+m^{i'j'}(0,\tilde{q}')\tilde{p}_{i'}\tilde{p}_{j'}}{2}\,.
$$
\begin{definition}
\label{de:domaind}
When $\hat{d}_{g}$ denotes the differential on
$\mathcal{E}_{loc}(\hat{d}_{g},\hat{F}_{g})$ introduced in
Definition~\ref{de:dhatE} and $b'\geq 0$\,,
 the differential $\hat{d}_{g,b'\fkh}$ equals
$$
\hat{d}_{g,b'\fkh}=e^{-b'\hat{\fkh}}\hat{d}_{g}e^{b'\hat{\fkh}}
=\hat{d}_{g}+b'd\hat{\fkh}\wedge\,.
$$
We keep the same notation for the operator 
in $L^{2}(X;F)$ defined by
\begin{eqnarray*}
&&  D(\hat{d}_{g,b'\fkh})=
\left\{s \in L^{2}(X;F)\cap \mathcal{E}_{loc}(\hat{d}_{g},\hat{F}_{g})\,,\, \hat{d}_{g,b'\fkh}s\in
  L^{2}(X;F)\right\}\\
&& \forall s\in D(\hat{d}_{g})\,,\quad
   \hat{d}_{g,b'\fkh}s=(d^{\nabla^{\fkf}}+b'd\hat{\fkh}\wedge)(s\big|_{X_{-}})+
  (d^{\nabla^{\fkf}}+b'd\hat{\fkh}\wedge) (s\big|_{X_{+}})\,.
\end{eqnarray*}
The operator $\overline{d}_{g,b'\fkh}$ on $\overline{X}_{-}$ is given by
\begin{eqnarray*}
  &&
D(\overline{d}_{g,b'\fkh})=\left\{s\in L^{2}(X_{-},F)\,,\quad s_{ev}\in D(\hat{d}_{g,b'\fkh})\right\}\,,\\
&& \forall s\in D(\overline{d}_{g,b'\fkh})\,,\quad
   \overline{d}_{g,b'\fkh}s=d^{\nabla^{\fkf}}s+b'd\fkh\wedge
   s\,.
\end{eqnarray*}
\end{definition}
\begin{proposition}
\label{pr:domaindH}
The operator $\hat{d}_{g,b'\fkh}$ given in $L^{2}(X;F)$ 
by Definition~\ref{de:domaind} 
is closed, satisfies $\hat{d}_{g,b' \fkh}\circ
\hat{d}_{g,b'\fkh}=0$\,, $\hat{d}_{g,b'\fkh}\circ
\Sigma_{\nu}=\Sigma_{\nu}\circ\hat{d}_{g,b'\fkh}$\,.
In particular, $\hat{d}_{g,b'\fkh}$ preserves the parity:
\begin{eqnarray*}
&&  D(\hat{d}_{g,b'\fkh})=D(\hat{d}_{g,b'\fkh})\cap
  L^{2}_{ev}(X;F)\oplus D(\hat{d}_{g,b'\fkh})\cap
  L^{2}_{odd}(X;F)\,,\\
\text{with}&&
\hat{d}_{g,b'\fkh}: D(\hat{d}_{g,b'\fkh})\cap L^{2}_{ev~odd}(X;F)\to L^{2}_{ev~odd}(X;F)\,.
\end{eqnarray*}
The subset $\mathcal{C}_{0,g}(\hat{F}_{g})$ of Definition~\ref{de:Cg}
 is dense in $D(\hat{d}_{g, b'\fkh})$\,. Additionally the dense subset
$\hat{\mathcal{D}}_{g,\nabla^{\fkf}}$ of $\mathcal{C}_{0,g}(\hat{F}_{g})$
given in  Proposition~\ref{pr:dod} satisfies
$\hat{d}_{g,b'\fkh}\mathcal{D}_{g,\nabla^{\fkf}}\subset
\mathcal{C}_{0,g}(\hat{F}_{g})\subset D(\hat{d}_{g,b'\fkh})$\,.\\
The domain of the operator
$\overline{d}_{g,b'\fkh}$ given in $L^{2}(X_{-};F)$ by
Definition~\ref{de:domaind} equals
$$
D(\overline{d}_{g,b'\fkh})=\left\{s\in L^{2}(X_{-}; F)\,,\quad 
d^{\nabla\fkf}_{b'\fkh}s\in L^{2}(X_{-};F)\,,\quad
  \frac{1-\hat{S}_{\nu}}{2}\mathbf{i}_{e_{1}}e^{1}\wedge 
s\big|_{X'}=0\right\}\,.
$$
The operator $\overline{d}_{g,b'\fkh}$ with this domain 
is closed and satisfies $\overline{d}_{g,b'\fkh}\circ
\overline{d}_{g,b'\fkh}=0$\,.\\
 The spaces
$\mathcal{C}^{\infty}_{0}(\overline{X}_{-};F)\cap D(\overline{d}_{g,b'\fkh})$\,,
$\mathcal{C}_{g}=\left\{s\in L^{2}(X_{-};F)\,,
  s_{ev}\in \mathcal{C}_{0,g}(\hat{F}_{g})\right\}$ 
and $\mathcal{D}_{g,\nabla^{\fkf}}=\left\{s\in
  L^{2}(X_{-};F)\,, s_{ev}\in
  \hat{\mathcal{D}}_{g,\nabla^{\fkf}}\right\}$ are dense in
$D(\overline{d}_{g,b'\fkh})$ with
$\overline{d}_{g,b'\fkh}\mathcal{D}_{g,\nabla^{\fkf}}\subset \mathcal{C}_{g}$\,.
\end{proposition}
\begin{proof}
 Let us first consider the operator
$\hat{d}_{g,b'\fkh}$\,. For a sequence  $(u_{n})_{n\in\nz}$ of
$D(\hat{d}_{g,b'\fkh})$ such that
$\lim_{n\to\infty}u_{n}=u$ and
$\lim_{n\to \infty}\hat{d}_{g,b'\fkh}u_{n}=v$ in $L^{2}(X;F)$\,, 
the convergence of $u_{n}\big|_{X_{\mp}}$ to 
$u\big|_{X_{\mp}}$
holds in  $\mathcal{E}_{loc}(d^{\nabla^{\fkf}},F\big|_{\overline{X}_{\mp}})$\,. The
  continuity of the trace map $j_{X'}^{*}$ of
  Proposition~\ref{pr:partialtrace}-b) implies $u\in
  \mathcal{E}_{loc}(\hat{d}_{g},\hat{F}_{g})$ and the identification
  $\hat{d}_{g,b'\fkh}u=v$ in $L^{2}_{loc}(X;F)$\,. Therefore
$(\hat{d}_{g,b'\fkh},D(\hat{d}_{g,b'\fkh})$ is closed. The property
$\hat{d}_{g,\fkh}\circ{\hat{d}_{g,b'\fkh}}$ and $\hat{d}_{g,b'\fkh}\circ
\Sigma_{\nu}=\Sigma_{\nu}\circ \hat{d}_{g,b'\fkh}$ were already
 proved in
Proposition~\ref{pr:dod}.\\
For the density, 
consider the cut-off function
$\chi_{n}=\chi\left(\frac{\hat{\fkh}}{n+1}\right)$ with
  $\chi\in \mathcal{C}^{\infty}_{0}(\rz;[0,1])$ equal to $1$ in a
  neighborhood of $0$\,. For $s\in D(\hat{d}_{g,b'\fkh})$\,, notice
  $\chi_{n}s\in
  \mathcal{E}_{comp}(\hat{d}_{g},\hat{F}_{g})\subset
  D(\hat{d}_{g,b'\fkh})$ and write with
    $s_{\mp}=s\big|_{X_{\mp}}$\,, 
$$
d^{\nabla^{\fkf}}_{b'\fkh}
[\chi_{n}s_{\mp}]=\chi_{n}(d^{\nabla^{\fkf}}_{b'\fkh}s_{\mp})+[d\chi_{n}\wedge
s_{\mp}]=\chi_{n}(d^{\nabla^{\fkf}}_{b'\fkh}s_{\mp})
+\frac{1}{n+1}[\chi'\left(\frac{\hat{\fkh}}{n+1}\right)(d\hat{\fkh} \wedge
s_{\mp})]\,. 
$$
With the coordinates $(\tilde{q},\tilde{p})$ and the metric
$\hat{g}^{E}$\,, we know $|d\tilde{q}^{i}|_{\hat{g}^{E}}=\mathcal{O}(\langle \tilde{p}\rangle^{-1/2})$\,,
 $|d\tilde{p}_{j}|_{\hat{g}^{E}}=\mathcal{O}(\langle \tilde{p}\rangle^{1/2})$ and
\begin{eqnarray*}
  &&
d\hat{\fkh}=\tilde{p}_{1}d\tilde{p}_{1}+m^{i'j'}(0,\tilde{q}')\tilde{p}_{i'}d\tilde{p}_{j'}
+\frac{1}{2}\tilde{p}_{i'}\tilde{p}_{j'}d_{\tilde{q}'}[m^{i'j'}(0,\tilde{q}')]\quad,\\
&&
|d\hat{\fkh}|_{\hat{g}^{E}}=\mathcal{O}(\langle p\rangle^{3/2})=\mathcal{O}(\hat{\fkh}^{3/4})\,.
\end{eqnarray*}
We deduce 
$$
\|d\chi_{n}\wedge s_{\mp}\|_{L^{2}}\leq \frac{C}{(n+1)^{1/4}}\|s_{\mp}\|_{L^{2}}\,
$$
and
$$
\lim_{n\to\infty}\|s-\chi_{n}s\|_{L^{2}}+\|\hat{d}_{g,b'\fkh}(s-\chi_{n}s)\|_{L^{2}}=0\,.
$$
We have just proved that $\mathcal{E}_{comp}(\hat{d}_{g},\hat{F}_{g})$ is
dense  (and continuously embedded) in
$D(\hat{d}_{g,b'\fkh})$\,. In Proposition~\ref{pr:dod} we proved that
 $\mathcal{C}_{0,g}(\hat{F}_{g})$ and $\hat{\mathcal{D}}_{g,\nabla^{\fkf}}$
is dense in $\mathcal{E}_{comp}(\hat{d}_{g},\hat{F}_{g})$ and that
$\hat{d}_{g}\hat{D}_{g,\nabla^{\fkf}}\subset
\mathcal{C}_{0,g}(\hat{F}_{g})$\,. 
By going back to the $\mathcal{C}^{\infty}$ structure $\Lambda
T^{*}M_{g}\otimes\pi_{M_{g}}^{*}(\fkf)$ of $\hat{F}_{g}$\,, it suffices to notice that
$\hat{h}=\frac{\tilde{p}_{1}^{2}+m^{i'j'}(0,\tilde{q}')\tilde{p}_{i'}\tilde{p}_{j'}}{2}$
is actually a $\mathcal{C}^{\infty}$ function on $M_{g}$ preserved by
$\tilde{S}_{1}$ to see that the multiplication by $e^{\pm b' \fkh}$ preserves
$\mathcal{C}_{0,g}(\hat{F}_{g})=S_{1,*}\mathcal{C}_{0}(\Lambda
T^{*}M_{g}\otimes \pi_{M_{g}}^{*}(\fkf))$ and
$\hat{D}_{g,\nabla^{\fkf}}=S_{1,*}\mathcal{C}^{\infty}_{0}(M_{g};\Lambda
T^{*}M_{g}\otimes \pi_{M_{g}}^{*}(\fkf))$\,.\\
Finally the properties of $(\overline{d}_{g,b'\fkh},
D(\overline{d}_{g,b'\fkh}))$ are obvious translations of the
properties of $(\hat{d}_{g,b'\fkh},D(\hat{d}_{g,b'\fkh}))$ because 
$\hat{d}_{g,b'\fkh}$ preserves the parity with respect to $\Sigma_{\nu}$
and $s\mapsto \frac{s_{ev}}{\sqrt{2}}$ is a unitary map from
$L^{2}(X_{-};F)$ to $L^{2}_{ev}(X;\hat{F}_{g})$\,. The condition
$\frac{1-\hat{S}_{\nu}}{2}\mathbf{i}_{e_{1}}e^{1}\wedge
s\big|_{X'}=0$ is simply the partial trace version in
$\mathcal{D}'(X';F\big|_{X'})$ of Proposition~\ref{pr:partrace}-b).
\end{proof}

\subsection{Comments}
\label{sec:commentsdomd}
The results of this section requires some specifications and
explanations.\\

\subsubsection{Dependence of the boundary condition with respect
  to $g^{TQ}$}
 All the analysis was carried out without using
any reference to a riemannian metric as it should be when one studies
the differential. The weight $\fkh$ could be replaced by another
function on $X$ with the suitable assumptions.  The map
$\hat{S}_{\nu}$ defined on $X'$ actually simply depends only  on the
identification of a normal vector to $Q'$\,. However the boundary
condtions for $\overline{d}_{g,b'\fkh}$ depend on the chosen metric
$g^{TQ}$ because the tangential trace is written
$\mathbf{i}_{e_{1}}e^{1}\wedge s\big|_{X'}$\,. Accordingly
the continuity condition in the vector bundles $\hat{E}_{g}$\,,
$\hat{F}_{g}$ or the differential $\mathcal{C}^{\infty}$ manifold
$M_{g}$ introduced in Subsection~\ref{sec:diffhatE} really depend on
the chosen metric $g^{TQ}$\,. Below is the verification that they must
not be confused with more usual boundary conditions for the
differential which would correspond to the metric $g_{0}^{TQ}$ where 
$\hat{E}_{g_{0}}=E=\Lambda T^{*}X$\,.
Let  $g^{TQ}=(dq^{1})^{2}\oplus^{\perp}m(q^{1})$ and
$g_{0}^{TQ}=(dq^{1})^{2}\oplus^{\perp}m(0)$ and use the frames
$(e^{i},\hat{e}_{j})$ for the first case and the frame
$(f^{i},\hat{f}_{j})$ for the second case. We forget the vector bundle $\fkf$ or simply
take $\nu=1$\,.
Writing a section $\omega\in \mathcal{C}^{\infty}_{0}(\overline{X}_{-};T^{*}X)$ in those two local frames gives
\begin{eqnarray*}
\omega&=&\omega_{i}e^{i}+\omega^{j}\hat{e}_{j}\\
&=&
\left[\omega_{1}-\Gamma_{j'1}^{k'}p_{k'}\omega^{j'}\right]f^{1}
+\left[\omega_{i'}-\Gamma_{i'1}^{k'}p_{k'}\omega^{1}-\Gamma_{i'j'}^{1}p_{1}\omega^{j'}\right]f^{i'}+\omega^{j}\hat{f}_{j}\,.
\end{eqnarray*}
The boundary condition for $g^{TQ}$ is
$$
\omega_{i'}(0,x')=\omega_{i'}(0,S_{1}(x'))\quad,\quad
\omega^{j}(0,x')=(-1)^{\delta_{1j}}\omega^{j}(0,S_{1}(x'))
$$
while for $g_{0}^{TQ}$ it is
\begin{eqnarray*}
  && 
\left[\omega_{i'}-\Gamma_{i'1}^{k'}p_{k'}\omega^{1}-\Gamma_{i'j'}^{1}p_{1}\omega^{j'}\right]=\left[\omega_{i'}-\Gamma_{i'1}^{k'}p_{k'}\omega^{1}-\Gamma_{i'j'}^{1}p_{1}\omega^{j'}\right](0,S_{1}(x'))\\
&&
\omega^{j}(0,x')=(-1)^{\delta_{j1}}\omega^{j}(0,S_{1}(x'))\,.
\end{eqnarray*}
They coincide iff 
$$
\forall i'\in \left\{2,\ldots,d\right\}\,,\quad
\Gamma_{i'1}^{k'}p_{k'}[\omega^{1}(0,x')+\omega^{1}(0,S_{1}(x'))]+2\Gamma_{i'j'}^{1}p_{1}
\omega^{j'}(0,x')=0\,,
$$
which is not true in general.

\subsubsection{The flat vector bundle $(\fkf,\nabla^{\fkf})$ and the
  spaces $\hat{\mathcal{D}}_{g,\nabla^{\fkf}}$\,,
  $\mathcal{D}_{g,\nabla^{\fkf}}$}
\label{sec:rqflat}
Once the flat vector bundle $\pi_{X}^{*}(\fkf)$ is made from the  two
pieces $\fkf\big|_{\overline{Q}_{-}}$ and
$\fkf\big|_{\overline{Q}_{+}}$ according to
Definition~\ref{de:doublef}, the vector bundle $\Lambda
T^{*}M_{g}\otimes \pi_{M_{g}}^{*}(\fkf)$ of Subsection~\ref{sec:diffhatE} can be considered as a
$\mathcal{C}^{\infty}$ flat vector bundle on $M_{g}$\,. However the
$\mathcal{C}^{\infty}$ structure depends on $\nabla^{\fkf}$ and this
why we keep track of the $\nabla^{\fkf}$-dependence in the subset
$\hat{\mathcal{D}}_{g,\nabla^{\fkf}}$ and
$\mathcal{D}_{g,\nabla^{\fkf}}$\,.  Let us consider the simple example
where a potential is added to the energy $\fkh$\,, with $\nu=1$\,. According to the
end of Subsection~\ref{sec:coornonsmooth} it can be formulated by
starting from $\fkf=\overline{Q}_{-}\otimes \cz$ with the trivial
connection $\nabla^{\fkf}=\nabla$ and the metric $g^{\fkf}(z)=
e^{-2b'V(q)}|z|^{2}$ or equivalently with the metric $g^{\fkf}(z)=|z|^{2}$ and the
connection $\nabla^{\fkf}=\nabla +b'dV(q)=e^{-b'V(q)}\nabla
e^{b'V(q)}$\,. Take the second choice. The corresponding differential
$d^{\nabla^{\fkf}}_{g,b'\fkh}$ on $X_{-}$ will be
$$
d^{\nabla^{\fkf}}_{g,b'\fkh}=d+b'd(\fkh+V)\wedge
$$
which is what we expect.
Smooth sections 
$\tilde{s}\in \mathcal{C}^{\infty}(M_{g};\Lambda T^{*}M_{g}\otimes
\pi_{M_{g}}^{*}(\fkf))$ are actually sections of $\Lambda T^{*}M$ such
that $e^{b'\hat{V}(q)}\tilde{s}\in \mathcal{C}^{\infty}(M_{g};\Lambda
T^{*}M_{g}\otimes \cz)$ with
$\hat{V}(\tilde{q})=V(-|\tilde{q}^{1}|,\tilde{q}')$\,. The set
$\hat{\mathcal{D}}_{g,\nabla^{\fkf}}$ was defined as the image of
$\mathcal{C}^{\infty}_{0}(M_{g};\Lambda T^{*}M_{g}\otimes
\pi_{M_{g}}^{*}(\fkf))$\,. The set of its even element is the image of the
even elements of $\mathcal{C}^{\infty}_{0}(M_{g};\Lambda
T^{*}M_{g}\otimes \pi_{M_{g}}^{*}(\fkf))$ with respect to
$\Sigma_{M_{g},*}$ that is via the symmetry
$(\tilde{q}^{1},\tilde{x}')\mapsto (-\tilde{q}^{1},\tilde{x}')$\,. If
we simply take a function $\tilde{s}\in
\mathcal{C}^{\infty}_{0}(M_{g};\cz)$ such that
$\tilde{s}(-\tilde{q}^{1},\tilde{x}')=\tilde{s}(\tilde{q}^{1},\tilde{x}')$
it must satisfy 
$\frac{\partial}{\partial
  \tilde{q}^{1}}(e^{b'\hat{V}}\tilde{s})(0,\tilde{q}')=0$\,. Written
in $\tilde{q}^{1}=0^{-}$ it means $\frac{\partial \tilde{s}}{\partial
  \tilde{q}^{1}}(0^{-},\tilde{x}')+b'\frac{\partial
  V}{\partial\tilde{q}^{1}}s(0^{-},\tilde{x}')=0$\,.
The corresponding section $s\in
\mathcal{D}_{g,\nabla^{\fkf}}$ belongs to
$\mathcal{C}^{\infty}_{0}(\overline{X}_{-};\cz)$ and satisfies
$$
s(0,\tilde{q}',\tilde{p}_{1},\tilde{p}')=s(0,\tilde{q}',-\tilde{p}_{1},\tilde{p}')\quad,\quad
\frac{\partial \tilde{s}}{\partial
  \tilde{q}^{1}}(0,\tilde{x}')+b'\frac{\partial
  V}{\partial\tilde{q}^{1}}s(0,\tilde{x}')=0\quad,\quad \tilde{x}'=(\tilde{q}',\tilde{p}_{1},\tilde{p}')\,.
$$
It clearly depends on the flat connection $\nabla^{\fkf}$\,.\\
While considering adjoints, like in Proposition~\ref{pr:partialtrace2}
we must use the anti-adjoint flat connection $\nabla^{\fkf'}$\,,
equal to $\nabla^{\fkf'}=\nabla-b'dV\wedge$ here, and 
the sign in front of
$\frac{\partial V}{\partial \tilde{q}^{1}}$ in the above condition is
changed. So the corresponding subset $\mathcal{D}_{g,\nabla^{\fkf'}}$
differs from $\mathcal{D}_{g,\nabla^{\fkf}}$\,. A similar subtlety
must be watched when we go further in the analysis and play with the
extrinsic curvature related with $\frac{\partial
  m}{\partial\underline{q}^{1}}(0,\underline{q}')$\,.\\
From this point of view working with the vector bundle $\hat{E}_{g}$
and $\hat{F}_{g}$ where only the continuity is considered along $X'$\,,
not only simplifies the correspondance $(0^{+},x')=(0^{-},y')$ into
$y'=x'$\,, but also 
prevents from mistakes or confusions.

\subsubsection{The interface condition of $\hat{d}_{g,b'\fkh}$ as
a jump condition}
Again, we forget the vector bundle $\pi_{X}^{*}(\fkf)$ or take simply
$\fkf=\cz$ with $\nu=1$\,.\\
The continuity condition for $s=s_{I}^{J}e^{I}\hat{e}_{J}\in
\mathcal{E}_{loc}(\hat{d}_{g},\hat{E}_{g})$ written:
$$
s_{I'}^{J}(0^{+},x')=s_{I'}^{J}(0^{-},x')
$$
can be written 
$$
j_{X'}^{*}(A_{+}s)=j_{X'}(A_{-}s)
$$ 
where $A_{\mp}$ is the vector bundle isomorphism sending the frame
$(e^{i}_{\mp},\hat{e}_{\mp,j})$ associated with $g_{\mp}^{TQ}$
to the frame $(f^{i},\hat{f}_{j})$ associated with $g_{0}^{TQ}$\,,
with $j_{X'}^{*}\tilde{s}$ written as $\mathbf{i}_{f^{1}}f^{1}\wedge
s\big|_{X'}=\mathbf{i}_{\frac{\partial}{\partial
    q^{1}}}(dq^{1}\wedge s)\big|_{X'}$\,. With basic linear algebra, it can also be written
$$
\mathbf{i}_{\frac{\partial}{\partial q^{1}}}(dq^{1}\wedge s)\big|_{\partial X_{+}}=
A\mathbf{i}_{\frac{\partial}{\partial q^{1}}} (dq^{1}\wedge
s)\big|_{\partial X_{-}}\,.
$$
By using the $\mathcal{C}^{\infty}$ natural structure of $E=\Lambda
T^{*}X$ the continuity conditions for $s\in
\mathcal{E}_{loc}(\hat{d}_{g},\hat{E}_{g})$ thus appears as a
jump condition in $E$\,.  In $\mathcal{D}'(X;\Lambda T^{*}X)$ an
element $s\in L^{2}_{loc}(X;\Lambda T^{*}X)$ belongs to
$\mathcal{E}_{loc}(\hat{d}_{g},\hat{E}_{g})$ iff
$$
ds=(ds_{-})+ (ds_{+})+\delta_{0}(q^{1})dq^{1}\wedge [(A-\Id)j_{X'}{*}s_{-}]\quad\text{in}~\mathcal{D}'(X;\Lambda T^{*}X)
$$
with $s_{\mp}=s\big|_{X_{\mp}}$\,, $ds_{\mp}\in
L^{2}_{loc}(\overline{X}_{\mp};\Lambda T^{*}X)$\,.
Checking $\hat{d}_{g}\mathcal{E}_{loc}(\hat{d}_{g},\hat{E}_{g})\subset\mathcal{E}_{loc}(\hat{d}_{g},\hat{E}_{g})
$ with $\hat{d}_{g}\circ\hat{d}_{g}=0$ means
 that the right-hand side of 
$$
d(\hat{d}_{g}s)=0+0
+\delta(q^{1})dq^{1}\wedge[j_{X'}(ds_{+})-j_{X'}(ds_{-})]=\delta(q^{1})dq^{1}\wedge
d'[j_{X'}(A-\Id)s_{-}]
$$
can be written $\delta(q^{1})\wedge
dq^{1}[(A-\Id)(d'j_{X'}s_{})]$\,. A sufficient condition is: the
vector bundle morphism $A$ satisfies $d'A=Ad'$ in
$\mathcal{D}'(X';\Lambda TX')$\,.\\
This can be checked by computations in terms of local coordinates by
using the expressions
\eqref{eq:fv}\eqref{eq:ff}\eqref{eq:ef}\eqref{eq:hef1}\eqref{eq:hef}
of the frames $(e^{i}_{\mp},\hat{e}_{\mp,j})$ and $(f^{i},\hat{f}_{j})$\,. But
this is not so simple and may involve the differentiation of the
Christoffel symbols $\Gamma_{ij}^{k}(q)$ which is
irrelevant. From this point of view introducing the proper
$\mathcal{C}^{\infty}$ structure of the  manifold $M_{g}$ in
Subsection~\ref{sec:diffhatE} is much more effective.\\
However note that it was introduced with the non
symplectic coordinates $(\tilde{q},\tilde{p})$\,. It is easier to
work with the differential structure of $X$\,, the coordinates
$(q,p)$\,,  and the piecewise $\mathcal{C}^{\infty}$ and continuous vector bundle
$\hat{E}_{g}$ when the symplectic structure is required.

\section{Symplectic and Bismut codifferential}
\label{sec:adjdiff}
Now that we have a good definition and properties of the closed
operator $\overline{d}_{g,\fkh}$ in $L^{2}(X_{-}; F)$\,, the codifferential can be defined as the adjoint operator,
for various duality products. Bismut's codifferential involves a
non-degenerate but non symmetric form, which makes a mixture of
symplectic and riemannian Hodge duality. It is therefore simpler to
start first from the duality between $F=\Lambda T^{*}X\otimes
\pi_{X}^{*}(\fkf)$
 and $F'=\Lambda TX\otimes \pi_{X}^{*}(\fkf)$\,, $\fkf\simeq \fkf'$ via
 the metric $g^{\fkf}$ or $\hat{g}^{\fkf}$\,,
which does not require any additional information than the symplectic volume
measure $dv_{X}$ on $X$ and then to transfer the information by various
isomorphisms from $F'$ to $F$\,. 
\subsection{Adjoints for the $(F',F)$ dual pair}
\label{sec:adjdualpair}
The volume measure on $X$ is the symplectic volume
$$
dv_{X}=\left|\frac{1}{d!}\sigma^{d}\right|=|dqdp|
$$
and coincides with the Lebesgue measure in any symplectic coordinate
system $(q^{i},p_{j})$\,.  However here it is more convenient to work
with the non symplectic coordinates $(\tilde{q},\tilde{p})$ of
Definition~\ref{de:tildeqp} where $dv_{X}$ does not have such a simple
form according to \eqref{eq:dvxtilde}.  However the trace results and
the integration by parts of Proposition~\ref{pr:partialtrace2} only
requires the volume form $dv_{X}$ and the defining function of the
boundary $x^{1}$ which is here $\tilde{q}^{1}=q^{1}$ in
$X_{(-\varepsilon,\varepsilon)}$ with $d\tilde{q}^{1}=dq^{1}=e^{1}$\,. Therefore the volume form
occuring in the boundary integrals is $dv_{X'}=|dq'dp|$\,.\\
Although the dual of $F=\Lambda T^{*}X \otimes
\pi_{X}^{*}(\fkf)$ can be identified with $F'=\Lambda T^{*}X\otimes
\pi_{X}^{*}(\fkf)$ via the metric $g^{\fkf}$ or $\hat{g}^{\fkf}$\,,
which provides the identification at the level of $L^{2}$-spaces, the
flat $\mathcal{C}^{\infty}$ connections $\pi_{X}^{*}(\nabla^{\fkf})$
and $\pi_{X}^{*}(\nabla^{\fkf'})$ differ and give rise to different
$\mathcal{C}^{\infty}$-structures of vector bundles, especially when
the metric $\hat{g}^{f}$ is used. An example is the case
$\nabla^{\fkf}=\nabla +b'd\hat{V}$ and $\nabla^{\fkf'}=\nabla -b'd\hat{V}$
where $\hat{V}(q^{1},q')=V(-|q^{1}|,q')$ already discussed in
Subsection~\ref{sec:rqflat}.\\
Because the operators $\overline{d}_{g,b'\fkh}$ and
$\hat{d}_{g,b'\fkh}$ are densely defined operators respectively in
$L^{2}(X_{-};F)$ and $L^{2}(X;F)$\,, we can consider their adjoints
via the duality product
$$
\langle t\,,\,s\rangle=\int_{X}\langle t\,,\,s\rangle_{F'_{x},F_{x}}~dv_{X}(x)\,,
$$
where the duality between $F'=\Lambda TX\otimes \pi_{X}^{*}(\fkf)$ and
$F=\Lambda T^{*}X\otimes \pi_{X}^{*}(\fkf)$ is made via the metric
$g^{\fkf}$ (or $\hat{g}^{\fkf}$) and 
for which $L^{2}(X_{-},F)'=L^{2}(X_{-},F')$
(resp. $L^{2}(X;F)'=L^{2}(X;F')$)\,.\\
When we work on $X_{-}$ or  $X_{-}\cup X_{+}$ with the symplectic
coordinates $(q,p)$\,, $dv_{X}=|dqdp|$ and without considering
boundary or interface conditions, the formal adjoint of
$d_{b'\fkh}=d^{\nabla^{\fkf}}_{b'\fkh}= d^{\nabla^{\fkf}}+b'd\fkh\wedge$ is
nothing but
\begin{equation}
  \label{eq:dtilde}
\tilde{d}_{b'\fkh}=
\tilde{d}^{\nabla^{\fkf'}}_{b'\fkh}=\tilde{d}^{\nabla^{\fkf'}}-b'\mathbf{i}_{d\fkh}
=-\mathbf{i}_{dq^{i}}\nabla^{\fkf'}_{\frac{\partial}{\partial
    q^{i}}}-\mathbf{i}_{dp_{j}}\frac{\partial}{\partial p_{j}}-b'\mathbf{i}_{d\fkh}\,.
\end{equation}
The notation $\tilde{\hat{d}}_{b'\fkh}$ on $X_{-}\cup X_{+}$ refers to the
case when $\fkh$ is replaced by the piecewise $\mathcal{C}^{\infty}$
and continuous function $\hat{\fkh}$\,. According to this, the
trace results of Proposition~\ref{pr:partialtrace2} lead to the
following definition.

\begin{definition}
\label{de:Elthatd}
 The space
$\mathcal{E}_{loc}(\tilde{\hat{d}}_{g,b'\fkh},
\hat{F}'_{g})=\mathcal{E}_{loc}(\tilde{\hat{d}}_{g,0}=\tilde{\hat{d}}_{g},
\hat{F}'_{g})$ is the set of sections in $t=L^{2}_{loc}(X;F')$\,, $t_{\mp}=t\big|_{X_{\mp}}$ such that 
\begin{eqnarray*}
  && \tilde{d}^{\nabla^{\fkf'}}t_{\mp} \in L^{2}_{loc}(X_{\mp};F')\,,
\\
&&
\mathbf{i}_{e^{1}}t_{-}\big|_{\partial
   X_{-}}=\mathbf{i}_{e^{1}}t_{+}\big|_{\partial X_{+}}\quad
   \text{in}~\mathcal{D}'(X';\hat{F}'_{g}\big|_{X'})\\
\text{or}&&
e_{1}\wedge \mathbf{i}_{e^{1}}t_{-}\big|_{\partial X_{-}}=
e_{1}\wedge \mathbf{i}_{e^{1}}t_{+}\big|_{\partial X_{+}}\quad
            \text{in}~\mathcal{D}'(X';\hat{F}'_{g}\big|_{X'})\,,\\
\text{and}&&
\tilde{\hat{d}}_{g,b'\fkh}t=(\tilde{d}^{\nabla^{\fkf'}}_{b'\hat{\fkh}}t_{-})+
(\tilde{d}^{\nabla^{\fkf'}}_{b'\hat{\fkh}}t_{+})\,.
\end{eqnarray*}
\end{definition}
\begin{proposition}
\label{pr:domtilded} The adjoint $\tilde{\hat{d}}_{g,b'\fkh}$ of
$(\hat{d}_{g,b'\fkh}, D(\hat{d}_{g,\fkh}))$ of
Definition~\ref{de:domaind} is closed and densely defined as
\begin{eqnarray*}
  &&D(\tilde{\hat{d}}_{g,b'\fkh})=\left\{t\in L^{2}(X;F')\cap
     \mathcal{E}_{loc}(\tilde{\hat{d}}_{g},\hat{F}'_{g})\,,\quad
     \tilde{\hat{d}}_{g,b'\fkh}t\in L^{2}(X;F')\right\}\\
&&\forall t\in D(\tilde{\hat{d}}_{g,b'\fkh})\,,\quad
\tilde{\hat{d}}_{g,b'\fkh}t= (\tilde{d}^{\nabla^{\fkf'}}_{b'\hat{\fkh}}t_{-})+
(\tilde{d}^{\nabla^{\fkf'}}_{b'\hat{\fkh}}t_{+})\quad,\quad t_{\mp}=t\big|_{X_{\mp}}\,.
\end{eqnarray*}
It satisfies $\tilde{\hat{d}}_{g,b'\fkh}\circ
\tilde{\hat{d}}_{g,b'\fkh}=0$ and $\tilde{\hat{d}}_{g,b'\fkh}\circ
\Sigma_{\nu}=\Sigma_{\nu}\tilde{\hat{d}}_{g,b'\fkh}$\,. In particular, $\tilde{\hat{d}}_{g,b'\fkh}$ preserves the parity:
\begin{eqnarray*}
&&  D(\tilde{\hat{d}}_{g,b'\fkh})=D(\tilde{\hat{d}}_{g,b'\fkh})\cap
  L^{2}_{ev}(X;F')\oplus D(\tilde{\hat{d}}_{g,b'\fkh})\cap
  L^{2}_{odd}(X;F')\,,\\
\text{with}&&
\tilde{\hat{d}}_{g,b'\fkh}: D(\tilde{\hat{d}}_{g,b'\fkh})\cap L^{2}_{ev~odd}(X;F')\to L^{2}_{ev~odd}(X;F')\,.
\end{eqnarray*}
The subset $\mathcal{C}_{0,g}(\hat{F}'_{g})$ of Definition~\ref{de:Cg}
 is dense in $D(\tilde{\hat{d}}_{g,b' \fkh})$\,. Additionally there exists a  dense subset
$\hat{\mathcal{D}}'_{g,\nabla^{\fkf'}}$ of $\mathcal{C}_{0,g}(\hat{F}'_{g})$
such that 
$\tilde{\hat{d}}_{g,b'\fkh}\hat{\mathcal{D}}'_{g,\nabla^{\fkf'}}\subset
\mathcal{C}_{0,g}(\hat{F}'_{g})\subset
D(\tilde{\hat{d}}_{g,b'\fkh})$\,.\\
The adjoint $\tilde{d}_{g,b'\fkh}$ of $(\overline{d}_{g,\fkh},
D(\overline{d}_{g,\fkh}))$ is densely defined and closed with
$$
D(\tilde{d}_{g,\fkh})=\left\{t\in L^{2}(X_{-},F')\,,\quad
  d^{\nabla^{\fkf'}}_{b'\fkh}t\in L^{2}(X_{-},F')\,,\quad
  \frac{1-\hat{S}_{\nu}}{2}e_{1}\wedge \mathbf{i}_{e^{1}}t\big|_{X'}=0\right\}\,.
$$
This adjoint operator $(\tilde{d}_{g,b'\fkh},D(\tilde{d}_{g,b'\fkh}))$
satisfies $\tilde{d}_{g,b'\fkh}\circ \tilde{d}_{g,b'\fkh}=0$\,.\\
The spaces $\mathcal{C}^{\infty}_{0}(\overline{X}_{-};F')\cap
D(\tilde{d}_{g,b'\fkh})$\,, $\mathcal{C}'_{g}=\left\{t\in
  L^{2}(X_{-},F')\,, t_{ev}\in \mathcal{C}_{0,g}(\hat{F}'_{g})\right\}$
and $\mathcal{D}'_{g,\nabla^{\fkf'}}=\left\{t\in L^{2}(X_{-};F)\,,
  t_{ev}\in \hat{\mathcal{D}}'_{g,\nabla^{\fkf'}}\right\}$ are dense in
$D(\tilde{d}_{g,b'\fkh})$ with
$\tilde{d}_{g,b'\fkh}\mathcal{D}'_{g,\nabla^{\fkf'}}\subset \mathcal{C}'_{g}$\,.
\end{proposition}
\begin{proof}
  For $\tilde{\hat{d}}_{g,b'\fkh}$ we use the manifold
  $M_{g}=X_{-}\cup X'\cup X_{+}$ introduced in
  subsection~\ref{sec:diffhatE}\,. By Lemma~\ref{le:Mg} the map
  $\tilde{S}_{1,*}: \Lambda T^{*}M_{g}\otimes
  \pi_{M_{g}}^{*}(\fkf)\big|_{M_{g,(-\varepsilon,\varepsilon)}}\to
  \hat{F}_{g}\big|_{X_{(-\varepsilon,\varepsilon)}}$  provides the
 transpose map $\tilde{S}_{1}^{*}:
 \hat{F}'_{g}\big|_{X_{(-\varepsilon,\varepsilon)}}\to \Lambda
 TM_{g,(-\varepsilon,\varepsilon)}\otimes \pi_{M_{g}}^{*}(\fkf)$ and 
$TM_{g,(-\varepsilon,\varepsilon)}\otimes
\pi_{M_{g}}^{*}(\fkf')=TM_{g,(-\varepsilon,\varepsilon)}\otimes \pi_{M_{g}}^{*}(\fkf)$ is a
$\mathcal{C}^{\infty}$-vector bundle when the flat connection
$\pi_{M_{g}}^{*}(\nabla^{\fkf'})$ is used. The operator
 $\hat{d}_{g,b'\fkh}$ was identified
with $e^{-b'\hat{\fkh}}d^{\nabla^{\fkf}}_{M_{g}}e^{b'\hat{\fkh}}$
acting on the smooth vector bundle $\Lambda T^{*}M_{g}\otimes
\pi_{M_{g}}^{*}(\fkf)$ and where $\hat{\fkh}$ is a smooth function on
$M_{g}$\,. We are thus led to consider the adjoint of the differential
on the smooth manifold $M_{g}$ without boundary, $\tilde{\hat{d}}_{g,b'\hat{\fkh}}$ is identified
with $e^{b'\hat{\fkh }}\tilde{d}^{\nabla^{\fkf'}}_{M_{g}}e^{-b'\hat{\fkh}}$ and all
the properties follow. In particular
$\hat{\mathcal{D}}_{g,\nabla^{\fkf'}}'$ is nothing but the image of
$\mathcal{C}^{\infty}_{0}(M_{g};\Lambda TM_{g}\otimes
\pi_{M_{g}}^{*}(\fkf'))$\,, where we write $\fkf'$ to remind that the
$\mathcal{C}^{\infty}$-structure is the one given by $\nabla^{\fkf'}$\,, by $\tilde{S}_{1,*}:\Lambda M_{g}\otimes
\pi_{M_{g}}^{*}(\fkf)\to \hat{F}'_{g}$\,.\\
The study of $(\tilde{d}_{g,b'\fkh}, D(\tilde{d}_{g,b'\fkh}))$ then
relies on parity arguments with respect to $\Sigma_{\nu}$ on
$\hat{F}_{g}'$ (or with respect to $\Sigma_{M_{g},*}$ on $\Lambda
T^{*}M_{g}\otimes \pi_{M_{g}}^{*}(\fkf')$)\,, as we did for
$(\overline{d}_{g,b'\fkh}, D(\overline{d}_{g,b'\fkh}))$\,.
\end{proof}

\subsection{Symplectic codifferential}
\label{sec:sympladj}
As a non degenerate $2$-form the symplectic form $\sigma$ on $TX$\,,  defines a morphism
 $\sigma: TX\to T^{*}X$ by writing $\sigma(S,T)=S.(\sigma
T)$ for $S,T\in TX$\,. By tensorization, this defines a morphism still
denoted $\sigma:E'=\Lambda TX\to E=\Lambda T^{*}X$ and $\sigma:
F'=\Lambda TX\otimes \pi_{X}^{*}(\fkf)\to F=\Lambda T^{*}X\otimes
\pi_{X}^{*}(\fkf)$ which fulfills the condition
\eqref{eq:estimunivphi} (see below for details). We can thus consider the $\sigma$-adjoint of
densely defined operators in $L^{2}(X;F)$ as closed operators in
$L^{2}(X;F)$\,, according to Definition~\ref{de:phiformadj} and deduce their properties from the $(F',F)$-adjoint
according to Proposition~\ref{pr:adjlr}. Because the symplectic form
is anti-symmetric the left and right adjoints are equal.
\begin{definition}
\label{de:sympl}
The operators 
$\hat{d}^{\sigma}_{g,b'\fkh}$ and $d^{\sigma}_{g,b'\fkh}$ are the
symplectic adjoints, that is the $\phi$-adjoint Definition~\ref{de:phiformadj} for
$\phi=\sigma: TX\to T^{*}X$\,, of  the
operators  $\hat{d}_{g, b'\fkh}$ and $\overline{d}_{g,b'\fkh}$ defined
in Definition~\ref{de:domaind} and characterized in
Proposition~\ref{pr:domaindH}.
\end{definition}
Before giving the properties of $\hat{d}^{\sigma}_{g,b'\fkh}$ and
$d^{\sigma}_{g,b'\fkh}$ let us specify some formulas.
\begin{lemma}
\label{le:sympladjextint} For $t\in L^{\infty}_{loc}(X;TX)$ and
$\omega\in L^{\infty}_{loc}(X;T^{*}X)$\,,  the symplectic adjoint of
$\mathbf{i}_{t}$ (resp. $\omega\wedge$) equals $d(\sigma t)\wedge$
(resp. $-\mathbf{i}_{\sigma^{-1}\omega}$)\,. 
\end{lemma}
\begin{proof}
  The general formula of Proposition~\ref{pr:adjlr} says
  $P^{\sigma}=\sigma\tilde{P}\sigma^{-1}={}^{t}\sigma\tilde{P}{}^{t}\sigma$ with ${}^{t}\sigma=-\sigma:TX\to
T^{*}X$\,. Therefore
\begin{eqnarray*}
&&
(\mathbf{i}_{t})^{\sigma}=\sigma
   \widetilde{(\mathbf{i}_{t})}\sigma^{-1}=\sigma(t\wedge)\sigma^{-1}=(\sigma t)\wedge\,,
\\
  && (\omega\wedge)^{\sigma}=\sigma \widetilde{(\omega\wedge)}\sigma^{-1}
     =\sigma
     \mathbf{i}_{\omega}\sigma^{-1}=\mathbf{i}_{{}^{t}\sigma^{-1}\omega}=\mathbf{i}_{-\sigma^{-1}\omega}=-\mathbf{i}_{\sigma^{-1}\omega}\,.
\end{eqnarray*}
\end{proof}
In particular when $\varphi$ is a locally Lipschitz continuous
function the $\sigma$-adjoint of $d\varphi\wedge$ is 
$$
(d\varphi\wedge)^{\sigma}=-\mathbf{i}_{\sigma^{-1}d\varphi}=-\mathbf{i}_{Y_{\varphi}}
$$
where $Y_{\varphi}$ is the Hamiltonian vector field characterized by
$\sigma(t,Y_{\varphi})=t.\sigma Y_{\varphi}=d\varphi(t)$ or
$Y_{\varphi}=\sigma^{-1}d\varphi$\,.\\
When we use the symplectic coordinates $(q,p)$ with $dv_{X}=|dqdp|$
the adjoint of the covariant derivative $\nabla^{\fkf}_{T}$ is
$-\nabla^{\fkf'}_{T}$ and the formal symplectic adjoint of
$$
d_{b'\fkh}=dq^{i}\wedge
\nabla^{\fkf}_{\frac{\partial}{\partial q^{i}}}+dp_{j}\wedge
\frac{\partial}{\partial p_{j}} +b'd\fkh\,,
$$
 with $\sigma(\frac{\partial}{\partial q^{i}})=dp_{i}$\,,
 $\sigma(\frac{\partial}{\partial p_{j}})=-dq^{j}$\,, equals
 \begin{equation}
   \label{eq:dsigma}
d^{\sigma}_{b'\fkh}= -\mathbf{i}_{\frac{\partial}{\partial
    p_{i}}}\nabla^{\fkf'}_{\frac{\partial}{\partial
    q^{i}}}+\mathbf{i}_{\frac{\partial}{\partial
    q^{j}}}\frac{\partial}{\partial p_{j}}-b'\mathbf{i}_{Y_{\fkh}}
=-\mathbf{i}_{\frac{\partial}{\partial
    p_{i}}}\nabla^{\fkf}_{\frac{\partial}{\partial
    q^{i}}}+\mathbf{i}_{\frac{\partial}{\partial
    q^{j}}}\frac{\partial}{\partial p_{j}}
-\mathbf{i}_{\frac{\partial}{\partial
    p_{i}}}\omega(\nabla^{\fkf},g^{\fkf})\left(\frac{\partial}{\partial
  q^{i}}\right)
-b'\mathbf{i}_{Y_{\fkh}}\,.
\end{equation}
The same formula holds on both sides $X_{-}$ and $X_{+}$\,, when $g^{\fkf}$ and $\fkh$ are replaced by
$\hat{g}^{\fkf}$ and $\hat{\fkh}$\,.\\
The boundary conditions for $\hat{d}_{g,b'\fkh}$\,,
$\overline{d}_{g,b'\fkh}$\,, $\tilde{\hat{d}}_{g,b'\fkh}$ and
$\tilde{d}_{g,b'\fkh}$ were studied with the non symplectic
coordinates $(\tilde{q},\tilde{p})$ but were finally formulated with
$e^{1}=dq^{1}$ and $e_{1}=\frac{\partial}{\partial
  \tilde{q}^{1}}$\,. Because $(e_{i},\hat{e}^{j})$ is a symplectic
basis with dual basis $(e^{i},\hat{e}_{j})$\,,  we can use
\begin{eqnarray*}
  &&
\sigma(e_{i})=\hat{e}_{i}\quad,\quad
\sigma(\hat{e}^{j})=-e^{j}\quad,\quad
\sigma^{-1}(\hat{e}_{j})=e_{j}\quad,\quad \sigma^{-1}(e^{i})=-\hat{e}^{i}\,,\\
\text{and}&&
\sigma \mathbf{i}_{e^{i}}\sigma^{-1}=\mathbf{i}_{-\hat{e}^{i}}\quad
\sigma\mathbf{i}_{\hat{e}_{j}}\sigma^{-1}=\mathbf{i}_{e_{j}}\quad
\sigma^{-1}\mathbf{i}_{e_{i}}\sigma=\mathbf{i}_{\hat{e}_{i}}\quad
\sigma^{-1}\mathbf{i}_{\hat{e}^{j}}\sigma=-\mathbf{i}_{e^{j}}\,,
\end{eqnarray*}
without referring to coordinates. 
It implies
$$
|\sigma(e_{i})|_{g^{E}}=|\hat{e}^{i}|_{g^{E}}=\langle
p\rangle_{q}^{1/2}=|e_{i}|_{g^{E'}}\quad,\quad
|\sigma(\hat{e}^{j})|_{g^{E}}=|e^{j}|_{g^{E}}=\langle p\rangle^{-1/2}_{q}=|\hat{e}^{j}|_{g^{E'}}\,,
$$
and the condition \eqref{eq:estimunivphi} is satisfied.\\
The operator
$\hat{S}_{\nu}$ of Definition~\ref{de:S1Snu} and involved in the
boundary conditions for $\overline{d}_{g,b'\fkh}$ and $\tilde{d}_{g,b'\fkh}$ satisfies 
$$
\sigma \hat{S}_{\nu}\sigma^{-1}=\hat{S}_{\nu}\,.
$$
Finally $\sigma$ belongs to $\mathcal{C}^{\infty}(\overline{X}_{-};
L(E',E))\cap\mathcal{C}^{\infty}(\overline{X}_{+}; L(E',E))\cap
\mathcal{C}^{0}(X;L(\hat{E}'_{g},\hat{E}_{g})) $\,, sends
$\mathcal{C}^{0}_{g}(\hat{F}'_{g})$ to
$\mathcal{C}^{0}_{g}(\hat{F}_{g})$ and preserves the parity with
respect to $\Sigma_{\nu}$\,.
\begin{proposition}
\label{pr:domdsigma}
The $\sigma$-adjoint of $(\hat{d}_{g,b'\fkh},D(\hat{d}_{g,b'\fkh}))$  equals
\begin{eqnarray*}
&&
D(\hat{d}^{\sigma}_{g,b'\fkh})=
\left\{s\in L^{2}(X;F)\cap \sigma\mathcal{E}_{loc}(\tilde{d},\hat{F}'_{g})\,,
  \quad \hat{d}^{\sigma}_{g,b'\fkh}s\in L^{2}(X;F)\right\}\,,\\
  &&
\forall s\in D(\hat{d}^{\sigma}_{g,b'\fkh})\,,\quad
     \hat{d}^{\sigma}_{g,b'\fkh}s=
     (\hat{d}^{\sigma}_{g,b'\fkh}s_{-})+(\hat{d}^{\sigma}_{b'\fkh}s_{+})\quad,\quad s_{\mp}=s\big|_{X_{\mp}}\,,\\
&&  \hat{d}^{\sigma}_{g,b'\fkh}s_{\mp}=-\mathbf{i}_{\frac{\partial}{\partial
    p_{i}}}\nabla^{\fkf}_{\frac{\partial}{\partial
    q^{i}}}s_{\mp}+\mathbf{i}_{\frac{\partial}{\partial
    q^{j}}}\frac{\partial s_{\mp}}{\partial p_{j}}
-\mathbf{i}_{\frac{\partial}{\partial
    p_{i}}}\omega(\nabla^{\fkf},\hat{g}^{\fkf})\left(\frac{\partial}{\partial
  q^{i}}\right)s_{\mp}
-b'\mathbf{i}_{Y_{\hat{\fkh}}}s_{\mp}\,.
\end{eqnarray*}
It satisfies $\hat{d}^{\sigma}_{g,b'\fkh}\circ
\hat{d}^{\sigma}_{g,b'\fkh}=0$ and $\hat{d}^{\sigma}_{g,b'\fkh}\circ
\Sigma_{\nu}=\Sigma_{\nu}\circ\hat{d}^{\sigma}_{g,b'\fkh}$\,. In particular, $\hat{d}^{\sigma}_{g,b'\fkh}$ preserves the parity:
\begin{eqnarray*}
&&  D(\hat{d}^{\sigma}_{g,b'\fkh})=D(\hat{d}^{\sigma}_{g,b'\fkh})\cap
  L^{2}_{ev}(X;F)\oplus D(\hat{d}^{\sigma}_{g,b'\fkh})\cap
  L^{2}_{odd}(X;F)\,,\\
\text{with}&&
\hat{d}^{\sigma}_{g,b'\fkh}: D(\hat{d}^{\sigma}_{g,b'\fkh})\cap L^{2}_{ev~odd}(X;F)\to L^{2}_{ev~odd}(X;F)\,.
\end{eqnarray*}
The subset $\mathcal{C}_{0,g}(\hat{F}_{g})$ of Definition~\ref{de:Cg}
 is dense in $D(\hat{d}^{\sigma}_{g,b' \fkh})$\,. Additionally there exists a  dense subset
$\sigma\hat{\mathcal{D}}'_{g,\nabla^{\fkf'}}$ of $\mathcal{C}_{0,g}(\hat{F}_{g})$
such that 
$\hat{d}^{\sigma}_{g,b'\fkh}(\sigma\hat{\mathcal{D}}'_{g,\nabla^{\fkf'}})\subset
\mathcal{C}_{0,g}(\hat{F}_{g})\subset
D(\hat{d}^{\sigma}_{g,b'\fkh})$\,.\\
The adjoint $d^{\sigma}_{g,b'\fkh}$ of $(\overline{d}_{g,b'\fkh},
D(\overline{d}_{g,b'\fkh}))$ is densely defined and closed with
\begin{eqnarray*}
  &&
D(d^{\sigma}_{g,b'\fkh})=\left\{s\in L^{2}(X_{-},F)\,,\quad
  d^{\sigma}_{b'\fkh}s\in L^{2}(X_{-},F)\,,\quad
  \frac{1-\hat{S}_{\nu}}{2}\hat{e}_{1}\wedge \mathbf{i}_{\hat{e}^{1}}s\big|_{X'}=0\right\}\,.
\\
&&
\forall s\in D(d^{\sigma}_{g,b'\fkh})\,,\quad
d^{\sigma}_{g,b'\fkh}s=
-\mathbf{i}_{\frac{\partial}{\partial
    p_{i}}}\nabla^{\fkf}_{\frac{\partial}{\partial
    q^{i}}}s+\mathbf{i}_{\frac{\partial}{\partial
    q^{j}}}\frac{\partial s}{\partial p_{j}}
-\mathbf{i}_{\frac{\partial}{\partial
    p_{i}}}\omega(\nabla^{\fkf},g^{\fkf})\left(\frac{\partial}{\partial
  q^{i}}\right)s
-b'\mathbf{i}_{Y_{\fkh}}s\,.
\end{eqnarray*}
This adjoint operator $(d^{\sigma}_{g,b'\fkh},D(d^{\sigma}_{g,b'\fkh}))$
satisfies $d^{\sigma}_{g,b'\fkh}\circ d^{\sigma}_{g,b'\fkh}=0$\,.\\
The spaces $\mathcal{C}^{\infty}_{0}(\overline{X}_{-};F)\cap
D(d^{\sigma}_{g,b'\fkh})$\,, $\mathcal{C}_{g}=\left\{s\in
  L^{2}(X_{-},F)\,, s_{ev}\in \mathcal{C}_{0,g}(\hat{F}_{g})\right\}$
and $\mathcal{D}_{g,\nabla^{\fkf'}}=\left\{s\in L^{2}(X_{-};F)\,,
  s_{ev}\in \sigma\hat{\mathcal{D}}'_{g,\nabla^{\fkf'}}\right\}$ are dense in
$D(d^{\sigma}_{g,b'\fkh})$ with
$d^{\sigma}_{g,b'\fkh}\mathcal{D}_{g,\nabla^{\fkf'}}\subset
\mathcal{C}_{g}$\,.
\end{proposition}
\begin{proof}
  It is a straightforward application of the general formula of Proposition~\ref{pr:adjlr}
  $$
(P^{\sigma},D(P^{\sigma}))=(\sigma \tilde{P}\sigma^{-1},
  \sigma D(\tilde{P}))\,.
$$
All the properties of $P^{\sigma}$\,, $P=\hat{d}_{g,b'\fkh}$ or
$P=\overline{d}_{g,b'\fkh}$\,,
 are obtained by conjugating
with $\sigma^{-1}$ the ones of $\tilde{P}$ and the properties of
$D(P^{\sigma})$ are obtained by transporting via $\sigma$ the ones of
$D(\tilde{P})$\,. Thus, it is just a translation of
Proposition~\ref{pr:domtilded} combined with the previous formulas and observations.
\end{proof}
\begin{remark}
\label{re:sympladj}
Below are some detailed explanations of the previous result:
\begin{description}
\item[a)] The sets $\mathcal{C}_{0,g}(\hat{F}_{g})$ and
  $\mathcal{C}_{g}$\,, respectively dense in
  $D(\hat{d}^{\sigma}_{g,b'\fkh})$ and $D(d^{\sigma}_{g,b'\fkh})$\,,
  are the same as in Proposition~\ref{pr:domaindH} where they are
  shown to be dense respectively in $D(\hat{d}_{g,b'\fkh})$ and
  $D(\overline{d}_{g,b'\fkh})$\,.
\item[b)] The term $-\mathbf{i}_{\frac{\partial}{\partial
    p_{i}}}\omega(\nabla^{\fkf},\hat{g}^{\fkf})\left(\frac{\partial}{\partial
  q^{i}}\right)$ (resp. $-\mathbf{i}_{\frac{\partial}{\partial
    p_{i}}}\omega(\nabla^{\fkf},g^{\fkf})
\left(\frac{\partial}{\partial  q^{i}}\right)$)   in the expression of $\hat{d}^{\sigma}_{g,b'\fkh}$
(resp. $d^{\sigma}_{g,b'\fkh}$) comes from the comparison between
$\nabla^{\fkf'}$\,, used for  the analysis in $\hat{F}_{g}'$
(resp. $F'\big|_{\overline{X}_{-}}$)\,, and the initial connection $\nabla^{\fkf}$ on
$\hat{F}_{g}$ (resp. $F\big|_{\overline{X}_{-}}$)\,.
\item[c)] The sets $\sigma \hat{\mathcal{D}}'_{g,\nabla^{\fkf'}}$ and
  $\mathcal{D}_{g,\nabla^{\fkf'}}$ differ from the sets
  $\hat{\mathcal{D}}_{g,\nabla^{\fkf}}$ and
  $\mathcal{D}_{g,\nabla^{\fkf}}$ of Proposition~\ref{pr:domaindH}\,, mainly because the two flat
  connections $\nabla^{\fkf}$
  and $\nabla^{\fkf'}$ are related with different
  $\mathcal{C}^{\infty}$ structures of $\pi_{\fkf}:\fkf\to Q$ when
  $\hat{g}^{\fkf}$ is only piecewise $\mathcal{C}^{\infty}$ and
  continuous.
\item[d)] While working with the symplectic structure the symplectic
  coordinates $(q,p)$ are more natural than the coordinates
  $(\tilde{q},\tilde{p})$ which were used in particular in
  Subsection~\ref{sec:diffhatE} for the
  $\mathcal{C}^{\infty}$-structure of $\hat{E}_{g}$ via the manifold
  $M_{g}$\,. When one uses the coordinates $(\tilde{q},\tilde{p})$ on
  $X$\,,  the symplectic form does not have a better regularity that
  the continuity at the interface. An example is given by the disc
  $\overline{Q}_{-}=\overline{D}(0,r_{0})$ in $\rz^{2}$ where the
  metric can be written
  $g_{-}^{TQ}=d(\underline{q}^{1})^{2}+(r_{0}+\underline{q}^{1})^{2}d\underline{q}^{2}$\,,
  where $\underline{q}^{1}$ is the radial coordinate and
  $\underline{q}^{2}$ the angular coordinate. Then the coordinates
  $(\tilde{q},\tilde{p})$ are given by
  $(\tilde{q},\tilde{p}_{1})=(q,p_{1})$ and
  $\tilde{p}_{2}=\frac{r_{0}}{r_{0}+q^{1}}p_{2}$ with
  $d\tilde{p}_{2}=-\frac{r_{0}p_{2}}{(r_{0}+q^{1})^{2}}dq^{1}+\frac{r_{0}}{r_{0}+q^{1}}dp_{2}$\,. For
  the metric $g_{+}^{TQ}$ simply replace $(r_{0}+q^{1})$ by
  $(r_{0}-q^{1})$\,. When the coordinates $(\tilde{q},\tilde{p})$ are
  constructed for
  $\hat{g}^{TQ}=1_{\overline{Q}_{-}}(\underline{q})g_{-}^{TQ}+1_{\overline{Q}_{+}}(\underline{q})g_{+}^{TQ}$\,,
  the symplectic volume $dv_{X}=|dqdp|$ equals
  $|\frac{r_{0}-|\tilde{q}^{1}|}{r_{0}}||d\tilde{q}d\tilde{p}|$
  which is clearly only piecewise $\mathcal{C}^{\infty}$ and
  continuous in those coordinates. Introducing the symplectic form
  breaks the $\mathcal{C}^{\infty}$ structure inherited from
  $M_{g}$\,. Again, this is a reason why we prefer to work with
  piecewise $\mathcal{C}^{\infty}$ and continuous vector bundles: This
  is the right framework were all the structures can be put together.
\item[e)] It is possible to express the symplectic codifferential with
  the frame $(e,\hat{e})$ and in terms of connections. If we work only
  with $E$\,, or equivalently with $\fkf=Q\times \cz$\,, $\nu=1$\,,
  and with $b'=0$\,,
  Bismut in \cite{Bis05} wrote
\begin{eqnarray*}
  &&d=e^{i}\wedge \nabla_{e_{i}}^{E}+ \hat{e}_{j}\wedge
     \nabla_{\hat{e}^{j}}^{X}+ \mathbf{i}_{R^{TQ}p}\,,\\
&& d^{\sigma}=-\mathbf{i}_{\hat{e}^{i}}\nabla^{E}_{e_{i}}+\mathbf{i}_{e_{j}}\nabla^{X}_{\hat{e}^{j}}+R^{TQ}p\wedge\,.
\end{eqnarray*}
When we work with the non smooth metric $\hat{g}^{TQ}$\,, the
zeroth order  term related with the curvature tensor $R^{TQ}$ is not
continuous along $X'$ and rather complicated. It is not obvious to
check $d\circ d=0$ and $d^{\sigma}\circ d^{\sigma}=0$ or to identify
easily dense domains of smooth sections. From this point of view, the
coordinates $(\tilde{q},\tilde{p})$ for the
$\mathcal{C}^{\infty}$-structure of the manifold $M_{g}$ and then the
symplectic coordinates $(q,p)$ make things more obvious.
\end{description}
\end{remark}

\subsection{Bismut codifferential}
\label{sec:biscod}
In the previous section we introduced the parameter $b'\geq 0$ in
front of $\fkh$ in order to make the comparison of local properties in
the case $b'=0$ and $b'=1$ self-contained. We now fix $b'=1$ and will
use the scaling of \cite{Bis05} recalled in the introduction where the
bilinear form $\eta_{\phi_{b}}$ on $TX$ is defined with a parameter
$b\in \rz^{*}$ as 
\begin{equation}
  \label{eq:etaphib}
\eta_{\phi_{b}}(U,V)=g^{TQ}(\pi_{X,*}(U),\pi_{X,*}(V))+b\sigma(U,V)=U.\phi_{b}V\,,\quad
U,V\in TX\,.
\end{equation}
The bilinear form $\phi_{b}$ is 
 associated with the map $\phi_{b}:TX\to T^{*}X$ written, by
taking a symplectic basis compatible with the horizontal-vertical
decompositions of $TX$ (like $(e_{i},\hat{e}^{j})$) and
$T^{*}X$ (like $(e^{i},\hat{e}_{j})$), as
\begin{equation}
  \label{eq:phib}
\phi_{b}=
\begin{pmatrix}
 g^{TQ}&-b\mathrm{Id}\\
b\mathrm{Id}&0
\end{pmatrix}\,,\quad b\neq 0\,.
\end{equation}
The dual bilinear form on $T^{*}X$
 is denoted $\eta^{*}_{\phi}$\,:
$$
\eta_{\phi_{b}}^{*}(\omega,\theta)=(\phi_{b}^{-1}\omega).\theta\quad,\quad
\phi_{b}^{-1}=\frac{1}{b^{2}}
\begin{pmatrix}
  0&b\mathrm{Id}\\
-b\mathrm{Id}&g^{TQ}
\end{pmatrix}\,.
$$
With the local bases $(e,\hat{e})$\,, the map $\phi_{b}$ can be 
be simply related with the operator $\sigma:TX\to T^{*}X$ associated
with the symplectic form.
With 
\begin{equation}
  \label{eq:lambda0}
\lambda_{0}=g^{TQ}_{ij}(q)e^{i}\wedge \mathbf{i}_{\hat{e}^{i}}: T^{*}X\to T^{*}X
\end{equation}
and  by assuming $(\pi_{X,*}e_{i}=\frac{\partial}{\partial \underline{q}^{i}})_{i=1,\ldots d}$
orthogonal along
$X_{\underline{q}_{0}}=T^{*}_{\underline{q}_{0}}Q$\,, writing
$$
\begin{pmatrix}
  1&-b\\
b&0
\end{pmatrix}=
\begin{pmatrix}
  1&\frac{1}{b}\\0&1
\end{pmatrix}
\begin{pmatrix}
  0&-b\\
b&0
\end{pmatrix}
=\exp\left[
\begin{pmatrix}
  0&\frac{1}{b}\\
0&0
\end{pmatrix}\right]
\begin{pmatrix}
  0&-b\\
b&0
\end{pmatrix}
$$
shows that $\phi_{b}=e^{\frac{\lambda_{0}}{b}}b\sigma$\,. In the decomposition
$$
{}^{t}\phi_{b}=b{}^{t}\sigma e^{\frac{{}^{t}\lambda_{0}}{b}}
={}^{t}\sigma e^{\frac{{}^{t}\lambda_{0}}{b}}({}^{t}\sigma^{-1})(b{}^{t}\sigma): TX\to T^{*}X
$$
the factor ${}^{t}\sigma e^{\frac{{}^{t}\lambda_{0}}{b}}{}^{t}\sigma^{-1}: T^{*}X\to
T^{*}X$ can also computed with the bases according to
$$
\begin{pmatrix}
  0&1\\
-1&0
\end{pmatrix}
\begin{pmatrix}
  0&0\\
\frac{1}{b}&0
\end{pmatrix}
\begin{pmatrix}
  0&-1\\
1&0
\end{pmatrix}
=
\begin{pmatrix}
  0&-\frac{1}{b}\\
0&0
\end{pmatrix}\,,
$$
which leads to
${}^{t}\phi_{b}=e^{-\frac{\lambda_{0}}{b}}({}^{t}b\sigma)=-e^{-\frac{\lambda_{0}}{b}}b\sigma$\,.\\
Here attention must be paid on the scaling with respect to $b\in
\rz^{*}$ while tensorizing $b\sigma$\,. Actually the multiplication by
$b$ on $TX$ or $T^{*}X$ is tensorized into the multiplication by
$b^{p}$ on $\Lambda^{p}TX$ or $\Lambda^{p}T^{*}X$\,. Therefore it is
better to use the notation
\begin{eqnarray*}
  &&
\sigma_{b}=(\otimes_{p=0}^{d}b^{p})\sigma: \Lambda TX\to \Lambda
     T^{*}X\\
&& \phi_{b}=e^{\frac{\lambda_{0}}{b}}\sigma_{b}: \Lambda TX\to \Lambda
   T^{*}X\,.
\end{eqnarray*}
\begin{definition}
  \label{de:phiBis} The linear maps $\phi_{b}$ and $\lambda_{0}$ are
  respectively given by \eqref{eq:phib}, extended by tensorization as
  a map $\phi_{b}=e^{\frac{\lambda_{0}}{b}}\sigma_{b}:\Lambda TX\to \Lambda T^{*}X$\,, and
  \eqref{eq:lambda0}. The same notation is used for
  $\phi_{b}=\phi_{b}\otimes \mathrm{Id}_{\fkf}$ and
  $\lambda_{0}=\lambda_{0}\otimes \mathrm{Id}_{\fkf}$ when $E=\Lambda
  T^{*}X$ or $E'=\Lambda TX$ are replaced by $F=E\otimes
  \pi_{X}^{*}(\fkf)$ and $F'=E'\otimes \pi_{X}^{*}(\fkf)$\,.\\
Accordingly the sesquilinear forms $\eta_{\phi_{b}}$ on $F'$ and
$\eta_{\phi_{b}}^{*}$ on $F$ are defined by
$$
\eta_{\phi_{b},\fkf}(U,V)=g^{\fkf}(U,\phi_{b}V)\quad,\quad 
\eta_{\phi_{b},\fkf}^{*}(\omega,\theta)=g^{\fkf}(\phi_{b}^{-1}\omega,\theta)\,.
$$
Finally the same notations $\phi_{b}$ and $\lambda_{0}$ are 
used with
$g=g^{TQ}(\underline{q}^{1},\underline{q}')$ replaced by
$\hat{g}=g^{TQ}(-|\underline{q}^{1}|,\underline{q}')$ and $E=\Lambda
T^{*}X,E'=\Lambda TX , F , F'$ replaced by the piecewise
$\mathcal{C}^{\infty}$ and continuous vector bundles
$\hat{E}_{g},\hat{E}'_{g},\hat{F}_{g},\hat{F}'_{g}$ of Definition~\ref{de:hatEF}
\end{definition}
The following lemma gather simple elementary properties of those maps
$\lambda_{0}$ and $\phi_{b}$\,.
\begin{lemma}
\label{le:phibdd}
On $\hat{F}_{g}$ the map $\lambda_{0}$ belongs to
$\mathcal{C}^{\infty}(\overline{X}_{-};L(F))\cap
\mathcal{C}^{\infty}(\overline{X}_{+};L(F))\cap
\mathcal{C}^{0}(X;L(F))$ with
$\Sigma_{\nu}\lambda_{0}=\lambda_{0}\Sigma_{\nu}$\,. In particular it is a
continuous endomorphism of $\mathcal{C}_{0,g}(\hat{F}_{g})$ and
$\mathcal{C}_{0,g,ev}(\hat{F}_{g})$\,.\\
The maps $\phi_{b}:F'\big|_{X_{\mp}}\to F\big|_{X_{\mp}}$ and
$\phi_{b}:\hat{F}'_{g}\to \hat{F}_{g}$ fulfill the condition
\eqref{eq:estimunivphi} with
$\phi_{b}=e^{\frac{\lambda_{0}}{b}}\sigma_{b}$ and
${}^{t}\phi_{b}=-e^{-\frac{\lambda_{0}}{b}}\sigma_{b}=\phi_{-b}$\,.\\
When $(P,D(P))$ is a densely defined operator in
$L^{2}(X;\hat{F}_{g})$ (or $L^{2}(X_{\mp};F)$) with a symplectic
adjoint $(P^{\sigma_{b}},D(P^{\sigma_{b}}))$\,, its left  and right $\phi_{b}$-adjoints equal
\begin{eqnarray*}
  &&P^{\phi_{b}}=\phi_{b}\tilde{P}\phi_{b}^{-1}=e^{\frac{\lambda_{0}}{b}}\sigma_{b}
\tilde{P}\sigma_{b}^{-1}e^{-\frac{\lambda_{0}}{b}}=
e^{\frac{\lambda_{0}}{b}}P^{\sigma_{b}}e^{-\frac{\lambda_{0}}{b}}\\
&&
   P^{{}^{t}\phi_{b}}=P^{\phi_{-b}}=e^{-\frac{\lambda_{0}}{b}}P^{\sigma_{b}}e^{\frac{\lambda_{0}}{b}}\\
\text{with}&&
              D(P^{\phi_{b}})=e^{-\frac{\lambda_{0}}{b}}D(P^{\sigma_{b}})\quad,\quad D(P^{{}^{t}\phi_{b}})=e^{\frac{\lambda_{0}}{b}}D(P^{\sigma_{b}})\,.
\end{eqnarray*}
\end{lemma}
\begin{proof}
  The first properties come from the definition of
$$
\lambda_{0}=e^{i}\wedge
\mathbf{i}_{\hat{e}^{i}}=e^{1}\wedge\mathbf{i}_{\hat{e}^{1}}+m_{ij}(-|q^{1}|,q')e^{i'}\wedge \mathbf{i}_{\hat{e}^{j'}}\,,
$$
where we use
$(e,\hat{e})=1_{Q_{-}}(q)(e_{-},\hat{e}_{-})+1_{Q_{+}}(q)(e_{+},\hat{e}_{+})$\,.
The regularity with respect to $x=(q,p)$ of $\lambda_{0}$ is inherited
from the one of $(e,\hat{e})$\,. The commutation with $\Sigma_{\nu}$
comes from
$$
\Sigma_{\nu}\left[s_{I}^{J}(q^{1},q',p_{1},p')e^{I}\hat{e}_{J}\right]=(-1)^{|I\cap
  \left\{1\right\}|+|J\cap\left\{1\right\}|}\nu s_{I}^{J}(-q^{1},q',-p_{1},p')e^{I}\hat{e}_{J}\,.
$$
We know that $\sigma$\,, and therefore $\sigma_{b}$ when $b\neq 0$\,, fulfills
the condition \eqref{eq:estimunivphi}. It thus suffices to check the
equivalence
\begin{equation}
  \label{eq:equivb}
\exists C_{b}>0\,, \forall x\in X\,, \forall \omega \in F_{x}\,,\quad
C_{b}^{-1}|e^{\frac{\lambda_{0}}{b}}\omega |_{g^{F}_{x}}\leq
|\omega|_{g^{F}_{x}}\leq C_{b}|e^{\frac{\lambda_{0}}{b}}\omega|_{g^{F}_{x}}\,.
\end{equation}
With coordinates such that $(\frac{\partial}{\partial
  \underline{q}^{1}},\ldots, \frac{\partial}{\partial
  \underline{q}^{d}})$ is orthonormal above a fixed $q_{0}\in
Q$\,, $g^{TQ}_{ij}(q_{0})=\delta_{ij}$\,, decompose $\omega\in
F_{(q_{0},p)}$ as $\omega=\omega^{H}\oplus^{\perp}\omega^{V}
\in  F_{x}=(T^{*}_{x}X^{H}\otimes
\pi_{X}^{*}(\fkf))\oplus^{\perp}(T^{*}_{x}X^{V}\otimes
\pi_{X}^{*}(\fkf))$  with $x=(q_{0},p)$\,.
Write simply $\omega=
\begin{pmatrix}
  \omega^{H}\\
\omega^{V}
\end{pmatrix}
$ and  $e^{\pm \lambda_{0}}\omega=
\begin{pmatrix}
  \omega^{H}\pm \frac{1}{b}\omega^{V}\\
\omega^{V}
\end{pmatrix}
$\,. The  $g^{F}_{x}$ norm of $\omega$ and $e^{\pm \lambda_{0}}\omega$
satisfy
\begin{eqnarray*}
  &&
|\omega|^{2}_{g^{F}_{x}}=\langle
p\rangle^{-1}|\omega^{H}|_{g^{\fkf}_{q_{0}}}^{2}+\langle p\rangle |\omega^{V}|_{g^{\fkf}_{q_{0}}}^{2}\\
&&
|e^{\pm \frac{\lambda_{0}}{b}}\omega|_{g^{F}_{x}}^{2}
=\langle
p\rangle^{-1}|\omega^{H}\pm \frac{1}{b}\omega^{V}|_{g^{\fkf}_{q_{0}}}^{2}+\langle
   p\rangle|\omega^{V}|_{g^{\fkf}_{q_{0}}}^{2}\\
&&
|e^{\pm \frac{\lambda_{0}}{b}}\omega|^{2}_{g^{F}_{x}}\leq
   2\langle p\rangle^{-1}|\omega^{H}|_{g^{\fkf}_{q_{0}}}^{2}+\langle
   p\rangle(1+\frac{2}{\langle p\rangle_{q}^{2}b^{2}})|\omega^{V}|_{g^{\fkf}_{q_{0}}}^{2}
\leq \max(2, 1+\frac{2}{b^{2}})|\omega|_{g^{F}_{x}}^{2}\,,
\,.
\end{eqnarray*}
Applying the last inequality with
$\omega=e^{\mp\frac{\lambda_{0}}{b}}\eta$ provides the reverse
inequality
$$
|\eta|_{g^{F}_{x}}^{2}\leq \max(2,1+\frac{2}{b^{2}})|e^{\mp\frac{\lambda_{0}}{b}}\eta|_{g^{F}_{x}}\,.
$$
The equivalence \eqref{eq:equivb} is thus proved with $C_{b}=\max(2,1+\frac{2}{b^{2}})$\,.
\end{proof}
\begin{proposition}
\label{pr:phiadj}
Take $b'=1$ and $b\neq 0$\,.\\
The $\phi_{b}$ left-adjoint of $(\hat{d}_{g,\fkh},D(\hat{d}_{g,\fkh}))$  equals
\begin{eqnarray*}
&&
D(\hat{d}^{\phi_{b}}_{g,b'\fkh})=
\left\{s\in L^{2}(X;F)\cap e^{\frac{\lambda_{0}}{b}}\sigma_{b}\mathcal{E}_{loc}(\tilde{d},\hat{F}'_{g})\,,
  \quad \hat{d}^{\phi_{b}}_{g,\fkh}s\in L^{2}(X;F)\right\}\,,\\
  &&
\forall s\in D(\hat{d}^{\phi_{b}}_{g,\fkh})\,,\quad
     \hat{d}^{\phi_{b}}_{g,\fkh}s=
     (\hat{d}^{\phi_{b}}_{g,\fkh}s_{-})+(\hat{d}^{\phi_{b}}_{\fkh}s_{+})\quad,\quad s_{\mp}=s\big|_{X_{\mp}}\,,\\
&&  \hat{d}^{\phi_{b}}_{g,\fkh}s_{\mp}=e^{-\frac{\lambda_{0}}{b}}d^{\sigma_{b}}_{\hat{\fkh}}e^{\frac{\lambda_{0}}{b}}s_{\mp}=\frac{1}{b}
e^{-\frac{\lambda_{0}}{b}}d^{\sigma_{b}}_{\hat{\fkh}}e^{\frac{\lambda_{0}}{b}}s_{\mp}\,.
\end{eqnarray*}
It satisfies $\hat{d}^{\phi_{b}}_{g,\fkh}\circ
\hat{d}^{\phi_{b}}_{g,\fkh}=0$ and $\hat{d}^{\phi_{b}}_{g,\fkh}\circ
\Sigma_{\nu}=\Sigma_{\nu}\circ\hat{d}^{\phi_{b}}_{g,\fkh}$\,. In particular, $\hat{d}^{\phi_{b}}_{g,\fkh}$ preserves the parity:
\begin{eqnarray*}
&&  D(\hat{d}^{\phi_{b}}_{g,\fkh})=D(\hat{d}^{\phi_{b}}_{g,\fkh})\cap
  L^{2}_{ev}(X;F)\oplus D(\hat{d}^{\phi_{b}}_{g,\fkh})\cap
  L^{2}_{odd}(X;F)\,,\\
\text{with}&&
\hat{d}^{\phi_{b}}_{g,\fkh}: D(\hat{d}^{\phi_{b}}_{g,\fkh})\cap L^{2}_{ev~odd}(X;F)\to L^{2}_{ev~odd}(X;F)\,.
\end{eqnarray*}
The subset $\mathcal{C}_{0,g}(\hat{F}_{g})$ of Definition~\ref{de:Cg}
 is dense in $D(\hat{d}^{\phi_{b}}_{g,\fkh})$\,. Additionally there exists a  dense subset
$e^{\frac{\lambda_{0}}{b}}\sigma_{b}\hat{\mathcal{D}}'_{g,\nabla^{\fkf'}}$ of $\mathcal{C}_{0,g}(\hat{F}_{g})$
such that 
$\hat{d}^{\phi_{b}}_{g,\fkh}(e^{\frac{\lambda_{0}}{b}}\sigma_{b} \hat{\mathcal{D}}'_{g,\nabla^{\fkf'}})\subset
\mathcal{C}_{0,g}(\hat{F}_{g})\subset
D(\hat{d}^{\phi_{b}}_{g,\fkh})$\,.\\
The adjoint $d^{\phi_{b}}_{g,\fkh}$ of $(\overline{d}_{g,\fkh},
D(\overline{d}_{g,\fkh}))$ is densely defined and closed with
\begin{eqnarray*}
  &&
D(d^{\phi_{b}}_{g,\fkh})=\left\{s\in L^{2}(X_{-},F)\,,\quad
  d^{\phi_{b}}_{\fkh}s\in L^{2}(X_{-},F)\,,\quad
  \frac{1-\hat{S}_{\nu}}{2}\hat{e}_{1}\wedge \mathbf{i}_{\hat{e}^{1}}s\big|_{X'}=0\right\}\,.
\\
&&
\forall s\in D(d^{\sigma_{b}}_{g,b'\fkh})\,,\quad
d^{\phi_{b}}_{g,\fkh}s=e^{-\frac{\lambda_{0}}{b}}d^{\sigma_{b}}_{g,\fkh}e^{\lambda_{0}}s=\frac{1}{b}e^{-\frac{\lambda_{0}}{b}}d^{\sigma_{b}}_{g,\fkh}e^{\lambda_{0}}s\,.
\end{eqnarray*}
This adjoint operator $(d^{\phi_{b}}_{g,\fkh},D(d^{\phi_{b}}_{g,\fkh}))$
satisfies $d^{\phi_{b}}_{g,\fkh}\circ d^{\phi_{b}}_{g,\fkh}=0$\,.\\
The spaces $\mathcal{C}^{\infty}_{0}(\overline{X}_{-};F)\cap
D(d^{\phi_{b}}_{g,\fkh})$\,, $\mathcal{C}_{g}=\left\{s\in
  L^{2}(X_{-},F)\,, s_{ev}\in \mathcal{C}_{0,g}(\hat{F}_{g})\right\}$
and $e^{\frac{\lambda_{0}}{b}}\mathcal{D}_{g,\nabla^{\fkf'}}=\left\{s\in L^{2}(X_{-};F)\,,
  s_{ev}\in e^{\frac{\lambda_{0}}{b}}\sigma_{b}\hat{\mathcal{D}}'_{g,\nabla^{\fkf'}}\right\}$ are dense in
$D(d^{\phi_{b}}_{g,\fkh})$ with
$d^{\phi_{b}}_{g,\fkh}e^{\frac{\lambda_{0}}{b}}\mathcal{D}_{g,\nabla^{\fkf'}}\subset
\mathcal{C}_{g}$\,.\\
Finally the $\phi_{b}$ right-adjoint is simply  the $\phi_{-b}$
left-adjoint.
\end{proposition}
\begin{proof}
  Most of the properties are derived from the properties of the symplectic
  adjoints by conjugation with $e^{-\frac{\lambda_{0}}{b}}$\,,
  according to Lemma~\ref{le:phibdd}\\
One thing to be checked is the simplified writing of the boundary
conditions for $s\in D(d^{\phi_{b}}_{g,\fkh})$\,. Actually $s\in D(d^{\phi_{b}}_{g,\fkh})=
e^{\frac{\lambda_{0}}{b}}D(d^{\sigma_{b}}_{g,\fkh})$ contains the boundary
condition
\begin{equation}
  \label{eq:bclambda0}
\frac{1-\hat{S}_{\nu}}{2}\hat{e}_{1}\wedge \mathbf{i}_{\hat{e}^{1}}e^{-\frac{\lambda_{0}}{b}}s\big|_{X'}=0\,.
\end{equation}
With
$$
\lambda_{0}=
\underbrace{e^{1}\wedge
  \mathbf{i}_{\hat{e}^{1}}}_{=\lambda_{0}^{1}}+\underbrace{m_{i'j'}(-|q^{1}|,q')e^{i'}\wedge
  \mathbf{i}_{\hat{e}^{j'}}}_{=\lambda_{0}'}\quad\text{with}\quad \lambda_{0}^{1}\lambda_{0}'=\lambda_{0}'\lambda_{0}^{1}
$$
and $\hat{e}_{1}\wedge
\mathbf{i}_{\hat{e}^{1}}\lambda_{0}'=\lambda_{0}' \hat{e}_{1}\wedge
\mathbf{i}_{\hat{e}^{1}}$ and $\hat{e}_{1}\wedge
\mathbf{i}_{\hat{e}^{1}}\lambda_{0}^{1}=0$ we obtain:
$$
\hat{e}_{1}\wedge
\mathbf{i}_{\hat{e}^{1}}e^{-\frac{\lambda_{0}}{b}}s=\hat{e}_{1}\wedge\mathbf{i}_{\hat{e}^{1}}e^{-\frac{\lambda_{0}^{1}}{b}}e^{-\frac{\lambda_{0}'}{b}}s=e^{-\frac{\lambda_{0}'}{b}}\hat{e}^{1}\wedge\mathbf{i}_{\hat{e}^{1}}s\,.
$$
Since $\hat{S}_{\nu}$ commutes with  $e^{-\frac{\lambda_{0}'}{b}}$ the
boundary condition \eqref{eq:bclambda0} is equivalent to 
$$
\frac{1-\hat{S}_{\nu}}{2}\hat{e}_{1}\wedge \mathbf{i}_{\hat{e}^{1}}s\big|_{X'}=0\,.
$$
\end{proof}
\begin{remark}
  \begin{description}
  \item[a)] No explicit expression was given for the differential
    operator $d^{\phi_{g}}_{\fkh}$\,. It is not really necessary and
    actually more confusing when properties along the boundaries are considered. Such an expression may be found in
    \cite{Bis05}:
    \begin{eqnarray*}
&& d^{\phi_{b}}_{\fkh}=\frac{1}{b}(d^{\sigma}_{0}-\mathbf{i}_{Y_{\fkh}})-\frac{1}{b^{2}}[d^{\sigma}_{0}-\mathbf{i}_{Y_{\fkh}},\lambda_{0}]
\\
&&\quad=-\frac{1}{b}\mathbf{i}_{\hat{e}^{i}}
    \big(\nabla^{F,g}_{e_{i}}+\omega(\fkf,g^{\fkf})(e_{i})-\underbrace{\nabla_{e_{i}}\fkh}_{{=0}}\big)
-\frac{1}{b}\mathbf{i}_{e_{i}}(\frac{\partial}{\partial p_{i}}-g^{ik}(q)p_{k})
-\frac{1}{b}R^{TQ}p\wedge\\
&&\hspace{4cm}
-\frac{1}{b^{2}}\mathbf{i}_{\hat{e}^{i}}(\frac{\partial}{\partial p_{i}}-g^{ik}(q)p_{k})\,.
    \end{eqnarray*}
\item[b)] Note that the set
  $e^{\frac{\lambda_{0}}{b}}\mathcal{D}_{g,\nabla^{\fkf'}}$ for the $\phi_{b}$
  left-adjoint and
  $e^{-\frac{\lambda_{0}}{b}}\mathcal{D}_{g,\nabla^{\fkf'}}$
  for the $\phi_{b}$ right-adjoint differ. In general, it is not possible to find a same core
  of smooth sections, which is sent simultaneously to
  $\mathcal{C}_{g}$ by $d^{\phi_{b}}_{g,\fkh}$ and
  $d^{\phi_{-b}}_{g,\fkh}$\,, especially when the second fundamental
  form of $Q'\subset (Q,g^{TQ})$ does not vanish. This is a curvature
  problem actually similar to the distinction between
  $\mathcal{D}_{g,\nabla^{\fkf}}$ and
  $\mathcal{D}_{g,\nabla^{\fkf'}}$\,.
\item[c)] However $\mathcal{D}_{g,\nabla^{\fkf}}$ is dense in
  $\mathcal{C}_{g}$ and therefore a core for all the operators
  $\overline{d}_{g,\fkh}$\,, $d^{\sigma_{b}}_{g,\fkh}$ and
  $d^{\phi_{b}}_{g,\fkh}$ with 
$$
\overline{d}_{g,\fkh}\mathcal{D}_{g,\nabla^{\fkf}}\subset
\mathcal{C}_{g}\subset D(\overline{d}_{g,\fkh})\cap
D(\overline{d}^{\sigma_{b}}_{g,\fkh})\cap D(\overline{d}^{\phi_{b}}_{g,\fkh})\,.
$$
Symmetric versions with $(\overline{d}_{g,\fkh},
\mathcal{D}_{g,\nabla^{\fkf}})$ replaced by
$(d^{\phi_{b}}_{g,\fkh},e^{\frac{\lambda_{0}}{b}}\mathcal{D}_{g,\nabla^{\fkf'}})$
for the $\phi_{b}$ left-adjoint or
$(d^{\phi_{-b}}_{g,\fkh},e^{-\frac{\lambda_{0}}{b}}\mathcal{D}_{g,\nabla^{\fkf'}})$
for the $\phi_{b}$ right-adjoint
hold true.
  \end{description}
\end{remark}

\section{Closed realizations of the hypoelliptic Laplacian}
\label{sec:closedhypo}
This section is split in several parts which follow the
general scheme for the analysis of the differential and its adjoints.
\begin{itemize}
\item 
In Subsection~\ref{sec:hyposmooth} we review the known results of
\cite{Bis05}\cite{BiLe}\cite{Leb1}\cite{Leb2} for the hypoelliptic Laplacian when
$(Q,g^{TQ})$ is a smooth compact riemannian manifold. In particular we
recall the class of Geometric Kramers-Fokker-Planck operators
introduced in \cite{Leb1}\cite{Leb2}.
\item The Subsection~\ref{sec:tracelocalB} focuses on trace theorem
  local forms of Geometric Kramers-Fokker-Planck operators.
\item The definitions of closed realizations of the hypoelliptic
  Laplacian acting on sections of $\hat{F}_{g}$ or of $F\big|_{X_{-}}$
  with boundary conditions are given in
  Subsection~\ref{sec:domainB}. Global subelliptic estimates derived
  from the one of the scalar case in \cite{Nie} are reviewed.
\item In Subsection~\ref{sec:bootstrap} improved global estimates are
  given for powers of the resolvent  and the semigroup associated with
  the maximal accretive closed realizations of the
  hypoelliptic Laplacian.
\item The commutation of the resolvent of the closed maximal accretive
  realizations of the hypoelliptic Laplacian, with the differential
  and Bismut's codifferential are proved in
  Subsection~\ref{sec:commutppty}. Because it concerns commutation of
  closed unbounded operators it is better to adapt the strategy of
  \cite{ABG} where instead of a $\mathcal{C}_{0}$-group  the closed
  realizations of the hypoelliptic Laplacian generate ``cuspidal''
  semigroups.
\item Finally Subsection~\ref{sec:PTsymm} is concerned with
  PT-symmetry which implies that the spectrum of the closed
  realizations of the hypoelliptic Laplacian is symmetric with respect
  to the real axis. This property, which actually holds only on dense
  set of the domain or at the formal level,  is crucial when the  asymptotic spectral  analysis in the specific regimes is considered.
\end{itemize}
\subsection{The hypoelliptic Laplacian in the smooth case}
\label{sec:hyposmooth}
We review here definitions and properties useful for the analysis of
Bismut's hypoelliptic Laplacian when $X=T^{*}Q$ and $(Q,g^{TQ})$ is a
smooth closed compact riemannian manifold. The vector bundle $F$
equals $\Lambda T^{*}X\otimes \pi_{X}^{*}(\fkf)$ where
$(\fkf,\nabla^{\fkf},g^{\fkf})$ is a smooth vector bundle on $Q$
endowed with the smooth hermitian metric $g^{\fkf}$\,, $\nabla^{\fkf}$
is a flat connection, $\fkf$ is identified with its antidual via the
hermitian metric and $\nabla^{\fkf'}$ denotes the antidual flat connection.
Bismut's hypoelliptic Laplacian is defined as the differential operator 
\begin{eqnarray*}
  &&
B^{\phi_{b}}_{\fkh}=\frac{1}{4}(d^{\phi_{b}}_{\fkh}+d_{\fkh})^{2}=\frac{1}{4}(d^{\phi_{b}}_{\fkh}d_{\fkh}+d_{\fkh}d^{\phi_{b}}_{\fkh})
\\
\text{with}&&
\fkh(q,p)=\frac{g^{TQ,ij}(q)p_{i}p_{j}}{2}\,,\\
&& \phi_{b}=e^{\frac{\lambda_{0}}{b}}\sigma_{b}\,,\quad b\in \rz^{*}\,,
\end{eqnarray*}
where $\phi_{b}$ and $\lambda_{0}$ were given in the Introduction and
in Definition~\ref{de:phiBis}.\\
It is a differential operator with $\mathcal{C}^{\infty}(X;L(F))$
coefficients, acts naturally on $\mathcal{C}^{\infty}_{0}(X;F)$ and
$\mathcal{D}'(X;F)$ and satisfies
$$
B^{\phi_{b}}_{\fkh}d_{\fkh}=d_{\fkh}B^{\phi_{b}}_{\fkh}\quad,\quad
B^{\phi_{b}}_{\fkh}d^{\phi_{b}}_{\fkh}=d^{\phi_{b}}_{\fkh}
B^{\phi_{b}}_{\fkh}d_{\fkh}\,.
$$
In \cite{BiLe}\cite{Leb1}\cite{Leb2}  Bismut and Lebeau developed the
functional and spectral analysis of this differential operator in the  $L^{2}$-space associated with the metric
$\tilde{g}^{F}=\langle p\rangle^{N_{H}+N_{V}}g^{F}$ and
$L^{2}(X;F,g^{F})=\langle
p\rangle_{q}^{\frac{N_{H}+N_{V}}{2}}L^{2}(X;F,\tilde{g}^{F})$\,. With
our choice of metric and the unitary equivalence $\langle
p\rangle_{q}^{\frac{N_{H}+N_{V}}{2}}:L^{2}(X;F,\tilde{g}^{F})\to
L^{2}(X;F,g^{F})$ their results concerns
$$
\langle p\rangle_{q}^{\frac{N_{H}+N_{V}}{2}}B^{\phi_{b}}_{\fkh}\langle
p\rangle_{q}^{-\frac{N_{H}+N_{V}}{2}}=B^{\phi_{b}}_{\fkh}+\langle
p\rangle_{q}^{\frac{N_{H}+N_{V}}{2}}[B^{\phi_{b}}_{\fkh}, \langle p\rangle_{q}^{-\frac{N_{H}+N_{V}}{2}}]\,.
$$
We checked in Proposition~\ref{pr:indepWmu}-e) and the following
discussion that $\langle
p\rangle_{q}^{\frac{N_{H}+N_{V}}{2}}:\mathcal{W}^{\mu}(X;F,\tilde{g}^{F})\to
\mathcal{W}^{\mu}(X;F,g^{F})$ is an isomorphism. The results of
\cite{Leb1}\cite{Leb2} allow to absorb error terms due to the
conjugation via $\langle p\rangle_{q}^{\pm\frac{N_{V}+N_{H}}{2}}$
because they are concerned with the following general class of
operators.
\begin{definition}
\label{de:GKFP} 
The metric $\gamma=g^{F}$\,, $\gamma=\tilde{g}^{F}=\langle
p\rangle_{q}^{N_{H}+N_{V}}g^{F}$ induces a norm $|~|_{\gamma}$ on
$L(F,F)$\,. The
associated class of $\gamma$-symbols of order $m\in \rz$ is defined as
the set of functions $M\in \mathcal{C}^{\infty}(X;L(F,F))$ such that 
\begin{eqnarray*}
&&\forall \alpha,\beta\in \nz^{d}\,, \exists C_{\alpha,\beta}>0\,,
   \forall x=(q,p)\in X\,,\quad |(\nabla^{F}_{e})^{\alpha}(\nabla^{F}_{\hat{e}})^{\beta}M(x)|_{\gamma}\leq
C_{\alpha,\beta}\langle p\rangle_{q}^{m-|\beta|}\,,
\\
\text{with}&&
(\nabla^{F}_{e})^{\alpha}(\nabla^{F})^{\beta}_{\hat{e}}=(\nabla_{e_{1}}^{F})^{\alpha_{1}}\ldots(\nabla_{e_{d}}^{F})^{\alpha_{d}}(\nabla_{\hat{e}^{1}}^{F})^{\beta_{1}}\ldots(\nabla_{\hat{e}^{d}}^{F})^{\beta_{d}}\,.
\end{eqnarray*}
A geometric Kramers-Fokker-Planck (GKFP) operator for the metric $\gamma$ is a differential
operator acting on $\mathcal{C}^{\infty}(X;E)$ and $\mathcal{D}'(X;E)$
of the following form:
\begin{eqnarray*}
&&\mathcal{A}_{\alpha,\mathcal{M}}=\mathcal{O}+\nabla^{F}_{\alpha Y_{\fkh}}+\mathcal{M}\,,
\\
\text{with}&&
              \mathcal{O}=\frac{-\Delta_{p}+|p|_{q}^{2}+2N_{V}-d}{2}=\frac{-g_{ij}(q)\nabla^{F}_{\frac{\partial}{\partial
              p_{i}}}\nabla^{F}_{\frac{\partial}{\partial p_{j}}}+g^{ij}(q)p_{i}p_{j}+2N_{V}-d}{2}\\
&&
  \fkh(q,p)=\frac{|p|_{q}^{2}}{2}=\frac{g^{ij}(q)p_{i}p_{j}}{2}\,,\quad
   \alpha\in \rz^{*}\,,\\
&& \nabla^{F}_{\alpha Y_{\fkh}}=\alpha g^{ij}(q)p_{j}\nabla^{F}_{e_{i}}\\
&&\mathcal{M}=\mathcal{M}_{0,j}\nabla^{F}_{\frac{\partial}{\partial
   p_{j}}}+\mathcal{M}_{0}^{j}p_{j}+\mathcal{M}_{0}\,,
\end{eqnarray*}
where $\mathcal{M}_{0,j}, \mathcal{M}_{0}^{j},\mathcal{M}_{0}$ are
$\gamma$-symbols of order $0$\,.
\end{definition}
Let us first consider the case $\fkf=Q\times \cz$  with the trivial
connection which, as we will
see below, is not a restriction. Because $\nabla_{e_{i}}^{F}\langle
p\rangle_{q}=0$ and $\nabla^{F}_{\hat{e}^{j }}\langle
p\rangle_{q}^{s}=\mathcal{O}(\langle p\rangle_{q}^{s-1})$ for
$s\in\rz$\,, the same discussion as the one following
Proposition~\ref{pr:indepWmu} about the equivariance of
$\mathcal{W}^{\mu}$-spaces, shows that conjugating by $G=\langle
p\rangle_{q}^{\pm\frac{N_{H}+N_{V}}{2}}$  transfers GKFP operators for
$g^{F}$ to  GKFP operators to $\tilde{g}^{F}$\,. Results of
\cite{BiLe}\cite{Leb1}\cite{Leb2} are formulated with the metric
$\tilde{g}^{F}$\,.  By working with those weighted metrics $g^{F}$ and
$\tilde{g}^{F}$\,, the definition of GKFP operators is the same when
$\nabla^{F}_{\alpha Y_{\fkh}}$ is replaced by the Lie derivative
$\mathcal{L}_{Y_{\alpha\fkh}}$ and we refer the reader to \cite{Leb1}-formula~(20) or
\cite{Nie}-(113)(114).
The term
$\mathcal{M}$ in $\mathcal{A}_{\alpha,\mathcal{M}}$ is actually a
perturbation which can be absorbed by the regularity estimates for
$\mathcal{A}_{\alpha,0}$\,. Additionally to the change of metric,
error terms due to partitions of unity which allow to localize the
analysis also appear as type $\mathcal{M}$ corrections:
\begin{itemize}
\item \textbf{Partition of unity in the $q$-variable:}  For a
  partition in unity in $q\in Q$\,, $Q$ is a closed compact manifold,
$\sum_{n=1}^{N}\chi_{n}(q)\equiv 1$\,, the comparison is given by
$$
\mathcal{A}_{\alpha,\mathcal{M}}-\sum_{n=1}^{N}\mathcal{A}_{\alpha,\mathcal{M}}\chi_{n}=\mathcal{M}_{\alpha,\mathcal{M},\chi}\,.
$$
So the analysis can be localized in a ball  $B(q_{0},\varrho)$ where
the coordinates $(q,p)$ can be used. Additionally changing
$\fkf\big|_{B(q_{0},\varrho)}$ with the connection $\nabla^{\fkf}$ by
$Q\times \cz^{d_{\fkf}}$ with the trivial connection and replacing
$g^{TQ}$ by the euclidean metric in $B(q_{0},\varrho)$\,, simply  adds a
term $\mathcal{M}_{g,\fkf}$\,. For $L^{2}$ or $\mathcal{W}^{\mu}$
estimates one rather uses a partition of unity
$\sum_{n=1}^{N}\chi_{n}^{2}(q)=1$ while comparing
$$
\|\mathcal{A}_{\alpha,\mathcal{M}}\omega\|_{\mathcal{W}^{\mu}}^{2}-\sum_{n=1}^{N}\|\mathcal{A}_{\alpha,\mathcal{M}}\chi_{n}\omega\|^{2}_{\mathcal{W}^{\mu}}
$$
but the idea is the same.
\item\textbf{Dyadic partition of unity in the $p$-variable:} After the
  localization in the $q$-variable and the reduction to the scalar
  case, a dyadic
partition of unity
$\sum_{j=0}^{\infty}\theta_{j}(p)=\chi_{0}(|p|_{q})+\sum_{j=1}^{\infty}\chi_{1}\left(2^{-j}|p|_{q}\right)\equiv
1$ is used with 
$$
\mathcal{A}_{\alpha,\mathcal{M}}-\sum_{j=0}^{\infty}\mathcal{A}_{\alpha,\mathcal{M}}\theta_{j}(p)=\mathcal{M}'_{\alpha,\mathcal{M},\chi}\,.
$$
Meanwhile using
$\sum_{j=0}^{\infty}\theta_{j}^{2}(p)=\chi_{0}^{2}(|p|_{q})+\sum_{j=1}^{\infty}\chi_{1}^{2}(2^{-j}|p|_{q})\equiv
1$\,, lead to  comparable error terms  for
$$
\|\mathcal{A}_{\alpha,\mathcal{M}}u\|^{2}_{\mathcal{W}^{\mu}}-\sum_{j=0}^{\infty}\|\mathcal{A}_{\alpha,\mathcal{M}}\theta_{j}u\|_{\mathcal{W}^{\mu}}^{2}\,.
$$
This allowed Lebeau to reduced the global subellipticity estimates to
parameter dependent local in a fixed ball or shell, and rather standard,
subelliptic estimates uniform with respect to  the small parameter
$h=2^{-j}\to 0$\,. 
\end{itemize}
Those use of partition of unity are actually the same as  the one
used for characterizing the spaces $\mathcal{W}^{\mu}(X;F)$ in terms
of standard  parameter dependent usual pseudodifferential calculus.\\
This led Lebeau in \cite{Leb2} to the following optimal results, that
we translate with our metric $g^{F}$ and our spaces
$L^{2}(X;F)$\,, $\mathcal{W}^{\mu}(X;F)$ for a GKFP operator
$\hat{A}_{\alpha,\mathcal{M}}$\,, $\alpha\neq 0$:
\begin{itemize}
\item  For $\mu\in \rz$  there exist
  two  constants $C_{\mu,\alpha,\mathcal{M}}>0$ and $C_{\alpha,\mathcal{M}}>0$\,, the latter
  independent of $\mu$\,, such that 
 the estimate
  \begin{multline}
\label{eq:maxhyp}
 \|\mathcal{O}s\|_{\mathcal{W}^{\mu}}+\|\nabla^{F}_{Y_{\fkh}}s\|_{\mathcal{W}^{\mu}}+\|s\|_{\mathcal{W}^{\mu+2/3}}+\delta_{0,\mu}\langle
 \lambda\rangle^{1/2}\|s\|_{\mathcal{W}^{\mu}}
\\
\leq C_{\mu,\alpha,\mathcal{M}}\|(C_{\alpha,\mathcal{M}}+\mathcal{A}_{\alpha,\mathcal{M}}-i\delta_{0,\mu}\lambda)s\|_{\mathcal{W}^{\mu}}\,,
  \end{multline}
holds true  for all $\lambda\in\rz$ and all $s\in
  \mathcal{S}'(X;F)$\,, satifying
  $(C_{\alpha,\mathcal{M}}+\mathcal{A}_{\alpha,\mathcal{M}}-i\delta_{0,\mu}\lambda)s
  \in \mathcal{W}^{\mu}(X;F)$\,. Remember that we do not use the
  $\lambda$-dependent $\mathcal{W}^{\mu}$-norms of \cite{Leb2} here.
\item The above constant  $C_{\alpha,\mathcal{M}}>0$ can be chosen
  such that $C_{\alpha,\mathcal{M}}\geq C_{0,\alpha,\mathcal{M}}$ and
  $C_{\alpha,\mathcal{M}}+\mathcal{A}_{\alpha,\mathcal{M}}$  with the domain 
  $D(\mathcal{A}_{\alpha,\mathcal{M}})=\left\{s\in
    L^{2}(X;F)\,,~\mathcal{A}_{\alpha,\mathcal{M}}s\in
    L^{2}(X;F)\right\}$ is maximally
  accretive  in $L^{2}(X;F)$\, with
$$
\forall s\in D(\mathcal{A}_{\alpha,\mathcal{M}})\,,\quad
\|\nabla_{p}^{F}s\|_{L^{2}}^{2}+\||p|_{q}s\|_{L^{2}}^{2}+\|u\|_{L^{2}}^{2}\leq
C_{\alpha,\mathcal{M}}\Real \langle s\,,\, (C_{\alpha,\mathcal{M}}+\mathcal{A}_{\alpha,\mathcal{M}})s\rangle\,.
$$
\item The subspaces $\mathcal{S}(X;E)$ and
  $\mathcal{C}^{\infty}_{0}(X;E)$ are dense in
  $D(\mathcal{A}_{\alpha,\mathcal{M}})$ endowed with its graph norm.
\item The adjoint $\mathcal{A}_{\alpha,\mathcal{M}}^{*}$\,, for the
  $L^{2}(X;F)$-scalar product, of
  $(\mathcal{A}_{\alpha,\mathcal{M}},D(\mathcal{A}_{\alpha,\mathcal{M}}))$ is a GKFP operator of
  the form $\mathcal{A}_{-\alpha,\mathcal{M}'}$
and has the same properties as $\mathcal{A}_{\alpha,\mathcal{M}}$\,.  
\end{itemize}
As a consequence of the maximal accretivity of
$(C_{\alpha,\mathcal{M}}+\mathcal{A}_{\alpha,\mathcal{M}},
D(\mathcal{A}_{\alpha,\mathcal{M}}))$ in $L^{2}(X;F)$ and the lower
bound
$C_{\alpha,\mathcal{M}}\|(C_{\alpha,\mathcal{M}}+\mathcal{A}_{\alpha,\mathcal{M}}-i\lambda)s\|_{L^{2}}\geq
\langle \lambda\rangle^{r}\|s\|_{L^{2}}$\,, $0<r<1$\, contained in
\eqref{eq:maxhyp} with $\mu=0$ and $r=\frac{1}{2}$ actually implies 
$$
\mathrm{Spec}(C_{\alpha,\mathcal{M}}+\mathcal{A}_{\alpha,\mathcal{M}})\subset
\left\{z\in \cz\,, \quad \Real z \geq C'_{\alpha,\mathcal{M}}|\Imag z|^{r}\right\}
$$
and the representation formula of the semigroup
$$
e^{-t\mathcal{A}_{\alpha,\mathcal{M}}}=\frac{1}{2i\pi}\int_{\gamma_{\alpha,\mathcal{M}}}\frac{e^{-tz}}{z-\mathcal{A}_{\alpha,\mathcal{M}}}~dz\,,\quad t>0\,,
$$
with $\gamma_{\alpha,\mathcal{M}}$ oriented from $+i\infty$ to
$-i\infty$ and given by
$$
\gamma_{\alpha,\mathcal{M}}=\left\{z\in \cz\,,\;\Real
  z+C_{\alpha,\mathcal{M}}\geq
  C'_{\alpha,\mathcal{M}}|\mathrm{Imag}~z|^{r}\right\}\,.
$$
This enters the class of  ``cuspidal semigroups'' for which the
 relationships
with subelliptic estimates and various
functional analytic  characterizations have been explained in
\cite{HerNi}, \cite{HeNi} and \cite{Nie}.\\
The fact that Bismut's hypoelliptic Laplacian $B^{\phi_{b}}_{\fkh}$\,,
or more exactly $2b^{2}B^{\phi_{b}}_{\fkh}=\mathcal{A}_{-b,\mathcal{M}}$\,, is
a GKFP operator actually comes from the Weitzenbock type formula given
in \cite{Bis05}
\begin{eqnarray}
\nonumber
 &&B^{\phi_{b}}_{\fkh}=\frac{1}{4b^{2}}
   \left[-\Delta_{p}+|p|_{q}^{2}-\frac{1}{2}\langle
   R^{TQ}(e_{i},e_{j})e_{k}\,,\,e_{\ell}\rangle
    e^{i}e^{j}\mathbf{i}_{\hat{e}^{k}\hat{e}^{\ell}}+2N_{V}-\dim Q\right]\\
\nonumber
&&\hspace{2cm}
-\frac{1}{2b}\Big[
\mathcal{L}_{Y_{\fkh}}+\frac{1}{2}\omega(\nabla^{\fkf},g^{\fkf})(Y_{\fkh})+\frac{1}{2}e^{i}\mathbf{i}_{\hat{e}^{j}}\nabla_{e_{i}}^{F}\omega(\nabla^{\fkf},g^{\fkf})(e_{j})
\\
\label{eq:weitzenbock}
&&
\hspace{5cm}
+\frac{1}{2}\omega(\nabla^{\fkf},g^{\fkf})(e_{i})\nabla^{F}_{\hat{e}^{i}}
\Big]\,.
\end{eqnarray}
We refer the reader to \cite{Leb1} for the detailed verification that
it is a GKFP and simply recall that the weighted metric
($\tilde{g}^{F}$ or $g^{F}$) is  convenient in the verification 
that $\mathcal{L}_{Y_{\fkh}}-\nabla^{F}_{Y_{\fkh}}$ enters in the
perturbation term like $\mathcal{M}$ in Definition~\ref{de:GKFP}.\\
Although it is simpler to work with the adjoint associated with the
usual $L^{2}(X;F)$ scalar product, example given  when the maximal accretivity is
considered, calculations which involve $d_{\fkh}$\,,
$d^{\phi_{b}}_{\fkh}$ and $B^{\phi_{b}}_{\fkh}$ are easier by using the
$\phi_{b}$ left or right-adjoint.
 The comparison of the standard
adjoint $P^{*}$ and the adjoints
$\tilde{P}$\,, $P^{\phi}$ and $P^{{}^{t}\phi}$ of a densely defined operator $(P,D(P))$ in
$L^{2}(X;F)$ was explained in Subsection~\ref{sec:globL2}. Remember
also that the $\phi_{b}$ right-adjoint is nothing but
$P^{{}^{t}\phi_{b}}=P^{\phi_{-b}}$\,, which is the $\phi_{-b}$
left-adjoint.\\
In particular the relation
$$
\forall s,s'\in \mathcal{S}(X;F)\,,\quad \langle
d^{\phi_{b}}_{\fkh}s\,,\, s'\rangle_{\phi_{b}}=\langle s\,,\, d_{\fkh}s'\rangle_{\phi_{b}}\,,
$$
leads to 
$$
\forall s,s'\in \mathcal{S}(X;F)\,,\quad
\langle B^{\phi_{b}}_{\fkh}s\,,\,s'\rangle_{\phi_{b}}
=\langle (d^{\phi_{b}}_{\fkh}d_{\fkh}+d_{\fkh}d^{\phi_{b}}_{\fkh})s\,,\,s'\rangle_{\phi_{b}}=
\langle s\,,\, B^{\phi_{-b}}_{\fkh}s'\rangle_{\phi_{b}}
$$
or
$$
(B^{\phi_{b}}_{\fkh})^{\phi_{-b}}=(B^{\phi_{b}}_{\fkh})^{{}^{t}\phi_{b}}=B^{\phi_{-b}}_{\fkh}\quad,\quad (B^{\phi_{-b}}_{\fkh})^{\phi_{b}}=B^{\phi_{b}}_{\fkh}\,.
$$
\subsection{Trace properties for local geometric Kramers-Fokker-Planck
operators}
\label{sec:tracelocalB}
Studying the existence of a trace along $X'$ is a local problem. In
order to take advantage of the local flexibility, we introduce a wider
class of GKFP-operators which have a good local behaviour. We consider
firstly the action of those operators on sections of the smooth vector
bundle $\pi_{F}:F\to X$ and we will in a second step consider their
properties when acting on sections of the restricted vector bundles
$F\big|_{\overline{X}_{\mp}}$ and sections of the piecewise
$\mathcal{C}^{\infty}$ and continuous vector bundle
$\hat{F}_{g}$\,. For the latter, a limited regularity of the
coefficients is required. 
\begin{definition}
\label{de:LGKFP} Consider the case where the metric $g=g^{TQ}$ is a
$\mathcal{C}^{\infty}$ metric {\red while the metric $g^{\fkf}$ on the flat vector bundle
$(\fkf,\nabla^{\fkf})$ piecewise $\mathcal{C}^{\infty}$ and continuous
like the metric $\hat{g}^{\fkf}$ of Definition~\ref{de:doublef}.}
A local geometric Kramers-Fokker-Planck (shortly
LGKFP) operator for the  metric $g=g^{TQ}$ is a differential 
operator  which can
be written locally above any local chart on $Q$ as:
\begin{eqnarray*}
  \mathcal{A}^{g,\kappa,\gamma}_{\alpha,\mathcal{M}}&=&(g^{-1}\kappa)^{ij}(q)p_{j}\nabla^{F}_{e_{i}}-\frac{\gamma_{ij}(q)\nabla^{F}_{\frac{\partial}{\partial
     p_{i}}}\nabla^{F}_{\frac{\partial}{\partial p_{j}}}}{2}+\mathcal{M}
\\
&=&\mathcal{A}^{g,\mathrm{Id},g}_{\alpha,0}+\mathcal{M}+\frac{\Delta_{p}^{g}-\Delta_{p}^{\gamma}}{2}+\alpha
    (g^{-1}(\kappa-\mathrm{Id}))^{ij}(q)p_{j}\nabla^{F}_{e_{i}}\,,
\\
\mathcal{A}^{g,\mathrm{Id},g}_{\alpha,0}&=&\alpha
  \nabla^{F}_{Y_{\fkh}}-\frac{\Delta_{p}^{g}}{2}=\alpha g^{ij}(q)p_{j}\nabla^{F}_{e_{i}}-\frac{g_{ij}(q)\nabla^{F}_{\frac{\partial}{\partial
  p_{i}}}\nabla^{F}_{\frac{\partial}{\partial
  p_{j}}}}{2}\,,
\\
\text{with}&&
\alpha\in\rz^{*}\quad,\quad
\mathcal{M}=\mathcal{M}_{j}(q,p)\nabla^{F}_{\frac{\partial}{\partial
    p_{j}}}+\mathcal{M}_{0}(q,p)\,,\\
\text{and}&& \mathcal{M}_{j}\,, \nabla^{F}_{\frac{\partial}{\partial
   p_{k}}}\mathcal{M}_{j}\,,  \mathcal{M}_{0}\in L^{\infty}_{loc}(X;L(F,F))\,.
\end{eqnarray*}
An element $\mathcal{M}$ like above will be called a locally admissible
perturbation. Admissible metrics $\gamma$ are Lipschitz continuous
metric $\gamma\in W^{1,\infty}(Q;T^{*}Q\odot T^{*}Q)$
 and admissible factors $\kappa$ belong to
 $\mathcal{C}^{\infty}(\overline{Q}_{-};L(T^{*}Q))\cap
 \mathcal{C}^{\infty}(\overline{Q}_{+};L(T^{*}Q))\cap
 \mathcal{C}^{0}(Q;L(T^{*}Q))$\,.
Both satisfy  $\|\gamma\|_{W^{1,\infty}}+\|\kappa\|_{W^{1,\infty}}\leq
R$ and
$\|\gamma-g\|_{L^{\infty}}+\|\kappa-\mathrm{Id}\|_{L^{\infty}}<\delta_{R,\alpha,g}$
with $\delta_{R,\alpha,g}>0$ small enough.
\end{definition}
{\red
The condition 
$\nabla^{F}_{\frac{\partial}{\partial p_{k}}}\mathcal{M}_{j}\in L^{\infty}_{loc}(X;L(F,F))$\,, implies
$$
\mathcal{M}=\mathcal{M}_{j}(q,p)\nabla^{F}_{\frac{\partial}{\partial
    p_{j}}}+\mathcal{M}_{0}(q,p)=
\nabla^{F}_{\frac{\partial}{\partial
    p_{j}}}\circ(\mathcal{M}_{j}(q,p)\times)
+\underbrace{\mathcal{M}_{0}(q,p)+(\nabla^{F}_{\frac{\partial}{\partial
      p_{j}}}\mathcal{M}_{j})(q,p)}_{\in L^{\infty}_{loc}(X;L(F,F))}\,.
$$
}
For any open set $U\subset X$\,, a LGKFP operator
$\mathcal{A}^{g,\kappa,\gamma}_{\alpha,\mathcal{M}}:\mathcal{C}^{\infty}_{0}(U;F)\to
L^{2}_{comp}(U;F)\subset \mathcal{D}'(U;F)$ has a formal
adjoint from $\mathcal{C}_{0}^{\infty}(U;F)\to
\mathcal{D}'(U;F)$  for the usual  $L^{2}(U;F)$ scalar product,
which is itself a LGKFP operator $\mathcal{A}^{g,\kappa',\gamma}_{-\alpha,\mathcal{M}'}$ with
$\alpha$ changed into $-\alpha$\,.  The map $\kappa'$ is nothing but
$\red{}^{t}(g^{-1}\kappa g)$\,. The difference between
$\mathcal{M}'$ and the formal adjoint $\mathcal{M}^{*}$ is due to:
\begin{itemize}
\item $\nabla^{F'}_{Y}=\nabla^{F}_{Y}-\omega(\fkf,g^{\fkf})(\pi_{X,*}Y)$
  where the terms $\omega(\fkf,g^{\fkf})(\frac{\partial}{\partial
    q^{i}})$ belong to $L^{\infty}(Q;\rz)$\,;
\item the
fact that $\nabla^{E}$ is not exactly the Levi-Civita connection for
the metric $g^{E}$ which includes the weight $\langle
p\rangle_{q}^{N_{V}-N_{H}}\pi_{X}^{*}(g^{\Lambda T^{*}Q}\otimes
g^{\Lambda TQ})$ (remember $e_{i}\langle p\rangle_{q}=0$ and
$\hat{e}^{j}\langle p\rangle_{q}^{r}=\mathcal{O}(\langle
p\rangle_{q}^{r-1})$)\,;
\item the derivatives $\nabla^{F}_{\frac{\partial}{\partial
      p_{j}}}\mathcal{M}_{j}\in L^{\infty}_{loc}(X;L(F,F))$\,;
\item{\red  the derivatives with respecto to $q^{i}$ of $\kappa$ which belong to
  $L^{\infty}(Q;\rz)$\,.}
\end{itemize}
{\red When the connection $\nabla^{E}$ (locally the flat vector bundle $\pi_{X}^{*}(\fkf\big|_{U})$
can be trivialized as $\pi_{X}^{-1}(U)\times \cz^{d_{f}}$) is replaced by a
 smooth connection $\tilde{\nabla}$ given by
$$
\tilde{\nabla}_{e_{i}}e^{\ell}=-\tilde{\Gamma}^{\ell}_{ik}(q)e^{k}\quad,\quad
\tilde{\nabla}_{e_{i}}\hat{e}_{j}=\tilde{\Gamma}^{\ell}_{ij}(q)\hat{e}_{k}\quad \tilde{\nabla}_{\hat{e}^{j}}e^{\ell}=\tilde{\nabla}_{\hat{e}^{j}}\hat{e}_{k}=0\,,
$$
with $\tilde{\Gamma}_{ik}^{\ell}\in L^{\infty}(Q;\rz)$\,,
then 
\begin{equation}
  \label{eq:LGKFPtriv}
\alpha (g\kappa)^{ij}(q)p_{j}\tilde{\nabla}_{e_{i}}-
\frac{\gamma_{ij}(q)
\tilde{\nabla}_{\frac{\partial}{\partial p_{i}}}
\tilde{\nabla}_{\frac{\partial}{\partial p_{j}}}}{2}
\end{equation}
is again a LGKFP operator for the metric $g^{TQ}$\,.
}
\begin{definition}
  \label{de:ElocB} Let  $\alpha\in \rz^{*}$ and
$\Omega=X$ or $\Omega=X_{\mp}$\,.
Let $\mathcal{A}^{g,\kappa,\gamma}_{\alpha,\mathcal{M}}$ be a LGKFP operator for the
metric $g^{TQ}$\,. The space
  $\mathcal{E}_{loc}(\mathcal{A}^{g,\kappa,\gamma}_{\alpha,\mathcal{M}},F\big|_{\overline{\Omega}})$ is defined according
  to Definition~\ref{de:calE} by
$$
\mathcal{E}_{loc}(\mathcal{A}^{g,\kappa,\gamma}_{\alpha,\mathcal{M}},F\big|_{\overline{\Omega}})=\left\{s\in
  L^{2}_{loc}(\overline{\Omega};F)\,, \quad \mathcal{A}^{g,\kappa,\gamma}_{\alpha,\mathcal{M}}s\in L^{2}_{loc}(\overline{\Omega};F)\right\}\,.
$$
\end{definition}
The topology of those spaces
$\mathcal{E}_{loc}(\mathcal{A}^{g,\kappa,\gamma}_{\alpha,\mathcal{M}};F\big|_{\overline{\Omega}})$ can be
given by the seminorms
$p_{\chi}(s)=\|\chi(\fkh)s\|_{L^{2}}+\|\chi(\fkh)\mathcal{A}^{g,\kappa,\gamma}_{\alpha,\mathcal{M}}s\|_{L^{2}}$\,.
\\
{\red
Those spaces are local on $\pi_{X}(\overline{\Omega})=Q$ or
$\overline{Q}_{\mp}$ according to the next lemma. This justifies the
writing of $\mathcal{A}_{\alpha,\mathcal{M}}^{g,\kappa,\gamma}$ with
the local frame of vector fields
$(e_{i},\hat{e}^{j}=\frac{\partial}{\partial p_{j}})$ in
Definition~\ref{de:LGKFP}.
\begin{lemma}
\label{le:ElocAloc} For any finite smooth partition of unity
$\sum_{j=1}^{J}\chi_{j}(q)\equiv 1$\,, $s$ belongs to
$\mathcal{E}_{loc}(\mathcal{A}^{g,\kappa,\gamma}_{\alpha,\mathcal{M}},F\big|_{\overline{\Omega}})$
if and only if $\chi_{j}s\in
\mathcal{E}_{loc}(\mathcal{A}^{g,\kappa,\gamma}_{\alpha,\mathcal{M}},F\big|_{\overline{\Omega}})$
for all $j=1,\ldots,I$\,.
\end{lemma}
\begin{proof}
When $\sum_{j=1}^{J}\chi_{j}(q)\equiv 1$ the difference
$
\mathcal{A}^{g,\kappa,\gamma}_{\alpha,\mathcal{M}}-\sum_{j=1}^{J}\mathcal{A}^{g,\kappa,\gamma}_{\alpha,\mathcal{M}}\chi_{j}(q)$
written locally as $\sum_{j=1}^{J}(g\kappa)^{ik}(q)p_{k}(\partial_{q^{i}}\chi_{j})(q)$
is a continuous endomorphism of $L^{2}_{loc}(\overline{\Omega};F\big|_{\overline{\Omega}})$\,. This yields the
equivalence
$$
\left(s\in \mathcal{E}_{loc}(\mathcal{A}^{g,\kappa,\gamma}_{\alpha,\mathcal{M}},F\big|_{\overline{\Omega}})\right)
\Leftrightarrow
\left(\forall j\in \left\{1,\ldots,J\right\}\,, \chi_{j}s\in \mathcal{E}_{loc}(\mathcal{A}^{g,\kappa,\gamma}_{\alpha,\mathcal{M}},F\big|_{\overline{\Omega}})\right)\,.
$$ 
\end{proof}
Other properties of the spaces
$\mathcal{E}_{loc}(\mathcal{A}^{g,\kappa,\gamma}_{\alpha,\mathcal{M}};F\big|_{\Omega})$
will be given in various cases.\\}
We first consider the smooth case without boundary $\Omega=X$\,.\\
{\red
We start with a lemma about  perturbed GKFP operators, which are of
course specific Local Geometric Kramers-Fokker-Planck operators.
\begin{lemma}
\label{le:pertGKFP} Consider the operator
$$
\mathcal{A}^{g,\mathrm{Id},\gamma}_{\alpha,\frac{|p|_{q}^{2}}{2}}=\mathcal{A}_{\alpha,0}-N_{V}+d/2+\frac{\Delta_{p}^{g}-\Delta_{p}^{\gamma}}{2}=\mathcal{A}_{\alpha,0}-N_{V}+d/2
+\frac{(g^{ij}(q)-\gamma^{ij}(q))\partial_{p_{i}}\partial_{p_{j}}}{2}
$$
according to Definition~\ref{de:LGKFP} and Definition~\ref{de:GKFP}
with $\alpha\neq 0$\,,  $\|\gamma-g\|_{L^{\infty}}\leq
\delta_{g,\alpha}$\,.\\
For $\delta_{g,\alpha}>0$ small enough,
$\mathcal{A}^{g,\mathrm{Id},\gamma}_{\alpha, \frac{|p|_{q}^{2}}{2}}$
(resp. its formal adjoint) is a relatively bounded perturbation in
$L^{2}(X;F)$ of
$\mathcal{A}_{\alpha,0}$ with domain
$D(\mathcal{A}_{\alpha,0})=\left\{s\in L^{2}(X;F),\;
  \mathcal{A}_{\alpha,0}s\in L^{2}(X;F)\right\}$ (resp of the adjoint
$\mathcal{A}_{\alpha,0}^{*}=\mathcal{A}_{-\alpha,\mathcal{M}'}$)\,. 
This operator with domain
$$
D(\mathcal{A}_{\alpha,\frac{|p|^{2}}{2}}^{g,\mathrm{Id},\gamma})
=D(\mathcal{A}_{\alpha,0})=\left\{s\in L^{2}(X;F),\;
  \mathcal{A}_{\alpha,\frac{|p|_{q}^{2}}{2}}^{g,\mathrm{Id},\gamma}s\in
L^{2}(X;F)\right\}\subset \mathcal{W}^{2/3}(X;F)
$$
is closed. Its adjoint has the domain
$$
D(\mathcal{A}_{-\alpha,0})=\left\{s\in L^{2}(X;F),\;  (\mathcal{A}_{\alpha,\frac{|p|_{q}^{2}}{2}}^{g,\mathrm{Id},\gamma})^{*}s\in
L^{2}(X;F)\right\}\subset \mathcal{W}^{2/3}(X;F)\,.
$$ 
By writing shortly
 $\mathcal{A}=\mathcal{A}_{\alpha,\frac{|p|_{q}^{2}}{2}}$
 or $\mathcal{A}=(\mathcal{A}_{\alpha,\frac{|p|_{q}^{2}}{2}}^{g,\mathrm{Id},\gamma})^{*}$
 and by recalling
 $\mathcal{O}=\frac{-\Delta_{p}+|p|_{q}^{2}+2N_{V}-d}{2}$\,, there
 exists a constant $C_{\alpha,g}>0$ such that
$(C_{\alpha,g}+\mathcal{A}):D(\mathcal{A})\to L^{2}(X;F)$ 
is invertible and
$$
(1+\mathcal{O})^{t}(C_{\alpha,g}+\mathcal{A})^{-1}(1+\mathcal{O})^{1-t}\in \mathcal{L}(L^{2}(X;F))
$$
is bounded in $L^{2}(X;F)$ for all
for all $t\in [0,1]$\,.
\end{lemma}
\begin{proof}
Because
$\mathcal{A}_{\alpha,0}$ is a GKFP operator, the global subelliptic
estimate \eqref{eq:maxhyp} says
$$
\forall s\in D(\mathcal{A}_{\alpha,0})\,,\quad
\|\mathcal{O}s\|_{L^{2}}+\|s\|_{\mathcal{W}^{2/3}}\leq C_{1,\alpha,g}\|(C_{1,\alpha,g}+\mathcal{A}_{\alpha,0})s\|_{L^{2}}
$$
with 
$$
D(\mathcal{A}_{\alpha,0})=\left\{s\in L^{2}(X;F)\,,\quad
  \mathcal{A}_{\alpha,0}s\in L^{2}(X;F)\right\}\,.
$$
The same holds for its  adjoint
$\mathcal{A}_{-\alpha,\mathcal{M}'}$ with the same constant
$C_{1,\alpha,g}$ and $\mathcal{C}^{\infty}_{0}(X;F)$ is a core for
both closed operators.\\
Because 
$$
\|\frac{\Delta_{p}^{g}-\Delta_{p}^{\gamma}}{2}s\|_{L^{2}}\leq \delta_{\alpha,g}\|\mathcal{O}s\|_{L^{2}}\,,
$$
it suffices to choose $\delta_{\alpha,g}>0$ small enough such that
$\mathcal{A}_{\alpha,0}^{g,\mathrm{Id},\gamma}$ and its formal adjoint
are respectively 
bounded perturbations of $(\mathcal{A}_{\alpha,0},D(\mathcal{A}_{\alpha,0})$
and its adjoint with bound $\leq \frac{1}{2}$\,. This ensures that
$(\mathcal{A}_{\alpha,0}^{g,\mathrm{Id},\gamma},D(\mathcal{A}_{\alpha,0}))$
and
$((\mathcal{A}_{\alpha,0}^{g,\mathrm{Id},\gamma}),D(\mathcal{A}_{\alpha,0}^{*}))$
are closed with the core $\mathcal{C}^{\infty}_{0}(X;F)$\,. They are
ajoint to each other. This yiels the characterization of 
$D(\mathcal{A})$ as $\left\{s\in L^{2}(X;F),\; \mathcal{A}s\in
  L^{2}(X;F)\right\}$ for
$\mathcal{A}=\mathcal{A}_{\alpha,0}^{g,\mathrm{Id},\gamma}$ and  $\mathcal{A}=\mathcal{A}_{\alpha,0}^{g,\mathrm{Id},\gamma}$\,.\\
Additionally there exists a constant $C_{\alpha,g}\geq 0$ such that
$$
\|\mathcal{O}s\|_{L^{2}}+\|s\|_{\mathcal{W}^{2/3}}\leq C_{\alpha,g}\|(C_{\alpha,g}+\mathcal{A})s\|_{L^{2}}\,.
$$
Hence
$(1+\mathcal{O})(C_{\alpha,g}+\mathcal{A})^{-1}$ and its adjoint
$(1+\mathcal{A}^{*})^{-1}(1+\mathcal{O})$ are bounded operators in $L^{2}(X;F)$\,. The
general result for $t\in[0,1]$ follows by interpolation.
\end{proof}
\begin{lemma}
  \label{le:kappachi} When $\kappa$ is the map of
  Definition~\ref{de:LGKFP} and $\sum_{j=1}^{J}\chi_{j}(q)\equiv 1$ is
  a smooth partition of unity on $Q$ subordinate to a chart atlas, the linear
  map $\kappa_{\chi}$ defined by
$$
\kappa_{\chi}(\sum_{j=1}^{J}\chi_{j}(q)s_{I}^{J}(q,p)e^{I}\hat{e}_{J})=
\sum_{j=1}^{J}\chi_{j}(q)s_{I}^{J}(q,\kappa(q)p) e^{I}\hat{e}_{J}
$$
is a continuous automorphism of $W^{\mu,2}_{loc}(X;F)$ for any $\mu\in[-1,1]$\,, of
$\mathcal{C}^{\infty}(\overline{X}_{-};F)\cap
\mathcal{C}^{\infty}(\overline{X}_{+};F)\cap \mathcal{C}^{0}(X;F)$ and
of $\mathcal{E}_{loc}(\Delta_{p}^{g};F)=\left\{s\in
  L^{2}_{loc}(X;F)\,,\; \Delta_{p}^{g}s\in
  L^{2}_{loc}(X;F)\right\}$\,.
It is also a continuous automorphism of the global space
$\mathcal{W}^{\mu}(X;F)$ for all $\mu\in [-1,1]$\,.
\end{lemma}
\begin{proof}
All the considered spaces are local spaces on $Q$\,. When $U$ is a
chart open set of  $Q$ we can assume $\fkf\big|_{U}=U\times
\cz^{d_{f}}$\,, 
section $s$ of $F$ supported in $U$ can be written
$s=s_{I}^{J}(q,p)e^{I}\hat{e}_{J}$ with $s_{I}^{J}(q,p)\in
\cz^{d_{f}}$\,. With the $\mathcal{C}^{\infty}$ frame
$(e^{i},\hat{e}_{j})$\,, $s\in
\mathcal{W}^{\mu}_{loc}(\pi_{X}^{-1}(U);F\big|_{U})$ (resp. 
$s\in \mathcal{C}^{\infty}(\overline{X}_{-}\cap \pi_{X}^{-1}(U);F)\cap
\mathcal{C}^{\infty}(\overline{X}_{+}\cap \pi_{X}^{-1}(U);F)\cap
\mathcal{C}^{0}(\pi_{X}^{-1}(U);F)$, resp. $s\in
\mathcal{E}_{loc}(\Delta_{p}^{g};F\big|_{\pi_{X}^{-1}(U)})$ )  if and only
if for all $I,J\subset \left\{1,\ldots,d\right\}$
\begin{eqnarray*}
  && s_{I}^{J}\in W^{\mu}_{loc}(\pi_{X}^{-1}(U);\cz^{d_{f}})\\
\text{resp.}&& s_{I}^{J}\in \mathcal{C}^{\infty}(\overline{X}_{-}\cap \pi_{X}^{-1}(U);\cz^{d_{f}})\cap
\mathcal{C}^{\infty}(\overline{X}_{+}\cap \pi_{X}^{-1}(U);\cz^{d_{f}})\cap
\mathcal{C}^{0}(\pi_{X}^{-1}(U);\cz^{d_{f}})\,,\\
\text{resp.}&& s_{I}^{J}\in \mathcal{E}_{loc}(\Delta_{p}^{g}; \cz^{d_{f}})\,.
\end{eqnarray*}
Because $\kappa\in
\mathcal{C}^{\infty}(\overline{Q}_{-};L(T^{*}Q))\cap
\mathcal{C}^{\infty}(\overline{Q}_{+};L(T^{*}Q))\cap
\mathcal{C}^{0}(Q;L(T^{*}Q))\subset W^{1,\infty}(Q;L(T^{*}Q))$\,,
those conditions are preserved by $s_{I}^{J}\mapsto
s_{I}^{J}(q,\kappa(q)p)$ for $\mu\in [-1,1]$ for the first space. The
same works for $\kappa_{\chi}^{-1}$\,.\\
For the global $\mathcal{W}^{\mu}(X;F)$ space for $\mu\in[-1,1]$\,, it
suffices to consider the cases $\mu=0$ and $\mu=1$ and to conclude by
duality and interpolation. For $\mu=0$\,, we recall that
$\|s\|_{L^{2}}^{2}$ is equivalent to
$$
N_{0}(s)^{2}=
\sum_{j,I,J}\int _{X}|\chi_{j}(q)s_{I}^{J}(q,p)|^{2}\langle p\rangle^{-|I|+|J|}~|dqdp|
$$
The estimate $N_{0}(\kappa_{\chi}s)\leq C_{\kappa,\chi}N_{0}(s)$ is then
obvious.\\
For $\mu=1$ the $\mathcal{W}^{1}(X;F)$\,,
$\|s\|_{\mathcal{W}^{1}}^{2}$ is equivalent to 
\begin{multline*}
N_{1}(s)^{2}=N_{0}(\langle p\rangle_{q}^{2}s)^{2}+
\\
\sum_{j,i_{1}, j_{1}I,J}\int_{X}\left[
|\chi_{j}(q)\partial_{q^{i_{1}}}s_{I}^{J}(q,p)|^{2}+
|\chi_{j}(q)\langle p\rangle_{q}\partial_{p_{j_{1}}}s_{I}^{J}(q,p)|^{2}
\right]\langle p\rangle_{q}^{-|I|+|J|}~|dqdp|
\end{multline*}
which leads to $N_{1}(\kappa_{\chi}s)\leq C_{\kappa,\chi}N_{1}(s)$\,.
\end{proof}
}
\begin{remark}
\red
  We do not use the natural push forward or pull back on sections of
  $F=\Lambda T^{*}X\otimes \pi_{X}^{*}(\fkf)$ of $(q,p)\mapsto
  (q,\kappa(q)p)$ because it involves its differential acting e.g on
  $e^{i}$ and $\hat{e}_{j}$ which is only $L^{\infty}_{loc}$ and not
  Lipschitz continous. Therefore $\kappa_{*}$ preserves the
 $W^{\mu}_{loc}$ regularity for $\mu\in
  [-1,1]$ only when it acts on functions. 
\end{remark}

\begin{lemma}
\label{le:ElocBX}
{\red Let us work in the framework of Definition~\ref{de:LGKFP} and
Definition~\ref{de:ElocB} with $\Omega=X$ while $g=g^{TQ}$ is a smooth
metric on $Q$\,, $\|\kappa\|_{W^{1,\infty}}+\|\gamma\|_{W^{1,\infty}}\leq R$ and
$\|\kappa-\mathrm{Id}\|_{L^{\infty}}+\|\gamma-g\|_{L^{\infty}}<\delta_{R,\alpha,g}$ with
$\delta_{\alpha,g}>0$ small enough. For a given partition of unity
$\sum_{j=1}^{J}\chi_{j}(q)\equiv 1$ subordinate to a chart atlas on
$Q$\,, $\kappa_{\chi}$ is the map defined in Lemma~\ref{le:kappachi}.
 The spaces
$\mathcal{E}_{loc}(\mathcal{A}^{g,\kappa,\gamma}_{\alpha,\mathcal{M}},F)$
satisfy the following properties:}
\begin{description}
\item[i)]
  $\mathcal{E}_{loc}(\mathcal{A}^{g,\kappa,\gamma}_{\alpha,\mathcal{M}},F)=\red
  \kappa_{\chi}\mathcal{E}_{loc}(\mathcal{A}^{g,\mathrm{Id},g}_{\alpha,0},F)\subset
  W^{2/3,2}_{loc}(X;F)\cap \mathcal{E}_{loc}(\Delta_{p}^{g};F)$
 with a continuous embedding. In
  particular the trace map $s\mapsto s\big|_{X'}$ is well-defined and
  continuous from $\mathcal{E}_{loc}(\mathcal{A}^{g,\kappa,\gamma}_{\alpha,\mathcal{M}},F)$ to
  $L^{2}_{loc}(X';F)$\,.
\item[ii)] A section $s$  belongs to
  $\mathcal{E}_{loc}(\mathcal{A}^{g,\kappa,\gamma}_{\alpha,\mathcal{M}},E)$ iff for any $\chi\in
  \mathcal{C}^{\infty}_{0}(\rz;\rz)$\,, {\red
    $\chi(\fkh)\kappa_{\chi}^{-1}s$ belongs to
    $D(\mathcal{A}^{g,\mathrm{Id},g}_{\alpha,\frac{|p|^{2}}{2}})$\,.  As a consequence
  $\mathcal{C}^{\infty}_{0}(\overline{X}_{-};F)\cap
  \mathcal{C}^{\infty}_{0}(\overline{X}_{+};F)\cap
  \mathcal{C}^{0}(X;F)$ is dense in $\mathcal{E}_{loc}(\mathcal{A}^{g,\kappa,\gamma}_{\alpha,\mathcal{M}};F)$\,.}
\item[iii)] {\red The equality
$
\mathcal{E}_{loc}(\mathcal{A}^{g,\kappa,\gamma}_{\alpha,\mathcal{M}},F)=\mathcal{E}_{loc}(\mathcal{A}^{g,\kappa,\gamma}_{\alpha_{1},\mathcal{M}},F)$
holds for any other choice of  $\alpha_{1}\in \rz^{*}$ as soon as
$$
\|\kappa-\mathrm{Id}\|_{L^{\infty}}+\|\gamma-g\|_{L^{\infty}}<\min(\delta_{R,\alpha,g},\delta_{R,\alpha_{1},g})\,.
$$
}
\item[iv)] Let $\mathcal{A}^{g,\kappa',\gamma}_{-\alpha,\mathcal{M}'}$ be the formal adjoint of
  $\mathcal{A}^{g,\kappa,\gamma}_{\alpha,\mathcal{M}}$\,,
 a section $s\in L^{2}_{loc}(X;F)$ belongs to
  $\mathcal{E}_{loc}(\mathcal{A}^{g,\kappa,\gamma}_{\alpha,\mathcal{M}},F)$ 
iff, for any compact subset
  $K\subset X$\,, there exists a constant $C_{K}>0$ such that
$$
\red \forall s'\in \mathcal{C}^{\infty}_{0}(\overline{X}_{-};F)\cap
\mathcal{C}^{\infty}_{0}(\overline{X}_{+};F)\cap \mathcal{C}^{0}(X;F)\,,~\supp s'\subset
K\,,
|\langle \mathcal{A}^{g,\kappa'\gamma}_{-\alpha,\mathcal{M}'}s'\,,\, s\rangle|\leq C_{K}\|s'\|_{L^{2}}\,.
$$
Finally the equality $\langle s'\,,\,
\mathcal{A}^{g,\kappa,\gamma}_{\alpha,\mathcal{M}}s'\rangle=\langle
\mathcal{A}^{g,\kappa',\gamma}_{-\alpha,\mathcal{M}'}s'\,,\,s\rangle$
holds for any compactly supported  $s'\in
\mathcal{E}_{loc}(\mathcal{A}^{g,\kappa',\gamma}_{-\alpha,\mathcal{M}'};F)$
when $s\in \mathcal{E}_{loc}(\mathcal{A}^{g,\kappa,\gamma}_{\alpha,\mathcal{M}})$\,.
\item[{\red v)}] When $\supp s\subset \pi_{X}^{-1}(U)$\,, where $U$ is a chart open
  set on $Q$ and $s=s_{I}^{J}(q,p)e^{I}\hat{e}_{J}$ on which
  $\fkf\big|_{U}\simeq U\times \cz^{d_{\fkf}}$\,, $s$ belongs to
  $\mathcal{E}_{loc}(\mathcal{A}^{g,\kappa,\gamma}_{\alpha,\mathcal{M}},F)$
  iff 
\red
$$
s_{I}^{J} (q,\kappa(q)^{-1}p)\quad\text{and}\quad \left(\alpha
g^{ij}(q)p_{j}e_{i}-\frac{g_{ij}(q)\partial_{p_{i}}\partial_{p_{j}}}{2}\right)[s_{I}^{J}(q,\kappa(q)^{-1}p)]
$$
belongs to $L^{2}_{loc}(\pi_{X}^{-1}(U);\cz^{d_{\fkf}})$ for all
$I,J\subset \left\{1,\ldots,d\right\}$
\end{description}
\end{lemma}
\begin{proof}
\red
The last statement \textbf{v)} is a technical point which will not be used
afterwards. We will first prove \textbf{i)} when $\kappa=\mathrm{Id}$ then
\textbf{v)} and only after the full result for~\textbf{i)}\,. The other
statements \textbf{ii) iii) iv)} will follow.\\
\noindent\textbf{i) when $\kappa=\mathrm{Id}$:}
When $s\in
\mathcal{E}_{loc}(\mathcal{A}^{g,\mathrm{Id},\gamma}_{\alpha,\mathcal{M}};F)$
and $\chi\in \mathcal{C}^{\infty}_{0}(\rz;\rz)$\,, the section
$\chi(\fkh)s$ satisfies
\begin{equation}
  \label{eq:btA}
\left[
C_{\alpha,g}+\mathcal{A}^{g,\mathrm{Id},\gamma}_{\alpha,\frac{|p|^{2}}{2}}\right]\chi(\fkh)s
=\chi(\fkh)\mathcal{A}^{g,\mathrm{Id},\gamma}_{\alpha,\mathcal{M}}s
+\left(C_{\alpha,g}+\frac{|p|^{2}}{2}-\mathcal{M}\right)\chi(\fkh)s
+\left[\mathcal{M}+\frac{\Delta_{p}^{\gamma}}{2}\,, \chi(\fkh)\right]s
\end{equation}
where our assumptions ensure that the right-hand side belong to
$(1+\mathcal{O})^{1/2}L^{2}(X;F)$\,. But Lemma~\ref{le:pertGKFP}
tells us that for $\delta_{R,\alpha}>0$ small enough and
$C_{\alpha,g}$ large enough, $\chi(\fkh)s$ belongs to
$(1+\mathcal{O})^{1/2}L^{2}(X;F)$\,. This holds for any $\chi\in
\mathcal{C}^{\infty}_{0}(\rz;\rz)$ and therefore $s$ and
$\nabla_{\frac{\partial}{\partial p_{j}}}^{F}s$ belong to
$L^{2}_{loc}(X;F)$\,.
This has two consequences:
\begin{itemize}
\item Equation \eqref{eq:btA} implies that $\chi(\fkh)s$ belongs to
  $D(\mathcal{A}^{g,\mathrm{Id},\gamma}_{\alpha,\frac{|p|^{2}}{2}})=D(\mathcal{A}^{g,\mathrm{Id},g}_{\alpha,\frac{|p|^{2}}{2}})$
  for $\delta_{\alpha,R}\leq \delta_{\alpha,g}$\,, according to
  Lemma~\ref{le:pertGKFP}. With   
$[\mathcal{A}^{g,\mathrm{Id},g},\chi(\fkh)]s=-\left[\frac{\Delta_{p}^{g}}{2}\,,
 \, \chi(\fkh)\right]s\in L^{2}(X;F)$\,, this implies $s\in
\mathcal{E}_{loc}(\mathcal{A}^{g,\mathrm{Id},g}_{\alpha,\frac{|p|^{2}}{2}};F)$\,. Thus
$\mathcal{E}_{loc}(\mathcal{A}^{g,\mathrm{Id},\gamma}_{\alpha,\mathcal{M}};F)\subset
\mathcal{E}_{loc}(\mathcal{A}^{g,\mathrm{Id},g}_{\alpha,\frac{|p|^{2}}{2}})$
and the reverse inclusion is due to the maximal hypoellipticity of $\mathcal{A}^{g,\mathrm{Id},g}_{\alpha,\frac{|p|^{2}}{2}}$\,.
\item With the maximal hypoellipticity result stated in
  Lemma~\ref{le:pertGKFP} for
  $\mathcal{A}^{g,\mathrm{Id},\gamma}_{\alpha,\frac{|p|^{2}}{2}}$ when
  $\delta_{\alpha,R}\leq \delta_{\alpha,g}$, \eqref{eq:btA} implies
  $\chi(\fkh)s\in W^{2/3,2}(X;F)$ and $\Delta_{p}^{g}\chi(\fkh)s\in
  L^{2}(X;F)$\,. This proves
  $\mathcal{E}_{loc}(\mathcal{A}^{g,\mathrm{Id},\gamma}_{\alpha,\mathcal{M}};F)\subset
  W^{2/3,2}_{loc}(X;F)\cap \mathcal{E}_{loc}(\Delta_{p}^{g},F)$\,.
\end{itemize}
\noindent\textbf{v)} Assume $s\in
\mathcal{E}_{loc}(\mathcal{A}^{g,\kappa,\gamma}_{\alpha,\mathcal{M}})$
with $\supp s\subset \pi_{X}^{-1}(U)$\,. We assume $\fkf\big|_{U}=U\times
\cz^{d_{f}}$ with the trivial connexion, which is not restrictive
because $\nabla^{\fkf}$ is flat, and we write
$s=s_{I}^{J}(q,p)e^{I}\hat{e}_{J}$ with $s_{I}^{J}(q,p)\in
L^{2}_{loc}(\pi_{X}^{-1}(U);\cz^{d_{f}})$\,. According to the
discussion around \eqref{eq:LGKFPtriv} we can replace the connexion
$\nabla^{F}$ by a trivial connexion simply by taking
$\tilde{\Gamma}_{ij}^{k}(q)=0$\,. The condition $s\in
\mathcal{E}_{loc}(\mathcal{A}^{g,\kappa,\gamma}_{\alpha,\mathcal{M}})$
simply means 
$$
[\alpha
(g\kappa)^{ij}(q)p_{j}e_{i}-\frac{\Delta_{p}^{\gamma}}{2}]s_{I}^{J}+(\mathcal{M}s)_{I}^{J}\in L^{2}_{loc}(\pi_{X}^{-1}(U);\cz^{d_{f}})\,.
$$
The change of variable $(q,p)\mapsto (q,\kappa^{-1}(q)p)$ gives
$dQ=dq$\,, $dP=(d\kappa^{-1}(q))dq p +dp$  and
$$
\frac{\partial}{\partial Q^{i}}=\frac{\partial}{\partial q^{i}}+
A_{ij}^{k}(q)p_{k}\frac{\partial}{\partial p_{j}}\quad,\quad
\frac{\partial}{\partial
  P}={}^{t}\kappa(q)^{-1}\frac{\partial}{\partial p}\,.
$$
By recalling $e_{i}=\frac{\partial}{\partial
  q^{i}}+\Gamma_{ij}^{k}(q)p_{k}\frac{\partial}{\partial p_{j}}$\,,  $\tilde{s}=s_{I}^{J}(q,\kappa^{-1}(q)p)e^{I}\hat{e}_{J}$
belongs to
$\mathcal{E}_{loc}(\mathcal{A}^{g,\mathrm{Id},\gamma'}_{\alpha,\mathcal{M}'};F\big|_{\pi_{X}^{-1}(U)})$
with $\supp \tilde{s}\subset \pi_{X}^{-1}(U)$ where
$\gamma'=\kappa(q)^{-1}\gamma {}^{t}\kappa(q)^{-1}$ satisfies
$$
\|\gamma'\|_{W^{1,\infty}}\leq (1+R)^{2}R\quad,\quad
\|\gamma'-g\|_{L^{\infty}}\leq 3)(1+R)^{2}\delta_{\alpha,R}\,.
$$
Taking $\delta_{\alpha,R}>0$ small enough such that
$3(1+R)^{2}\delta_{\alpha,R}\leq \delta_{\alpha,g}$ and the analysis of
\textbf{i)} with $\kappa=\mathrm{Id}$ implies that $\tilde{s}\in
\mathcal{E}_{loc}(\mathcal{A}^{g,\mathrm{Id},g}_{\alpha,0};F\big|_{\pi_{X}^{-1}(U)})$\,. But
this is equivalent to 
$$
\tilde{s}_{I}^{J}\quad\text{and}\quad [\alpha
g^{ij}(q)p_{j}e_{i}-\frac{\Delta_{p}^{g}}{2}]\tilde{s}_{I}^{J}\in L^{2}_{loc}(\pi_{X}^{-1};\cz^{d_{f}})
$$
for all $I,J\subset\left\{1,\ldots,d\right\}$\,, with
$\tilde{s}_{I}^{J}(q,p)=s_{I}^{J}(q,\kappa(q)^{-1}p)$\,. The reverse
way comes from the maximal hypoellipticity of $\mathcal{A}^{g,\mathrm{Id},g}_{\alpha,\frac{|p|^{2}}{2}}$\,.\\
\noindent\textbf{End of i)} Applying \textbf{v)} to all $\chi_{j}(q)s$
which is supported in a chart open set,  with Lemma~\ref{le:ElocAloc},
says that $s\in
\mathcal{E}_{loc}(\mathcal{A}^{g,\kappa,\gamma}_{\alpha,\mathcal{M}};F)$
if and only if $\kappa_{\chi}^{-1}s\in
\mathcal{E}_{loc}(\mathcal{A}^{g,\mathrm{Id},g}_{\alpha,0};F)$\,. This
proves
$\mathcal{E}_{loc}(\mathcal{A}^{g,\kappa,\gamma}_{\alpha,\mathcal{M}};F)=\kappa_{\chi}\mathcal{E}_{loc}(\mathcal{A}^{g,\mathrm{Id},g}_{\alpha,0};F)$
and the embedding in $W^{2/3,2}_{loc}(X;F)\cap
\mathcal{E}_{loc}(\eta_{p}^{g};F)$ is a consequence of Lemma~\ref{le:kappachi}.

\noindent\textbf{ii)} When $\kappa=\mathrm{Id}$\,, the commutation
$[\mathcal{A}^{g,\mathrm{Id},g}_{\alpha,\frac{|p|^{2}}{2}},\chi(\fkh)]=\left[-\frac{\Delta_{p}^{g}}{2},\chi(\fkh)\right]$
with
$\mathcal{E}_{loc}(\mathcal{A}^{g,\mathrm{Id},\gamma}_{\alpha,\mathcal{M}})=\mathcal{E}_{loc}(\mathcal{A}^{g,\mathrm{Id},g}_{\alpha,\frac{|p|^{2}}{2}};F)\subset
\mathcal{E}_{loc}(\Delta_{p}^{g};F)$ implies that $s\in
\mathcal{E}_{loc}(\mathcal{A}^{g,\mathrm{Id},\gamma}_{\alpha,\mathcal{M}};F)$
if an only if $\chi(\fkh)s\in L^{2}(X;F)$ and
$\mathcal{A}^{g,\mathrm{Id},g}_{\alpha,\frac{|p|^{2}}{2}}\in
L^{2}(X,F)$\,, which is $\chi(\fkh)s\in
D(\mathcal{A}^{g,\mathrm{Id},g}_{\alpha,\frac{|p|^{2}}{2}})$\,. It is
clear that the topology of
$\mathcal{E}_{loc}(\mathcal{A}^{g,\mathrm{Id},\gamma}_{\alpha,\mathcal{M}})$
is equivalently given by the family of seminorms
$\tilde{p}_{\chi}(s)=\|\chi(\fkh)s\|_{L^{2}}+\|\mathcal{A}^{g,\mathrm{Id},g}_{\alpha,\frac{|p|^{2}}{2}}\chi(\fkh)s\|_{L^{2}}$
and the sequence $s_{n}=\chi(\frac{\fkh}{n+1})s$ converges to $s\in
\mathcal{E}_{loc}(\mathcal{A}^{g,\mathrm{Id},\gamma}_{\alpha,\mathcal{M}})$
for this topology. Once the problem is reduced to a compactly
supported $s\in
\mathcal{E}_{loc}(\mathcal{A}^{g,\mathrm{Id},\gamma}_{\alpha,\mathcal{M}})$
we use the properties of the GKFP operator  $\mathcal{A}^{g,\mathrm{Id},g}_{\alpha,\frac{|p|^{2}}{2}}$\,.
We know in this case that $\mathcal{C}^{\infty}_{0}(X;F)$ is dense in
$D(\mathcal{A})^{g,\mathrm{Id},g}_{\alpha,\frac{|p|^{2}}{2}}$\,. Because
$$
\mathcal{C}^{\infty}_{0}(X;F)\subset
\mathcal{C}^{\infty}_{0}(\overline{X}_{-};F)\cap
\mathcal{C}^{\infty}_{0}(X;F)\cap \mathcal{C}^{0}(X;F)\subset
D(\mathcal{A}^{g,\mathrm{Id},g}_{\alpha,\frac{|p|^{2}}{2}})
$$
where the last embedding is a simple application of the jump formula
for  the first order derivative  $p_{1}\partial_{q^{1}}$ transverse to
$X'$\,. The two results hold for $\kappa=\mathrm{Id}$\,. For a general
$\kappa$ we use
$\mathcal{E}_{loc}(\mathcal{A}^{g,\kappa,\gamma}_{\alpha,\mathcal{M}})=\kappa_{\chi}\mathcal{E}_{loc}(\mathcal{A}^{g,\mathrm{Id},g}_{\alpha,\frac{|p|^{2}}{2}})$
and the fact that $\kappa_{\chi}$ is a isomorphim of
$\mathcal{C}^{\infty}_{0}(\overline{X}_{-};F)\cap
\mathcal{C}^{\infty}_{0}(\overline{X}_{+};F)\cap \mathcal{C}^{0}(X;F)$\,.\\ 
\noindent\textbf{iii)} A section $s\in
\mathcal{E}_{loc}(\mathcal{A}^{g,\kappa,\gamma}_{\alpha,\mathcal{M}},
F)$ iff $\chi(\fkh)\kappa_{\chi}s\in
D(\mathcal{A}^{g,\mathrm{Id},g}_{\alpha,\frac{|p|^{2}}{2}})$ when
$\delta_{\alpha,R}>0$ is chosen small enough. But
the maximal subelliptic estimate \eqref{eq:maxhyp} with $\mu=0$ ensures
$D(\mathcal{A}^{g,\mathrm{Id},g}_{\alpha,\frac{|p|^{2}}{2}})=D(\mathcal{A}^{g,\mathrm{Id},g}_{\alpha_{1},\frac{|p|^{2}}{2}})$
for any $\alpha,\alpha_{1}\in \rz^{*}$\,.

\noindent\textbf{iv)} As a differential operator
$\mathcal{A}^{g,\kappa,\gamma}_{\alpha,\mathcal{M}}$ is the formal adjoint of the LGKFP
operator  $\mathcal{A}^{g,\kappa',\gamma}_{-\alpha,\mathcal{M}'}$\,. So for $s\in
L^{2}_{loc}(X;F)$ the condition of \textbf{iv)} with test functions
$s'\in \mathcal{C}^{\infty}_{0}(X;F)$ 
is nothing but the weak formulation of $\mathcal{A}^{g,\kappa,\gamma}_{\alpha,\mathcal{M}}s\in
L^{2}_{loc}(X;F)$\,. Therefore the condition of \textbf{iv)} with
$s'\in \mathcal{C}^{\infty}_{0}(X;F)$ is equivalent to 
$s\in
\mathcal{E}_{loc}(\mathcal{A}^{g,\kappa,\gamma}_{\alpha,\mathcal{M}})$\,.
By assuming
$$
\left|\langle
  \mathcal{A}^{g,\kappa',\gamma}_{-\alpha,\mathcal{M}'}s'\,,\,s\rangle\right|\leq C_{K}\|s'\|_{L^{2}}
$$
for all $s'\in \mathcal{C}^{\infty}_{0}(X;F)$ with $\supp s'\subset
K$\,, the question is whether it holds for all $s'\in
\mathcal{C}^{\infty}_{0}(\overline{X}_{-};F)\cap
\mathcal{C}^{\infty}_{0}(\overline{X}_{+};F)\cap
\mathcal{C}^{0}(X;F)=:D$\,. We know $s\in
\mathcal{E}_{loc}(\mathcal{A}^{g,\kappa,\gamma'}_{\alpha,\mathcal{M}})\subset
\mathcal{E}_{loc}(\Delta_{p}^{g},F)$ while $D\subset
W^{1,2}_{comp}(X;F)$\,. Because
$\mathcal{A}^{g,\kappa',\gamma}_{-\alpha,\mathcal{M}'}$ contains only
first order derivatives in the variable $q$\,, any sequence $s_{n}'\in
\mathcal{C}^{\infty}_{0}(X;F)$ converging to $s'\in W_{comp}^{1,2}(X;F)$ with
a fixed compact support $K'\supset K$ will satisfy
$$
\lim_{n\to \infty}\langle
\mathcal{A}^{g,\kappa',\gamma}_{-\alpha,\mathcal{M}'}s_{n}'\,,\,s\rangle=\langle
\mathcal{A}^{g,\kappa',\gamma}_{-\alpha,\mathcal{M}'}s'\,,\,s\rangle
$$
By taking the limit in $\left|\langle
  \mathcal{A}^{g,\kappa',\gamma}_{-\alpha,\mathcal{M}'}s'_{n}\,,\,s\rangle\right|\leq
C_{K'}\|s'_{n}\|_{L^{2}}$\,, this proves
$$
\forall s'\in D\,, \supp s'\subset K\,, \quad |\langle
\mathcal{A}^{g,\kappa',\gamma'}_{-\alpha,\mathcal{M}'}s'\,,s\rangle|\leq C_{K'}\|s'\|_{L^{2}}\,.
$$
The last equality for a compactly supported $s'\in
\mathcal{E}_{loc}(\mathcal{A}^{g,\kappa',\gamma}_{-\alpha,\mathcal{M}'})$
is a consequence of \textbf{ii)}\,.
\end{proof}
\begin{remark}
  Actually we could have used in the proof the local maximal hypoellipticity
  of the scalar operator
  $g^{ij}(q)p_{i}\partial_{q^{j}}-\frac{\Delta_{p}}{2}$ instead of Lebeau's
  global result, which is actually derived from this local result via
  the dyadic partition of unity in $p$\,. Our writing is more
  straightorward {\red for our purpose. We had however to use
a reduction to  the scalar case via Lemma~\ref{le:ElocBX}-\textbf{v)}.}
\end{remark}
The previous result is concerned with the case without boundary with the smooth
vector bundle $\pi_{F}:F\to X$\,. It relies on the control
of terms which contain $\partial_{p_{j}}$-derivatives by the main part. We used the maximal
hypoellipticity because the $\mathcal{W}^{2/3}_{loc}(X;E)$-regularity
will be required later. It could have been done by using intregration
by parts with $2\langle u\,,\, (C+\mathcal{O}) u\rangle\geq
\sum_{j=1}^{d}\|\partial_{p_{j}}u\|_{L^{2}}^{2}+\|p_{j}u\|_{L^{2}}^{2}$\,. None of those
techniques are relevant for the case $\Omega=X_{\mp}$ or
($\Omega=X$ and $F$ replaced by $\hat{F}_{g}$) as long as boundary or
interface conditions  are not specified. \\
On one side there is a natural way to define the
$\mathcal{E}_{loc}(\mathcal{A}^{\hat{g},\mathrm{Id},\hat{g}}_{\alpha,\mathcal{M}}, \hat{F}_{g})$ by using the
isometry $\hat{\Psi}^{g,g_{0}}$ of diagram \eqref{eq:psiXgg} and this
is where the perturbative terms $\kappa-\mathrm{Id}$ and $\gamma-g$
enter in the game.  On the other side it is possible to write a trace
theorem for elements of
$\mathcal{E}_{loc}(\mathcal{A}^{g,\mathrm{Id},g}_{\alpha,\mathcal{M}}F\big|_{\overline{X}_{\mp}})\cap
\mathcal{E}_{loc}(\nabla^{F}_{\frac{\partial}{\partial p_{1}}},
F\big|_{\overline{X}_{\mp}})$ by following the approach presented in
\cite{Nie}. We check below that the two different approaches are
actually coherent and that the additional required regularity for
$\nabla^{F}_{\frac{\partial}{\partial p_{1}}}$ when $\Omega=X_{\mp}$
is actually provided by the symmetrization technique.
\begin{lemma}
\label{le:ElocBXpm}
Consider the case $\Omega=X_{\mp}$\,, where all the data for the
vector bundle $F\big|_{\overline{\Omega}}$ are
$\mathcal{C}^{\infty}(\overline{\Omega})$\,. Let $\mathcal{A}^{g,\mathrm{Id},g}_{\alpha,\mathcal{M}}$ be a LGKFP
operator with $\kappa=\mathrm{Id}$\,, $\gamma=g$ and $\alpha\in
\rz^{*}$
 and let $\mathcal{E}_{loc}(A^{g,\mathrm{Id},g}_{\alpha,\mathcal{M}}, F\big|_{\overline{\Omega}})$ be
given as in Definition~\ref{de:ElocB}. We assume additionally that the
coefficients $\mathcal{M}_{j},\mathcal{M}_{0}$ of
$\mathcal{M}=\mathcal{M}_{j}(q,p)\nabla^{F}_{\frac{\partial}{\partial
  p_{j}}}+\mathcal{M}_{0}(q,p)$ belong to
$\mathcal{C}^{\infty}(\overline{\Omega}; L(F,F))$\,.\\ 
Any element of $\mathcal{E}_{loc}(\mathcal{A}^{g,\mathrm{Id},g}_{\alpha,\mathcal{M}},F\big|_{\overline{\Omega}})\cap
\mathcal{E}_{loc}(\nabla^{F}_{\frac{\partial}{\partial
    p_{1}}},F\big|_{\overline{\Omega}})$ admits a trace along $X'=\partial
\Omega$\,.
More precisely for any $\chi\in \mathcal{C}^{\infty}_{0}(\rz;\rz)$ the
map 
$s\to \chi(\fkh)s\big|_{X'}$ is well defined an continuous from
$\mathcal{E}_{loc}(\mathcal{A}^{g,\mathrm{Id},g}_{\alpha,\mathcal{M}},F\big|_{\overline{\Omega}})\cap
\mathcal{E}_{loc}(\nabla^{F}_{\frac{\partial}{\partial p_{1}}},F\big|_{\overline{\Omega}})$
to $L^{2}(\rz,|p_{1}|dp_{1}; \red \mathcal{D}^{-2}_{T^{*}Q'})${\red where
$\mathcal{D}^{-2}_{T^{*}Q'}$ is a $W^{-2,2}$-space defined on $T^{*}Q'$\,}.\\
When   $\mathcal{A}^{g,\mathrm{Id},g}_{-\alpha,\mathcal{M}'}$ is the formal adjoint
of $\mathcal{A}^{g,\mathrm{Id},g}_{\alpha,\mathcal{M}}$ the integration by parts
$$
\langle s\,,\, \mathcal{A}^{g,\mathrm{Id},g}_{\alpha,\mathcal{M}}s'\rangle-\langle
\mathcal{A}^{g,\mathrm{Id},g}_{-\alpha,\mathcal{M}'}s\,,\, s'\rangle=\pm \int_{X'}\langle s\,,\,
s'\rangle_{g^{F}}~p_{1}|dp_{1}dq'dp'|\,.
$$
holds for all $s\in \mathcal{C}^{\infty}_{0}(\overline{X_{\mp}};F)$ and
$s'\in \mathcal{E}_{loc}(\mathcal{A}^{g,\mathrm{Id},g}_{\alpha,\mathcal{M}},F\big|_{\overline{X}_{\mp}})\cap
\mathcal{E}_{loc}(\nabla^{F}_{\frac{\partial}{\partial p_{1}}},
F\big|_{\overline{X}_{\mp}})$\,.
\end{lemma}
\begin{proof}
  We focus on $\Omega=\overline{X}_{-}$\,, $\partial X_{-}=X'$ while
  the other case $\Omega=\overline{X}_{+}$ is symmetric.
With a partition of unity $\sum_{j=1}^{J}\chi_{j}(q)\equiv 1$ on
$\overline{Q}_{-}$\,, {\red Lemma~\ref{le:ElocAloc}
gives} the equivalence
$$
\left(s\in \mathcal{E}_{loc}(\mathcal{A}^{g,\mathrm{id},g}_{\alpha,\mathcal{M}},F\big|_{\overline{X}_{-}})\right)
\Leftrightarrow
\left(\forall j\in \left\{1,\ldots,J\right\}\,, \chi_{j}s\in \mathcal{E}_{loc}(\mathcal{A}^{g,\mathrm{Id},g}_{\alpha,\mathcal{M}},F\big|_{\overline{X}_{-}})\right)\,,
$$
while the same equivalence is obvious when $\mathcal{A}^{g,\mathrm{Id},g}_{\alpha,\mathcal{M}}$ is
replaced by $\nabla^{F}_{\frac{\partial}{\partial p_{1}}}$\,.\\
Since the existence of trace is a local problem, we may assume that
$s=s_{I}^{J}e^{I}\hat{e}_{J}\in
\mathcal{E}_{loc}(\mathcal{A}^{g,\mathrm{Id},g}_{\alpha,\mathcal{M}},F\big|_{\pi_{X}^{-1}(U)})$ is supported in
$\pi_{X}^{-1}(U)$\,, $U$ open chart set in $\overline{X}_{-}$
surrounding $q_{0}\in \partial X_{-}=X'$\,, and replace the connection $\nabla^{F}$ by a
connection $\tilde{\nabla}$ which is trivial in the frame
$(e^{i},\hat{e}_{j})$ in $\pi_{X}^{-1}(U)$ with $\fkf\big|_{U}\simeq
U\times\cz^{d_{\fkf}}$\,. We did not, and we actually cannot, get
rid of the term $\mathcal{M}$ and we obtain for all $I,J\subset\left\{1,\ldots,d\right\}$
$$
s_{I}^{J}\in L^{2}_{loc}(\pi_{X}^{-1}(U);\cz^{d_{\fkf}})\quad\text{and}\quad
\left(g^{ij}(q)p_{j}e_{i}-\frac{g_{ij}(q)\partial_{p_{i}}\partial_{p_{j}}}{2}\right)s_{I}^{J}\in W^{-1,2}_{loc}(\pi_{X}^{-1}(U);\cz^{d_{\fkf}})\,,
$$
{\red The local coordinates may be chosen so that 
$g^{TQ}=(dq^{1})^{2}\oplus^{\perp}m(q^{1})$\, while we recall:}
$$
e_{i}=\frac{\partial}{\partial
  q^{i}}-\Gamma_{ij}^{k}(q)p_{k}\frac{\partial}{\partial p_{j}}\,.
$$
Meanwhile the condition $s\in
\mathcal{E}_{loc}(\nabla^{F}_{\frac{\partial}{\partial p_{1}}},F)$
says
$$
\frac{\partial s_{I}^{J}}{\partial p_{1}}\in
L^{2}_{loc}(\pi_{X}^{-1}U;\cz^{d})\quad\text{and}\quad
(1+\mathcal{O}_{1})^{1/2}s_{I}^{J}\in L^{2}_{loc}(\pi_{X}^{-1}(U),F)\,,
$$
where $\mathcal{O}_{1}=\frac{-\Delta_{p_{1}}+p_{1}^{2}-1}{2}$ is the
vertical one dimensional harmonic hamiltonian in the variable $p_{1}$\,.\\
By introducing an arbitrary cut-off $\chi(\fkh)$\,
$\chi\in \mathcal{C}^{\infty}_{0}(\rz;\rz)$\,,  and setting
$\tilde{s}_{I}^{J}=\chi(\fkh)s_{I}^{J}$\,,  we end with the essentially
scalar problem
\begin{eqnarray}
\label{eq:ts1}
  && \hspace{-1.5cm}
\tilde{s}_{I}^{J}\in
     (1+\mathcal{O}_{1})^{-1/2}L^{2}(\pi_{X}^{-1}(U);\cz^{d})\subset
  (1+\mathcal{O}_{1})^{-1/2}L^{2}(\rz_{-}\times \rz,
     |dp_{1}dq^{1}|;\fkF)\,,\\
\label{eq:ts2}
&&\hspace{-1.5cm} (1+\mathcal{O}_{1})^{-1}\left(p_{1}\partial
   _{q^{1}}+\frac{-\Delta_{p_{1}}+p_{1}^{2}+1}{2}\right)\tilde{s}_{I}^{J}\in 
 (1+\mathcal{O}_{1})^{-1/2}L^{2}(\rz_{-}\times \rz,
   |dp_{1}dq^{1}|;\fkF)\,,\\
\label{eq:ts3}
\text{with}&& \fkF=W^{-{\red 2},2}(T^{*}Q';\cz^{d_{\fkf}})\,.
\end{eqnarray}
{\red
Since we work with a compactly supported $\tilde{s}$\,, any global definition
$W^{\mu,2}(T^{*}Q';\cz^{d_{\fkf}})$ can be chosen.}\\
This is exactly  the situation studied in \cite{Nie}-Chap~2) with the
for the $\fkF$-valued, $\fkF$ a Hilbert space, section on
$T^{*}\rz_{-}$\,  for which a trace at $q^{1}=0$ is defined as
$$
\tilde{s}_{I}^{J}\big|_{q^{1}=0}\in L^{2}(\rz, |p_{1}||dp_{1}|;\fkF)\,.
$$
We deduce
$$
\tilde{s}\big|_{q^{1}\red=0}\in L^{2}(\rz,|p_{1}||dp_{1}|;
\red \mathcal{D}^{-2}_{T^{*}Q'})$$
{\red which is  the trace result if we chose $\|\omega\|^{2}_{\mathcal{D}^{-2}_{T^{*}Q'}}=\sum_{j,I,J}|(\chi_{j}(q')\omega)_{I}^{J}|^{2}_{W^{-2,2}(T^{*}Q';\cz^{df})}$.}\\
The integration by part is the standard one for $s,s'\in
\mathcal{C}^{\infty}_{0}(\overline{X}_{-};F)$  {\red where the boundary term
comes from $p_{1}\partial_{q^{1}}$\,. For a 
general $s\in \mathcal{E}_{loc}(\mathcal{A}^{g,\mathrm{Id},g}_{\alpha,\mathcal{M}},F\big|_{\overline{\Omega}})\cap
\mathcal{E}_{loc}(\nabla^{F}_{\frac{\partial}{\partial
    p_{1}}},F\big|_{\overline{\Omega}})$\,,  we replace by $\tilde{s}=\chi(\fkh)s$
with $\tilde{\chi}\equiv 1$ in a neighborhood of $\supp s'$ and
$\tilde{s}$ satisfies \eqref{eq:ts1}\eqref{eq:ts2}\eqref{eq:ts2} while
$s'\in \mathcal{C}^{\infty}_{0}(\overline{X}_{-};F)$ implies 
${s'}_{I}^{J}\in \mathcal{S}((-\infty,0]\times
\rz_{p_{1}};W^{2,2}(T^{*}Q';\cz^{d_{\fkf}})$\,. It thus suffices to apply
\cite{Nie}-Proposition~2.10 while replacing the
$W^{-2,2}(T^{*}Q')$
scalar product by the $W^{2}-W^{-2}$ duality product.}
\end{proof} 
We aim at providing a good domain definition for  Bismut's
hypoelliptic Laplacian $\hat{B}^{\phi_{b}}_{\fkh}$ acting on sections
of the piecewise $\mathcal{C}^{\infty}$ and continuous vector bundle
$\hat{F}_{g}$ associated with the metric
$\hat{g}^{TQ}=1_{Q_{-}(q)}g_{-}^{TQ}+1_{Q_{+}}(q)g_{+}^{TQ}$\,.  This
means that as a differential operator $\hat{B}^{\phi_{b}}_{\fkh}$ is
defined like $B^{\phi_{b}}_{\fkh}$ on $X_{-}$ and $X_{+}$ with the
metric $g^{F}_{-}$ on $X_{-}$ and the metric $g_{+}^{F}$ on $X_{+}$
and accordingly the energy $\fkh$ replaced by
$\hat{\fkh}=\frac{\hat{g}^{ij}(q)p_{i}p_{j}}{2}$\,. The continuity of
sections of $\hat{F}_{g}$ is expressed in the frame
$(e,\hat{e})=1_{X_{\mp}}(e_{\mp},\hat{e}_{\mp})$ with the
identification $e_{+}^{i}\big|_{\partial
  X_{+}}=e_{-}^{i}\big|_{\partial X_{-}}$ and
$\hat{e}^{j}_{+}\big|\partial_{X_{+}}=\hat{e}^{j}_{-}\big|_{\partial X_{-}}$\,.
More
generally we may consider LGKFP operators $\mathcal{A}^{\hat{g},\mathrm{Id},\hat{g}}_{\alpha,\mathcal{M}}$defined on $X_{-}\cup X_{+}$
associated with the metric $\hat{g}^{TQ}$ with the suitable interface
condition along $X'=\partial X_{-}=\partial X_{+}$\,. As a
differential operator on $X_{-}\cup X_{+}$ it is written
$$
\mathcal{A}^{\hat{g},\mathrm{Id},\hat{g}}_{\alpha,\mathcal{M}}=\alpha\nabla_{Y_{\hat{\fkh}}}^{F,\hat{g}}-\frac{\hat{g}^{ij}(q)\partial_{p_{i}}\partial_{p_{j}}}{2}+\mathcal{M}\,,
$$
where the coefficients $\mathcal{M}_{j}$\,, $j\in
\left\{0,1,\ldots,d\right\}$ of the perturbation
$\mathcal{M}=\mathcal{M}_{j}\nabla^{F,\hat{g}}_{\frac{\partial}{\partial
  p_{j}}}+\mathcal{M}_{0}$ belong to
$\mathcal{C}^{\infty}(\overline{X}_{\mp};L(F))$ and therefore to 
$L^{\infty}_{loc}(X;L(\hat{F}_{g}))$\,.
With the
coordinates $(\tilde{q},\tilde{p})$ of Definition~\ref{de:tildeqp} the
energy $\hat{h}$ satisfies
$$
2\hat{h}=\tilde{p}_{1}^{2}+m^{i'j'}(0,\tilde{q}')\tilde{p}_{i'}\tilde{p}_{j'}\,,
$$ 
while formulas, $\tilde{p'}=\psi(q^{1},q')p'$\,, \eqref{eq:tv1}\eqref{eq:tvv}\eqref{eq:tvh} imply
\begin{eqnarray}
\label{eq:Yhtilde}
  &&
     Y_{\hat{\fkh}}=\tilde{p}_{1}e_{1}+m^{i'k'}(-|\tilde{q}^{1}|,\tilde{q}')(\psi^{-1}(\tilde{q}^{1},\tilde{q}'))_{k'}^{j'}\tilde{p}_{j'}e_{i'}\\
\label{eq:eitilde}
\text{with}&& e_{1}=\frac{\partial}{\partial\tilde{q}^{1}}\quad,\quad
              e_{i'}=\frac{\partial}{\partial
              \tilde{q}^{i'}}+M_{i'j}^{k}(\tilde{q})\tilde{p}_{k}
\frac{\partial}{\partial  \tilde{p}_{j}}\\
\label{eq:ejtilde}
&&\hat{e}^{1}=\frac{\partial}{\partial \tilde{p}_{1}}\quad,\quad
   \hat{e}^{j'}=\psi^{j'}_{k'}(\tilde{q})\frac{\partial}{\partial \tilde{p}_{k'}}
\end{eqnarray}
where the coefficients $M_{i'j}^{k}$\,, $\psi_{k'}^{j'}$ and $m^{i'j'}(-|\tilde{q}^{1}|,\tilde{q})=\hat{g}^{i'j'}(\tilde{q})$ are $\mathcal{C}^{\infty}$ on
$\overline{Q}_{\mp}$ with the additional property 
$$
\hat{g}^{ij}(0,\tilde{q}')=g^{ij}_{0}(0,\tilde{q}')\quad,\quad \psi^{k'}_{j'}(0,\tilde{q'})=\delta^{k'}_{j'}\,.
$$ 
By using the isometry
$\hat{\Psi}^{g,g_{0}}:F\to \hat{F}_{g}$ which induces an isomorphism
$(\hat{\Psi}^{g,g_{0}}_{X})_{*}:L^{2}(X_{(-\varepsilon,\varepsilon)};F\big|_{X_{(-\varepsilon,\varepsilon)}},\hat{g}^{F}_{0})\to
L^{2}(X_{-\varepsilon,\varepsilon};F\big|_{X_{(-\varepsilon,\varepsilon)}},\hat{g}^{F})$ according to
Proposition~\ref{pr:L2glob}-e)\,, we deduce that
$$
(\hat{\Psi}^{g,g_{0}}_{X})_{*}^{-1}\mathcal{A}^{\hat{g},\mathrm{Id},\hat{g}}_{\alpha,\mathcal{M}}(\hat{\Psi}^{g,g_{0}}_{X^{*}})_{*}=\mathcal{A}_{\alpha,\mathcal{M}'}^{g_{0},\kappa,\gamma}
$$
is a LGKFP operator for the metric
$g_{0}^{TQ}=(dq^{1})^{2}+m_{i'j'}(0,q')dq^{i'}dq^{j'}$\,, for which we
recall that $\hat{F}_{g_{0}}=F$ is a smooth vector bundle.
More precisely the perturbations $\kappa$ and $\gamma$ satisfy
\begin{itemize}
\item $\kappa\in
  \mathcal{C}^{\infty}(\overline{Q}_{\mp};L(T^{*}Q)\big|_{Q_{\mp}})$\,,
  $\gamma\in \mathcal{C}^{\infty}(\overline{Q}_{\mp};T^{*}Q\odot
  T^{*}Q\big|_{Q_{\mp}})$\,,
\item $\kappa\in \mathcal{C}^{0}(Q;L(T^{*}Q))$\,, $\gamma\in
  \mathcal{C}^{0}(Q;T^{*}Q\odot T^{*}Q)$\,,
\item $\kappa\big|_{Q'}=\mathrm{Id}$\,, $\gamma\big|_{Q'}=g_{0}\big|_{Q'}$\,.
\end{itemize}
For trace problems along $X'=\pi_{X}^{-1}(Q')$\,, 
restricting the analysis to $X_{(-\varepsilon,\varepsilon)}=\pi_{X}^{-1}(Q_{(-\varepsilon,\varepsilon)})$
with $\varepsilon>0$ small enough is possible by using a finite
partition of unity $\sum_{j=1}^{J}\chi_{j}(q)\equiv 1$
Lemma~\ref{le:ElocBX}-\textbf{v)}\,. But on
$X_{(-\varepsilon,\varepsilon)}$ those coefficients $\kappa$ and
$\gamma$ satisfy
$\|\kappa\|_{W^{1,\infty}}+\|\gamma\|_{W^{1,\infty}}\leq R_{g}$\,,
independent of $\varepsilon>0$ and
$$
\|\kappa-\mathrm{Id}\|_{L^{\infty}}+\|\gamma-g_{0}\|_{L^{\infty}}=\mathcal{O}(\varepsilon)\,.
$$ 
This leads to the following natural
definition, {\red after recalling the definition of $\kappa_{\chi}$ in
Lemma~\ref{le:kappachi} and the identification of Lemma~\ref{le:ElocBX}-\textbf{i)}.}
\begin{definition}
\label{de:tildeE}
Let $\alpha\in
\rz^{*}$ and let $\mathcal{A}^{\hat{g},\mathrm{Id},\hat{g}}_{\alpha,\mathcal{M}}$ a LGKFP operator for the
metric $\hat{g}$\,, the space
$\mathcal{E}_{loc}(\mathcal{A}^{\hat{g},\mathrm{Id},\hat{g}}_{\alpha,\mathcal{M}},\hat{F}_{g}\big|_{X_{(-\varepsilon,\varepsilon)}})$
equals 
$$
(\hat{\Psi}^{g,g_{0}}_{X})_{*}\mathcal{E}_{loc}(\mathcal{A}^{g_{0},\kappa,\gamma}_{\alpha,\mathcal{M}'},F\big|_{X_{(-\varepsilon,\varepsilon)}})
=
(\hat{\Psi}^{g,g_{0}}_{X})_{*}\mathcal{E}_{loc}{\red \kappa_{\chi}}(\mathcal{A}^{g_{0},\mathrm{Id},g_{0}}_{\alpha,0},F\big|_{X_{(-\varepsilon,\varepsilon)}})
$$
for $\varepsilon<\varepsilon_{g,\alpha}$ and
$\varepsilon_{g,\alpha}>0$ small enough.
\end{definition}

The following statement specifies the relationship between this
definition and the previous trace results.

\begin{proposition}
\label{pr:traceAhat}
Let $\alpha\in \rz^{*}$ and
let $\mathcal{A}^{\hat{g},\mathrm{Id},\hat{g}}_{\alpha,\mathcal{M}}$ be  a LGKFP operator for the
metric $\hat{g}$\,. A section $s\in L^{2}_{loc}(X;F)$ belongs to
$\mathcal{E}_{loc}(\mathcal{A}^{\hat{g},\mathrm{Id},\hat{g}}_{\alpha,\mathcal{M}},\hat{F}_{g}\big|_{X_{(-\varepsilon,\varepsilon)}})$
with $\varepsilon<\varepsilon_{g,\alpha}$\,,
$\varepsilon_{g,\alpha}>0$ small enough, iff
one of the following condition is satisfied:
\begin{description}
\item[i)] $(\hat{\Psi}^{g,g_{0}}_{X})_{*}^{-1}s\in
  {\red \kappa_{\chi}}\mathcal{E}_{loc}(\mathcal{A}^{g_{0},\mathrm{Id},g_{0}}_{\alpha,0},F\big|_{(-\varepsilon,\varepsilon)})$\,,
  which implies
  $s\in
  W^{2/3,2}_{loc}(X_{(-\varepsilon,\varepsilon)};\hat{F}_{g}\big|_{X_{(-\varepsilon,\varepsilon)}})$
  and $s\in \mathcal{E}_{loc}(\Delta_{p}^{\hat{g}};\hat{F}_{g}\big|_{X_{(-\varepsilon,\varepsilon)}})$\,;
\item[ii)] the restrictions $s_{\mp}=s\big|_{X_{\mp}}$ belong
  respectively  to
  $$
\mathcal{E}_{loc}(\mathcal{A}^{g_{\mp},\mathrm{Id},g_{\mp}}_{\alpha,\mathcal{M}},
  F\big|_{X_{\mp [0,\varepsilon)}})\cap
  \mathcal{E}_{loc}(\nabla^{g_{\mp}}_{\frac{\partial}{\partial
      p_{1}}}, F\big|_{\overline{X}_{\mp[0,\varepsilon)}})
$$
 and the traces
  $s_{-}\big|_{\partial X_{-}}$ and $s_{+}\big|_{\partial X_{+}}$ of
  Lemma~\ref{le:ElocBXpm} coincide after the identification
  $(e_{-},\hat{e}_{-})=(e_{+},\hat{e}_{+})$ along $X'=\partial
  X_{-}=\partial X_{+}$\,.
\item[iii)] If $\mathcal{A}^{\hat{g},\mathrm{Id},\hat{g}}_{-\alpha,\mathcal{M}'}$ denotes the formal
  adjoint of $\mathcal{A}^{\hat{g},\mathrm{Id},\hat{g}}_{\alpha,\mathcal{M}}$ (defined on the open
  set $X_{-}\cup X_{+}$)\,,  for all $s'\in \mathcal{C}_{0,g}(\hat{F}_{g}\big|_{X_{(-\varepsilon,\varepsilon)}})$\,, the equality 
$$
\langle s'\,,\, \mathcal{A}^{\hat{g},\mathrm{Id}\hat{g}}_{\alpha,\mathcal{M}}s\rangle=\langle
\mathcal{A}^{\hat{g};\mathrm{Id},\hat{g}}_{-\alpha,\mathcal{M}'}s'\,,\, s\rangle
$$
holds with for any compact set $K\subset X$\,, the existence of a
constant $C_{K}>0$ such that
$$
\forall s'\in \mathcal{C}_{0,g}(\hat{F}_{g}\big|_{X_{(-\varepsilon,\varepsilon)}})\,,\supp s'\subset K\,,\quad
|\langle \mathcal{A}^{\hat{g},\mathrm{Id},\hat{g}}_{-\alpha,\mathcal{M}'}s' \,,\, s\rangle|\leq C_{K}\|s'\|_{L^{2}}\,.
$$ 
\end{description}
\end{proposition}
\begin{proof}
The  statement \textbf{i)} is essentially the definition of
$\mathcal{E}_{loc}(\mathcal{A}^{\hat{g},\mathrm{Id},\hat{g}}_{\alpha,\mathcal{M}},\hat{F}_{g}\big|_{X_{(-\varepsilon,\varepsilon)}})$ with the
additional information $\red s\in
W^{2/3,2}_{loc}(X_{(-\varepsilon,\varepsilon)};\hat{F}^{g}\big|_{X_{(-\varepsilon,\varepsilon)}})\cap
\mathcal{E}_{loc}(\Delta_{p}^{\hat{g}};\hat{F}^{g}\big|_{X_{(-\varepsilon,\varepsilon)}})$\,.\\
{\red
By Lemma~\ref{le:ElocBX}-\textbf{i)}\,,
$(\hat{\Psi}^{g,g_{0}}_{X})_{*}^{-1}s\in \red \kappa_{\chi}\mathcal{E}_{loc}(A^{g_{0},\mathrm{Id},g_{0}}_{\alpha,0},)$
implies 
$$
(\hat{\Psi}^{g,g_{0}}_{X})_{*}^{-1}s\in
W^{2/3,2}_{loc}(X_{(-\varepsilon,\varepsilon)};F\big|_{X_{(-\varepsilon,\varepsilon)}})\cap
\mathcal{E}_{loc}(\Delta_{p}^{g_{0}};F\big|_{X_{(-\varepsilon,\varepsilon)}})\,.
$$
Since $(\hat{\Psi}^{g,g_{0}}_{X})_{*}:
\mathcal{W}^{2/3}_{loc}(X;F)\to
\mathcal{W}_{loc}^{2/3}(X;\hat{F}_{g})$  is a continuous isomorphism
by the local version of Proposition~\ref{pr:hatPsiW} with $\mathcal{W}^{2/3}_{loc}=W^{2/3}_{loc}$,
this implies $s\in
W^{2/3,2}_{loc}(X;\hat{F}_{g})$\,. Remember that $u\in
\mathcal{W}^{2/3}_{loc}(X;\hat{F}_{g})$ means
$u_{\mp}=u\big|_{X_{\mp}}\in
\mathcal{W}^{2/3}_{loc}(\overline{X}_{\mp[0,\varepsilon)};F\big|_{\mp[0,\varepsilon)})$ with the equality of
the traces $u_{-}\big|_{\partial X_{-}}=u_{+}\big|_{\partial X_{+}}$
(always with $(e_{-},\hat{e}_{-})=(e_{+},\hat{e}_{+})$ along
$X'$)\,. The   vertical regularity $s\in
\mathcal{E}_{loc}(\Delta_{p}^{\hat{g}};F\big|_{X_{(-\varepsilon,\varepsilon)}})$
is even simpler.}\\
\noindent\textbf{i) implies ii):}
The previous characterization including the
$\mathcal{W}^{2/3}_{loc}(X_{(-\varepsilon,\varepsilon)};\hat{F}_{g}\big|_{X_{(-\varepsilon,\varepsilon)}})$ regularity clearly implies
$s_{\mp}=s\big|_{X_{\mp[0,\varepsilon)}}\in \mathcal{E}_{loc}(\mathcal{A}^{g_{\mp},\mathrm{\Id},g_{\mp}}_{\alpha,\mathcal{M}},
F\big|_{\overline{X}_{\mp[0,\varepsilon)}})\cap
\mathcal{E}_{loc}(\nabla^{F,g_{\mp}}_{\frac{\partial}{\partial
    p_{1}}},F\big|_{X_{\mp[0,\varepsilon)}})$ and the equality of the traces
$s_{-}\big|_{\partial X_{-}}=s_{+}\big|_{\partial
  X_{+}}=s\big|_{\partial X'}\in L^{2}_{loc}(X';F)$\,. This ends the
proof of
$\mathbf{i)}\Rightarrow\mathbf{ii)}$\,.\\
By assuming \textbf{ii)} the integration by part of Lemma~\ref{le:ElocBXpm},
where the sum of boundary terms along $\partial X_{-}=X'=\partial
X_{+}$ vanishes, implies \textbf{iii)}\,.\\
\noindent\textbf{iii) implies i):} Let $\mathcal{A}^{\hat{g},\mathrm{Id},\hat{g}}_{-\alpha,\mathcal{M}'}$
be the formal adjoint of $\mathcal{A}^{\hat{g},\mathrm{Id},\hat{g}}_{\alpha,\mathcal{M}}$ defined on
$X_{(-\varepsilon,0)}\cup X_{(0,+\varepsilon)}$\,. Although
$\hat{\Psi}^{g,g_{0}}_{X}:(F,\hat{g}_{0}^{F})\to (\hat{F}_{g},\hat{g}^{F})$ is an
isometry, the isomorphism
$(\hat{\Psi}^{g,g_{0}}_{X})_{*} : L^{2}(X;F,\hat{g}_{0})\to
L^{2}(X;F,\hat{g})$ is not unitary because
$dv_{X}=|\det(\psi^{-1}(\tilde{q}))||d\tilde{q}d\tilde{p}|$ according
to \eqref{eq:dvxtilde}. However it can be made unitary by multiplying
by the piecewise $\mathcal{C}^{\infty}$ and continuous function of
$q$\,, $|\det(\psi^{-1}(q))|^{1/2}$\,. When
$(\hat{\Psi}^{g,g_{0}}_{X})_{*}$ is replaced by the unitary map
$\tilde{\Psi}^{g,g_{0}}=|\det(\psi^{-1}(q))|^{1/2}(\hat{\Psi}^{g,g_{0}}_{X})_{*}$\,,
it 
simply modifies the admissible perturbation term $\mathcal{M}$ in the
action 
by conjugation on LGKFP-operators. Therefore the operator
$\mathcal{A}^{\hat{g},\mathrm{Id},\hat{g}}_{\alpha,\mathcal{M}}$ and its formal adjoint
$\mathcal{A}^{\hat{g},\mathrm{Id},\hat{g}}_{-\alpha,\mathcal{M}'}$ can be written
$$
\mathcal{A}^{\hat{g},\mathrm{Id},\hat{g}}_{\alpha,\mathcal{M}}=\tilde{\Psi}^{g,g_{0}}\mathcal{A}^{g_{0},\kappa,\gamma}_{\alpha,\mathcal{M}_{1}}(\tilde{\Psi}^{g,g_{0}})^{-1}\quad,\quad \mathcal{A}^{\hat{g},\mathrm{Id},\hat{g}}_{-\alpha,\mathcal{M}'}=\tilde{\Psi}^{g,g_{0}}\mathcal{A}^{g_{0},\kappa',\gamma}_{-\alpha,\mathcal{M}_{1}'}(\tilde{\Psi}^{g,g_{0}})^{-1}\,,
$$
where $\mathcal{A}^{g_{0},\kappa',\gamma}_{-\alpha,\mathcal{M}_{1}'}$ is the formal
adjoint of $\mathcal{A}^{g_{0},\kappa,\gamma}_{\alpha,\mathcal{M}_{1}}$ in
$X_{-}\cup X_{+}$\,. {\red The result is just a consequence of
Lemma~\ref{le:ElocBX}-\textbf{iv)} if we notice that
$\hat{\Psi}^{g,,g_{0}}$ sends 
$\mathcal{C}^{\infty}_{0}(X_{(-\varepsilon,0]};F)\cap
\mathcal{C}^{\infty}_{0}(X_{[0,\varepsilon)};F)\cap
\mathcal{C}^{0}(X_{(-\varepsilon,\varepsilon)};F)$ to $\mathcal{C}_{0}^{g}(\hat{F}\big|_{X_{(-\varepsilon,\varepsilon)}})$\,.} 
\end{proof}
{\red We now use
$\mathcal{E}_{loc}(\mathcal{A}^{g_{0},\kappa,\gamma}_{\alpha,0},F)=\mathcal{E}_{loc}(\mathcal{A}^{g_{0},\kappa,\gamma}_{-\alpha,0};F)$
stated in Lemma~\ref{le:ElocBX}-iii) when  
$$
\|\kappa-\mathrm{Id}\|_{L^{\infty}}+\|\gamma-g\|_{L^{\infty}}< \min(\delta_{R,\alpha,g},\delta_{R,-\alpha,g})\,,
$$
and the fact that the
formal adjoint in $L^{2}(X_{\mp};F)$ of $\nabla^{F}_{Y_{\hat{\fkh}}}$ is
$-\nabla^{F'}_{Y_{\hat\fkh}}=-\nabla^{F}_{Y_{\hat\fkh}}-\omega(\hat{f},g^{\hat{f}})(\pi_{X,*}Y_{\hat{\fkh}})$
in order to prove an integration by part. When we work globally this
singular framework, this
will provide a priori upper bound of $\|\mathcal{O}^{1/2}s\|_{L^{2}}$
before proving subelliptic estimates.}
\begin{proposition}
\label{pr:IPPloc}
For a  LGKFP operator $\mathcal{A}^{\hat{g},\mathrm{Id},\hat{g}}_{\alpha,\mathcal{M}}$ the space
$\mathcal{E}_{loc}(\mathcal{A}^{\hat{g},\mathrm{Id},\hat{g}}_{\alpha,\mathcal{M}},\hat{F}_{g}\big|_{X_{(-\varepsilon,\varepsilon)}})$
equals $\mathcal{E}_{loc}(\mathcal{A}^{\hat{g},\mathrm{Id},\hat{g}}_{\alpha_{1},\mathcal{M}'},\hat{F}_{g}\big|_{X_{(-\varepsilon,\varepsilon)}})$ for
any other $\alpha_{1}\in \rz^{*}$\,,  any other admissible local
perturbation $\mathcal{M}'$ as soon as
$\varepsilon<\varepsilon_{\alpha,\alpha_{1},g}$ with
$\varepsilon_{\alpha,\alpha_{1},g}>0$ small enough.
\,.\\
With
$\mathcal{M}=\mathcal{M}_{j}(q,p)\nabla^{X,\hat{g}}_{\frac{\partial}{\partial
  p_{j}}}+\mathcal{M}_{0}(q,p)$\,, and the adjoints
$\mathcal{M}_{j}^{*}$ of $\mathcal{M}_{j}$\,,
the integration  part formula
\begin{eqnarray*}
  \Real\langle s\,,\, \mathcal{A}^{\hat{g},\mathrm{Id},\hat{g}}_{\alpha,\mathcal{M}}s\rangle
&=& \frac{1}{2}\Real \langle s\,, -\Delta^{\hat{g}}_{p} s\rangle
+\Real\langle s\,,\,
    \mathcal{M}_{j}\nabla^{F,\hat{g}}_{\frac{\partial}{\partial
    p_{j}}}s \rangle +\Real\langle s\,,\, \mathcal{M}_{0}s\rangle\\
&&\hspace{6cm}
-\frac{\alpha}{2}\Real\langle s\,,\, \omega(\hat{\fkf},\hat{g}^{\fkf})(\pi_{X,*}Y_{\hat{\fkh}})s \rangle
\\
&\geq &\sum_{j=1}^{d}\frac{1}{2}\left[\|\nabla^{F,\hat{g}}_{\frac{\partial}{\partial p_{j}}}s\|^{2}_{L^{2}}
-{\red  (C_{g}\|s\|_{L^{2}}+2\|\mathcal{M}_{j}^{*}s\|_{L^{2}})}\|\nabla^{F,\hat{g}}_{\frac{\partial}{\partial
        p_{j}}}s\|_{L^{2}}\right]+\Real\langle s\,,\, \mathcal{M}_{0}s\rangle
\\
&&\hspace{6cm}-\frac{\alpha}{2}\Real\langle s\,,\, \omega(\hat{\fkf},\hat{g}^{\fkf})(\pi_{X,*}Y_{\hat{\fkh}})s \rangle\,,
\end{eqnarray*}
holds true for any compactly supported section $s\in
\mathcal{E}_{loc}(\mathcal{A}^{\hat{g},\mathrm{Id},\hat{g}}_{\alpha,\mathcal{M}},\hat{F}_{g}|_{X_{(-\varepsilon,\varepsilon)}})$
and $\varepsilon\in (0,\varepsilon_{\alpha,g})$\,,
$\varepsilon_{\alpha,g}>0$ small enough.
\end{proposition}
\begin{remark}
  Note that the term $-\frac{\alpha}{2}\Real\langle s\,,\,
  \omega(\hat{\fkf},\hat{g}^{\fkf})(\pi_{X,*}Y_{\hat{\fkh}})s \rangle$
  is due to the fact that we used the flat and possibly non unitary
  connection $\nabla^{\fkf}$ on $\pi_{\fkf}:\fkf\to Q$\,. {\red The term
  $-C_{g}\|s\|_{L^{2}}\|\nabla_{\frac{\partial}{\partial
      p}}s\|_{L^{2}}$ comes from the fact that the adjoint of
  $\nabla^{F}_{\frac{\partial}{\partial p_{j}}}$ equals
  $-\nabla^{F}_{\frac{\partial}{\partial p_{j}}}+R_{j}$ with $R_{j}\in
  \mathcal{L}(L^{2}(X;F))$ when we use the weighted metric 
$\langle p\rangle_{q}^{-N_{H}+N_{V}}\pi_{X}^{*}(g^{\Lambda
  T^{*}Q}\otimes \Lambda TQ)$ on $E$\,.}
\end{remark}
\begin{proof}
  The first result is simply a consequence of
  $$
\red
\mathcal{E}_{loc}(\mathcal{A}^{\hat{g},\mathrm{Id},\hat{g}}_{\alpha,\mathcal{M}},\hat{F}_{g})=(\hat{\Psi}^{g,g_{0}}_{X})_{*}\kappa_{\chi}\mathcal{E}_{loc}(\mathcal{A}^{g_{0},\mathrm{Id},g_{0}}_{\alpha,0},F)=(\hat{\Psi}^{g,g_{0}}_{X})_{*}\kappa_{\chi}\mathcal{E}_{loc}(\mathcal{A}^{g_{0},\mathrm{Id},g_{0}}_{\alpha_{1},0},F)
$$
  for any $\alpha_{1}\in\rz^{*}$\,.\\
{\red The formal adjoint of
$\mathcal{A}^{\hat{g},\mathrm{Id},\hat{g}}_{\alpha,0}=\alpha\nabla^{F}_{Y_{\hat{\fkh}}}+\frac{-\Delta_{p}^{\hat{g}}}{2}$
equals
$$
\mathcal{A}^{\hat{g},\mathrm{Id},\hat{g}}_{-\alpha,\mathcal{M}'}
=-\alpha \nabla^{F}_{Y_{\hat{\fkh}}}+\frac{-\Delta_{p}}{2}+\mathcal{M}'\,,
$$
where $\mathcal{M}'$ gathers
$$-\alpha \omega(\hat{f},\hat{g}^{f})(\pi_{X,*}Y_{\hat{\fkh}})=-\alpha g^{ij}(-|q^{1}|,q')p_{j}\omega(\hat{\fkf},\hat{g}^{\fkf})(\frac{\partial}{\partial
  q^{i}})\times
$$
and other terms coming from the fact that the adjoint of
$\nabla^{F}_{\frac{\partial}{\partial p_{j}}}$ is
$-\nabla^{F}_{\frac{\partial}{\partial p_{j}}}+R_{j}$\,.
By using the map
$\hat{\Psi}^{g,g_{0}}=|\det(\psi^{-1}(q))|^{1/2}\hat{\Psi}^{g,g_{0}}_{X}$\,,
introduced in the proof of Proposition~\ref{pr:traceAhat} and which is
unitary from $L^{2}(X; F,g_{0})$ to $L^{2}(X;F,\hat{g})$\,, we get
$$
(\hat{\Psi}^{g,g_{0}})^{-1}\mathcal{A}^{\hat{g},\mathrm{Id},\hat{g}}_{\alpha,0}\hat{\Psi}^{g,g_{0}}=\mathcal{A}^{g_{0},\kappa,\gamma}_{\alpha,
\mathcal{M}_{1}}\quad\text{and}\quad
(\hat{\Psi}^{g,g_{0}})^{-1}\mathcal{A}^{\hat{g},\mathrm{Id},\hat{g}}_{-\alpha,\mathcal{M}'}\hat{\Psi}^{g,g_{0}}=\mathcal{A}^{g_{0},\kappa,\gamma}_{-\alpha,
\mathcal{M}_{1}'}\,,
$$
where $=\mathcal{A}^{g_{0},\kappa,\gamma}_{-\alpha,
\mathcal{M}_{1}'}$ is the formal adjoint of
$\mathcal{A}^{g_{0},\kappa,\gamma}_{\alpha,\mathcal{M}_{1}}$\,.
From Lemma~\ref{le:ElocBX}-\textbf{iv)}\,, we deduce that for any
compactly supported $s\in
\mathcal{E}_{loc}(\mathcal{A}^{\hat{g},\mathrm{Id},\hat{g}}_{\alpha,0};F)$
$$
\langle s\,,\,
\mathcal{A}^{\hat{g},\mathrm{Id},\hat{g}}_{\alpha,0}s\rangle=\langle \mathcal{A}^{\hat{g},\mathrm{Id},\hat{g}}_{-\alpha,\mathcal{M}'}s\,,\,s\rangle
$$
and 
$$
\Real \langle s\,,\, \mathcal{A}^{\hat{g},\mathrm{Id},\hat{g}}_{\alpha,0}s\rangle
=\frac{1}{2}\langle s\,,\,
(\mathcal{A}^{\hat{g},\mathrm{Id},\hat{g}}_{\alpha,0}+\mathcal{A}^{\hat{g},\mathrm{Id},\hat{g}}s\rangle
= \frac{1}{2}\Real \langle s\,, -\Delta^{\hat{g}}_{p} s\rangle
-\frac{\alpha}{2}\Real\langle s\,,\, \omega(\hat{\fkf},\hat{g}^{\fkf})(\pi_{X,*}Y_{\hat{\fkh}})s \rangle\,,
$$
and this ends the proof.}
\end{proof}
\subsection{Boundary conditions and closed realizations of the
  hypoelliptic Laplacian}
\label{sec:domainB}
As a differential operator $\hat{B}^{\phi_{b}}_{\hat{\fkh}}$ is
defined as the Bismut hypoelliptic Laplacian on the open set
$X_{-}\cup X_{+}=X\setminus X'$
for the metric
$\hat{g}=\hat{g}^{TQ}=1_{Q_{-}}g_{-}^{TQ}+1_{Q_{+}}g_{+}^{TQ}$\,, $b\in
\rz^{*}$ and
 the energy $\hat{\fkh}(q,p)=\frac{\hat{g}^{ij}(q)p_{i}p_{j}}{2}$\,.
Using the fact that $2b^{2}B^{\phi_{b}}_{\fkh}$ is a local
geometric Kramers-Fokker-Planck operator with $\alpha=-b$, we can define its closed
realization in $L^{2}(X;\hat{F}_{g})$\,. By mimicking the symmetry
argument used for $\overline{d}_{g,\fkh}$\,,
 and $\overline{d}^{\phi_{b}}_{g,\fkh}$\,, one deduces
 boundary conditions and a closed realization of 
$B^{\phi_{b}}_{\fkh}$ in  $L^{2}(X_{-};F)$\,. Additional
properties for both operators are specified afterwards.
\begin{proposition}
\label{pr:domhatB} 
Let $\hat{B}^{\phi_{b}}_{\fkh}$ be the Bismut hypoelliptic Laplacian
defined as a differential operator on $X_{-}\cup X_{+}=X\setminus X'$
for the metric
$\hat{g}=\hat{g}^{TQ}=1_{Q_{-}}g_{-}^{TQ}+1_{Q_{+}}g_{+}^{TQ}$\,, $b\in
\rz^{*}$ and
 the energy $\hat{\fkh}(q,p)=\frac{\hat{g}^{ij}(q)p_{i}p_{j}}{2}$\,.
In $L^{2}(X;F)$ it is defined with the domain
$$
D(\hat{B}^{\phi_{b}}_{\fkh})=\left\{s\in
  L^{2}(X;F)\cap \mathcal{E}_{loc}(\hat{B}^{\phi_{b}}_{\fkh},\hat{F}_{g}\big|_{X_{(-\varepsilon,\varepsilon)}})\,,\quad
  B^{\hat\phi}_{\alpha\hat{\mathcal{H}}}s\in L^{2}(X;\hat{F}_{g})
\right\}\,,
$$
where $\mathcal{E}_{loc}(\hat{B}^{\phi_{b}}_{\fkh},\hat{F}_{g}\big|_{X_{(-\varepsilon,\varepsilon)}})$ is
given by Definition~\ref{de:tildeE} with
$\varepsilon<\varepsilon_{g,b}$\,, $\varepsilon_{g,b}>0$ small
enough.\\
The operator
$(\hat{B}^{\phi_{b}}_{\fkh},D(\hat{B}^{\phi_{b}}_{\fkh}))$ satisfies
the following properties:
\begin{description}
\item[a)] 
With
$D(\hat{B}^{\phi_{b}}_{\fkh})\subset
\mathcal{E}_{loc}(\hat{B}^{\phi_{b}}_{\fkh},\hat{F}_{g})\subset
\mathcal{W}^{2/3}_{loc}(X;\hat{F}_{g})\red =W^{2/3,2}_{loc}(X;\hat{F}_{g})$\,, any element $s\in
D(\hat{B}^{\phi_{b}}_{\fkh})$ admits a trace in
$L^{2}_{loc}(X';\red \hat{F}^{g}\big|_{X'})$\,.
\item[b)]
The operator
$(\hat{B}^{\phi_{b}}_{\fkh},D(\hat{B}^{\phi_{b}}_{\fkh}))$
is closed with a dense domain.
\item[c)]
There exists a constant $C_{b,g}>0$ such that
the
inequality
$$
\Real \langle s\,,\,
(C_{b,g}+\hat{B}^{\phi_{b}}_{\fkh})s\rangle \geq
\frac{1}{4b^{2}}\langle s\,,\, (1+\mathcal{O})s\rangle
$$
holds for all $s\in D(\hat{B}^{\fkh}_{g})$\,. 
\item[d)]
The space $\mathcal{C}_{0,g}(\hat{F}_{g})$ is dense in
$D(\hat{B}^{\phi_{b}}_{\fkh})$ endowed with its graph norm.
\item[e)]
The ${}^{t}\phi_{b}=\phi_{-b}$ left-adjoint of
$\hat{B}^{\phi_{b}}_{\fkh}$ is nothing but
$(\hat{B}^{\phi_{-b}}_{\fkh}, D(\hat{B}^{\phi_{-b}}_{\fkh}))$\,.
\item[f)]
The operator $\hat{B}^{\phi_{b}}_{\fkh}$ commutes with
$\Sigma_{\nu}$ and 
\begin{eqnarray*}
 && D(\hat{B}^{\phi_{b}}_{\fkh})=[L^{2}_{ev}(X;E)\cap
  D(\hat{B}^{\phi_{b}}_{\fkh})]\oplus[L^{2}_{odd}(X;E)\cap
 D(\hat{B}^{\phi_{b}}_{\fkh})]\\
&&
\hat{B}^{\phi_{b}}_{\fkh}: L_{ev~odd}^{2}(X;E)\cap
   D(\hat{B}^{\phi_{b}}_{\fkh})\to L_{ev~odd}^{2}(X;E)\,.
\end{eqnarray*}
\end{description}
\end{proposition}
\begin{proof}
\noindent\textbf{a)}  The definition of
  $D(\hat{B}^{\phi_{b}}_{\fkh})\subset
  \mathcal{E}_{loc}(\hat{B}^{\phi_{b}}_{\fkh},\hat{F}_{g}\big|_{X_{(-\varepsilon,\varepsilon)}})$\,,
  while the graph norm of $s\in D(\hat{B}^{\phi_{b}}_{\fkh})$ is nothing but
  $\|s\|_{L^{2}}+\|\hat{B}^{\phi_{b}}_{\fkh}s\|_{L^{2}}$\,,
  combined with Proposition~\ref{pr:traceAhat} which says
  $\mathcal{E}_{loc}(\hat{B}^{\phi_{b}}_{\fkh},\hat{F}_{g}\big|_{X_{(-\varepsilon,\varepsilon)}})\subset
  W^{2/3,2}_{loc}(X;\hat{F}_{g})$ implies that $s\mapsto s\big|_{X'}$ is
  continuous from $D(\hat{B}^{\phi_{b}}_{\fkh})$ to
  $L^{2}_{loc}(X';E)$ (with the identification
  $(e,\hat{e})=1_{\overline{X}_{\mp}}(e_{\mp},\hat{e}_{\mp})$)\,. Therefore
  any element $s\in D(\hat{B}^{\phi_{b}}_{\fkh})$
  admits a trace in $L^{2}_{loc}(X';\red{\hat{F}_{g}}\big|_{X'})$\,.\\
\noindent\textbf{b)} Let us check that
$(\hat{B}^{\phi_{b}}_{\fkh},D(\hat{B}^{\phi_{b}}_{\fkh}))$
is closed. According to Definition~\ref{de:tildeE},  $2b^{2}(\hat{B}^{\phi_{b}}_{\fkh})$ is a LGKFP
operator
$\mathcal{A}^{\hat{g},\mathrm{Id},\hat{g}}_{\alpha,\mathcal{M}}$
with  $\alpha=-b$  with a formal adjoint
$\mathcal{A}^{\hat{g},\mathrm{Id},\hat{g}}_{b,\mathcal{M}'}$\,.
For a sequence $(u_{n})_{n\in\nz}$ in
$D(\hat{B}^{\phi_{b}}_{\fkh})$ such that
$\lim_{n\to\infty}\|u_{n}-u\|_{L^{2}}=0$ and $\lim_{n\to
  \infty}\|\hat{B}^{\phi_{b}}_{\fkh}u_{n}-v\|_{L^{2}}=0$\,,
Proposition~\ref{pr:traceAhat}-iii) after a partition of unity in
$q$\,, leads to 
$$
\forall s'\in 
\mathcal{C}_{0,g}(X;\hat{F}_{g})\,,\quad
\langle s'\,,\,
\hat{B}^{\phi_{b}}_{\fkh}u_{n}\rangle=\langle
(2b^{2})^{-1}\mathcal{A}^{\hat{g},\mathrm{Id},\hat{g}}_{+b,\mathcal{M}'}s'\,,\, u_{n}\rangle\,.
$$
The right-hand side converges to $\langle (2b^{2})^{-1}\mathcal{A}^{\hat{g},\mathrm{Id},\hat{g}}_{+b,\fkh}s'\,,\,u\rangle$ while the left-hand side
converges to $\langle s'\,,\, v\rangle$\,. We deduce
\begin{eqnarray*}
  &&\forall s'\in \mathcal{C}_{0,g}(\overline{X}_{-};\hat{F}_{g})\,,\quad
\langle s'\,,\, v\rangle=\langle
(2b^{2})^{-1}\mathcal{A}^{\hat{g},\mathrm{Id},\hat{g}}_{+b,\mathcal{M}'}s'\,,\, u\rangle\,,\\
&&
\forall s'\in \mathcal{C}_{0,g}(X;\hat{F}_{g })\,,\quad
|\langle \mathcal{A}^{\hat{g},\mathrm{Id},\hat{g}}_{+b,\mathcal{M}'}s'\,, u\rangle|\leq 
\|v\|_{L^{2}}\|s'\|_{L^{2}}\,,
\end{eqnarray*}
and Proposition~\ref{pr:traceAhat}-iii) implies $u\in
\mathcal{E}_{loc}(\hat{B}^{\phi_{b}}_{\fkh},\hat{F}_{g})$
and $\hat{B}^{\phi_{g}}_{\fkh}u=v$ in
$L^{2}_{loc}(X;F)$ while $v\in L^{2}(X;F)$\,. This proves that
$(\hat{B}^{\phi_{b}}_{\fkh},
D(\hat{B}^{\phi_{b}}_{\fkh}))$ is closed.\\
\textbf{c)}
 For a finite partition of unity
$\sum_{k=1}^{K}\theta_{k}^{2}(q)\equiv 1$ we have:
$$
\hat{B}^{\phi_{b}}_{\fkh}=\sum_{k=1}^{K}\theta_{k}(q)\hat{B}^{\phi_{b}}_{\fkh}\theta_{k}(q)
$$
and we reduce the problem to $\supp~s\subset X_{(-\varepsilon,\varepsilon)}$\,.\\
Consider the dyadic partition of unity $\sum_{k=0}^{\infty}\chi_{k}^{2}(t)=\chi_{0}^{2}(t)+\sum_{k=1}^{\infty}\chi^{2}(\frac{t}{2^{2k}})\equiv
 1$ on $\rz$
 with $\chi_{0}\in \mathcal{C}^{\infty}_{0}(\rz)$ $\chi_{0}\equiv
 1$ in a neighborhood of $0$ a $\chi\in
 \mathcal{C}^{\infty}(]r_{1},r_{2}[)$\,. Because
$
\hat{B}^{\phi_{b}}_{\fkh}$ is a GKFP operator we get
for $s\in D(\hat{B}^{\phi_{b}}_{\fkh})$:
\begin{eqnarray*}
  &&
\hat{B}^{\phi_{b}}_{\fkh}
-\sum_{k=0}^{\infty}\chi_{k}(\hat{\fkh})
\hat{B}^{\phi_{b}}_{\fkh}
\chi_{k}(\hat{\fkh})
=
-\frac{1}{4b^{2}}\sum_{k=0}^{\infty}\hat{g}_{ij}(q)(\partial_{p_{i}}\chi_{k}(\frac{\hat{\fkh}}{2^{2k}}))(\partial_{p_{j}}\chi_{k}(\frac{\hat{\fkh}}{2^{2k}}))\,,
\\
\text{with}&&\partial_{p_{\ell}}[\chi_{k}(\hat{\fkh})]=\nabla_{p}\big[\chi\big(\frac{\hat{g}^{ij}(q)p_{j}p_{j}}{2^{2k}}\big)\big]
              =\frac{\hat{g}^{\ell
              j}(q)p_{\ell}}{2^{2k}}\chi'(\frac{\hat{\fkh}}{2^{2k}})\quad\text{for}~k\geq
              1\,,\\
&&\frac{1}{2}\hat{g}_{ij}(q)(\partial_{p_{i}}\chi_{k}(\frac{\hat{\fkh}}{2^{2k}}))(\partial_{p_{j}}\chi_{k}(\frac{\hat{\fkh}}{2^{2k}}))=\frac{\hat{\fkh}}{2^{4k}}
|\chi'(\frac{\hat{\fkh}}{2^{2k}})|^{2k}=
\mathcal{O}(2^{-2k})=\mathcal{O}(\langle
              p\rangle_{q}^{-2})\,.
\end{eqnarray*}
Hence there exists $C^{1}_{b,g}>0$ such that
$$
\Real \langle s\,,\,
B^{\hat{\phi}}_{\fkh}s\rangle
\geq \left[\sum_{k=0}^{\infty}\Real \langle \chi_{k}(\hat{\fkh})s\,,\, \hat{B}^{\phi_{b}}_{\fkh}\chi_{k}(\fkh)s\rangle\right]-C^{1}_{b,g}\|s\|^{2}\,.
$$
The operator 
$2b^{2}\hat{B}^{\phi_{b}}_{\fkh}$ is a LGKFP operator for the metric
$\hat{g}$ according to Definition~\ref{de:tildeE}, we can use
Proposition~\ref{pr:IPPloc} with
$$
\mathcal{M}_{j}(q,p)=\mathcal{M}_{0,j}(q,p)\quad,\quad \mathcal{M}_{0}(q,p)=\mathcal{M}_{0}^{j}(q,p)p_{j}+\mathcal{M}_{0,0}(q,p)+\frac{\hat{g}^{ij}(q)p_{j}p_{j}}{4b^{2}}
$$
and $\mathcal{M}_{0,j}, \mathcal{M}_{0}^{j},\mathcal{M}_{0,0}$ are
uniformly bounded\,.
Proposition~\ref{pr:IPPloc} applied
 to the compactly supported  $s_{k}=\chi_{k}(\fkh)s$  of
$\mathcal{E}_{loc}(\hat{B}^{\phi_{b}}_{\fkh},\hat{F}_{g})$ leads to 
\begin{eqnarray*}
\Real \langle s_{k}\,,\,
\hat{B}^{\phi_{b}}_{\fkh}s_{k}\rangle
&\geq& 
\sum_{j=1}^{d}\frac{1}{4b^{2}}\left[
  \|\nabla^{F}_{\frac{\partial}{\partial
      p_{j}}}s_{k}\|_{L^{2}}^{2}+\|p_{j}s_{k}\|^{2}_{L^{2}}\right]
\\
&&\hspace{3cm}-C_{b,g}^{2}\|s_{k}\|_{L^{2}}\left[\|\nabla^{F}_{\frac{\partial}{\partial
      p_{j}}}s_{k}\|_{L^{2}}+\|p_{j}s_{k}\|_{L^{2}}+\|s_{k}\|_{L^{2}}\right]\\
&\geq & 
\frac{1}{6b^{2}}
\sum_{j=1}^{d}\left[
  \|\nabla^{F}_{\frac{\partial}{\partial
      p_{j}}}s_{k}\|^{2}_{L^{2}}+\|p_{j}s_{k}\|^{2}_{L^{2}}\right]
-C^{3}_{b,g}\|s_{k}\|^{2}_{L^{2}}\,,
\end{eqnarray*}
for any $\delta>0$ and some
$C_{b,g}^{2},C_{b,g}^{3}>0$\,. By putting all together and absorbing
in a similar way the error terms,  there is a constant
$C_{b,g}>0$ such that
$$
\Real \langle s\,,\,
(C_{b,g}+\hat{B}^{\phi_{b}}_{\fkh})s\rangle \geq
\frac{1}{8b^{2}}\left[\sum_{j=1}^{d}\|\nabla^{F}_{\frac{\partial}{\partial
    p_{j}}}s\|_{L^{2}}^{2}+\|p_{j}s\|_{L^{2}}^{2}+\|s\|_{L^{2}}^{2}\right]\geq
\frac{1}{4b^{2}}\langle s\,,\,(1+\mathcal{O})s\rangle\,.
$$
\noindent\textbf{d)}
Let us first consider the effect of a truncation on $s\in
D(\hat{B}^{\phi_{b}}_{\fkh})$: Take $\chi\in
\mathcal{C}^{\infty}_{0}(\rz;\rz)$\,, $\chi\equiv 1$ in a neighborhood of
$0$ and set $\chi_{n}(t)=\chi(t/n)$\,. When $s\in
D(\hat{B}^{\phi_{b}}_{\fkh})$\,,
$\chi_{n}(\hat{\fkh})s\in
D(\hat{B}^{\phi_{b}}_{\fkh})$ while
\begin{eqnarray*}
  &&
     \hat{B}^{\phi_{b}}_{\fkh}\chi_{n}(\hat{\fkh})s=\chi_{n}(\hat{\fkh})\hat{B}^{\phi_{b}}_{\fkh}-\frac{1}{2}[\Delta^{\hat{g}}_{p},\chi_{n}(\hat{\fkh})]s\\
&&
\|[\Delta^{\hat{g}}_{p},\chi_{n}(\hat{\fkh})]s\|\leq
\frac{\|s\|_{L^{2}}}{n}\left[\sum_{k=1}^{d}\|\nabla^{F}_{\frac{\partial}{\partial
      p_{k}}}s\|_{L^{2}}\right]\leq 
C_{b,g} \frac{\Real\langle s\,,\,
  (C_{b,g}+\hat{B}^{\phi_{b}}_{\fkh})s\rangle}{n}
\end{eqnarray*}
implies that $\chi_{n}(\hat{\fkh})s$ converges to $s$ in
$D(\hat{B}^{\phi_{b}}_{\fkh})$ endowed with its graph
norm.
Now for $s_{N}=\chi_{N}(\fkh)s$ the problem is reduced to the
approximation of the
compactly supported element of
$\mathcal{E}_{loc}(\hat{B}^{\phi_{b}}_{\fkh},\hat{F}_{g})$ by elements
of $\mathcal{C}_{0,g}(\hat{F}_{g})$\,. By using a partition of unity
in $q$\,, $\sum_{k=1}^{K}\theta_{k}(q)\equiv 1$ with 
$$
\left[\hat{B}^{\phi_{b}}_{\fkh},\theta_{k}(q)\right]=g^{ij}(q)p_{j}\frac{\partial
\theta_{k}}{\partial
q^{i}}
$$
the problem is reduced to a compactly supported element $s_{N}$ of
 $$
\mathcal{E}_{loc}(\hat{B}^{\phi_{b}}_{\fkh},\hat{F}_{g}\big|_{X_{(-\varepsilon,\varepsilon)}})=(\hat{\Psi}^{g,g_{0}}_{X})_{*}\mathcal{E}_{loc}(\mathcal{A}^{g_{0},\kappa,\gamma}_{-b,\mathcal{M}'},F\big|_{X_{(-\varepsilon,\varepsilon)}})\,.
$$
{\red
Then the approximation of a compactly supported element $s_{N}\in
\mathcal{E}_{loc}(\hat{B}^{\phi_{b}}_{\fkh},\hat{F}_{g}\big|_{X_{(-\varepsilon,\varepsilon)}})$
results from Lemma~\ref{le:ElocBX}-\textbf{ii)}\,.}\\
\noindent\textbf{e)} By construction the  isomorphism 
$J_{b,g}:L^{2}(X;F)\to
L^{2}(X,F)$ given by 
$$
\forall s'\in L^{2}(X;F)\,,\quad 
\langle J_{b,g}s\,,\, s'\rangle_{L^{2}}=\langle s\,,\,s'\rangle_{\phi_{-b}}
$$
and its inverse
preserve $\mathcal{C}_{0,g}(\hat{F}_{g})$ which is a core for
$(\hat{B}^{\phi_{b}}_{\fkh},D(\hat{B}^{\phi_{b}}_{\fkh})$ and its
$L^{2}$-adjoint $(B^{\phi_{b}}_{\fkh})^{*}$\,.
The $L^{2}$-adjoint is the closure of the operator defined on
$\mathcal{C}_{0,g}(\hat{F}_{g})$ by
$$
\forall s,s'\in \mathcal{C}_{0,g}(\hat{F}_{g})\,,\quad \langle s\,,\,
\hat{B}^{\phi_{b}}_{\fkh}s'\rangle=\langle
(\hat{B}^{\phi_{b}}_{\fkh})^{*}s\,,\, s' \rangle\,.
$$
This gives
$$
\forall s,s'\in \mathcal{C}_{0,g}(\hat{F}_{g})\,,\quad
\langle J_{b,g}^{-1}s\,,\,
\hat{B}^{\phi_{b}}_{\fkh}s'\rangle_{\phi_{-b}}= 
\langle J_{b,g}^{-1}(\hat{B}^{\phi_{b}}_{\fkh})^{*}s\,,\,s'\rangle_{\phi_{-b}}\,.
$$
We deduce that the $\phi_{-b}$ left-adjoint, $(\hat{B}^{\phi_{b}}_{\fkh})^{\phi_{-b}}$ of
$\hat{B}^{\phi_{b}}_{\fkh}$ satisfies
$$
\forall s\in \mathcal{C}_{0,g}(\hat{F}_{g})\,,\quad 
(\hat{B}^{\phi_{b}}_{\fkh})^{\phi_{-b}}s=J_{b,g}^{-1}(\hat{B}^{\phi_{b}}_{\fkh})^{*}J_{b,g}s\,.
$$
Taking a test function $s'\in \mathcal{C}^{\infty}_{0}(X_{-}\cup
X_{+};F)$ allows to make the integration by part for $s\in
\mathcal{C}_{0,g}(\hat{F}_{g})$\,,
$$
\langle s\,,\, \hat{B}^{\phi_{b}}s'\rangle_{\phi_{-b}}
=\langle s\,,\,
(d^{\phi_{b}}_{\hat{\fkh}}d_{\hat{\fkh}}+d_{\hat{\fkh}}d^{\phi_{b}}_{\hat{\fkh}})s'\rangle_{\phi_{-b}}
=
\langle (d^{\phi_{-b}}_{\hat{\fkh}})d_{\hat{\fkh}}+d_{\hat{\fkh}}d^{\phi_{-b}}_{\hat{\fkh}})s\,,\,s'\rangle_{\phi_{-b}}
$$
without any boundary term. This gives 
$$
(\hat{B}^{\phi_{b}}_{\fkh})^{\phi_{-b}}s=J_{b,g}^{-1}(\hat{B}^{\phi_{b}}_{\fkh})^{*}J_{b,g}s=\hat{B}^{\phi_{-b}}_{\fkh}s
\quad\text{in}~\mathcal{D}'(X_{-}\cup X_{+},F)
$$
while the left-hand side belongs to $L^{2}(X;F)$\,. With $s\in
\mathcal{C}_{0,g}(\hat{F}_{g})\subset D(\hat{B}^{\phi_{-b}}_{\fkh})$ this gives
$$
(\hat{B}^{\phi_{b}}_{\fkh})^{\phi_{-b}}s=\hat{B}^{\phi_{-b}}_{\fkh}s\,,
$$
and, because $\mathcal{C}_{0,g}(\hat{F}_{g})$ is a core for both
operators,  the $\phi_{-b}$ left-adjoint of
$(\hat{B}^{\phi_{b}}_{\fkh}, D(\hat{B}^{\phi_{b}}_{\fkh}))$ equals 
$(\hat{B}^{\phi_{-b}}_{\fkh},D(\hat{B}^{\phi_{-b}}_{\fkh}))$\,.\\
\noindent\textbf{f)}
By construction $\hat{B}^{\phi_{b}}_{\fkh}$ commutes with $\Sigma_{\nu}$
as a differential operator on $X_{-}\cup X_{+}=X\setminus X'$\,, that
is in $\mathcal{D}'(X_{-}\cup X_{+};F)$\,. Meanwhile
$\mathcal{C}_{0,g}(\hat{F}_{g})$ is left invariant by
$\Sigma_{\nu}$\,. This proves
$$
\forall s\in \mathcal{C}_{0,g}(\hat{F}_{g})\,,\quad \Sigma_{\nu}\hat{B}^{\phi_{b}}_{\fkh}s=\hat{B}^{\phi_{b}}_{\fkh}\Sigma_{\nu}s\,.
$$
Since $\mathcal{C}_{0,g}(\hat{F}_{g})$ is a core for
$(\hat{B}^{\phi_{b}}_{\fkh},D(\hat{B}^{\phi_{b}}_{\fkh}))$\,, the
equality holds for all $s\in D(\hat{B}^{\phi_{b}}_{\fkh})$\,.
\end{proof}

\begin{definition}
\label{de:barBg}
In $L^{2}(X_{-};F\big|_{X_{-}})$\,, the operator $\overline{B}^{\phi}_{g,\alpha\mathcal{H}}$ is
defined  with the domain
$$
D(\overline{B}^{\phi}_{g,\alpha\mathcal{H}})
=\left\{s\in L^{2}(X_{-};F\big|_{X_{-}})\,,\quad s_{ev}\in D(\hat{B}^{\phi_{b}}_{\fkh})\right\}\,.
$$
\end{definition}
\begin{theorem}
\label{th:domainB} For $b\in \rz^{*}$ the operator 
$(\overline{B}^{\phi_{b}}_{\fkh},D(\overline{B}^{\phi_{b}}_{\fkh}))$ satisfies
the following properties
\begin{description}
\item[a)] It is closed and a section $s\in
  D(\overline{B}^{\phi_{b}}_{\fkh})$ has trace  $s\big|_{X'}\in
  L^{2}_{loc}(X';F\big|_{X'})$\,, $X'=\partial X_{-}$\,,
 such that $\hat{S}_{\nu}s\big|_{X'}=s\big|_{X'}$ where $\hat{S}_{\nu}$
is  defined by \eqref{eq:dehSnu1}.
\item[b)] The space 
$$
\mathcal{C}^{\infty}_{0}(\overline{X}_{-};F)\cap D(\overline{B}^{\phi_{b}}_{\fkh})=
\left\{s\in \mathcal{C}^{\infty}_{0}(\overline{X}_{-};F)\,,\quad \hat{S}_{\nu}s\big|_{X}=s\big|_{X'}\right\}
$$
is dense in $D(\overline{B}^{\phi_{b}}_{\fkh})$ endowed
with its graph norm.
\item[c)] The $\phi_{-b}$ left-adjoint of
  $(\overline{B}^{\phi_{b}}_{\fkh},
  D(\overline{B}^{\phi_{b}}_{\fkh}))$ equals
  $(\overline{B}^{\phi_{-b}}_{\fkh},
  D(\overline{B}^{\phi_{-b}}_{\fkh}))$\,.
\item[d)] There  exists a constant $C_{b,g}>0$ such that
$$
\Real \langle s\,,\, (C_{b,g}+\overline{B}^{\phi_{b}}_{\fkh})s\rangle
\geq \frac{1}{4b^{2}}\langle s\,,\, (1+\mathcal{O})s\rangle
$$
holds for all $s\in D(\overline{B}^{\phi_{b}}_{\fkh})$\,.
\item[e)] The operator $(C_{b,g}+\overline{B}^{\phi_{b}}_{\fkh},
  D(\overline{B}^{\phi_{b}}_{\fkh}))$ is maximal accretive and 
 estimate
$$
\left.
  \begin{array}[c]{l}
\sum_{j=1}^{d}[\|\nabla^{F}_{\frac{\partial}{\partial
    p_{j}}}s\|_{L^{2}}+
\|p_{j}s\|_{L^{2}}]\\
+\|s\|_{\mathcal{W}^{1/3}}+\|\langle
p\rangle_{q}^{-1}s\big|_{X'}\|_{L^{2}(X',|p_{1}|dv_{X'})}\\
+\langle \lambda\rangle^{1/4}\|s\|
  \end{array}
\right\}
  \leq
  C'_{b,g}\|(\overline{B}^{\phi_{b}}_{\fkh}+C'_{b,g}+i\lambda)s\|_{L^{2}}\quad\,,\quad
$$
holds for some $C'_{b,g}>0$\,, all $\lambda\in \rz$ and all $s\in
D(\overline{B}^{\phi_{b}}_{\fkh})$\,. In particular
$(\overline{B}^{\phi_{b}}_{\fkh}, D(\overline{B}^{\phi_{b}}_{\fkh}))$
has a compact resolvent.
\item[f)] Finally the domain $D(\overline{B}^{\phi_{b}}_{\fkh})$
  equals
$$
D(\overline{B}^{\phi_{b}}_{\fkh})=
\left\{s\in L^{2}(X_{-};F)\,,\quad
  \nabla^{F}_{\frac{\partial}{\partial p_{1}}}s  \,,\,  B^{\phi_{b}}_{\fkh}s\in
  L^{2}(X_{-};F)\quad,\quad
\hat{S}_{\nu}s\big|_{X'}=s\big|_{X'}\right\}\,.
$$
\end{description}
\end{theorem}
\begin{proof}
  The statements \textbf{a),b),c),d)} are straightforward
  consequences of Proposition~\ref{pr:domhatB} and
  Definition~\ref{de:barBg} after recalling that $L^{2}(X_{-};F)\ni
  s\mapsto 2^{-1/2}s_{ev}\in L^{2}_{ev}(X;F)$ is unitary.\\
The results of \textbf{e)} are deduced from the results of \cite{Nie}
for scalar Kramers-Fokker-Planck operators. After localization in $q$
after partition of unity and a possible change of connection, one can
write
$$
B^{\phi_{b}}_{\fkh}(s_{I}^{J}e^{I}\hat{e}_{J})
= \left([g^{ij}(q)p_{j}e_{i}+\mathcal{O}]s_{I}^{J}\right)e_{I}\hat{e}^{J}+\mathcal{M}(s_{I}^{J}e^{I}\hat{e}_{J})
$$
where $\mathcal{M}$ is a global admissible perturbation (see Definition~\ref{de:GKFP}), such that
$$
\|\mathcal{M} s\|_{L^{2}}\leq C\Real \langle s\,,\, B^{\phi_{b}}_{\fkh}s\rangle\,.
$$
Meanwhile the boundary  conditions written
$$
s_{I}^{J}(0,q',p_{1},p')=\nu(-1)^{|\left\{1\right\}\cap
  I|+|\left\{1\right\}\cap J|}s_{I}^{J}(0,q',-p_{1},p')
$$
where $\nu$ can be replaced by $\pm 1$ are the ones considered in
\cite{Nie}-Theorem~1.1. Finally note that the a priori condition, 
$\nabla^{F}_{\frac{\partial}{\partial p_{j}}}s_{I}^{J}\,,
p_{j}s_{I}^{J} \in L^{2}(X_{-};\cz^{d_{\fkf}})$ is actually provided
by the integration by part \textbf{d)}. The subelliptic estimates and
the maximal accretivity proved in \cite{Nie} for scalar
Kramers-Fokker-Planck operators, which is stable under
admissible global perturbations $\mathcal{M}$ while adapting the
constants in the inequalities, are thus valid for
$(\hat{B}^{\phi_{b}}_{\fkh}, D(B^{\phi_{b}}_{\fkh}))$\,.\\
\noindent\textbf{f)} Clearly 
$$
D(\overline{B}^{\phi_{b}}_{\fkh})\subset 
\left\{s\in L^{2}(X_{-};F)\,,\quad
  \nabla^{F}_{\frac{\partial}{\partial p_{1}}}s  \,,\,  B^{\phi_{b}}_{\fkh}s\in
  L^{2}(X_{-};F)\quad,\quad
\hat{S}_{\nu}s\big|_{X'}=s\big|_{X'}\right\}\,.
$$
For the reverse inclusion, the condition 
$s\in \mathcal{E}_{loc}(B^{\phi_{b}}_{\fkh},F\big|_{X_{(-\varepsilon,0]}})\cap
\mathcal{E}_{loc}(\nabla^{F}_{\frac{\partial}{\partial p_{1}}},
F\big|_{X_{(-\varepsilon,0]}})$ and Lemma~\ref{le:ElocBXpm} imply that $s$ admits a
trace along $X'=\partial X_{-}$\,. 
{\red The condition
$\hat{S}_{\nu}s\big|_{X'}=s\big|_{X'}$ makes sense and this ensures that
$s_{ev}$ admits a trace in $L^{2}_{loc}(X';\hat{F}_{g}\big|_{X'})$\,. 
 By using  tha map $\hat{\Psi}^{g,g_{0}}_{X}$\,, we deduce that
 $\omega_{ev}=(\Psi^{g,g_{0}}_{X})s_{ev}$ admits a trace in $L^{2}_{loc}(X';F)$ while
$$
(\hat{\Psi}^{g,g_{0}}_{X})_{*}^{-1}\hat{B}^{\phi_{b}}_{\fkh}(\hat{\Psi}^{g,g_{0}}_{X})_{*}=\mathcal{A}^{g_{0},\kappa,\gamma}_{-b,\mathcal{M}'}\,.
$$
Because $\mathcal{A}^{g_{0},\kappa,\gamma}_{-b,\mathcal{M}'}$ is
contains only one derivative $\frac{\partial}{\partial q^{1}}$ with
while all the coefficients are $\mathcal{C}^{\infty}$ above
$\overline{X}_{-}$ and $\overline{X}_{+}$\,, the jump formula says
$$
\mathcal{A}^{g_{0},\mathrm{Id},g_{0}}_{-b,0}\omega_{ev}=\mathcal{A}^{g,\mathrm{Id},g_{0}}_{-b,0}\omega_{ev}\big|_{X\setminus
X'} \quad\text{in}~\mathcal{D}'(X;F)\,,
$$
while the right-hand side belongs to
$L^{2}(X_{(-\varepsilon,\varepsilon)};F\big|_{X_{(-\varepsilon,\varepsilon)}})$\,. By
going back to $s_{ev}$\,, this implies 
 $s_{ev}\in D(\hat{B}^{\phi_{b}}_{\fkh})$\,.
}
\end{proof}
The result of Theorem~\ref{th:domainB}-\textbf{e)} can be translated
for the operator 
$\hat{B}^{\phi_{b}}_{\fkh}$ after recalling
$$
D(\hat{B}^{\phi_{b}}_{\fkh})=L^{2}_{ev}(X;E)\cap
D(B^{\hat{\phi}})\oplus L^{2}_{odd}(X;E)\cap D(\hat{B}^{\phi_{b}}_{\fkh})
$$
and writing $s=s_{ev}+s_{odd}$. Although only $s_{ev}$ has been
treated,  
$s_{odd}$ actually enters in the same framework after
replacing  the unitary involution $\nu$ of $\fkf|_{Q'}$
with $-\nu$\,.
\begin{corollary}
\label{cor:sub} The results of Theorem~\ref{th:domainB}-\textbf{e)} hold,
mutatis mutandis, for the operator
$(\hat{B}^{\phi_{b}}_{\fkh}, D(\hat{B}^{\phi_{b}}_{\fkh}))$\,.
\end{corollary}
The summary about ``cuspidal semigroups'' (terminology introduced in
\cite{Nie}) in Subsection~\ref{sec:hyposmooth} applies now to
$(\overline{B}^{\phi_{b}}_{\fkh},D(\overline{B}^{\phi_{b}}_{\fkh}))$
and to $(\hat{B}^{\phi_{b}}_{\fkh},D(\hat{B}^{\phi_{b}}_{\fkh}))$ with
the exponent $r=1/4$\,.
\begin{corollary}
\label{cor:semigroup}
 For $b\in \rz^{*}$ and $C_{b,g}>0$ large
enough, the operators
$(A=C_{b,g}+\overline{B}^{\phi_{b}}_{\fkh},
D(\overline{B}^{\phi_{b}}_{\fkh}))$ and
$(A=C_{b,g}+\hat{B}^{\phi_{b}}_{\fkh},D(\hat{B}^{\phi_{b}}_{\fkh}))$
are maximal accretive and their resolvent are compact.\\
Their spectrum is contained in 
$$
\left\{z\in \cz\,, \Real z \geq C_{b,g}^{-1}\langle \Imag z\rangle^{1/4}\right\}\,.
$$
If $\gamma_{b,g}$ is the contour $\left\{z\in\cz\,, \Real
  z=\frac{1}{2 C_{b,g}}\langle \Imag z\rangle^{1/4}\right\}$ oriented from $+i\infty$
to $-i\infty$\,, the semigroups are given by
$$
\forall t>0\,,   e^{-tA}=\frac{1}{2i\pi}\int_{\gamma_{b,g}}\frac{e^{-tz}}{z-A}~dz\,.
$$
Those semigroups satisfy in particular $e^{-tA}\in D(A^{N})$ for any $t>0$\,.
\end{corollary}

\subsection{Bootstrapped regularity for the powers of the 
 resolvent and the semigroup}
\label{sec:bootstrap}
The  $\mathcal{W}^{1/3}(X_{-};\hat{F})$ (resp. $\mathcal{W}^{1/3}(X;\hat{F}_{g})$)
 global regularity estimate of  ~\ref{th:domainB}-\textbf{e)} (resp. Corollary~\ref{cor:sub}) for $s\in D(\overline{B}_{\fkh})$ (resp. $s\in
D(\hat{B}^{\phi_{b}}_{\fkh})$) does not correspond to the
 maximal hypoellipticity
result, $s\in \mathcal{W}^{2/3}(X;F)$ obtained by Lebeau in
\cite{Leb2} in the smooth case without boundary.
 As pointed out in
\cite{Nie} it is related to the extrinsic curvature of $\partial Q$
and it is not yet known whether it can be improved. 
However the estimates of Theorem~\ref{th:domainB} and
Corollary~\ref{cor:sub} suffice to get higher regularity estimates for
high enough powers of the resolvent and subsequently for the semigroup.\\
We start with weighted estimates which do not use any other regularity
properties than the one stated in Theorem~\ref{th:domainB} and
Corollary~\ref{cor:sub}.
\begin{lemma}
\label{le:weightB} Let $b\in \rz^{*}$ and set
$A=\overline{B}^{\phi_{b}}_{\fkh}$ or
$A=\hat{B}^{\phi_{b}}_{\fkh}$\,.
For any $n\in\nz$\,, there exists a constant $C_{n,b,g}>0$ such that
\begin{eqnarray*}
  && \langle p\rangle_{q}^{m+1}(C+A)^{-1}\langle
     p\rangle_{q}^{-m}\quad,\quad (1+\mathcal{O})^{1/2}\langle
     p\rangle_{q}^{m}(C+A)^{-1}\langle p\rangle_{q}^{-m}\\
  \text{and}&&\langle p\rangle_{q}^{m}(C+A)^{-m}\quad
               (1+\mathcal{O})^{1/2}\langle p\rangle_{q}^{m}(C+A)^{-m-1}
\end{eqnarray*}
are bounded for all $m$\,, $0\leq m\leq n$ and all $C\geq C_{n,b,g}$\,.\\
Finally for any $t>0$ and any $n\in\nz$\,, $(1+\mathcal{O})^{1/2}
\langle
p\rangle_{q}^{n}e^{-tA}$ is a bounded operator.
\end{lemma}
\begin{proof}
We focus on the case $A=\overline{B}^{\phi_{b}}_{\fkh}$ and the case
$A=\hat{B}^{\phi_{b}}_{\fkh}$ can be recovered by symmetrization like Corollary~\ref{cor:sub}.\\
With $1+\mathcal{O}\geq \frac{\langle p\rangle_{q}^{2}}{2}$ and
$$
\langle p\rangle_{q}^{n}(C+A)^{-n}=\prod_{k=0}^{n-1} \langle
p\rangle_{q}^{n-k}(C+A)^{-1}\langle p\rangle_{q}^{-n+1+k}
$$
all the results are consequences of  the boundedness of 
 $(1+\mathcal{O})^{1/2}\langle
p\rangle_{q}^{m}(C+A)^{-1}\langle p\rangle_{q}^{-m}$ for $C\geq
\tilde{C}_{m,b,g}\geq \tilde{C}_{m-1,b,g}$\,.\\
Take $u\in \mathcal{C}^{\infty}_{0}(\overline{X}_{-};F)\cap
D(\overline{B}^{\phi_{b}}_{\fkh})$\,,  and write
with $(C+A)u=f$:
$$
(C+A)\langle p\rangle_{q}^{-m}u=\langle p\rangle_{q}^{-m}f
+[-\frac{\Delta_{p}}{2},\langle p\rangle_{q}^{-m}]u\,.
$$
It  becomes
$$
\left(C+A+\mathcal{M}_{m,j}(q,p)\nabla_{\partial p_{j}}+\mathcal{M}_{m}\right)
\langle p\rangle_{q}^{-m}=\langle p\rangle_{q}^{-m}f
$$
where $\mathcal{M}_{m,j}$\,, $\mathcal{M}_{m}$ are symbols of order
$0$ according to Definition~\ref{de:GKFP}. With the integration by
part of Theorem~\ref{th:domainB}-\textbf{d)}\,,  the operator 
$\mathcal{M}_{m}=\mathcal{M}_{m,j}(q,p)\nabla_{\partial
  p_{j}}+\mathcal{M}_{m}$ is a relatively bounded perturbation of $A$ with
infinitesimal bound. Therefore   for
$C\geq \tilde{C}_{m,b,g}\leq \tilde{C}_{m-1,b,g}$ with $\tilde{C}_{m,b,g}$large enough, $(C+A+\mathcal{M})$ is
maximal  accretive and we get the uniform bound 
$$
\|\langle p\rangle_{q}^{-m}f\|=\|(C+A+\mathcal{M})\langle p\rangle_{q}^{-m}u\|\geq 
\frac{1}{2}\|(C+A)\langle
p\rangle_{q}^{-m}u\|
\geq C_{m}^{-1}\|(1+\mathcal{O})^{1/2}\langle
p\rangle_{q}^{-m}u\|
$$
Approaching any $u\in \langle p\rangle_{q}^{m}D(A)$ by elements
$u_{n}$ in
$\mathcal{C}^{\infty}_{0}(\overline{X}_{-};F)\cap D(A)$
  proves the boundedness of $(1+\mathcal{O})^{1/2}\langle
  p\rangle^{m}(C+A)^{-1}\langle p\rangle^{-m}$ for $C\geq \tilde{C}_{m,b,g}$\,.\\
The final statement about the semigroup is a consequence of 
$$
(1+\mathcal{O})^{1/2}\langle p\rangle_{q}^{n}e^{-tA}=(1+\mathcal{O})^{1/2}\langle p\rangle_{q}^{n}(C_{n,\alpha}+A)^{-n-1}(C_{n,\alpha}+A)^{n+1}e^{-tA}\,.
$$
\end{proof}
We will use Lebeau's maximal hypoellipticity estimate
\eqref{eq:maxhyp} with various values of ${\red \mu}\in[0,1]$\,.
We already used the fact that
$\hat{B}^{\phi_{b}}_{\fkh}=\mathcal{A}^{\hat{g},\mathrm{Id},\hat{g}}_{-b,\mathcal{M}}$ or more precisely
$(\hat{\Psi}^{g,g_{0}}_{X})_{*}^{-1}\hat{B}^{\phi_{b}}_{\fkh}(\hat{\Psi}^{g,g_{0}}_{X})_{*}=\mathcal{A}^{g_{0},\kappa,\gamma}_{-b,\mathcal{M}'}$
is a LGKFP operator in  order to the existence of traces and local
regularity properties in Propositions~\ref{pr:traceAhat} and
\ref{pr:IPPloc}. Let us look globally at those
transformations. Remember that  $(\hat{\Psi}_{X}^{g,g_{0}})_{*}$ provides a
continuous isomorphism from
$\mathcal{W}^{\red\mu}(X_{(-\varepsilon,\varepsilon)};F\big|_{X_{(-\varepsilon,\varepsilon)}},g_{0}){\red=\mathcal{W}^{\mu}(X;F)}$
to
$\mathcal{W}^{\red\mu}(X_{(-\varepsilon,\varepsilon)};\hat{F}_{g},\hat{g})$
for $\red\mu\in [-1,1]$ while Lemma~\ref{le:kappachi} provides for a
given partition of unity $\sum_{j=1}^{J}\chi_{j}(q)\equiv 1$ the map
$\kappa_{\chi}$ which is an automorphism of $\mathcal{W}^{\mu}(X;F)$
for $\mu\in [-1,1]$\,.
\begin{itemize}
\item Firstly the vector field $Y_{\hat{\fkh}}$ remains expressed with the
  vector fiels $(e_{i})_{1\leq i\leq d}$ which differ from the vector
  fields $(f_{i})_{1\leq i\leq}$ associated with the metric
  $g_{0}^{TQ}$\,. From \eqref{eq:eitilde} we can write
  \begin{eqnarray*}
e_{i'}&=&\frac{\partial}{\partial
  \tilde{q}^{i'}}+M_{i'j}^{k}(\tilde{q})\tilde{p}_{k}\frac{\partial}{\partial
  \tilde{p}_{j}}\\
&=&\underbrace{\frac{\partial}{\partial
  q'^{i}}+\Gamma_{i',j'}^{k'}(0,\tilde{q}')\tilde{p}_{k}\frac{\partial}{\partial\tilde{p}_{j'}}}_{\tilde{f}_{i'}}
+\underbrace{[M_{i'j}^{k}(\tilde{q})-(1-\delta_{1j})(1-\delta_{1k})\Gamma^{k}_{i'j}(0,\tilde{q})]\tilde{p}_{k}\frac{\partial}{\partial
  \tilde{p}_{j}}}_{\tilde{R}_{i'}}\,,
  \end{eqnarray*}
with
$(\hat{\Psi}^{(g,g_{0})}_{X})_{*}^{-1}(\tilde{f}_{i})(\hat{\Psi}^{g,g_{0}}_{X})_{*}=f_{i}$\,,
$(\hat{\Psi}^{(g,g_{0})}_{X})_{*}^{-1}(\tilde{R}_{i'})(\hat{\Psi}^{g,g_{0}}_{X})_{*}=R_{i'}$
and $R_{i}$ behaves like $p\times \frac{\partial}{\partial p}$\,.
Hence we deduce that
$$
(\hat{\Psi}^{(g,g_{0})}_{X})_{*}^{-1}(Y_{\hat{\fkh}})(\hat{\Psi}^{g,g_{0}}_{X})_{*}
=p_{1}\frac{\partial}{\partial q^{1}}+(g_{0}\kappa)^{i'j'}p_{j'}f_{i'} + (g_{0}\kappa)^{i'j'}p_{j'}R_{i'}
$$
By conjugating with the map $\red\kappa_{\chi}$ of
Lemma~\ref{le:kappachi} one obtains similarly
$$
{\red\kappa_{\chi}}(\hat{\Psi}^{(g,g_{0})}_{X})_{*}^{-1}(Y_{\hat{\fkh}})(\hat{\Psi}^{g,g_{0}}_{X})_{*}{\red\kappa_{\chi}^{-1}}
=Y_{\fkh_{0}}+R
$$
where $\fkh_{0}=\frac{g_{0}^{ij}(q)p_{i}p_{j}}{2}$ and $R$ is an
operator  with terms like
$M_{i}^{jk}(q)p_{j}p_{k}\frac{\partial}{\partial p_{k}}$\,.
\item The Levi-Civita connections for the metrics $g^{TQ}$ and
  $g_{0}^{TQ}$ differ but we already noticed that such a change of
  connection adds a perturbative term
  $\mathcal{M}^{j}(q)p_{j}+\mathcal{M}_{0}$ when expressed in the
  basis $(e,\hat{e})$\,. The change of frame to $(f,\hat{f})$  will
  lead to a  perturbative term  $\mathcal{M}^{jk}(q)p_{j}p_{k}+\mathcal{M}^{j}(q)p_{j}+\mathcal{M}^{0}(q)$\,.
\item The conjugation by $\kappa_{\chi}$ changes the term
  $\Delta_{p}^{\gamma}$ into $\Delta^{\gamma'}_{p}$ with
  $\gamma'=\kappa^{-1}\gamma{}^{t}\kappa^{-1}$\,.
\item All the coefficients belong to
  $\mathcal{C}^{\infty}(\overline{X}_{\pm};\rz)$ and may be
  discontinuous, except the metric $\gamma'$ which coincides with
  $g_{0}^{TQ}$ along $Q'$ on both sides. 
\end{itemize}
Hence we can write
\begin{equation}
  \label{eq:KFPpertglob}
{\red\kappa_{\chi}}(\hat{\Psi}^{g,g_{0}}_{X})_{*}^{-1}\hat{B}^{\phi_{b}}_{\fkh}(\hat{\Psi}^{g,g_{0}}_{X})_{*}{\red\kappa_{\chi}^{-1}}
=
\underbrace{-b\nabla^{F,g_{0}}_{Y_{\fkh_{0}}}+\mathcal{O}^{g_{0}}}_{\mathcal{A}_{-b,0}^{g_{0}}}
+\frac{\Delta^{g_{0}}_{p}-\Delta^{\gamma'}_{p}}{2}
+\langle p\rangle_{g_{0},q}^{2}\mathcal{M}
\end{equation}
with
$$
\mathcal{M}=[\mathcal{M}^{0}_{j}(q,p)+1_{\rz_{+}}(q^{1})\mathcal{M}^{0}_{+,j}(q,p)]
\frac{\partial}{\partial p_{j}}+[\mathcal{M}^{0}(q,p)+1_{\rz_{+}}(q^{1})\mathcal{M}^{0}_{+}(q,p)]\,,
$$
where
$\mathcal{M}^{0}_{j},\mathcal{M}^{0}_{+,j},\mathcal{M}^{0},\mathcal{M}^{0}_{+}$
are symbols of order $0$  on $X_{(-\varepsilon,\varepsilon)}$ for the
metric $g_{0}^{TQ}$\,, and all the superscript $^{g_{0}}$ recall  that
the GKFP operator $\mathcal{A}_{-b,0}^{g_{0}}$ and the connection
$\nabla^{F,g_{0}}$ are the ones associated with the metric
$g_{0}^{TQ}$\,.\\
If $\mathcal{A}^{g_{0}}_{-b,0}$ is a GKFP operator in the smooth case,
the additional term
$$
\frac{\Delta_{p}^{g_{0}}-\Delta^{\gamma'}_{p}}{2}+\langle p\rangle_{g_{0},q}^{2}\mathcal{M}
$$
is not an admissible perturbation according to
Definition~\ref{de:GKFP}. The term
$\frac{\Delta_{p}^{g_{0}}-\Delta^{\gamma'}_{p}}{2}$ will be absorbed
by a relative boundedness with bound less than $1$ argument. 
For $\langle p\rangle_{q}^{2}\mathcal{M}$ we have two difficulties:
\begin{itemize}
\item The weight $\langle p\rangle^{2}_{g_{0},q}$ is too high but this
  will be handled via Lemma~\ref{le:weightB}.
\item The discontinuity along $X'$ prevent for high regularity
  estimates. This will be handled by using the one dimensional product
  rule for Sobolev spaces
\begin{equation}
  \label{eq:Sobprod1}
\left(
  \begin{array}[c]{l}
    \varphi_{1}\in W^{s_{1},2}(\rz)\quad,\quad \varphi_{2}\in
    W^{s_{2},2}(\rz)\\
s_{1},s_{2}\geq s_{3}\\
s_{1}+s_{2}>s_{3}+\frac{1}{2}
  \end{array}
\right)
\Rightarrow (\varphi_{1}\varphi_{2}\in W^{s_{3},2}(\rz))\,,
\end{equation}
while noticing $1_{\rz_{+}}(q^{1})\in W^{1/2-\delta,2}_{loc}(\rz)$ for
all $\delta>0$\,.
\end{itemize}
\begin{lemma}
\label{le:regpowers} Consider the operator
$P=\mathcal{A}^{g_{0}}_{\alpha,0}+\frac{\Delta_{p}^{g_{0}}-\Delta^{\gamma'}}{2}+\langle
p\rangle_{g_{0},q}^{2}\mathcal{M}$ in $X_{(-\varepsilon,\varepsilon)}$
for $\varepsilon>0$ small enough. If $u\in
\mathcal{W}^{\mu}(X_{(-\varepsilon,\varepsilon)};F\big|_{X_{(-\varepsilon,\varepsilon)}})$ satisfies $\supp
u\subset X_{(-\varepsilon/2,\varepsilon/2)}$ and $(1+\mathcal{O}^{g_{0}})^{1/2}\langle
p\rangle^{m+2}u\in
\mathcal{W}^{\mu}(X_{(-\varepsilon,\varepsilon)};F\big|_{X_{(-\varepsilon,\varepsilon)}})$
and $\langle p\rangle_{g_{0},q}^{m}(Pu)\in \mathcal{W}^{\mu}(X;F)$ 
for $\mu\in [0,1/2)$ and $m\in \nz$\,, then $\langle p\rangle_{g_{0},q}^{m}u\in
\mathcal{W}^{\mu'}(X_{(-\varepsilon,\varepsilon)};F\big|_{X_{(-\varepsilon,\varepsilon)}})$
for all $\mu'\in [0,\mu+2/3)$ with
$$
\|\langle p\rangle_{g_{0},q}^{m}u\|_{\mathcal{W}^{\mu'}}\leq
C_{b,g_{0},\gamma',\mathcal{M},\mu,\mu',m}\left[
\|(1+\mathcal{O}^{g_{0}})^{1/2}\langle
p\rangle_{g_{0},q}^{m+2}u\|_{\mathcal{W}^{\mu}}+\|\langle p\rangle_{g_{0},q}^{m}(Pu)\|_{\mathcal{W}^{\mu}}\right]\,.
$$
\end{lemma}
\begin{proof}
Write simply $\langle p\rangle=\langle p\rangle_{g_{0},q}$\,,
$\mathcal{O}=\mathcal{O}^{g_{0}}$ and compute
$$
P(\langle p\rangle^{m}u)=\langle
p\rangle^{m}(Pu)+\left[-\frac{\Delta_{b}^{\gamma'}}{2}+\langle p\rangle^{2}\mathcal{M},\langle
p\rangle^{m}\right]\langle p\rangle^{-m-2}\langle p\rangle^{m+2}u\,.
$$
This gives
$$
\left[\mathcal{A}^{g_{0}}_{-b,0}+\frac{\Delta_{g}^{g_{0}}-\Delta_{p}^{\gamma}}{2}\right](\langle p\rangle^{m}u)=\langle p
\rangle^{m}(Pu)
+\langle p\rangle^{2}\mathcal{M}'(\langle p\rangle^{m}u)
$$
where $\mathcal{M}'$ has the same structure as $\mathcal{M}$:
\begin{eqnarray*}
\mathcal{M}'&=&[\mathcal{M}^{0\prime}_{j}(q,p)+1_{\rz_{+}}(q^{1})\mathcal{M}^{0\prime}_{+,j}(q,p)]
\frac{\partial}{\partial
  p_{j}}+[\mathcal{M}^{0\prime}(q,p)+1_{\rz_{+}}(q^{1})\mathcal{M}^{0\prime}_{+}(q,p)]\\
&=&\mathcal{M}'_{-}+1_{\rz_{+}}(q^{1})\mathcal{M}_{+}'\,.
\end{eqnarray*}
But we know 
$$
\|\langle p\rangle^{2}\mathcal{M}_{\mp}'(\langle p\rangle^{m})u\|_{\mathcal{W}^{\mu}}\leq
C_{b,g_{0},\gamma',\mathcal{M},m}\|(1+\mathcal{O})^{1/2}\langle p\rangle^{m+2}u\|_{\mathcal{W}^{\mu}}\,.
$$
For $\mu=0$ this gives
$$
\|\langle p\rangle^{2}\mathcal{M}'\langle
p\rangle^{m}u\|_{\mathcal{W}^{0}}
\leq C'_{b,g_{0},\mathcal{M},m}\|(1+\mathcal{O})^{1/2}\langle p\rangle^{m+2}u\|_{\mathcal{W}^{0}}\,,
$$
while for $\mu\in ]0,1/2[$\,, $\mu''<\mu$\,, $1_{\rz_{+}}\in
W^{1/2-\delta_{\mu,\mu''},2}(\rz)$\,,  $\delta_{\mu,\mu''}\leq 1/2-\mu$\,, $\mu-\delta_{\mu,\mu''}>\mu''$\,,
the multiplication rule \eqref{eq:Sobprod1} leads to 
$$
\|\langle p\rangle^{2}\mathcal{M}'\langle
p\rangle^{m}u\|_{\mathcal{W}^{\mu''}}\leq
C_{b,g_{0},\gamma',\mathcal{M},m,\mu,\mu''}\|(1+\mathcal{O})^{1/2}\langle p\rangle^{m}u\|_{\mathcal{W}^{\mu}}\,.
$$
In all cases $\langle p\rangle^{m}u\in
\mathcal{W}^{\mu}(X_{(-\varepsilon,\varepsilon)},F\big|_{X_{(-\varepsilon,\varepsilon)}})$
with $\supp u\subset X_{(-\varepsilon/2,\varepsilon/2)}$ solves
$$
\left[\mathcal{A}^{g_{0}}_{-b,0}+\frac{\Delta_{g}^{g_{0}}-\Delta_{p}^{\gamma}}{2}\right](\langle p\rangle^{m}u)=\hat{f}
$$
with 
$$
\|\hat{f}\|_{\mathcal{W}^{\mu''}}\leq C_{b,g,\gamma',\mathcal{M},m,\mu,\mu''}\left[
\|(1+\mathcal{O})^{1/2}\langle
p\rangle^{m+2}u\|_{\mathcal{W}^{\mu}}+\|\langle p\rangle^{m}(Pu)\|_{\mathcal{W}^{\mu}}
\right]
$$
for all $\mu''<\mu$\,. 
But the maximal subelliptic result of Lebeau in \cite{Leb2}, recalled
in \eqref{eq:maxhyp} implies
$$
\|\mathcal{O}(\langle p\rangle^{m}u)\|_{\mathcal{W}^{\mu''}}+\|\langle
p\rangle^{m}u\|_{\mathcal{W}^{\mu''+2/3}}
\leq C_{b,g_{0},s''}\left[\|\mathcal{A}^{g_{0}}_{-b,0}(\langle p\rangle^{m}u)\|_{\mathcal{W}^{\mu''}}+\|\langle p\rangle^{m}u\|_{\mathcal{W}^{\mu''}}\right]
$$
while 
$$
\|\frac{\Delta_{p}^{\gamma'}-\Delta^{g_{0}}_{p}}{2}(\langle
p\rangle^{m}u)\|_{\mathcal{W}^{\mu''}}\leq
C_{g_{0},\gamma'}\varepsilon\|\mathcal{O}(\langle p\rangle^{m}u)\|_{\mathcal{W}^{\mu''}}\,.
$$
Taking $\varepsilon>0$ small enough implies $\langle p\rangle^{m}u\in
\mathcal{W}^{\mu''+2/3}(X_{(-\varepsilon,\varepsilon)};F\big|_{X_{(-\varepsilon,\varepsilon)}})$
for all $\mu''<\mu$ with the inequality written for $\mu'=\mu''+2/3$\,.
\end{proof}
By using a partition of unity in $q$\,, while the result for $u$ with
$\supp u\subset X_{\rz\setminus(-\varepsilon/2,\varepsilon/2)}$ comes
from the maximal subelliptic estimate \eqref{eq:maxhyp}, we deduce the
\begin{lemma}
\label{le:regpowers2} Let $(\hat{B}^{\phi_{b}}_{\fkh},
D(\hat{B}^{\phi_{b}}_{\fkh}))$ be the Bismut hypoelliptic Laplacian
studied in Proposition~\ref{pr:domhatB} for $b\in \rz^{*}$\,. If $u\in
\mathcal{W}^{\mu}(X;\hat{F}_{g})\cap D(\hat{B}^{\phi_{b}}_{\fkh})$ satisfies 
$(1+\mathcal{O}^{g})^{1/2}\langle
p\rangle_{q}^{m+2}u\in
\mathcal{W}^{\mu}(X;F)$
and $\langle p\rangle^{m}_{q}(\hat{B}^{\phi_{b}}_{\fkh}u)\in \mathcal{W}^{\mu}(X;\hat{F}_{g})$ 
for $\mu\in [0,1/2)$ and $m\in \nz$\,, then $\langle p\rangle^{m}u\in
\mathcal{W}^{\min(\mu',1)}(X_{(-\varepsilon,\varepsilon)};F\big|_{X_{(-\varepsilon,\varepsilon)}})$
for all $\mu'\in [0,\mu+2/3)$  with 
$$
 \|\langle p\rangle_{q}^{m}u\|_{\mathcal{W}^{\min(1,\mu')}}\leq
 C_{b,g,\mu,\mu'}
\left[
\|(1+\mathcal{O}^{g})^{1/2} \langle p\rangle_{q}^{m+2}u\|_{\mathcal{W}^{\mu}}
 +\|\langle p\rangle_{q}^{m}(\hat{B}^{\phi_{b}}_{\fkh}u)\|_{\mathcal{W}^{\mu}}
\right]\,.
$$
In particular
$(1+\mathcal{O}^{g})^{1/2}\langle p\rangle_{q}^{m}u\in
\mathcal{W}^{\mu''}(X;\hat{F}_{g})$ for all $\mu''\in [0,\mu+\frac{1}{6})$\,.
\end{lemma}
\begin{proof}
As said before it is a consequence of Lemma~\ref{le:regpowers} and the
interior maximal subelliptic estimate \eqref{eq:maxhyp} after using a
partition of unity in $q$\,.\\
  The only thing to be recalled is the fact that
  $(\hat{\Psi}^{g,g_{0}}_{X})_{*}$ and $K$ are isomorphism of
  $\mathcal{W}^{\mu}$-spaces only for $\mu\in [-1,1]$\,. This explains the
  exponent $\min(1,\mu')$\,.
\end{proof}
\begin{proposition}
\label{pr:bootstr}
For $b\in\rz^{*}$ and $C\geq C_{b,g}>0$\,, $C_{b,g}$ large enough, the
operators $(\hat{B}^{\phi_{b}}_{\fkh},D(\hat{B}^{\phi_{b}}_{\fkh}))$
of Proposition~\ref{pr:domhatB} and the operator
$(\overline{B}^{\phi_{b}}_{\fkh},D(\overline{B}^{\phi_{b}}_{\fkh}))$
of Definition~\ref{de:barBg}  the maps
\begin{eqnarray*}
  &&
     (C+\hat{B}^{\phi_{b}}_{\fkh})^{-9}:L^{2}(X;F)\to
     \mathcal{W}^{1}(X;\hat{F}_{g})\,,\\
&& 
(C+\hat{B}^{\phi_{b}}_{\fkh})^{-9}:\mathcal{W}^{-1}(X;\hat{F}_{g})\to
   L^{2}(X;F)\,,\\
&& (C+\overline{B}^{\phi_{b}}_{\fkh})^{-9}:
   L^{2}(X_{-};F)\to \mathcal{W}^{1}(\overline{X}_{-};F)
\end{eqnarray*}
are all bounded.
\end{proposition}
\begin{proof}
For $m\in \nz$ to be fixed later we call
$u=(C+\hat{B}^{\phi_{b}}_{\fkh})^{-m-3}u_{0}$ and
$f=(C+\hat{B}^{\phi_{b}}_{\fkh})^{-m-2}$ and
$f_{1}=(C+\hat{B}^{\phi_{b}}_{\fkh})^{-m-1}u_{0}$ for $u_{0}\in
L^{2}(X;F)$ so that $u, f\in D(\hat{B}^{\phi_{b}}_{\fkh})$ satisfy
$$
(C+\hat{B}^{\phi_{b}}_{\fkh})u=f\quad,\quad (C+\hat{B}^{\phi_{b}}_{\fkh})f=f_{1}\,.
$$
From Lemma~\ref{le:weightB} we know
$$
(1+\mathcal{O}^{g})^{1/2}\langle p\rangle_{q}^{m}u\in
\mathcal{W}^{0}(X;\hat{F}_{g})\quad,\quad
\langle p\rangle_{q}^{m}f\in \mathcal{W}^{0}(X;\hat{F}_{g})\,.
$$
But Lemma~\ref{le:regpowers2} applied to the pair $(u,f)$ then implies
\begin{equation}
  \label{eq:weu}
(1+\mathcal{O}^{g})^{1/2}\langle p\rangle^{m'+2}u\in
\mathcal{W}^{\mu'}(X;\hat{F}_{g})
\end{equation}
for $\mu'\in [0,1/6)$ and $m'=m-2$\,.
Applied to the pair $(f,f_{1})$ with $m$ replaced by $m-1$ it says
$$
\langle p\rangle_{q}^{m-1}f \in
\mathcal{W}^{\min(1)\mu'')}(X;\hat{F}_{g})~\text{for}~\mu''\in [0,2/3)\,,
$$
and this implies
\begin{equation}
  \label{eq:wef}
  \langle p\rangle_{q}^{m'}f\in \mathcal{W}^{\mu'}(X;F)\,.
\end{equation}
for $m'=m-2$ and $\mu'\in [0,1/6)$\,.
From \eqref{eq:weu}\eqref{eq:wef} with $\mu'\in [0,1/6)$\,, we can
apply again Lemma~\ref{le:regpowers2} with any $\mu\in [0,1/6)$ with
$m$ replaced by $m-2$\,. This leads to \eqref{eq:weu}\eqref{eq:wef}
with $\mu'\in [0,1/3)$ and $m'=m-4$\,. By doing it once more we obtain
\eqref{eq:weu}\eqref{eq:wef} for any $\mu'\in [0,1/2)$ and
$m'=m-6$\,. Applying Lemma~\ref{le:regpowers2} a last time gives
$$
\langle p\rangle_{q}^{m-6}u\in \mathcal{W}^{1}(X;\hat{F}_{g})\,.
$$
Taking $m=6$ proves that
$$
(C+\hat{B}^{\phi_{b}}_{\fkh})^{-6-3}: L^{2}(X;F)\to \mathcal{W}^{1}(X;\hat{F}_{g})
$$
is continuous.\\
The continuity of $(C+\hat{B}^{\phi_{b}}_{\fkh})^{-9}:\mathcal{W}^{-1}(X;\hat{F}_{g})\to
L^{2}(X;F)$ is deduced by duality after recalling that the
$L^{2}$-adjoint $(\hat{B}^{\phi_{b}}_{\fkh})^{*}$ is a GKFP operator
with the same properties as $\hat{B}^{\phi_{b}}_{\fkh}$\,.\\
The property for $\overline{B}^{\phi_{b}}_{\fkh}$ is deduced from the
fact that $(C+\hat{B}^{\phi_{b}}_{\fkh})^{-1}:L^{2}_{ev}(X;F)\to L^{2}_{ev}(X;F)$\,.
\end{proof}
\begin{corollary}
\label{cor:semid} For $b\in \rz^{*}$\,, $k\in\nz$ 
and $t>0$ the following
operators are well defined and continuous:
\begin{eqnarray*}
  && \hat{d}_{g,\fkh}(\hat{B}^{\phi_{b}}_{\fkh})^{k}e^{-t\hat{B}^{\phi_{b}}_{\fkh}}: L^{2}(X;F)\to
     L^{2}(X;F)\,,\\
&& \hat{d}^{\phi_{b}}_{g,\fkh}
(\hat{B}^{\phi_{b}}_{\fkh})^{k}e^{-t\hat{B}^{\phi_{b}}_{\fkh}}: L^{2}(X;F)\to  L^{2}(X;F)\,,\\
&&
 \overline{d}_{g,\fkh}(\overline{B}^{\phi_{b}}_{\fkh})^{k}e^{-t \overline{B}^{\phi_{b}}_{\fkh}}: L^{2}(X_{-};F)\to L^{2}(X_{-};F)\,,\\
 &&
 \overline{d}^{\phi_{b}}_{g,\fkh}(\overline{B}^{\phi_{b}}_{\fkh})^{k}e^{-t \overline{B}^{\phi_{b}}_{\fkh}}: L^{2}(X_{-};F)\to L^{2}(X_{-};F)\,.
\end{eqnarray*}
By taking the $\phi_{\pm b}$ left-adjoints the reverse products,
initially defined on $D(\hat{d}_{g,\fkh})$\,,
$D(\hat{d}^{\phi_{b}}_{g,\fkh})$\,,
$D(\overline{d}^{\phi_{b}}_{g,\fkh})$ or
$D(\overline{d}^{\phi_{b}}_{g,\fkh})$\,, extend as
bounded operators in $L^{2}(X;F)$ and $L^{2}(X_{-};F)$\,.
\end{corollary}
\begin{proof}
 It suffices to write for $A=\hat{B}^{\phi_{b}}_{\fkh}$ 
  $A^{k}e^{-tA}=(C+A)^{-9}A^{k}(C+A)^{-k}(C+A)^{k+9}e^{-tA}=A^{k}(C+A)^{-k}(C+A)^{k+9}e^{-tA}(C+A)^{-9}$
  owing to Corollary~\ref{cor:semigroup},
  and to notice that 
$$
\mathcal{W}^{1}(X;\hat{F}_{g})
\subset D(\hat{d}_{\alpha\hat{\mathcal{H}}})\cap D(\hat{d}^{\hat{\phi}}_{\alpha\hat{\mathcal{H}}})\,,
$$
while $\hat{d}_{\alpha\hat{\mathcal{H}}}$ and
$\hat{d}^{\hat\phi}_{\alpha\hat{\mathcal{H}}}$ are continuous from
$L^{2}(X;E)$ to $\mathcal{W}^{-1}(X;\hat{F}_{g})$\,.\\
The result for $\overline{B}^{\phi_{b}}_{\fkh}$\,,
$\overline{d}_{g,\fkh}$\,, $\overline{d}^{\phi_{b}}_{g,\fkh}$ are again
deduced from the general construction with the parity with respect to $\Sigma_{\nu}$\,.
\end{proof}
\subsection{Commutation property}
\label{sec:commutppty}
We now prove the commutation of the differential and Bismut
codifferential with the resolvent of the hypoelliptic Laplacian.
\begin{proposition}
\label{pr:commut}
Let
$(\overline{B}^{\phi_{b}}_{\fkh},D(\overline{B}^{\phi_{b}}_{\fkh}))$
be the operator of Definition~\ref{de:barBg} and
Theorem~\ref{th:domainB} for $b\in\rz^{*}$\,. Let
$(\overline{d}_{g,\fkh},D(\overline{d}_{g,\fkh})$
and
$(\overline{d}^{\phi_{b}}_{g,\fkh},D(\overline{d}^{\phi_{b}}_{g,\fkh}))$ 
be the closed realizations of the differential and Bismut
codifferential studied in Proposition~\ref{pr:domaindH} and
Proposition~\ref{pr:phiadj}.\\
The semigroup $(e^{-t\overline{B}^{\phi_{b}}_{\fkh}})_{t\geq 0}$ preserves
on $D(\overline{d}_{g,\fkh})$  (resp. $D(\overline{d}^{\phi_{b}}_{g,\fkh})$) with 
\begin{eqnarray*}
  &&\forall s\in D(\overline{d}_{g,\fkh})\,, \quad
     \overline{d}_{g,\fkh}e^{-t\overline{B}^{\phi_{b}}_{\fkh}}s=e^{-t\overline{B}^{\phi_{b}}_{\fkh}}
\overline{d}_{g,\fkh}s\\
   \text{resp.}&&\forall s\in 
 D(\overline{d}^{\phi_{b}}_{\fkh})\,, \quad
      \overline{d}^{\phi_{b}}_{g,\fkh}
 e^{-t\overline{B}^{\phi_{b}}_{\fkh}}s=
 e^{-t\overline{B}^{\phi_{b}}_{\fkh}}
 \overline{d}^{\phi_{b}}_{g,\fkh}s\,,
\end{eqnarray*}
for all $t\geq 0$\,.\\
Hence for any $z\in \cz\setminus\textrm{Spec}(\overline{B}^{\phi_{b}}_{\fkh})$ the
following holds
\begin{eqnarray*}
  &&\forall  s\in D(\overline{d}_{g,\fkh})\,,\quad
     (z-B^{\phi_{b}}_{g,\fkh})^{-1}s\in
     D(\overline{d}_{g,\fkh})\\
&&\quad \text{and}\quad
\overline{d}_{g,\fkh}(z-\overline{B}^{\phi_{b}}_{\fkh})^{-1}s=
(z-\overline{B}^{\phi_{b}}_{\fkh})^{-1}\overline{d}_{g,\fkh}s\,,\\
\text{resp.}&&
\forall  s\in D(\overline{d}^{\phi_{b}}_{g,\fkh})\,,\quad
     (z-\overline{B}^{\phi_{b}}_{\fkh})^{-1}s\in
     D(\overline{d}^{\phi_{b}}_{g,\fkh})\\
&&\quad \text{and}\quad
\overline{d}^{\phi_{b}}_{g,\fkh}(z-\overline{B}^{\phi_{b}}_{\fkh})^{-1}s=(z-\overline{B}^{\phi_{b}}_{\fkh})^{-1}\overline{d}^{\phi_{b}}_{g,\fkh}s\,.
\end{eqnarray*}
\end{proposition}
Although we seek properties of
$\overline{B}^{\phi_{b}}_{\fkh}$\,, it is again more
convenient to work with $\hat{B}^{\phi_{b}}_{\fkh}$ and
then to translate the results via the parity w.r.t $\Sigma_{\nu}$\,.
We start with a lemma.
\begin{lemma}
  \label{le:commut} Let $\mathcal{C}_{0,g}(\hat{F}_{g})$\,,
  $\hat{\mathcal{D}}_{g,\nabla^{\fkf}}$ and 
  $e^{\frac{\lambda_{0}}{b}}\sigma_{b}\hat{\mathcal{D}}'_{g,\nabla^{\fkf'}}$ be the
  spaces introduced respectively in
  Definition~\ref{de:Cg}, Proposition~\ref{pr:domaindH} and
  Proposition~\ref{pr:phiadj}.  For all $\omega\in \hat{\mathcal{D}}_{g,\nabla^{\fkf}}$ (resp $\omega\in e^{\frac{\lambda_{0}}{b}}\sigma_{b}\hat{\mathcal{D}}'_{g,\nabla^{\fkf'}}$) the
  following properties hold:
  \begin{eqnarray*}
    && \omega \in D(\hat{d}_{g,\fkh})\cap
       D(\hat{B}^{\phi_{b}}_{\fkh})\quad
\hat{d}_{g,\fkh}\omega \in D(\hat{B}^{\phi_{b}}_{\fkh})\quad
\hat{B}^{\phi_{b}}_{\fkh}\omega\in D(\hat{d}_{g,\fkh})
\\
\text{and}&&
\hat{B}^{\phi_{b}}_{\fkh}\hat{d}_{g,\fkh}\omega=\hat{d}_{g,\fkh}\hat{B}^{\phi_{b}}_{\fkh}\omega\,,\\
\text{resp.}&&
\omega \in D(\hat{d}^{\phi_{b}}_{g,\fkh})\cap
       D(\hat{B}^{\phi_{b}}_{\fkh})\quad
       \hat{d}^{\phi_{b}}_{g,\fkh}\omega \in
       D(\hat{B}^{\phi_{b}}_{\fkh})\quad
       \hat{B}^{\phi_{b}}_{\fkh}\omega\in
       D(\hat{d}^{\phi_{b}}_{g,\fkh})\\
\text{and}&&
\hat{B}^{\phi_{b}}_{g,\fkh}\hat{d}^{\phi_{b}}_{\fkh}\omega=\hat{d}^{\phi_{b}}_{g,\fkh}\hat{B}^{\phi_{b}}_{g,\fkh}\omega\,.
  \end{eqnarray*}
\end{lemma}
\begin{proof}
  By definition $\hat{\mathcal{D}}_{g,\nabla^{\fkh}}\subset
  \mathcal{C}_{0,g}(\hat{F}_{g})\subset
  D(\hat{d}_{g,\fkh})\cap D(\hat{B}^{\phi_{b}}_{\fkh})$ 
while Proposition~\ref{pr:domaindH} ensures
  $\hat{d}_{g,\fkh}\omega=d_{\hat{\fkh}}\omega  \in
  \mathcal{C}_{0,g}(\hat{F}_{g})\subset D(\hat{B}^{\phi_{b}}_{\fkh})$
  when $\omega\in \hat{D}_{g,\nabla^{\fkh}}$\,.
Additionally as differential operators on $X_{-}\cup X_{+}=X\setminus X'$\,, we know
$$
d_{\hat\fkh}\hat{B}^{\phi_{b}}_{\fkh}=d_{\hat\fkh}(d_{\hat\fkh}d^{\phi_{b}}_{\hat{\fkh}}+d^{\phi_{b}}_{\hat{\fkh}}d_{\hat\fkh})
=\hat{B}^{\phi_{b}}_{\fkh}d_{\hat{\fkh}}\,. 
$$
The question to be answered is $\hat{B}^{\phi_{b}}_{\fkh}\omega \in
D(\hat{d}_{g,\fkh})$ when $\omega\in
\hat{\mathcal{D}}_{g,\nabla^{\fkf}}$\,.  For $\omega\in
\hat{\mathcal{D}}_{g,\nabla^{\fkf}}$\,, write
$$
\langle t'\,,\,
\underbrace{\hat{B}^{\phi_{b}}_{\fkh}\hat{d}_{g,\fkh}\omega}_{\in L^{2}(X;F)}\rangle
=\lim_{\varepsilon \to 0}\langle 1_{\rz\setminus[-\varepsilon,
  \varepsilon]}(q^{1})t'\,,\, \hat{B}^{\phi_{b}}_{\fkh}d_{\hat{\fkh}}\omega\rangle
=\langle t'\,,\, d_{\hat{\fkh}}\hat{B}^{\phi_{b}}_{\fkh}\omega\big|_{X\setminus X'}\rangle
$$
for $t'\in \mathcal{C}_{0,g}(X;\hat{F}'_{g})$\,, by  using the duality
product
$$
\langle t\,,\, s\rangle=\int_{X}\langle t\,,\, s\rangle_{F'_{x},F_{x}}~dv_{X}(x)\,.
$$ 
It  implies 
$$
\forall t'\in \mathcal{C}_{0,g}(X;\hat{F}_{g})\,,\quad \left|\langle
  t'\,,\,
  d_{\hat{\fkh}}(\hat{B}^{\phi_{b}}_{\fkh}\omega)\big|_{X\setminus X'} \rangle\right|\leq C\|t'\|_{L^{2}}\,.
$$
In particular when we take $t'\in
\mathcal{C}^{\infty}_{0}(X_{\mp};F')$\,, it says that 
$u_{\mp}=\hat{B}^{\phi_{b}}_{\fkh}\omega\big|_{X_{\mp}}$ belongs to
$\mathcal{E}_{loc}(d_{\hat{\fkh}};F)$ and $j_{\partial X_{\mp}}u_{\mp}$ is
well defined.\\
By working in $M_{g}$ as we did in the proof of  Proposition
 and taking
test sections in
$\mathcal{C}^{\infty}_{0}(M_{g,(-\varepsilon,\varepsilon)};\Lambda
T^{*}M_{g}\otimes \pi_{X}^{*}(\fkf'))$\, which make a subset of
$\mathcal{C}_{0,g}(\hat{F}_{g}')$\,, we  finally obtain that
$u=\hat{B}^{\phi_{b}}_{\fkh}\omega =1_{X_{-}}u_{-}+1_{X_{+}}u_{+}$
belongs to $D(\hat{d}_{g,\fkh})$\,.\\
When $\omega\in
e^{\frac{\lambda_{0}}{b}}\sigma_{b}\hat{\mathcal{D}}'_{g,\nabla^{\fkf'}}$
we know $\hat{d}^{\phi_{b}}_{g,\fkh}\omega \in
\mathcal{C}_{0,g}(\hat{F}_{g})\subset
D(\hat{B}^{\phi_{b}}_{\fkh})$ and we want to prove
$\hat{B}^{\phi_{b}}_{\fkh}\omega \in D(\hat{d}^{\phi_{b}}_{g,\fkh})$\,. By
taking $\omega'\in \hat{D}_{g,\nabla^{\fkf}}$ we write
$$
\langle \omega'\,,\,
\hat{B}^{\phi_{b}}_{\fkh}\hat{d}^{\phi_{b}}_{g,\fkh}\omega\rangle_{\phi_{-b}}
=\langle
\hat{d}_{g,\fkh}\hat{B}^{\phi_{-b}}_{\fkh}\omega'\,,\,\omega\rangle_{\phi_{-b}}
=\langle \hat{B}^{\phi_{-b}}_{g}\hat{d}_{g,\fkh}\omega'\,,\,
\omega\rangle_{\phi_{-b}}
=\langle \hat{d}_{g,\fkh}\omega'\,,\, \hat{B}^{\phi_{b}}_{\fkh}\omega\rangle_{\phi_{-b}}
$$
where all the identity make sense for the closed operators  by the
first result.
But the left-hand side implies that
$v=\hat{B}^{\phi_{b}}_{\fkh}\omega\in L^{2}(X;\hat{F}_{g})$ satisfies 
$$
\forall \omega'\in \hat{\mathcal{D}}_{g,\nabla^{\fkf}}\,,\quad
\left|\langle \hat{d}_{g,\fkh}\omega'\,,\, v\rangle_{\phi_{-b}}\right|\leq C\|\omega'\|_{L^{2}}\,.
$$
But since $\hat{\mathcal{D}}_{g,\nabla^{\fkf}}$ is a core for $\hat{d}_{g,\fkh}$
and the $\phi_{-b}$ right-adjoint of $\hat{d}_{g,\fkh}$ is
$\hat{d}^{\phi_{b}}_{g,\fkh}$ we deduce $v=\hat{B}^{\phi_{b}}_{\fkh}\in
D(\hat{d}^{\phi_{b}}_{g,\fkh})$ and 
$$
\hat{B}^{\phi_{b}}\hat{d}^{\phi_{b}}_{g,\fkh}\omega=\hat{d}^{\phi_{b}}_{g,\fkh}\hat{B}^{\phi_{b}}_{\fkh}\omega
$$
\end{proof}
\begin{proof}[Proof of Proposition~\ref{pr:commut}]
We work with $\hat{B}^{\phi_{b}}_{\fkh}$\,,
$\hat{d}_{g,\fkh}$ and
$\hat{d}^{\phi_{b}}_{g,\fkh}$\,.
For $s_{0}\in L^{2}(X;F)$\,, Corollary~\ref{cor:semid} says that
$\hat{d}_{g,\fkh}e^{-t\hat{B}^{\phi_{b}}_{\fkh}}s_{0}=\hat{d}_{g,\fkh}s_{t}$ is a
$\mathcal{C}^{1}(]0,+\infty[; L^{2}(X;F))$ function with 
$$
\forall t>0\,,
\frac{d}{dt}[\hat{d}_{g,\fkh}e^{-t\hat{B}^{\phi_{b}}_{\fkh}}s_{0}]=
\hat{d}_{g,\fkh}\hat{B}^{\phi_{b}}_{\fkh}e^{-t\hat{B}^{\phi_{b}}_{\fkh}}s_{0}
=\hat{d}_{g,\fkh}\hat{B}^{\phi_{b}}_{\fkh}
s_{t}\,.
$$
For $\omega\in e^{\frac{\lambda_{0}}{-b}}\sigma_{-b}\hat{\mathcal{D}}'_{g,\nabla^{\fkf'}}$ and $t>0$\,,
Lemma~\ref{le:commut} allows to write
$$
\langle \omega'\,,\,
\hat{d}_{g,\fkh}\hat{B}^{\phi_{b}}_{\fkh}s_{t}\rangle_{\phi_{-b}}=
\langle \hat{B}^{\phi_{-b}}\hat{d}^{\phi_{-b}}_{g,\fkh}\omega'\,,\,
 s_{t}\rangle_{\phi_{-b}}
=\langle
\hat{d}^{\phi_{-b}}_{g,\fkh}\hat{B}^{\phi_{-b}}_{\fkh}\omega'\,,\,
s_{t}\rangle_{\phi_{b}}
=\langle \hat{B}^{\phi_{-b}}_{\fkh}\omega'\,,\, \hat{d}_{g,\fkh}s_{t}\rangle_{\phi_{-b}}\,.
$$
This implies in particular 
$$
\forall \omega'\in e^{\frac{\lambda_{0}}{-b}}\sigma_{-b}\hat{\mathcal{D}}'_{g,\nabla^{\fkf'}}\,,\quad
\left| \langle
  \hat{B}^{\phi_{-b}}_{\fkh}\omega'\,,\,
  \hat{d}_{g,\fkh}s_{t}\rangle\right|\leq C_{t}\|\omega'\|_{L^{2}}\,.
$$
But $\hat{\mathcal{D}'}_{g,\nabla^{\fkf'}}$ is dense in
$\mathcal{C}_{0,g}(\hat{F}'_{g})$ while
$e^{\frac{\lambda_{0}}{-b}}\sigma_{-b}$ is continuous from
$\mathcal{C}_{0,g}(\hat{F'}_{g})$ to $\mathcal{C}_{0,g}(\hat{F}_{g})$
and $\mathcal{C}_{0,g}(\hat{F}_{g})$ is continuously and densely
embedded in $D(\hat{B}^{\phi_{-b}}_{\fkh})$\,. Since
$\hat{B}^{\phi_{b}}_{\fkh}$ is the $\phi_{-b}$ right-adjoint of
$\hat{B}^{\phi_{-b}}_{\fkh}$ we deduce
$$
\forall t>0\,,\quad \hat{d}_{g,\fkh}s_{t}\in
D(\hat{B}^{\phi_{b}}_{\fkh})\quad \text{and}\quad
\frac{d}{dt}\hat{d}_{g,\fkh}s_{t}=\hat{B}^{\phi_{b}}_{\fkh}\hat{d}_{g,\fkh}s_{t}\,.
$$
Since for $t_{0}>0$\,, $\hat{d}_{g,\fkh}s_{t_{0}}\in
L^{2}(X;F)$ by Corollary~\ref{cor:semid},  this implies
$$
\forall t>t_{0}>0\,,\quad
\hat{d}_{g,\fkh}s_{t}=e^{-(t-t_{0})\hat{B}^{\phi_{b}}_{\fkh}}\hat{d}_{g,\fkh}s_{t_{0}}\,,
$$
or
$$
\forall t>t_{0}>0\,,\quad
\hat{d}_{g,\fkh}e^{-t\hat{B}^{\phi_{b}}_{\fkh}}s_{0}=e^{-(t-t_{0})\hat{B}^{\phi_{b}}_{\fkh}}\hat{d}_{g,\fkh}(e^{-t_{0}\hat{B}^{\phi_{b}}_{\fkh}}s_{{0}})\,.
$$
Let us assume now $s_{0}\in D(\hat{d}_{g,\fkh})$  and
take the limit as $t_{0}\to 0^{+}$\,. From $\lim_{t_{0}\to
  0^{+}}e^{-t_{0}\hat{B}^{\phi_{b}}_{\fkh}}s_{0}=s_{0}$
in $L^{2}(X;F)$\,, Corollary~\ref{cor:semid} yields
$$
\lim_{t_{0}\to 0^{+}}e^{-(t-t_{0})\hat{B}^{\phi_{b}}_{\fkh}}\hat{d}_{g,\fkh}(e^{-t_{0}\hat{B}^{\phi_{b}}_{\fkh}}s_{{0}})=e^{-t\hat{B}^{\phi_{b}}_{\fkh}}
\hat{d}_{g,\fkh}s_{0}
$$
We have proved that for $s_{0}\in
D(\hat{d}_{g,\fkh})$ 
$$
\forall t>0\,, \quad \hat{d}_{g,\fkh}
e^{-t\hat{B}^{\phi_{b}}_{\fkh}}s_{0}=
e^{-t\hat{B}^{\phi_{b}}_{\fkh}}
\hat{d}_{g,\fkh}s_{0}\,,
$$
while the result is obvious for $t=0$\,.\\
Because $e^{-t\hat{B}^{\phi_{b}}_{\fkh}}\hat{d}^{\phi_{b}}_{g,\fkh}$
and $\hat{d}^{\phi_{b}}_{g,\fkh}e^{-t\hat{B}^{\phi_{b}}_{\fkh}}$ are
the $\phi_{b}$ left-adjoints respectively of
$\hat{d}_{g,\fkh}e^{-t\hat{B}^{\phi_{-b}}_{\fkh}}$ and
$e^{-t\hat{B}^{\phi_{-b}}_{\fkh}}\hat{d}_{g,\fkh}$\,, the same result
holds when $\hat{d}_{g,\fkh}$ is replaced with
$\hat{d}^{\phi_{b}}_{g,\fkh}$\,.\\

Finally the commutation with the resolvent are proved after writing
for $\Real z\leq -C$ with $C>0$ large enough
$$
(\hat{B}^{\phi_{b}}_{\fkh}-z)^{-1}=\int_{0}^{+\infty}e^{-t(\hat{B}^{\phi_{b}}_{\fkh}-z)}~dt\,,
$$
and by analytic continuation  for the extension to any $z\in
\cz\setminus \mathrm{Spec}(\hat{B}^{\phi_{b}}_{\fkh})$\,.
\end{proof}

\subsection{$PT$-symmetry}
\label{sec:PTsymm}
While working with the metric $\pi_{X}^{*}(g^{\Lambda T^{*}Q}\otimes
g^{\Lambda TQ}\otimes g^{\fkf})$ without the weight $\langle
p\rangle_{q}^{N_{V}-N_{H}}$\,, Bismut in \cite{Bis05} establishes an
important property for the spectral analysis of the hypoelliptic
Laplacian: the $PT$-symmetry or more precisely a formal version of it.
Let us first recall Bismut's construction in the smooth case
and then we will show how
this can be adapted easily to our case  with interface or
boundary conditions. 
Remember
$$
\phi_{b}=\begin{pmatrix}
  g^{TQ}&-b\mathrm{Id}\\
b\mathrm{Id}&0
\end{pmatrix}
=
\begin{pmatrix}
  g&-b\mathrm{Id}\\
b\mathrm{Id}&0
\end{pmatrix}
\quad,\quad
\phi_{b}^{-1}=
\begin{pmatrix}
  0&\frac{\mathrm{Id}}{b}\\
-\frac{\mathrm{Id}}{b}&\frac{1}{b^{2}}g^{T^{*}Q}
\end{pmatrix}
=
\begin{pmatrix}
  0&\frac{\mathrm{Id}}{b}\\
-\frac{\mathrm{Id}}{b}&\frac{1}{b^{2}}g^{-1}
\end{pmatrix}\,.
$$
The tangent bundle is endowed with the new metric
\begin{equation}
  \label{eq:fkgb}
\mathfrak{g}_{b}^{TX}=
\begin{pmatrix}
  g^{TQ}&b\mathrm{Id}\\
 b\mathrm{Id}&2b^{2}g^{T^{*}Q}
\end{pmatrix}
=
\begin{pmatrix}
  g&b\mathrm{Id}\\
 b\mathrm{Id}&2b^{2}g^{-1}
\end{pmatrix}
\end{equation}
or by calling $\mathfrak{p}U$ the vertical component of $U$
$$
\langle U,U\rangle_{\mathfrak{g}_{b}^{TX}}=
\langle \pi_{X,*}U,\pi_{X,*}U\rangle_{g^{TQ}}
+2b\langle \pi_{X,*}U\,,\, \mathfrak{p}U\rangle_{TQ,T^{*}Q}
+
2b^{2}\langle \mathfrak{p}U\,,\, \mathfrak{p}U\rangle_{g^{T^{*}Q}}\,.
$$
Meanwhile the dual metric is given by 
\begin{equation}
  \label{eq:fkgbm1}
\mathfrak{g}_{b}^{T^{*}Q}=
\begin{pmatrix}
  2g^{-1}&-\frac{\mathrm{Id}}{b}\\
-\frac{\mathrm{Id}}{b}&\frac{g}{b^{2}} 
\end{pmatrix}\,.
\end{equation}
This defines a new metric
$\mathfrak{g}_{b}^{\fkf}=\mathfrak{g}_{b}^{\Lambda T^{*}X}\otimes \pi_{X}^{*}(g^{\fkf})$ which
is uniformly equivalent to $\pi_{X}^{*}(g^{\Lambda T^{*}Q}\otimes
g^{\Lambda TQ}\otimes g^{\fkf})$\,. In particular the sesquilinear
form 
\begin{equation}
  \label{eq:dualfkgb}
\langle s\,\, s\rangle_{\mathfrak{g}_{b}^{\fkf}}=\int_{X}
\mathfrak{g}_{b}^{\fkf}(s\,,\,s')~dv_{X}(x)
\end{equation}
is continuous on $L^{2}(X;F, \pi_{X}^{*}(g^{\Lambda T^{*}Q}\otimes
g^{\Lambda TQ}\otimes g^{\fkf})$ and it is neither continuous nor
everywhere defined on $L^{2}(X;F,g^{F})$\,. However it is well defined
on $\mathcal{C}^{\infty}_{0}(X;F)$  which is dense in all the
considered $L^{2}$-spaces, in the smooth case.
An additional modification is used by introducing the maps $F_{b}:TX\to TX$ and its transpose
$\tilde{F}_{b}:T^{*}X\to T^{*}X$ given by the matrix
\begin{equation}
  \label{eq:FbtFb}
F_{b}=
\begin{pmatrix}
  \mathrm{Id}_{TQ}&2bg^{-1}\\
0&-\mathrm{Id}_{T^{*}Q}
\end{pmatrix}
\quad,\quad \tilde{F}_{b}=
\begin{pmatrix}
  \mathrm{Id}_{T^{*}Q}& 0\\
2bg^{-1}&-\mathrm{Id}_{TQ}
\end{pmatrix}\,.
\end{equation}
We need also the maps $H:TX\to TX$ and its transpose
$\tilde{H}:T^{*}X\to T^{*}X$ and $r:X\to X$ given by
$$
H=
\begin{pmatrix}
  \mathrm{Id}_{TQ}&0\\
0&-\mathrm{Id}_{T^{*}Q}
\end{pmatrix}
\quad,\quad
\tilde{H}=
\begin{pmatrix}
  \mathrm{Id}_{T^{*}Q}&0\\
0&-\mathrm{Id}_{TQ}
\end{pmatrix}\quad,\quad r(q,p)=(q,-p)\,.
$$
After tensorization of $\tilde{F}_{b}$\,, $\tilde{H}$\,, the map
$u_{b}:\mathcal{C}^{\infty}_{0}(X;F)\to \mathcal{C}^{\infty}_{0}(X;F)$
is defined by
\begin{equation}
  \label{eq:ub}
u_{b}s(q,p)=(\tilde{F}_{b}s)(q,-p)=r_{*}\tilde{H}\tilde{F}_{b}s\,.
\end{equation}
Then the hermitian form $\mathfrak{H}_{b}$ is defined on
$\mathcal{C}^{\infty}_{0}(X;F)$ by
\begin{equation}
  \label{eq:fkHb}
\langle s\,,\,s'\rangle_{\mathfrak{H}_{b}}=\langle u_{b}s\,,\,s'\rangle_{\mathfrak{g}_{b}^{\fkf}}\,.
\end{equation}
From
$r_{*}H(\mathfrak{g}_{b}^{\fkf})^{-1}=(\mathfrak{g}_{b}^{\fkf})^{-1}r_{*}\tilde{H}$\,,
and the relation
$$
H(\mathfrak{g}_{b}^{\fkf})^{-1}\tilde{F}_{b}=
\begin{pmatrix}
  \mathrm{Id}&0\\
0&-\mathrm{Id}
\end{pmatrix}
\begin{pmatrix}
  2g^{-1}&-\frac{\mathrm{Id}}{b}
\\-\frac{\mathrm{Id}}{b}&\frac{g}{b^{2}}
\end{pmatrix}
\begin{pmatrix}
\mathrm{Id}&0\\
2bg^{-1}&-\mathrm{Id}
\end{pmatrix}
=
\begin{pmatrix}
  0&\frac{\mathrm{Id}}{b}\\
-\frac{\mathrm{Id}}{b}&\frac{g}{b^{2}}
\end{pmatrix}
=\phi_{b}^{-1}
$$
we deduce
$$
\langle s\,,\, s'\rangle_{\mathfrak{H}_{b}}=\langle s\,,r_{*}s'\rangle_{\phi_{b}}\,.
$$
Finally since $d_{\fkh}r_{*}=r_{*}d_{\fkh}$\,, the
$\mathfrak{H}_{b}$ formal adjoint of $d_{\fkh}$ is
$d^{\phi_{b}}_{\fkh}$ and since $\mathfrak{H}_{b}$ is hermitian, 
$(d_{\fkh}+d_{\fkh}^{\phi_{b}})$ and its square $B^{\phi_{b}}_{\fkh}$
are formally self-ajoint for $\mathfrak{H}_{b}$\,.\\
The piecewise $\mathcal{C}^{\infty}$ and continuous version for the
vector bundle $\hat{F}_{g}$ for the metric $\hat{g}^{TQ}$ is defined
as follows.
\begin{definition}
\label{de:hatPT} The  map $\hat{u}_{b}$ and hermitian form $\hat{\mathfrak{H}}_{g}$ on
$\mathcal{C}_{0,g}(\hat{F}_{g})$ are given by the same formula as
\eqref{eq:ub}\eqref{eq:fkHb} after replacing $g^{TQ}$ by
$\hat{g}^{TQ}=1_{\overline{Q}_{-}}g_{-}^{TQ}+1_{Q_{+}}g_{+}^{TQ}$ in \eqref{eq:fkgb}\eqref{eq:fkgbm1}\eqref{eq:dualfkgb}\eqref{eq:FbtFb}.
\end{definition}
By construction $\hat{u}_{b}:\mathcal{C}_{0,b}(\hat{F}_{g})\to
\mathcal{C}_{0,b}(\hat{F}_{b})$ and it preserves the parity with respect
to $\Sigma_{\nu}$\,, while $\hat{\mathfrak{H}}_{g}$ well defined on
$\mathcal{C}_{0,g}(\hat{F}_{b})$  and the direct sum
$\mathcal{C}_{0,g,ev}(\hat{F}_{g})\oplus
\mathcal{C}_{0,g,odd}(\hat{F}_{g})$ is $\hat{\mathfrak{H}}_{g}$
orthogonal.
We deduce the following proposition.
\begin{proposition}
  \label{pr:PTsym} The hermitian form $\hat{\mathfrak{H}}_{g}$ is
  well-defined and continuous on $\langle
  p\rangle_{q}^{-d/2}L^{2}(X;F)$\,.\\
The identity  
\begin{equation}
  \label{eq:almostPT}
\langle \hat{B}^{\phi_{b}}_{\fkh}s\,,\,
s'\rangle_{\hat{\mathfrak{H}}_{b}}=
\langle s\,,\, \hat{B}^{\phi_{b}}_{\fkh}s'\rangle_{\hat{\mathfrak{H}}_{b}}
\end{equation}
holds for all 
$s,s'\in \langle
p\rangle_{q}^{-d/2}D(\hat{B}^{\phi_{b}}_{\fkh})\subset
D(\hat{B}^{\phi_{b}}_{\fkh})\cap \langle p\rangle_{q}^{-d/2}L^{2}(X;F)$\,.\\
As a consequence $\mathrm{Spec}(\hat{B}^{\phi_{b}}_{\fkh})$ is symmetric
with respect to the real axis.\\
Finally the same results hold for
$(\overline{B}^{\phi_{b}}_{\fkh},D(\overline{B}^{\phi_{b}}_{\fkh}))$
after using the restriction $\mathfrak{H}_{b}$ of
$\hat{\mathfrak{H}}_{b}$ to $F|_{X_{-}}$\,. 
\end{proposition}
\begin{proof}
  The map $\langle p\rangle_{q}^{-d/2}$ is continuous from
  $L^{2}(X;F,g^{F})$ to $L^{2}(X;F,\pi_{X}^{*}(g^{\Lambda T^{*}Q}\otimes
  g^{\Lambda TQ}\otimes g^{\fkf}))$ while $\mathfrak{H}_{b}$ is
  continuous on $L^{2}(X;F,\pi_{X}^{*}(g^{\Lambda T^{*}Q}\otimes
  g^{\Lambda TQ}\otimes g^{\fkf})$\,. All those spaces contain
  $\mathcal{C}_{0,g}(X;\hat{F}_{g})$ as a dense subspace.\\
The set $\mathcal{C}_{0,g}(X;\hat{F}_{g})$ is dense in $\langle
p\rangle^{-d/2}D(\hat{B}^{\phi_{b}}_{\fkh})$\,. The continuous
embedding of  $\langle
p\rangle_{q}^{-d/2}D(\hat{B}^{\phi_{b}}_{\fkh})$ comes from the fact
that $\langle p\rangle^{d/2}\hat{B}^{\phi_{b}}_{\fkh}\langle
p\rangle_{q}^{-d/2}$ is a relatively bounded perturbation of
$(\hat{B}^{\phi_{b}}_{\fkh},D(\hat{B}^{\phi_{b}}_{\fkh})$ with
infinitesimal bound\,. The identity \eqref{eq:almostPT} is valid for
$s,s'\in \mathcal{C}_{0,g}(\hat{F}_{g})$ and by density extends to
$s,s'\in \langle p\rangle_{q}^{-d/2}D(\hat{B}^{\phi_{b}}_{\fkh})$\,.\\
Let us consider the spectral problem. We know that
$\mathrm{Spec}(\hat{B}^{\phi_{b}}_{\fkh})$ is discrete. Additionally
because $(C_{b}+\hat{B}^{\phi_{b}}_{\fkh})$ is maximal accretive we
know that for all $z\in \cz$\,,
$(\hat{B}^{\phi_{b}}_{\fkh}-z)(1+C_{b}+\hat{B}^{\phi_{b}}_{\fkh})^{-1}$
is a Fredholm operator with index $0$ and $\lambda\in
\mathrm{Spec}(\hat{B}^{\phi_{b}}_{\fkh})$ iff
$\ker(\hat{B}^{\phi_{b}}-\lambda)\neq \left\{0\right\}$\,.
 When $\lambda$
is an eigenvalue of $\hat{B}^{\phi_{b}}_{\fkh}$ there exists
$s_{\lambda}\in D(\hat{B}^{\phi_{b}}_{\fkh})$ such that
$$
(C+\hat{B}^{\phi_{b}}_{\fkh})s_{\lambda}=(C+\lambda)s_{\lambda}\,.
$$
By chosing $C>0$ such that $(C+\hat{B}^{\phi_{b}}_{\fkh})$ is
invertible,  we obtain $s_{\lambda}\in
(C+\hat{B}^{\phi_{b}}_{\fkh})^{-n}L^{2}(X;F)\subset \langle
p\rangle_{q}^{-n}D(\hat{B}^{\phi_{b}}_{\fkh})$ for any $n\in\nz$ by
Lemma~\ref{le:weightB}. Hence we can use \eqref{eq:almostPT} which
implies
$$
\forall s'\in \mathcal{C}_{0,g}(\hat{F}_{g})\,,\quad 
\langle s_{\lambda}\,, (\hat{B}^{\phi_{b}}-\overline{\lambda})s'\rangle_{\hat{\mathfrak{H}_{b}}}=0\,.
$$
But $\mathfrak{H}_{g}$ is also continuous on  $\langle
p\rangle^{-d}L^{2}(X;F)\times L^{2}(X;F)$ while $s_{\lambda}\in
\langle p\rangle^{-d}L^{2}(X;F)$ and the density of
$\mathcal{C}_{0,g}(\hat{F}_{g})$ in
$D(\hat{B}^{\phi_{b}}_{\fkh})$\,. Finally as a non zero element of
$\langle p\rangle^{-d}L^{2}(X;F)$ there exists $s''_{\lambda}\in
L^{2}(X;F)$ such that $\langle s_{\lambda}\,,\,
s''_{\lambda}\rangle_{\hat{\mathfrak{H}}_{b}}\neq 0$\,. Therefore
$(\hat{B}^{\phi_{b}}_{\fkh}-\overline{\lambda})$ is not onto
 and
$\overline{\lambda}\in \mathrm{Spec}(\hat{B}^{\phi_{b}}_{\fkh})$\,.
\end{proof}
\begin{remark}
The symmetry with respect to the real axis of
$\mathrm{Spec}(\hat{B}^{\phi_{b}}_{\fkh})$ and
$\mathrm{Spec}{\overline{B}^{\phi_{b}}_{\fkh}}$ are not the only
issues of Proposition~\ref{pr:PTsym}. Actually the relation
\eqref{eq:almostPT} is a crucial point while studying the spectrum in
a neighborhood of $0$ in various
asymptotic regimes e.g. $b\to 0^{+}$ (see 
\cite{She}\cite{BiLe}\cite{HHS}). In particular those asymptotic
regimes correspond to cases where the spectral spaces concentrate
asymptotically to the kernel of the  scalar vertical harmonic
oscillator hamiltonian, multiples of a scaled gaussian function in
$p$\,, on which the restricted hermitian form $\mathfrak{H}_{g}$ is
positive definite. This property with \eqref{eq:almostPT} helps to reduce the
asymptotic spectral analysis of the hypoelliptic Laplacian on
$X=T^{*}Q$ to the the more standard asymptotic spectral analysis of
the Hodge or Witten Laplacian on $Q$\,. We refer the reader to
\cite{Nie14}-Proposition~15.2 for a short abstract version of those arguments.
\end{remark}

\noindent\textbf{Acknowledgements:} The authors thank Michael Hitrik for early
discussions about this work.\\
They aknowledge the support of the
following french
ANR-projects:
\begin{itemize}
\item Francis Nier:  QUAMPROCS ANR-19-CE40-0010;
\item  Shu Shen: ADYCT, ANR-20-CE40-0017\,.
\end{itemize}
Although the COVID crisis restrained collaborative works,
they initially obtained the support of the CNRS IEA-project 00137. 

\end{document}